\documentclass[11pt]{amsart}

\usepackage[utf8]{inputenc}
\usepackage[english]{babel}
\usepackage{amsmath}
\usepackage{amssymb}
\usepackage[pdftex]{graphicx}
\usepackage{amsthm}
\usepackage{enumerate}
\usepackage[numbers,sort&compress]{natbib}
\usepackage{amsthm}
\usepackage[pdfborder={0 0 0}]{hyperref}
\usepackage{fge}
\usepackage{bm}
\usepackage{turnstile}
\usepackage{esvect}
\usepackage{color}
\usepackage{setspace} 
\usepackage{comment}
\usepackage[textsize=footnotesize,color=yellow, bordercolor=white]{todonotes}

\newcommand{\comm}[1]{}

\DeclareRobustCommand{\SkipTocEntry}[5]{}

\newtheorem{satz}{Satz}[section]
\newtheorem{thm}[satz]{Theorem}

\newtheorem{lemma}[satz]{Lemma}
\newcounter{cl}[satz]
\newtheorem{claim}[cl]{Claim}
\newcounter{subcl}[cl]
\newtheorem{subclaim}[subcl]{Subclaim}
\newtheorem{definition}[satz]{Definition}
\newtheorem{cor}[satz]{Corollary}

\numberwithin{subsection}{section}

\newcounter{anhang}
\newcommand{\thistheoremname}{}
\newtheorem*{genericthm}{\thistheoremname}
\newenvironment{namedthm}[1]
  {\renewcommand{\thistheoremname}{#1}%
   \begin{genericthm}}
  {\end{genericthm}}

\newcommand{\st}{\mid}

\newcommand{\card}[1]{|#1|}

\newcommand{\ZFC}{\ensuremath{\operatorname{ZFC}} }

\newcommand{\CH}{\ensuremath{\operatorname{CH}} }
\newcommand{\OR}{\ensuremath{\operatorname{Ord}} }
\newcommand{\Ord}{\ensuremath{\operatorname{Ord}} }
\newcommand{\HOD}{\ensuremath{\operatorname{HOD}} }
\newcommand{\OD}{\ensuremath{\operatorname{OD}} }

\newcommand{\TC}{\ensuremath{\operatorname{TC}} }
\newcommand{\AC}{\ensuremath{\operatorname{AC}} }
\newcommand{\AD}{\ensuremath{\operatorname{AD}} }
\newcommand{\crit}{\ensuremath{\operatorname{crit}} }

\newcommand{\Det}{\ensuremath{\operatorname{Det}} }

\newcommand{\Col}{\ensuremath{\operatorname{Col}} }
\newcommand{\lh}{\ensuremath{\operatorname{lh}} }
\newcommand{\ran}{\ensuremath{\operatorname{ran}} }
\newcommand{\pred}{\ensuremath{\operatorname{pred}} }
\newcommand{\WO}{\ensuremath{\operatorname{WO}} }

\newcommand{\Pot}{\mathcal{P}}

\newcommand{\defeq}{\overset{def}{=}}

\newcommand{\Ult}{\ensuremath{\operatorname{Ult}}}

\newcommand{\forces}[2]{\Vdash^{#1}_{#2}}

\newcommand{\pwimg}{\, "}

\newcommand{\cf}{\ensuremath{\operatorname{cf}}}

\newcommand{\dom}{\ensuremath{\operatorname{dom}}}

\newcommand{\Lpn}{
  Lp^n}

\newcommand{\Lpoddn}{
  Lp^{2n-1}}

\newcommand{\BS}{{}^\omega\omega}

\newcommand{\PI}{\boldsymbol\Pi}
\newcommand{\SIGMA}{\boldsymbol\Sigma}

\newcommand{\Q}{\mathcal{Q}}
\newcommand{\M}{\mathcal{M}}
\newcommand{\N}{\mathcal{N}}
\newcommand{\T}{\mathcal{T}}
\newcommand{\U}{\mathcal{U}}
\newcommand{\R}{\mathbb{R}}

\newcommand{\ODdet}{\mathrm{OD}\text{-determinacy}}

\newcommand{\game}{\Game}

\theoremstyle{definition}

\theoremstyle{remark}

\newtheorem*{remark}{Remark}

\makeindex

\title[Mice from optimal determinacy hypotheses]{Mice with finitely many Woodin cardinals from optimal
  determinacy hypotheses}

\author{Sandra Müller}
\thanks{First author formerly known as Sandra Uhlenbrock. The first
    author gratefully acknowledges her non-material and financial
    support from the ``Deutsche Telekom Stiftung''. The material in
    this paper is part of the first author's Ph.D. thesis written
    under the supervision of the second author.} 
\author{Ralf Schindler}
\author{W. Hugh Woodin}

\begin{document}

\pagenumbering{arabic}

\maketitle


\section*{Abstract}

We prove the following result which is due to the third author.

Let $n \geq 1$. If $\PI^1_n$ determinacy and $\Pi^1_{n+1}$ determinacy both 
hold true and there is no $\SIGMA^1_{n+2}$-definable $\omega_1$-sequence of 
pairwise distinct reals, then $M_n^\#$ exists and is $\omega_1$-iterable.
The proof yields that
$\PI^1_{n+1}$ determinacy implies that $M_n^\#(x)$ exists and is
$\omega_1$-iterable for all reals $x$.

A consequence is the Determinacy Transfer
  Theorem for arbitrary 
$n \geq 1$, namely the statement that $\PI^1_{n+1}$ determinacy implies 
  \emph{$\game^{(n)}(<\omega^2 - \PI^1_1)$} determinacy. 


\tableofcontents 

\section*{Preface}

In 1953 David Gale and Frank M. Stewart developed a basic theory of
infinite games in \cite{GS53}. For every set of reals, that means set
of sequences of natural numbers, $A$ they considered a two-player game
$G(A)$ of length $\omega$, where player $\mathrm{I}$ and player
$\mathrm{II}$ alternately play natural numbers. They defined that
player $\mathrm{I}$ wins the game $G(A)$ if and only if the sequence
of natural numbers produced during a run of the game $G(A)$ is
contained in $A$ and otherwise player $\mathrm{II}$ wins. Moreover the
game $G(A)$ or the set $A$ itself is called determined if and only if
one of the two players has a winning strategy, that means a method by
which they can win, no matter what their opponent does, in the game
described above.

Already in \cite{GS53} the authors were able to prove that every open
and every closed set of reals is determined using $\ZFC$. But
furthermore they proved that determinacy for all sets of reals
contradicts the Axiom of Choice. This leads to the natural question of
what the picture looks like for definable sets of reals which are more
complicated than open and closed sets. After some partial results by
Philip Wolfe in \cite{Wo55} and Morton Davis in \cite{Da64}, Donald
A. Martin was finally able to prove in \cite{Ma75} from $\ZFC$ that
every Borel set of reals is determined.

In the meantime the development of so called \emph{Large Cardinal
  Axioms} was proceeding in Set Theory. In 1930 Stanisław Ulam first
defined measurable cardinals and at the beginning of the 1960's
H. Jerome Keisler \cite{Kei62} and Dana S. Scott \cite{Sc61} found a
way of making them more useful in Set Theory by reformulating the
statements using elementary embeddings.

About the same time, other set theorists were able to prove that
\emph{Determinacy Axioms} imply \emph{regularity properties} for sets
of reals. More precisely Banach and Mazur showed that under the Axiom
of Determinacy (that means every set of reals is determined), every
set of reals has the Baire property. Mycielski and Swierczkowski
proved in \cite{MySw64} that under the same hypothesis every set of
reals is Lebesgue measurable. Furthermore Davis showed in \cite{Da64}
that under this hypothesis every set of reals has the perfect set
property. Moreover all three results also hold if the Determinacy
Axioms and regularity properties are only considered for sets of reals
in certain pointclasses. This shows that Determinacy Axioms have a
large influence on the structure of sets of reals and therefore have a
lot of applications in Set Theory.

In 1965 Robert M. Solovay was able to prove these regularity
properties for a specific pointclass, namely $\SIGMA^1_2$, assuming
the existence of a measurable cardinal instead of a Determinacy Axiom
(see for example \cite{So69} for the perfect set property). Then
finally Donald A. Martin was able to prove a direct connection between
Large Cardinals and Determinacy Axioms: he showed in 1970 that the
existence of a measurable cardinal implies determinacy for every
analytic set of reals (see \cite{Ma70}).

Eight years later Leo A. Harrington established that this result is in
some sense optimal. In \cite{Ha78} he proved that determinacy for all
analytic sets implies that $0^\#$, a countable active mouse which can
be obtained from a measurable cardinal, exists. Here a mouse is a fine
structural, iterable model. Together with Martin's result mentioned
above this yields that the two statements are in fact equivalent.

This of course motivates the question of whether a similar result can
be obtained for larger sets of reals, so especially for determinacy in
the projective hierarchy. The right large cardinal notion to consider
for these sets of reals was introduced by the third author in 1984 and
is nowadays called a Woodin cardinal. Building on this, Donald
A. Martin and John R. Steel were able to prove in \cite{MaSt89} almost
twenty years after Martin's result about analytic determinacy that,
assuming the existence of $n$ Woodin cardinals and a measurable
cardinal above them all, every $\SIGMA^1_{n+1}$-definable set of reals
is determined.

In the meantime the theory of mice was further developed. At the level
of strong cardinals it goes back to Ronald B. Jensen, Robert
M. Solovay, Tony Dodd and Ronald B. Jensen, and William J.
Mitchell. Then it was further extended to the level of Woodin
cardinals by Donald A. Martin and John R. Steel in \cite{MaSt94} and
William J. Mitchell and John R. Steel in \cite{MS94}, where some
errors were later repaired by John R. Steel, Martin
Zeman, and the second author in \cite{SchStZe02}. Moreover Ronald B. Jensen developed another
approach to the theory of mice at the level of Woodin cardinals in
\cite{Je97}.

In 1995 Itay Neeman was able to improve the result from \cite{MaSt89}
in \cite{Ne95}. He showed that the existence and
$\omega_1$-iterability of $M_n^\#$, the minimal countable active mouse
at the level of $n$ Woodin cardinals, is enough to obtain that every
$\game^n(< \omega^2 - \Pi^1_1)$-definable set of reals is
determined. Here ``$\game$'' is a quantifier which is defined by a
game. Neeman's result implies that if for all reals $x$ the premouse
$M_n^\#(x)$ exists and is $\omega_1$-iterable, then in particular
every $\SIGMA^1_{n+1}$-definable set of reals is determined. For odd
$n$ this latter result was previously known by the third author. The
converse of this latter result was announced by the third author in
the 1980's but a proof has never been published. The goal of this
paper is to finally provide a complete proof of his result.


\section*{Overview}

The purpose of this paper is to give a proof of the following theorem,
which connects inner models with Woodin cardinals and descriptive set
theory at projective levels in a direct level-wise way. This theorem
is due to the third author and announced for example in the addendum
(§5) of \cite{KW08} and in Theorem $5.3$ of \cite{Sch10}, but so far a
proof of this result has never been published.

\begin{namedthm}{Theorem \ref{cor:bfDetSharps}}
  Let $n \geq 1$ and assume $\PI^1_{n+1}$ determinacy holds. Then
  $M_n^\#(x)$ exists and is $\omega_1$-iterable for all $x \in \BS$.
\end{namedthm}

The converse of Theorem \ref{cor:bfDetSharps} also holds true and is
for odd $n$ due to the third author (unpublished) and for even $n$ due
to Itay Neeman (see \cite{Ne95}). From this we can obtain the
following corollary, which is Theorem $1.10$ in \cite{KW08} for odd
$n$. We will present a proof of the Determinacy Transfer Theorem for
the even levels $n$ as a corollary of Theorem \ref{cor:bfDetSharps} in
Section \ref{sec:Applications}, using Theorem $2.5$ in \cite{Ne95} due
to Itay Neeman.

\begin{namedthm}{Corollary \ref{cor:transferThm}}[Determinacy Transfer
  Theorem]
  Let $n \geq 1$. Then $\PI^1_{n+1}$ determinacy is equivalent to
  \emph{$\game^{(n)}(<\omega^2 - \PI^1_1)$} determinacy.
\end{namedthm}

In fact we are going to prove the following theorem which
will imply Theorem \ref{cor:bfDetSharps}.

\begin{namedthm}{Theorem \ref{thm:newWoodin1}}
  Let $n \geq 1$ and assume there is no $\SIGMA^1_{n+2}$-definable
  $\omega_1$-sequence of pairwise distinct reals. Moreover assume that
  $\PI^1_n$ determinacy and $\Pi^1_{n+1}$ determinacy hold. Then
  $M_n^\#$ exists and is $\omega_1$-iterable.
\end{namedthm}

This is a part of the following theorem.

\begin{namedthm}{Theorem \ref{cor:newWoodin1}}
  Let $n \geq 1$ and assume there is no $\SIGMA^1_{n+2}$-definable
  $\omega_1$-sequence of pairwise distinct reals. Then the following
  are equivalent.
\begin{enumerate}[(1)]
\item $\PI^1_n$ determinacy and $\Pi^1_{n+1}$ determinacy,
\item for all $x \in \BS$, $M_{n-1}^\#(x)$ exists and is
  $\omega_1$-iterable, and $M_n^\#$ exists and is $\omega_1$-iterable,
\item $M_n^\#$ exists and is $\omega_1$-iterable.
\end{enumerate}
\end{namedthm}

Here the direction $(3)$ implies $(1)$ follows from Theorem $2.14$ in
\cite{Ne02} and is due to the third author for odd $n$ (unpublished)
and due to Itay Neeman for even $n$. Moreover the direction $(2)$
implies $(1)$ and the equivalence of $(2)$ and $(3)$ as proven in
\cite{Ne02} do not need the background hypothesis that there is no
$\SIGMA^1_{n+2}$-definable $\omega_1$-sequence of pairwise distinct
reals.

Furthermore we get the following relativized version of Theorem
\ref{cor:newWoodin1}.

\begin{namedthm}{Corollary \ref{cor:newWoodin1rel}}
  Let $n \geq 1$. Then the following are equivalent.
\begin{enumerate}[(1)]
\item $\PI^1_{n+1}$ determinacy, and
\item for all $x \in \BS$, $M_n^\#(x)$ exists and is $\omega_1$-iterable.
\end{enumerate}
\end{namedthm}

This gives that $\PI^1_{n+1}$ determinacy is an optimal hypothesis for 
proving the existence and $\omega_1$-iterability of $M_n^\#(x)$ for all
$x \in \BS$. 

\begin{remark}
  In contrast to the statement of Theorem $1.4$ in \cite{Ne02} it is
  open whether $\PI^1_n$ determinacy and $\Pi^1_{n+1}$ determinacy
  alone imply the existence of an $\omega_1$-iterable $M_n^\#$ for
  $n > 1$ (see also Section \ref{sec:openproblemsproj}). Whenever we
  are citing \cite{Ne02} in this paper we make no use of any
  consequence of this result stated there.
\end{remark}

\addtocontents{toc}{\SkipTocEntry}
\subsection*{Outline}

This paper is organized as follows. In Section \ref{ch:introproj} we
give a short introduction to determinacy and inner model theory. In
particular we state some known results concerning the connection of
determinacy for certain sets of reals and the existence of mice with
large cardinals.

In Section \ref{ch:WCFromDet} we will construct a proper class inner
model with $n$ Woodin cardinals from $\PI^1_n$ determinacy and
$\Pi^1_{n+1}$ determinacy. For that purpose we will prove in Lemma
\ref{lem:KS} from the same determinacy hypothesis that for a cone of
reals $x$ the premouse $M_{n-1}(x) | \delta_x$ is a model of $\ODdet$,
where $\delta_x$ denotes the least Woodin cardinal in
$M_{n-1}(x)$. This generalizes a theorem of Kechris and Solovay to the
context of mice (see Theorem $3.1$ in \cite{KS85}).

Afterwards we will prove in Section \ref{ch:it} that, assuming
$\PI^1_{n+1}$ determinacy, there is in fact an $\omega_1$-iterable
model which has $n$ Woodin cardinals. More precisely, we will prove
under this hypothesis that $M_n^\#(x)$ exists and is
$\omega_1$-iterable for all reals $x$. The proof of this result
divides into different steps. In Sections \ref{sec:prep} and
\ref{sec:corrnsuit} we will introduce the concept of $n$-suitable
premice and show that if $n$ is odd, using the results in Section
\ref{ch:WCFromDet}, such $n$-suitable premice exist assuming $\PI^1_n$
determinacy and $\Pi^1_{n+1}$ determinacy.  The rest of Section
\ref{ch:it} will also be divided into different cases depending if $n$
is even or odd.

In Section \ref{sec:assumption} we will show that $\PI^1_{n+1}$
determinacy already implies that every $\SIGMA^1_{n+2}$-definable
sequence of pairwise distinct reals is countable. The proof of Theorem
\ref{cor:bfDetSharps} is organized as a simultaneous induction for odd
and even levels $n$. We will show in Section \ref{sec:odd} that the
hypothesis that there is no $\SIGMA^1_{2n+1}$-definable
$\omega_1$-sequence of pairwise distinct reals in addition to
$\PI^1_{2n-1}$ determinacy and $\Pi^1_{2n}$ determinacy suffices to
prove that $M_{2n-1}^\#$ exists and is $\omega_1$-iterable. This
finishes the proof of Theorem \ref{cor:bfDetSharps} for the odd levels
$n$ of the inductive argument. In Section \ref{sec:even} we will
finally prove the analogous result for even $n$, that means we will
show that if every $\SIGMA^1_{2n+2}$-definable sequence of pairwise
distinct reals is countable and $\PI^1_{2n}$ determinacy and
$\Pi^1_{2n+1}$ determinacy hold, then $M_{2n}^\#$ exists and is
$\omega_1$-iterable. The proof is different for the odd and even
levels of the projective hierarchy because of the periodicity in terms
of uniformization and correctness.

We close this paper with proving the Determinacy Transfer Theorem for
even $n$ as an application of Theorem \ref{cor:bfDetSharps} and
mentioning related open questions in Section \ref{ch:conclusion}.

Finally we would like to thank the referee for carefully reading this
paper and making several helpful comments and suggestions.

\section{Introduction}
\label{ch:introproj}

In this section we will introduce some relevant notions such as games
and mice and their basic properties. In particular we will summarize
some known results about the connection between large cardinals and
the determinacy of certain games. Then we will have a closer look at
mice with finitely many Woodin cardinals and introduce the premouse
$M_n^\#$.

\subsection{Games and Determinacy}
\label{sec:GamesDet}

 Throughout this paper we will consider games in the sense of
 \cite{GS53} if not specified otherwise. We will always identifiy
 $^\omega 2$, $\BS$, and $\R$ with each other, so that we can define
 Gale-Stewart games as follows. 

 \begin{definition}[Gale, Stewart] \label{def:GSgame} Let
   $A \subseteq \mathbb{R}$. By $G(A)$\index{Game@$G(A)$} we denote
   the following game.

   \begin{figure}[h]
     \centering
     \begin{tabular}{c|cccccc}
       $\mathrm{I}$ & $i_0$ & & $i_2$ & & $\dots$ & \\ \hline
       $\mathrm{I}\mathrm{I}$ & & $i_1$ & & $i_3$ & & $\dots$
     \end{tabular}
     \; \; for $i_n \in \{0,1\}$ and $n \in \omega$.
   \end{figure}

   We say player $\mathrm{I}$ wins the game $G(A)$ iff
   $(i_n)_{n < \omega} \in A$. Otherwise we say player
   $\mathrm{I}\mathrm{I}$ wins.
 \end{definition}

 \begin{definition}
   Let $A \subseteq \mathbb{R}$. We say $G(A)$ (or the set $A$ itself)
   is \emph{determined}\index{determined} iff one of the players has a
   winning strategy in the game $G(A)$ (in the obvious sense).
 \end{definition}

 Some famous results concerning the question which sets of reals are
 determined are the following. The first three theorems can be proven
 in $\ZFC$. 

 \begin{thm}[Gale, Stewart in \cite{GS53}]
   Let $A \subset \mathbb{R}$ be open or closed and assume the
   Axiom of Choice. Then $G(A)$ is determined.
 \end{thm}

 \begin{thm}[Gale, Stewart in \cite{GS53}]
   Assuming the Axiom of Choice there is a set of reals which is not
   determined.
 \end{thm}

 \begin{thm}[Martin in \cite{Ma75}]
   Let $A \subset \mathbb{R}$ be a Borel set and assume the
   Axiom of Choice. Then $G(A)$ is determined.
 \end{thm}

 To prove stronger forms of determinacy we need to assume for example
 large cardinal axioms.

 \begin{thm}[Martin in \cite{Ma70}]
   Assume $\ZFC$ and that there is a measurable cardinal. Let
   $A \subseteq \mathbb{R}$ be an analytic, i.e.
   $\SIGMA^1_1$-definable, set. Then $G(A)$ is determined.
 \end{thm}

 Determinacy in the projective hierarchy can be obtained from finitely
 many Woodin cardinals which were introduced by the third author in
 1984 and are defined as follows.

 \begin{definition}
   Let $\kappa<\delta$ be cardinals and $A \subseteq \delta$. Then
   $\kappa$ is called \emph{$A$-reflecting in $\delta$} iff for every
   $\lambda < \delta$ there exists a transitive model of set theory
   $M$ and an elementary embedding $\pi: V \rightarrow M$ with
   critical point $\kappa$, such that $\pi(\kappa) > \lambda$ and
   \[ \pi(A) \cap \lambda = A \cap \lambda. \]
  \end{definition}

  \begin{definition}
    A cardinal $\delta$ is called a \emph{Woodin
      cardinal}\index{Woodin cardinal} iff for all
    $A \subseteq \delta$ there is a cardinal $\kappa < \delta$ which
    is $A$-reflecting in $\delta$.
  \end{definition}

 \begin{thm}[Martin, Steel in \cite{MaSt89}]\label{thm:MSPD}
   Let $n \geq 1$. Assume $\ZFC$ and that there are $n$ Woodin
   cardinals with a measurable cardinal above them all. Then every
   $\SIGMA^1_{n+1}$-definable set of reals is determined.
 \end{thm}

 See for example Chapter $13$ in \cite{Sch14} or Section $5$ in
 \cite{Ne10} for modern write-ups of the proof of Theorem
 \ref{thm:MSPD}.

The existence of infinitely many Woodin cardinals with a measurable
cardinal above them all yields a much stronger form of determinacy as
shown in the following theorem due to the third author. 

 \begin{thm}[\cite{KW10}]
   Assume $\ZFC$ and that there are $\omega$ Woodin cardinals with a
   measurable cardinal above them all. Then every set of reals in
   $L(\mathbb{R})$ is determined.
 \end{thm}

 The goal of this paper is to prove results in the converse
 direction. That means we want to obtain large cardinal strength from
 determinacy axioms. This is done using inner model theoretic concepts
 which we start to introduce in the next section.


\subsection{Inner Model Theory}

 The most important concept in inner model theory is a
 mouse. Therefore we briefly review the definition of mice in this
 section and mention some relevant properties without proving
 them. The reader who is interested in a more detailed introduction to
 mice is referred to Section $2$ of \cite{St10}. 

 We assume that the reader is familiar with some fine structure theory
 as expounded for example in \cite{MS94} or \cite{SchZe10}.

 In general the models we are interested in are of the form
 $L[\vec{E}]$ for some coherent sequence of extenders $\vec{E}$. This
 notion goes back to Ronald B. Jensen and William J. Mitchell and is
 made more precise in the following definition.

 \begin{definition}
   We say $M$ is a \emph{potential premouse}\index{potential premouse}
   iff
   \[ M = (J_\eta^{\vec{E}}, \in, \vec{E} \upharpoonright \eta,
   E_\eta) \]
   for some fine extender sequence $\vec{E}$ and some ordinal
   $\eta$. We say that such a potential premouse $M$ is
   \emph{active}\index{active premouse} iff $E_\eta \neq \emptyset$.

   Moreover if $\kappa \leq \eta$, we write
   \[ M | \kappa = (J_\kappa^{\vec{E}}, \in, \vec{E} \upharpoonright
   \kappa, E_\kappa). \]
 \end{definition}

 Here \emph{fine extender sequence} is in the sense of Definition
 $2.4$ in \cite{St10}. This definition of a fine extender sequence
 goes back to Section $1$ in \cite{MS94} and \cite{SchStZe02}.

 \begin{definition}
   Let $M$ be a potential premouse. Then we say $M$ is a
   \emph{premouse}\index{premouse} iff every proper initial segment of
   $M$ is $\omega$-sound.
 \end{definition}

 We informally say that a \emph{mouse} is an iterable premouse, but
 since there are several different notions of iterability we try to
 avoid to use the word ``mouse'' in formal context, especially if it
 is not obvious what sort of iterability is meant. Nevertheless
 whenever it is not specified otherwise ``iterable'' in this paper
 always means ``$\omega_1$-iterable'' as defined below and therefore a
 ``mouse'' will be an $\omega_1$-iterable premouse.

 \begin{definition}
   We say a premouse $M$ is
   \emph{$\omega_1$-iterable}\index{iterable@$\omega_1$-iterable} iff
   player $\mathrm{II}$ has a winning strategy in the iteration game
   $\mathcal{G}_\omega(M, \omega_1)$ as described in Section $3.1$ of
   \cite{St10}.
 \end{definition}

 The iteration trees which are considered in Section $3$ in
 \cite{St10} are called
 \emph{normal}\index{normaliterationtrees@normal iteration trees}
 iteration trees.

 Whenever not specified otherwise we will assume throughout this
 paper that all iteration trees are normal to simplify the
 notation. Since normal iteration trees do not suffice to prove for
 example the Dodd-Jensen Lemma (see Section $4.2$ in \cite{St10}) it
 is necessary to consider stacks of normal trees. See Definition $4.4$
 in \cite{St10} for a formal definition of iterability for stacks of
 normal trees.

 All arguments to follow easily generalize to countable stacks of
 normal trees of length $<\omega_1$ instead of just normal trees of
 length $<\omega_1$. The reason for this is that the iterability we
 will prove in this paper for different kinds of premice will in fact
 always be obtained from the sort of iterability for the model $K^c$
 which is proven in Chapter $9$ in \cite{St96}.

 Throughout this paper we will use the notation from \cite{St10} for
 iteration trees.


\subsection{Mice with Finitely Many Woodin Cardinals}

We first fix some notation and give a short background on the mouse
$M_n^\#$. In this paper we always assume $M_n^\#$ to be
$\omega_1$-iterable if not specified otherwise.

The premice we are going to consider in this paper will mostly have
the following form.

\begin{definition}
  Let $n \geq 1$. A premouse $M$ is called
  \emph{$n$-small}\index{small@$n$-small} iff for every critical point
  $\kappa$ of an extender on the $M$-sequence
  \[ M | \kappa \nvDash \text{``there are } n \text{ Woodin
    cardinals''}. \]
  Moreover we say that a premouse $M$ is \emph{$0$-small} iff $M$ is
  an initial segment of Gödel's constructible universe $L$.
\end{definition}

Moreover $\omega$-small premice are defined analogously.

\begin{definition}
  Let $n \geq 1$ and $x \in \BS$. Then
  \emph{$M_n^\#(x)$}\index{Mnsharp@$M_n^\#(x)$} denotes the unique
  countable, sound, $\omega_1$-iterable $x$-premouse which is not
  $n$-small, but all of whose proper initial segments are $n$-small,
  if it exists and is unique.
\end{definition}

\begin{definition}
  Let $n \geq 1$, $x \in \BS$ and assume that $M_n^\#(x)$ exists. Then
  $M_n(x)$\index{Mn@$M_n(x)$} is the unique $x$-premouse which is
  obtained from $M_n^\#(x)$ by iterating its top measure out of the
  universe.
\end{definition}

\begin{remark}
  We denote $M_n^\#(0)$ and $M_n(0)$ by $M_n^\#$ and $M_n$ for $n \geq
  0$. 
\end{remark}

\begin{remark}
  We say that $M_0^\#(x) = x^\#$ for all $x \in \BS$, where $x^\#$
  denotes the least active $\omega_1$-iterable premouse if it
  exists. Moreover we say that $M_{0}(x) = L[x]$ is Gödel's
  constructible universe above $x$.
\end{remark}

The two correctness facts due to the third author about the premouse
$M_n^\#(x)$ which are stated in the following lemma are going to be
useful later, because they help transferring projective statements
from $M_n^\#(x)$ to $V$ or the other way around.

\begin{lemma}
\label{lem:corr}
Let $n \geq 0$ and assume that $M_n^\#(x)$ exists and is
$\omega_1$-iterable for all $x \in \BS$. Let $\varphi$ be a
$\Sigma^1_{n+2}$-formula.
\begin{enumerate}[(1)]
\item Assume $n$ is even and let $x \in \BS$. Then we have for every
  real $a$ in $M_n^\#(x)$,
  \[ \varphi(a) \; \leftrightarrow \; M_n^\#(x) \vDash \varphi(a). \]
  That means $M_n^\#(x)$ is $\SIGMA^1_{n+2}$-correct in $V$.
\item Assume $n$ is odd, so in particular $n \geq 1$, and let
  $x \in \BS$. Then we have for every real $a$ in $M_n^\#(x)$,
  \[ \varphi(a) \; \leftrightarrow \; \;
  \forces{M_n^\#(x)}{\Col(\omega,\delta_0)} \varphi(a), \]
  where $\delta_0$ denotes the least Woodin cardinal in
  $M_n^\#(x)$. Furthermore we have that $M_n^\#(x)$ is
  $\SIGMA^1_{n+1}$-correct in $V$.
\end{enumerate}
\end{lemma}

For notational simplicity we sometimes just write $a$ for the standard
name $\check{a}$ for a real $a$ in $M_n^\#(x)$.

\begin{proof}[Proof of Lemma \ref{lem:corr}]
  For $n=0$ this lemma holds by Shoenfield's Absoluteness Theorem (see
  for example Theorem $13.15$ in \cite{Ka03}) applied to the model we
  obtain by iterating the top measure of the active premouse
  $M_0^\#(x)$ and its images until we obtain a model of height
  $\geq \omega_1^V$, because this model has the same reals as
  $M_0^\#(x)$.

  We simultaneously prove that $(1)$ and $(2)$ hold for all $n \geq 1$
  inductively. In fact we are proving a more general statement: We
  will show inductively that both $(1)$ and $(2)$ hold for all
  $n$-iterable $x$-premice for all reals $x$ in the sense of
  Definition $1.1$ in \cite{Ne95} which have $n$ Woodin cardinals
  which are countable in $V$ instead of the concrete $x$-premouse
  $M_n^\#(x)$ as in the statement of Lemma \ref{lem:corr}. Therefore
  notice that we could replace the $z$-premouse $M_n^\#(z)$ in the
  following argument by any $z$-premouse $N$ which is $n$-iterable and
  has $n$ Woodin cardinals which are countable in $V$.

  \textbf{Proof of $(2)$:} We start with proving $(2)$ in the
  statement of Lemma \ref{lem:corr} for $n$, assuming inductively that
  $(1)$ and $(2)$ hold for all $m < n$. For the downward implication
  assume that $n$ is odd and let $\varphi$ be a
  $\Sigma^1_{n+2}$-formula such that $\varphi(a)$ holds in $V$ for a
  parameter $a \in M_n^\#(z) \cap \BS$ for a $z \in \BS$. That means
  \[ \varphi(a) \equiv \exists x \forall y \, \psi(x,y,a) \]
  for a $\Sigma^1_n$-formula $\psi(x,y,a)$. Now fix a real $\bar{x}$
  in $V$ such that
  \[ V \vDash \forall y \, \psi(\bar{x},y,a). \]
  We aim to show that
  \[ \forces{M_n^\#(z)}{\Col(\omega, \delta_0)} \; \varphi(a), \]
  where $\delta_0$ denotes the least Woodin cardinal in $M_n^\#(z)$.

  We first use Corollary $1.8$ in \cite{Ne95} to make $\bar{x}$
  generic over an iterate $M^*$ of $M_n^\#(z)$ for the collapse of the
  image of the bottom Woodin cardinal $\delta_0$ in $M_n^\#(z)$. That
  means there is an iteration tree $\T$ on the $n$-iterable
  $z$-premouse $M_n^\#(z)$ of limit length and a non-dropping branch
  $b$ through $\T$ such that if \[ i: M_n^\#(z) \rightarrow M^* \]
  denotes the corresponding iteration embedding we have that $M^*$ is
  $(n-1)$-iterable and if $g$ is $\Col(\omega, i(\delta_0))$-generic
  over $M^*$, then $\bar{x} \in M^*[g]$. 

  We have that $M^*[g]$ can be construed as an
  $(z \oplus \bar{x})$-premouse satisfying the inductive hypothesis
  and if we construe $M^*[g]$ as an $(z \oplus \bar{x})$-premouse we
  have that in fact $M^*[g] = M_{n-1}^\#(z \oplus \bar{x})$ (see for
  example \cite{SchSt09} for the fine structural details). Therefore
  we have inductively that the premouse $M^*[g]$ is
  $\SIGMA^1_n$-correct in $V$ (in fact it is even
  $\SIGMA^1_{n+1}$-correct in $V$, but this is not necessary here) and
  using downwards absoluteness for the $\Pi^1_{n+1}$-formula
  ``$\forall y \psi(\bar{x},y,a)$'' it follows that
  \[ M^*[g] \vDash \forall y \, \psi(\bar{x},y,a), \]
  because $\bar{x},a \in M^*[g]$. Since the forcing
  $\Col(\omega, i(\delta_0))$ is homogeneous, we have that
  \[ \forces{M^*}{\Col(\omega, i(\delta_0))} \; \exists x \forall y \,
  \psi(x,y,a), \]
  and by elementarity of the iteration embedding
  $i: M_n^\#(z) \rightarrow M^*$ it follows that
  \[ \forces{M_n^\#(z)}{\Col(\omega, \delta_0)} \; \exists x \forall y
  \, \psi(x,y,a), \] as desired.

  For the upward implication of $(2)$ in the statement of Lemma
  \ref{lem:corr} let $n$ again be odd, let $z$ be a real and assume
  that
  \[ \forces{M_n^\#(z)}{\Col(\omega, \delta_0)} \; \varphi(a), \]
  where as above
  $\varphi(a) \equiv \exists x \forall y \, \psi(x,y,a)$ is a
  $\Sigma^1_{n+2}$-formula, $\psi(x,y,a)$ is a $\Sigma^1_n$-formula
  and $a$ is a real such that $a \in M_n^\#(z)$. Let $g$ be
  $\Col(\omega, \delta_0)$-generic over $M_n^\#(z)$ and pick a real
  $\bar{x}$ such that
  \[ M_n^\#(z)[g] \vDash \forall y \, \psi(\bar{x},y,a). \]
  Since $M_n^\#(z)|(\delta_0^+)^{M_n^\#(z)}$ is countable in $V$, we
  can in fact pick the generic $g$ such that we have $g \in V$. Then
  we have that $\bar{x} \in V$. Similar as above $M_n^\#(z)[g]$ can be
  construed as a $z^*$-premouse for some real $z^*$ which satifies the
  inductive hypothesis for $n-1$, in fact we again have that
  $M_n^\#(z)[g] = M_{n-1}^\#(z^*)$ if $M_n^\#(z)[g]$ is construed as a
  $z^*$-premouse. Since $n-1$ is even and we inductively assume that
  $(1)$ in the statement of Lemma \ref{lem:corr} holds for all
  $m < n$, it follows from $(1)$ applied to the $\Pi^1_{n+1}$-formula
  ``$\forall y \psi(\bar{x},y,a)$'' and the premouse $M_n^\#(z)[g]$
  that
  \[ V \vDash \forall y \, \psi(\bar{x},y,a), \]
  and therefore \[ V \vDash \exists x \forall y \, \psi(x,y,a), \] as
  desired. 

  The fact that in this situation $M_n^\#(z)$ is
  $\SIGMA^1_{n+1}$-correct in $V$ also follows from the inductive
  hypothesis, because $n-1$ is even and the inductive hypothesis for
  $n-1$ can be applied to the premouse $M_n^\#(z)$.

  \textbf{Proof of $(1)$:} Now we turn to the proof of $(1)$ in the
  statement of Lemma \ref{lem:corr}. Let $n$ be even and assume
  inductively that $(1)$ and $(2)$ hold for all $m < n$. We again
  start with the proof of the downward implication, that means we want
  to prove that
  \[ M_n^\#(z) \vDash \varphi(a), \]
  where as above $\varphi$ is a $\Sigma^1_{n+2}$-formula which holds
  in $V$ for $a \in M_n^\#(z) \cap \BS$ and $z \in \BS$. That means we
  again have
  \[ \varphi(a) \equiv \exists x \forall y \, \psi(x,y,a) \]
  for a $\Sigma^1_n$-formula $\psi(x,y,a)$. Since $n$ is even, it
  follows from Moschovakis' Second Periodicity Theorem that the
  pointclass $\Pi^1_{n+1}(a)$ has the uniformization property (see
  Theorem $6C.5$ in \cite{Mo09}), because Theorem $2.14$ in
  \cite{Ne02} (see also Corollary \ref{cor:newWoodin1rel}) yields that
  $\PI^1_{n}$ determinacy holds from the hypothesis that
  $M_{n-1}^\#(z)$ exists for all reals $z$. Consider the
  $\Pi^1_{n+1}(a)$-definable set
  \[ \{ x \st \forall y \, \psi(x,y,a) \}. \]
  The uniformization property yields the existence of a real $\bar{x}$
  such that we have $\{ \bar{x} \} \in \Pi^1_{n+1}(a)$ and
  \[ V \vDash \forall y \, \psi(\bar{x},y,a). \]
  So let $\rho$ be a $\Pi^1_{n+1}$-formula such that
  \[ x = \bar{x} \leftrightarrow \rho(x,a) \]
  for all $x \in \BS$. That means we have
  \[ V \vDash \rho(\bar{x},a) \wedge \forall y \,
  \psi(\bar{x},y,a). \]
  Now we use, as in the proof of $(2)$ above, the $n$-iterability of
  $M_n^\#(z)$ and Corollary $1.8$ in \cite{Ne95} to make $\bar{x}$
  generic over an iterate $M^*$ of $M_n^\#(z)$ for the collapse of the
  image of the bottom Woodin cardinal $\delta_0$ in $M_n^\#(z)$. As in
  the proof of $(2)$ this means that in fact there is an iteration
  embedding
  \[ i: M_n^\#(z) \rightarrow M^* \]
  such that $M^*$ is $(n-1)$-iterable and if $g \in V$ is
  $\Col(\omega, i(\delta_0))$-generic over $M^*$ then
  $\bar{x} \in M^*[g]$. Since $a$ is a real in $M_n^\#(z)$ we have
  that
  \[ \forces{M^*[g]}{\Col(\omega, \bar{\delta})} \; \rho(\bar{x},a)
  \wedge \forall y \, \psi(\bar{x},y,a), \]
  where $\bar{\delta}$ denotes the least Woodin cardinal inside
  $M^*[g]$, by the inductive hypothesis applied to the premouse
  $M^*[g]$, construed as an $(z \oplus \bar{x})$-premouse, and the
  $\Pi^1_{n+1}$-formula
  ``$\rho(\bar{x},a) \wedge \forall y \, \psi(\bar{x},y,a)$'', because
  $n-1$ is odd. As above we have that $M^*[g]$, construed as a
  $(z \oplus \bar{x})$-premouse, satisfies the inductive
  hypothesis. Moreover we have as above that the inductive hypothesis
  applied to the model $M^*[g]$ and the $\Pi^1_{n+1}$-formula
  ``$\rho(x,a) \wedge \forall y \, \psi(x,y,a)$'' yields
  that in fact for all $x^* \in M^*[g]$
  \[ \forces{M^*[g]}{\Col(\omega, \bar{\delta})} \; \rho(x^*,a) \wedge
  \forall y \, \psi(x^*,y,a) \; \text{ iff } \; V \vDash \rho(x^*,a)
  \wedge \forall y \, \psi(x^*,y,a). \]

  By the homogeneity of the forcing $\Col(\omega, i(\delta_0))$ this
  implies that the witness $\bar{x}$ for $x$ with $\rho(x,a)$ already
  exists in the ground model $M^*$, since $a \in M^*$ and $\bar{x}$ is
  still the unique witness to the fact that the statement
  ``$\forces{M^*[g]}{\Col(\omega, \bar{\delta})} \; \rho(\bar{x},a)
  \wedge \forall y \, \psi(\bar{x},y,a)$''
  holds true. Therefore it follows by downward absoluteness that
  \[ M^* \vDash \rho(\bar{x},a) \wedge \forall y \,
  \psi(\bar{x},y,a). \] This implies that in particular
  \[ M^* \vDash \exists x \forall y \, \psi(x,y,a). \]
  Using the elementarity of the iteration embedding we finally get
  that
  \[ M_n^\#(z) \vDash \exists x \forall y \, \psi(x,y,a). \]

  For the proof of the upward implication in $(1)$ let $n$ again be
  even and assume that
  $M_n^\#(z) \vDash \exists x \forall y \, \psi(x,y,a)$ for
  $z \in \BS$ and a fixed real $a \in M_n^\#(z)$. Furthermore fix a
  real $\bar{x} \in M_n^\#(z)$ such that
  \[ M_n^\#(z) \vDash \forall y \, \psi(\bar{x},y,a). \]
  Then we obviously have that $\bar{x} \in V$. We want to show that
  $V \vDash \forall y \, \psi(\bar{x},y,a).$ Assume not. That means
  \[ V \vDash \exists y \, \neg \psi(\bar{x},y,a), \]
  where ``$\exists y \, \neg \psi(\bar{x},y,a)$'' is a
  $\Sigma^1_{n+1}$-formula. Therefore the downward implication we
  already proved applied to the formula
  ``$\exists y \, \neg \psi(\bar{x},y,a)$'' and the parameters
  $\bar{x},a \in M_n^\#(z)$ yields that
  \[ M_n^\#(z) \vDash \exists y \, \neg \psi(\bar{x},y,a), \] which is a
  contradiction.
\end{proof}

The proof of Lemma \ref{lem:corr} with Shoenfield absoluteness
replaced by $\SIGMA^1_1$ absoluteness immediately gives the following
lemma.

\begin{lemma}
\label{lem:ctblcorr}
Let $n \geq 0$ and let $M$ be a countable $x$-premouse with $n$ Woodin
cardinals for some $x \in \BS$ such that $M \vDash \ZFC^-$ and $M$ is
$\omega_1$-iterable. Let $\varphi$ be a $\Sigma^1_{n+1}$-formula.
\begin{enumerate}[(1)]
\item Assume $n$ is even. Then we have for every real $a$ in $M$,
  \[ \varphi(a) \; \leftrightarrow \; M \vDash \varphi(a). \]
  That means $M$ is $\SIGMA^1_{n+1}$-correct in $V$.
\item Assume $n$ is odd, so in particular $n \geq 1$. Then we have for
  every real $a$ in $M$,
  \[ \varphi(a) \; \leftrightarrow \; \;
  \forces{M}{\Col(\omega,\delta_0)} \varphi(a), \]
  where $\delta_0$ denotes the least Woodin cardinal in
  $M$. Furthermore we have that $M$ is $\SIGMA^1_{n}$-correct in $V$.
\end{enumerate}
\end{lemma}

The following lemma shows that Lemma \ref{lem:corr} $(1)$ does not
hold if $n$ is odd. Therefore the periodicity in the statement of
Lemma \ref{lem:corr} is necessary indeed.

\begin{lemma}\label{lem:oddnotcorr}
  Let $n \geq 1$ be \emph{odd} and assume that $M_n^\#(x)$ exists and
  is $\omega_1$-iterable for all $x \in \BS$. Then $M_n^\#(x)$ is
  \emph{not} $\SIGMA^1_{n+2}$-correct in $V$.
\end{lemma}

\begin{proof}[Proof sketch]
  Consider for example the following $\Sigma^1_{n+2}$-formula
  $\varphi$, where $\Pi^1_{n+1}$-iterability is defined in Definition
  $1.6$ in \cite{St95} (see also Section \ref{sec:prepKS} in this
  paper for some results related to $\Pi^1_{n+1}$-iterability).
  \begin{align*} \varphi(x) \equiv \exists & N \text{ such that } N
                                             \text{ is a countable }x\text{-premouse } \\
                                           & \text{ which is }
                                             \Pi^1_{n+1}\text{-iterable
                                             and not } n\text{-small}.
    \end{align*}
    The statement ``$N$ is $\Pi^1_{n+1}$-iterable'' is
    $\Pi^1_{n+1}$-definable uniformly in any code for $N$ (see Lemma
    $1.7$ in \cite{St95}). Therefore $\varphi$ is a
    $\Sigma^1_{n+2}$-formula.

    We have that $\varphi(x)$ holds in $V$ for all reals $x$ as
    witnessed by the $x$-premouse $M_n^\#(x)$, because
    $\omega_1$-iterability implies $\Pi^1_{n+1}$-iterability, since we
    assumed that $M_n^\#(x)$ exists for all $x \in \BS$ (see Lemma
    \ref{lem:steel2.2} $(2)$ which uses Lemma $2.2$ in
    \cite{St95}). 

    Assume toward a contradiction that $\varphi(x)$ holds in
    $M_n^\#(x)$ as witnessed by some $x$-premouse $N$ in $M_n^\#(x)$
    which is $\Pi^1_{n+1}$-iterable and not $n$-small. Since $n$ is
    odd, Lemma $3.1$ in \cite{St95} implies that
    $\mathbb{R} \cap M_n(x) \subseteq \mathbb{R} \cap N$, which is a
    contradiction as $N \in M_n^\#(x)$.

    Therefore $\varphi(x)$ cannot hold in $M_n^\#(x)$ and thus
    $M_n^\#(x)$ is not $\SIGMA^1_{n+2}$-correct in $V$.
  \end{proof}

  See for example \cite{St95} for further information on the premouse
  $M_n^\#$.


\subsection{Mice and Determinacy}

Some of the results mentioned in Section \ref{sec:GamesDet} can be
improved using the existence of certain mice instead of large
cardinals in $V$ as an hypothesis. In this section we will list
results of that form and mention some things known about the converse
direction.

In the context of analytic determinacy Harrington was able to prove
the converse of Martin's result from \cite{Ma70} and therefore
obtained the following theorem.

\begin{thm}[Harrington in \cite{Ha78}, Martin in \cite{Ma70}]
  The following are equivalent over $\ZFC$.
  \begin{enumerate}[$(i)$]
  \item The mouse $0^\#$ exists, and
  \item every $\Pi^1_1$-definable set of reals is determined.
  \end{enumerate}
\end{thm}

In the projective hierarchy Neeman improved the result of
\cite{MaSt89} as follows. Here ``$\Game$'' denotes the game quantifier
as also used in Section \ref{sec:odd} later (see Section $6D$ in
\cite{Mo09} for a definition and some basic facts about the game
quantifier ``$\Game$'').

\begin{thm}[Neeman in \cite{Ne02}]
  \label{thm:Neeman}
  Let $n \geq 1$ and assume that $M_n^\#$ exists and is
  $\omega_1$-iterable. Then every
  $\Game^{(n)} (< \omega^2 - \Pi^1_1)$-definable set of reals is
  determined and thus in particular every $\Pi^1_{n+1}$-definable set
  of reals is determined.
\end{thm}

In this paper we will present a proof of the boldface version of a
converse direction of this theorem due to the third author which is
only assuming that every $\PI^1_{n+1}$-definable set of reals is
determined, see Corollary \ref{cor:newWoodin1rel}. The lightface
version of the analogous converse direction of Theorem
\ref{thm:Neeman} is still open for $n > 1$ (see also Section
\ref{sec:openproblemsproj}). For $n = 1$ it is due to the third
author, who proved the following theorem (see Corollary $4.17$ in
\cite{StW16}).

\begin{thm}
  The following are equivalent over $\ZFC$.
  \begin{enumerate}[$(i)$]
  \item $M_1^\#$ exists and is $\omega_1$-iterable, and
  \item every $\PI^1_1$-definable and every $\Delta^1_2$-definable set
    of reals is determined. 
  \end{enumerate}
\end{thm}

Here the implication $(i) \Rightarrow (ii)$, which also follows from
Theorem $2.14$ in \cite{Ne02} (see Theorem \ref{thm:Neeman} above),
was first shown by the third author in unpublished work.

At the level of infinitely many Woodin cardinals we have that the
following equivalence, which is also due to the third author, holds
true (see Theorem $8.4$ in \cite{KW10}).

  \begin{thm}[\cite{KW10}]
    The following are equivalent over $\ZFC$.
    \begin{enumerate}[$(i)$]
    \item $M_\omega^\#$ exists and is countably iterable, and
    \item $\AD^{L(\R)}$ holds and $\R^\#$ exists.
    \end{enumerate}
  \end{thm}

  Here we mean by ``$\AD^{L(\R)}$ holds'' that every set of reals in
  $L(\R)$ is determined.


\section[Woodin Cardinals
  from Determinacy Hypotheses]{A Model with Woodin Cardinals
  from Determinacy Hypotheses}
\label{ch:WCFromDet}

In this section we are going to construct a proper class model with
$n$ Woodin cardinals from $\PI^1_n$ determinacy together with
$\Pi^1_{n+1}$ determinacy, but the model constructed in this section
need not be iterable. We will treat iterability issues for models like
the one constructed in this section later in Section \ref{ch:it}.

\subsection{Introduction}
\label{sec:introKS}

The main goal of Sections \ref{ch:WCFromDet} and \ref{ch:it} is to
give a proof of the following theorem due to the third author.

\begin{thm}
\label{cor:bfDetSharps}
  Let $n \geq 1$ and assume $\PI^1_{n+1}$ determinacy holds. Then $M_n^\#(x)$
  exists and is $\omega_1$-iterable for all $x \in \BS$.
\end{thm}

The converse of Theorem \ref{cor:bfDetSharps} also holds true. For odd
$n$ it is due to the third author in never-published work and for even
$n$ it is due to Itay Neeman in \cite{Ne95}. This yields the following
corollary, where the case $n = 0$ is due to D. A. Martin (see
\cite{Ma70}) and L. Harrington (see \cite{Ha78}).

\begin{cor}\label{cor:newWoodin1rel}
  Let $n \geq 0$. Then the following are equivalent.
\begin{enumerate}[(1)]
\item $\PI^1_{n+1}$ determinacy, and
\item for all $x \in \BS$, $M_n^\#(x)$ exists and is $\omega_1$-iterable.
\end{enumerate}
\end{cor}

The proof of Theorem \ref{cor:bfDetSharps} is organized
inductively. Thereby Harrington's result that analytic determinacy
yields the existence of $0^\#$ (see \cite{Ha78}) is the base step of
our induction. So we will assume throughout the proof of Theorem
\ref{cor:bfDetSharps} at the $n$'th level, that Theorem
\ref{cor:bfDetSharps} holds true at the level $n-1$. In fact by
Theorem $2.14$ in \cite{Ne02} we can assume during the proof at the
level $n$ that the existence and $\omega_1$-iterability of
$M_{n-1}^\#(x)$ for all $x \in \BS$ is equivalent to $\PI^1_n$
determinacy (see Corollary \ref{cor:newWoodin1rel}, for odd $n$ this
result is due to the third author). We will use this in what follows
without further notice.

We first fix some notation we are going to use for the rest of this
paper.

If $M$ is a premouse let $\delta_M$\index{deltaM@$\delta_M$} denote
the least Woodin cardinal in $M$, if it exists. For $n \geq 2$ and
$x \in {}^\omega\omega$ let $\delta_x = \delta_{M_{n-1}(x)}$ denote
the least Woodin cardinal in $M_{n-1}(x)$, whenever this does not lead
to confusion. Moreover in case we are considering $L[x] = M_0(x)$ and
a confusion is not possible let $\delta_x$ denote the least
$x$-indiscernible in $L[x]$.

\begin{remark}
  Recall that a real $x \in \BS$ is \emph{Turing reducible} to a real
  $y \in \BS$ (write ``$x \leq_T y$'') iff $x$ is recursive in $y$ or
  equivalently iff there exists an oracle Turing machine that computes
  $x$ using $y$ as an oracle. Moreover we write $x \equiv_T y$ iff
  $x \leq_T y$ and $y \leq_T x$ and say in this case that $x$ and $y$
  are \emph{Turing equivalent}. \index{Turingreducible@$x \leq_T y$
    and $x \equiv_T y$}
\end{remark}

The following lemma generalizes a theorem of Kechris and Solovay (see
Theorem $3.1$ in \cite{KS85}) to the context of mice with finitely
many Woodin cardinals. It is one key ingredient for building inner
models with finitely many Woodin cardinals from determinacy hypotheses
and therefore in particular for proving Theorem
\ref{cor:bfDetSharps}. The following two sections will be devoted to
the proof of this lemma.

\begin{lemma}
\label{lem:KS}
Let $n \geq 1$. Assume that $M_{n-1}^\#(x)$ exists and is
$\omega_1$-iterable for all $x \in {}^\omega\omega$ and that all
$\Sigma^1_{n+1}$-definable sets of reals are determined. Then there
exists a real $y_0$ such that for all reals $x \geq_T y_0$,
\[ M_{n-1}(x)|\delta_x \vDash \ODdet. \]
\end{lemma}

The main difficulty in proving Lemma \ref{lem:KS} in our context is
the fact that for $n > 1$ the premouse $M_{n-1}(x) | \delta_x$ has
lots of total extenders on its sequence. This is the main reason why
we cannot generalize the proof of Theorem $3.1$ in \cite{KS85}
straightforwardly. Therefore we need to prove some preliminary lemmas
concerning comparisons and $L[E]$-constructions in our context in the
following section. Some of this could have been avoided, if we would
only want to prove Lemma \ref{lem:KS} for models like for example
lower part models (see Definition \ref{def:Lp}), which do not contain
total extenders on their sequence.


\subsection{Preliminaries}
\label{sec:prepKS}

In this section we prove a few general lemmas about $(n-1)$-small
premice which we are going to need for the proof of Lemma \ref{lem:KS}
and which are also going to be helpful later on.

The following models, called $\Q$-structures, can serve as witnesses
for iterability by guiding an iteration strategy as in Definition
\ref{def:Qstritstr}. See also for example the proof of Lemma
\ref{lem:iterability} for an application of this iteration
strategy.

Informally a $\Q$-structure for a cofinal well-founded branch $b$
through $\T$ is the longest initial segment of $\M_b^\T$ at which
$\delta(\T)$ is still seen to be a Woodin cardinal. Such
$\Q$-structures are also introduced for example in Definition $6.11$
in \cite{St10}.

\begin{definition}\label{def:Qstructure}
  Let $N$ be an arbitrary premouse and let $\T$ be an iteration tree
  on $N$ of limit length.
  \begin{enumerate}[(1)]
  \item We say a premouse $\Q = \Q(\T)$ is a \emph{$\Q$-structure for
      $\T$} \index{Qstructure@$\Q$-structure for $\T$} iff
    $\M(\T) \unlhd \Q$, $\delta(\T)$ is a cutpoint of $\Q$, $\Q$ is
    $\omega_1$-iterable above $\delta(\T)$,
    \[ \Q \vDash \text{``} \delta(\T) \text{ is a Woodin
      cardinal''}, \] if $\Q \neq \M(\T)$ and either
    \begin{enumerate}[$(i)$]
    \item over $\Q$ there exists an $r\Sigma_n$-definable set
      $A \subset \delta(\T)$ such that there is no
      $\kappa < \delta(\T)$ such that $\kappa$ is strong up to
      $\delta(\T)$ with respect to $A$ as being witnessed by extenders
      on the sequence of $\Q$ for some $n < \omega$, or
  \item  $\rho_n(\Q) < \delta(\T)$ for some $n<\omega$.
    \end{enumerate}
  \item Let $b$ be a cofinal well-founded branch through $\T$. Then we
    say a premouse $\Q = \Q(b, \T)$ is a \emph{$\Q$-structure for $b$
      in $\T$}\index{Qstructureb@$\Q$-structure for $b$ in $\T$} iff
    $\Q = \M_b^\T | \gamma$, where $\gamma \leq \M_b^\T \cap \Ord$ is
    the least ordinal such that either
    \[ \gamma < \M_b^\T \cap \Ord \text{ and } \M_b^\T | (\gamma+1)
    \vDash \text{``}\delta(\T) \text{ is not Woodin'',} \] or
    \[ \gamma = \M_b^\T \cap \Ord \text{ and }
    \rho_n(\M_b^\T) < \delta(\T) \]
    for some $n<\omega$ or over $\M_b^\T$ there exists an
    $r\Sigma_n$-definable set $A \subset \delta(\T)$ such that there
    is no $\kappa < \delta(\T)$ such that $\kappa$ is strong up to
    $\delta(\T)$ with respect to $A$ as being witnessed by extenders
    on the sequence of $\M_b^\T$ for some $n < \omega$.\\
    If no such ordinal $\gamma \leq \M_b^\T \cap \Ord$ exists, we let
    $\Q(b,\T)$ be undefined.
  \end{enumerate}
\end{definition}

For the notion of an $r\Sigma_n$-definable set see for example §$2$
in \cite{MS94}.

\begin{remark}
  We are also going to use the notion of a \emph{$\Pi^1_n$-iterable
    $\Q$-structure $\Q(\T)$}, meaning that $\Q(\T)$ is
  $\Pi^1_n$-iterable above $\delta(\T)$, $\delta(\T)$-solid and that
  $\Q(\T)$ satisfies all properties in $(1)$ except for the
  $\omega_1$-iterability above $\delta(\T)$. It will be clear from the
  context if we include $\omega_1$-iterability in the definition of
  $\Q$-structure or not. Here $\Pi^1_n$-iterability is defined as in
  Definitions $1.4$ and $1.6$ in \cite{St95} (see also the
  explanations before Lemma \ref{lem:steel2.2}).
\end{remark}

\begin{remark}
  In Case $(1)(i)$ in Definition \ref{def:Qstructure} we have that in
  particular $\rho_n(\Q) \leq \delta$ and if $\Q \lhd M$ for a
  premouse $M$ and if we let $\gamma = \Q \cap \Ord$, then we have
  that $\delta(\T)$ is not Woodin in $J^M_{\gamma + 1}$. The same
  thing for $M$ as above holds true in Case $(1)(ii)$ in Definition
  \ref{def:Qstructure}, because in this case $\delta(\T)$ is not even
  a cardinal in $J^M_{\gamma + 1}$.
\end{remark}

\begin{remark}
  Let $n \geq 1$ and assume that $M_n^\#(x)$ exists and is
  $\omega_1$-iterable for all reals $x$. Then any $\Q$-structure $\Q$
  for an iteration tree $\T$ on an $n$-small premouse is unique by a
  comparison argument as the one we will see in the proof of Lemma
  \ref{lem:iterability}.
\end{remark}

\begin{definition}
  \label{def:Qstritstr} Let $N$ be a premouse. Then a possibly partial
  iteration strategy $\Sigma$ for $N$ is the \emph{$\Q$-structure
    iteration strategy for
    $N$}\index{Qstructureiterationstrategy@$\Q$-structure iteration
    strategy for $N$} or we say that $\Sigma$ is \emph{guided by
    $\Q$-structures}, iff $\Sigma$ is defined as follows. For a tree
  $\U$ on $N$ of limit length and a cofinal branch $b$ through $\U$ we let
  \[ \Sigma(\U) = b \text{ iff } \Q(\U) \text{ exists and } \Q(b,\U) =
  \Q(\U), \]
  if such a branch $b$ exists and is unique. If no such unique branch
  $b$ through $\U$ exists, we let $\Sigma(\U)$ be undefined.
\end{definition}

\begin{lemma}\label{lem:Qstrunique}
  Let $n \geq 1$ and assume that $M_n^\#(x)$ exists and is
  $\omega_1$-iterable for all reals $x$. Let $N$ be an $n$-small
  premouse and let $\T$ be an iteration tree on $N$. Then the branch
  $b$ through $\T$ which satisfies $\Q(b,\T) = \Q(\T)$ as in the
  definition of the $\Q$-structure iteration strategy $\Sigma$ is in
  fact unique.
\end{lemma}

This lemma holds true because for two branches $b$ and $c$ through an
iteration tree $\T$ as above, we have that $\Q(b, \T) = \Q(c, \T)$
implies that $b = c$ by the Branch Uniqueness Theorem (see Theorem
$6.10$ in \cite{St10}). This is proven in Corollary $6$ in §6 of
\cite{Je97} (written by Martin Zeman).

\begin{remark}
  Let $n \geq 1$ and assume that $M_n^\#(x)$ exists and is
  $\omega_1$-iterable for all reals $x$. Let $N$ be an $n$-small
  premouse and let $\T$ be an iteration tree on $N$. Then the branch
  $b$ through $\T$ which satisfies $\Q(b,\T) = \Q(\T)$ as in the
  definition of the $\Q$-structure iteration strategy $\Sigma$ is in
  fact cofinal.
\end{remark}

The following notion will be important in what follows to ensure that
$\Q$-structures exist.

\begin{definition}\label{def:notdefWdn}
  Let $M$ be a premouse and let $\delta$ be a cardinal in $M$ or
  $\delta = M \cap \Ord$. We say that \emph{$\delta$ is not definably
    Woodin over $M$} \index{definable Woodin cardinal} iff there
  exists an ordinal $\gamma \leq M \cap \Ord$ such that
  $\gamma \geq \delta$ and either
  \begin{enumerate}[$(i)$]
  \item over $J_\gamma^M$ there exists an $r\Sigma_n$-definable set
    $A \subset \delta$ such that there is no $\kappa < \delta$ such
    that $\kappa$ is strong up to $\delta$ with respect to $A$ as
    witnessed by extenders on the sequence of $M$ for some
    $n < \omega$, or
\item $\rho_n(J_\gamma^M) < \delta \text{ for some } n < \omega.$
  \end{enumerate}
\end{definition} 

For several iterability arguments to follow we need our premice to
satisfy the following property. By a fine structural argument this
property is preserved during an iteration and will therefore ensure
that $\Q$-structures exist in an iteration of a premouse $M$
satisfying this property.

\begin{definition}\label{def:nodefWdns}
  Let $M$ be a premouse. We say $M$ has \emph{no definable Woodin
    cardinals} \index{definable Woodin cardinals} iff for all
  $\delta \leq M \cap \Ord$ we have that $\delta$ is not definably
  Woodin over $M$.
\end{definition} 

\begin{remark}
  Let $M$ be a premouse which has no definable Woodin cardinals. Note
  that $M$ might still have Woodin cardinals. Consider for example the
  premouse $M_1^\#$, which by definition has a Woodin cardinal, but no
  definable Woodin cardinals since $\rho_\omega(M_1^\#) = \omega$.
\end{remark}

In what follows we sometimes want to consider premice which are
obtained from $M_n^\#$ ``constructed on top'' of a premouse $N$. The
following definition makes precise what we mean by that.

\begin{definition}
\label{def:M(N)}
Let $n \geq 1$ and assume that $M_n^\#(x)$ exists for all reals
$x$. Let $N$ be a countable $x$-premouse for some $x \in \BS$. Then we
say $M_n^\#(N)$ \index{MnsharpN@$M_n^\#(N)$} is the smallest
$x$-premouse $M \unrhd N$ with
\[ \rho_\omega(M) \leq N \cap \Ord \]
which is $\omega_1$-iterable above $N \cap \Ord$, sound above
$N \cap \Ord$ and such that either $M$ is not fully sound, or $M$ is
not $n$-small above $N \cap \Ord$.
\end{definition}

In the first case, i.e. if $M$ is not fully sound, we sometimes say that the
construction of $M_n^\#(N)$ breaks down.

\begin{remark}
  We can define a premouse $M_n(N)$ in a similar fashion, by iterating
  the top extender of $M = M_n^\#(N)$ out of the universe in the case
  that $M$ is not $n$-small above $N \cap \Ord$. In the case that $M$
  is not fully sound, we just let $M_n(N) = M$. So in particular
  $M_n(N)$ is a proper class model in the first case and a set in the
  latter case.
\end{remark}

In what follows we will point out if $M_n^\#(N)$ denotes the premouse
constructed in the sense of Definition \ref{def:M(N)} or if it denotes
the premouse $M_n^\#(x)$ in the usual sense, where $x$ is for example
a real coding the countable premouse $N$. Note that these two notions
are different since extenders on the $N$-sequence are included in the
$M_n^\#(N)$-sequence if $M_n^\#(N)$ is constructed in the sense of
Definition \ref{def:M(N)}.

We will need the following notation. 

\begin{definition}
  \begin{enumerate}[(1)]
  \item Let $x, y \in \BS$ be such that $x = (x_n \st n < \omega)$ for
    $x_n \in \omega$ and $y = (y_n \st n < \omega)$ for
    $y_n \in \omega$. Then we let
    $x \oplus y = (x_0, y_0, x_1, y_1, \dots ) \in \BS$.
  \item Let $M$ and $N$ be countable premice. We say a real $x$ codes
    $M \oplus N$ iff $x \geq_T x_M \oplus x_N$ for a real $x_M$ coding $M$
    and a real $x_N$ coding $N$.
  \end{enumerate}
\end{definition}

The following lemma proves that under the right hypothesis comparison
works for certain $\omega_1$-iterable premice instead of
$(\omega_1 + 1)$-iterable premice as in the usual Comparison Lemma
(see Theorem $3.11$ in \cite{St10}). Moreover the proof of this lemma
will use arguments that are explained here in full detail and will
show up in several different proofs throughout this paper again with
possibly less details given.

Recall that $\delta_x$ denotes the least Woodin cardinal in
$M_{n-1}^\#(x)$ for $n \geq 2$.

\begin{lemma}\label{lem:iterability}
  Let $n \geq 1$ and assume that $M_{n-1}^\#(x)$ exists and is
  $\omega_1$-iterable for all reals $x$. Let $M$ and $N$ be countable
  premice, such that $M$ and $N$ have a common cutpoint
  $\delta$. Assume that $M$ and $N$ both do not have definable Woodin
  cardinals above $\delta$ and that every proper initial segment of
  $M$ and $N$ is $(n-1)$-small above $\delta$. 
  \begin{enumerate}[(1)]
  \item Let $x$ be an arbitrary real and let $n \geq 2$. Then
    $H_{\delta_x}^{M_{n-1}^\#(x)}$ is closed under the operation
    \[ a \mapsto M_{n-2}^\#(a) \]
    and moreover this operation $a \mapsto M_{n-2}^\#(a)$ for
    $a \in H_{\delta_x}^{M_{n-1}^\#(x)}$ is contained in
    $M_{n-1}^\#(x)| \delta_x$.
  \item Let $x$ be a real coding $M$ and assume that the premouse $M$
    is $\omega_1$-iterable above $\delta$. If $\Sigma$ denotes the
    $\omega_1$-iteration strategy for $M$ above $\delta$, then for
    every iteration tree $\T \in H_{\delta_x}^{M_{n-1}^\#(x)}$ on $M$
    above $\delta$ of limit length,
    $\Sigma(\T) \in M_{n-1}^\#(x)|\delta_x$ and in fact the operation
    $\T \mapsto \Sigma(\T)$ is in $M_{n-1}^\#(x)|\delta_x$ for every
    such $\T$.
  \item Let $x$ be a real coding $M \oplus N$ and assume that the
    premice $M$ and $N$ are both $\omega_1$-iterable above
    $\delta$. Moreover assume that \[ M | \delta = N | \delta. \]
    Then we can successfully coiterate $M$ and $N$ above $\delta$
    inside the model $M_{n-1}^\#(x)$. That means there are iterates
    $M^*$ of $M$ and $N^*$ of $N$ above $\delta$ such that we have

 {\onehalfspacing
    \begin{enumerate}[$(a)$]
    \item $M^* \unlhd N^*$ and the iteration from $M$ to $M^*$ does
      not drop, or
    \item $N^* \unlhd M^*$ and the iteration from $N$ to $N^*$ does
      not drop. 
    \end{enumerate}
    In particular the coiteration is successful in $V$ in the same
    sense.}

  \item Let $x$ be a real coding $M \oplus N$ and assume that the
    premice $M$ and $N$ are both $\omega_1$-iterable above
    $\delta$. Moreover assume that \[ M | \delta = N | \delta, \]
    $M$ and $N$ are $\delta$-sound, $\rho_\omega(M) \leq \delta$ and
    $\rho_\omega(N) \leq \delta$. Then we have
    \[ M \unlhd N \text{ or } N \unlhd M. \]
\end{enumerate}
\end{lemma}

\begin{remark}
  This lemma holds in particular for $\delta = \omega$. That means if
  we assume that $M_{n-1}^\#(x)$ exists and is $\omega_1$-iterable for
  all reals $x$ and $M$ and $N$ are $\omega_1$-iterable countable
  premice such that both do not have definable Woodin cardinals and
  such that every proper initial segment of $M$ and $N$ is
  $(n-1)$-small, then we can successfully compare $M$ and $N$ as in
  Lemma \ref{lem:iterability} $(3)$. 
\end{remark}

\begin{proof}[Proof of Lemma \ref{lem:iterability}]
  \textbf{Proof of $(1)$:} Let $a \in H_{\delta_x}^{M_{n-1}^\#(x)}$ be
  arbitrary and perform a fully backgrounded extender construction
  $L[E](a)^{M_{n-1}^\#(x)|\kappa}$ in the sense of \cite{MS94} (with
  the smallness hypothesis weakened to allow $\omega$-small premice in
  the construction) above $a$ inside the model
  $M_{n-1}^\#(x) | \kappa$, where $\kappa$ denotes the critical point
  of the top measure of the active premouse $M_{n-1}^\#(x)$. Then we
  have that the premouse $L[E](a)^{M_{n-1}^\#(x)|\kappa}$ has $n-1$
  Woodin cardinals by a generalization of Theorem $11.3$ in
  \cite{MS94}. So in particular it follows that
  $L[E](a)^{M_{n-1}^\#(x)|\kappa}$ is not $(n-2)$-small.  Moreover the
  premouse $L[E](a)^{M_{n-1}^\#(x)|\kappa}$ inherits the iterability
  from $M_{n-1}^\#(x)$ and therefore we have that
  \[ M_{n-2}^\#(a) \lhd L[E](a)^{M_{n-1}^\#(x)|\kappa}. \]

  In fact the operation $a \mapsto M_{n-2}^\#(a)$ for
  $a \in H_{\delta_x}^{M_{n-1}^\#(x)}$ is contained in
  $M_{n-1}^\#(x)|\delta_x$, because $M_{n-2}^\#(a)$ can be obtained
  from an $L[E]$-construction above $a$.

  \textbf{Proof of $(2)+(3)+(4)$:} We prove $(2)$,$(3)$ and $(4)$
  simultaneously by an inductive argument. For $n=1$ there is nothing
  to show, because we defined $M$ to be $0$-small iff $M$ is an
  initial segment of Gödel's constructible universe $L$. That means if
  $M$ and $N$ are such that every proper initial segment of $M$ or $N$
  is $0$-small above some common cutpoint $\delta$ as in the statement
  of Lemma \ref{lem:iterability}, then every iteration tree on $M$ or
  $N$ above $\delta$ is linear and there is nothing to show for
  $(2)$. Moreover we easily get that one is an initial segment of the
  other since every proper initial segment of $M$ or $N$ is above
  $\delta$ as an initial segment of $L$. Therefore $(3)$ and $(4)$
  hold.

  So let $n \geq 2$ and assume that $(2)$,$(3)$ and $(4)$ hold for
  $n-2$. We first want to show that $(2)$ holds for $n-1$. Let us
  assume for notational simplicity that $\delta = \omega$ and let $M$
  be an $\omega_1$-iterable premouse, such that every proper initial
  segment of $M$ is $(n-1)$-small and such that $M$ has no definable
  Woodin cardinals. Let $x$ be a real coding the premouse $M$ and
  assume that $M_{n-1}^\#(x)$ exists and is
  $\omega_1$-iterable. Further let $\Sigma$ be the
  $\omega_1$-iteration strategy for $M$ and let $\T$ be an iteration
  tree on $M$ in $V$ of length $\lambda + 1$ for some limit ordinal
  $\lambda < \omega_1^V$ such that $\T$ is according to $\Sigma$ and
  \[ \T \upharpoonright \lambda \in H_{\delta_x}^{M_{n-1}^\#(x)}. \]
  By assumption the premouse $M$ is $\omega_1$-iterable in $V$ and has
  no definable Woodin cardinals. Therefore $\Sigma$ is the
  $\Q$-structure iteration strategy (see Definition \ref{def:Qstritstr})
  and there exists a $\Q$-structure $\Q \unlhd \M_\lambda^\T$ for
  $\T \upharpoonright \lambda$ which is $\omega_1$-iterable in $V$.
  We first want to show that such an $\omega_1$-iterable
  $\Q$-structure already exists in the model $M_{n-1}^\#(x)|\delta_x$.

  First consider the case \[ \Q = \M(\T \upharpoonright \lambda), \]
  where the latter denotes the common part model of
  $\T \upharpoonright \lambda$. In this case $\Q$ is also a
  $\Q$-structure for $\T \upharpoonright \lambda$ inside the model
  $M_{n-1}^\#(x)|\delta_x$ for trivial reasons, because the condition
  that $\Q$ needs to be $\omega_1$-iterable above
  $\delta(\T \upharpoonright \lambda) = \M(\T \upharpoonright \lambda)
  \cap \Ord$
  is empty here and therefore $\Q$ can be isolated as the
  $\Q$-structure for $\T \upharpoonright \lambda$ in
  $M_{n-1}^\#(x)|\delta_x$.

  So we can assume now that
  \[ \M(\T \upharpoonright \lambda) \lhd Q. \]
  That means $\delta(\T \upharpoonright \lambda) \in \Q$ and therefore
  $\Q$ is by definition the longest initial segment of $\M_\lambda^\T$
  in $V$ such that
  \[ \Q \vDash \text{``} \delta(\T \upharpoonright \lambda) \text{ is
    Woodin''.} \]
  Every proper initial segment of $\M_\lambda^\T$ is $(n-1)$-small
  since the same holds for $M$. Thus every proper initial segment of
  $\Q$ is $(n-1)$-small. Together with the fact that
  \[ \Q \vDash \text{``} \delta(\T \upharpoonright \lambda) \text{ is
    Woodin'',} \]
  this implies that every proper initial segment of $\Q$ in fact has
  to be $(n-2)$-small above $\delta(\T \upharpoonright \lambda)$.

  Now consider the premouse
  $M_{n-2}^\#(\M(\T \upharpoonright \lambda))$ in the sense of
  Definition \ref{def:M(N)}, which exists inside
  $M_{n-1}^\#(x)|\delta_x$ because of $(1)$, since we have that
  $\T \upharpoonright \lambda \in H_{\delta_x}^{M_{n-1}^\#(x)}$. Note
  that all proper initial segments of
  $M_{n-2}^\#(\M(\T \upharpoonright \lambda))$ are $(n-2)$-small above
  \[ \delta(\T \upharpoonright \lambda) = \M(\T \upharpoonright
  \lambda) \cap \Ord, \]
  regardless of whether the case that
  $M_{n-2}^\#(\M(\T \upharpoonright \lambda))$ is not fully sound or
  the case that $M_{n-2}^\#(\M(\T \upharpoonright \lambda))$ is not
  $(n-2)$-small above $\delta(\T \upharpoonright \lambda)$ in the
  definition of $M_{n-2}^\#(\M(\T \upharpoonright \lambda))$ (see
  Definition \ref{def:M(N)}) holds. Moreover we have by definition
  that
  \[ \Q \, | \, \delta(\T \upharpoonright \lambda) = \M(\T
  \upharpoonright \lambda) = M_{n-2}^\#(\M(\T \upharpoonright
  \lambda)) \, | \, \delta(\T \upharpoonright \lambda). \]
  Thus a coiteration of $\Q$ and
  $M_{n-2}^\#(\M(\T \upharpoonright \lambda))$ would take place above
  $\delta(\T \upharpoonright \lambda)$. Moreover
  $M_{n-2}^\#(\M(\T \upharpoonright \lambda))$ and $\Q$ are both
  $\delta(\T \upharpoonright \lambda)$-sound and project to
  $\delta(\T \upharpoonright \lambda)$, so in particular they both do
  not have definable Woodin cardinals above
  $\delta(\T \upharpoonright \lambda)$. This implies by the inductive
  hypothesis $(4)$ that the comparison of $\Q$ and
  $M_{n-2}^\#(\M(\T \upharpoonright \lambda))$ is successful in $V$,
  because all proper initial segments of both $\Q$ and
  $M_{n-2}^\#(\M(\T \upharpoonright \lambda))$ are $(n-2)$-small above
  $\delta(\T \upharpoonright \lambda)$. Moreover both sides do not
  move in the comparison as in $(4)$ and therefore we can distinguish
  two cases as follows.

\begin{minipage}{\textwidth}
  \bigskip
  \textbf{Case 1.} Assume that
  \[ M_{n-2}^\#(\M(\T \upharpoonright \lambda)) \lhd \Q. \]
  \medskip
\end{minipage}

Then we have that
  \[ M_{n-2}^\#(\M(\T \upharpoonright \lambda)) \vDash \text{``}
  \delta(\T \upharpoonright \lambda) \text{ is Woodin'',} \]
  because we have by definition of the $\Q$-structure $\Q$ (see
  Definition \ref{def:Qstructure}) that
  \[ \Q \vDash \text{``} \delta(\T \upharpoonright \lambda) \text{ is
      Woodin''.} \] Moreover our case assumption implies that
  $M_{n-2}^\#(\M(\T \upharpoonright \lambda))$ is fully sound. That
  means we were able to construct the full premouse $M_{n-2}^\#$ on
  top of $\M(\T \upharpoonright \lambda)$ as in the sense of
  Definition \ref{def:M(N)}. In this case
  $M_{n-2}^\#(\M(\T \upharpoonright \lambda))$ is not $(n-1)$-small,
  because $\delta(\T \upharpoonright \lambda)$ is a Woodin cardinal in
  $M_{n-2}^\#(\M(\T \upharpoonright \lambda))$. But we argued earlier
  that every proper initial segment of $\Q$ is $(n-1)$-small, which
  contradicts $M_{n-2}^\#(\M(\T \upharpoonright \lambda)) \lhd \Q$.

\begin{minipage}{\textwidth}
  \bigskip
  \textbf{Case 2.} Assume that
  \[ \Q \unlhd M_{n-2}^\#(\M(\T \upharpoonright \lambda)). \]
  \medskip
\end{minipage}

This implies that $\Q$ is in $M_{n-1}^\#(x)|\delta_x$ and furthermore
is $\omega_1$-iterable above $\delta(\T \upharpoonright \lambda)$ in
$M_{n-1}^\#(x)|\delta_x$ since the same holds for
$M_{n-2}^\#(\M(\T \upharpoonright \lambda))$ by part $(1)$ of this
lemma.

So we showed that in both cases there exists an $\omega_1$-iterable
$\Q$-structure $\Q$ for $\T \upharpoonright \lambda$ in
$M_{n-1}^\#(x)|\delta_x$. We now aim to show that the cofinal
well-founded branch through $\T$ in $V$ which is given by the
$\Q$-structure iteration strategy $\Sigma$, that means the branch $b$
for which we have
\[ \Q(b, \T \upharpoonright \lambda) = \Q, \]
is also contained in $M_{n-1}^\#(x)|\delta_x$.

Consider the statement
\begin{eqnarray*}
\phi(\T \upharpoonright \lambda, \Q) \equiv &\text{``there is a cofinal
  branch } b \text{ through } \T \upharpoonright \lambda \text{ such
  that } \\ 
& \text{ there is a } \Q^* \unlhd \M_b^\T \text{ with } \Q^* \cong \Q \text{''.}   
\end{eqnarray*}
This statement is $\Sigma^1_1$-definable uniformly in codes for
$\T \upharpoonright \lambda$ and $\Q$ and obviously true in $V$ as
witnessed by the branch $b$ given by the $\omega_1$-iterability of the
premouse $M$. Since the iteration tree
$\T \upharpoonright \lambda \in H_{\delta_x}^{M_{n-1}^\#(x)}$ need not
be countable in the model $M_{n-1}^\#(x)|\delta_x$, we consider the
model $(M_{n-1}^\#(x)|\delta_x)^{\Col(\omega, \gamma)}$ instead, where
$\gamma < \delta_x$ is an ordinal such that
\[ \T \upharpoonright \lambda, \Q \in
(M_{n-1}^\#(x)|\delta_x)^{\Col(\omega, \gamma)} \]
are countable inside the model
$(M_{n-1}^\#(x)|\delta_x)^{\Col(\omega, \gamma)}$, where with the
model $(M_{n-1}^\#(x)|\delta_x)^{\Col(\omega, \gamma)}$ we denote an
arbitrary $\Col(\omega, \gamma)$-generic extension of the model
$M_{n-1}^\#(x)|\delta_x$.

Then it follows by $\SIGMA^1_1$-absoluteness that the statement
$\phi(\T \upharpoonright \lambda, \Q)$ holds in
$(M_{n-1}^\#(x)|\delta_x)^{\Col(\omega, \gamma)}$ as witnessed by some
branch $\bar{b}$ and some model $\bar{\Q}$.

Since by the argument above $\Q$ is a $\Q$-structure for $\T$ in
$M_{n-1}^\#(x)|\delta_x$, we have that $\Q = \bar{\Q}$ and it follows
that $\bar{b}$ is the unique cofinal branch through $\T$ with
$\Q(\bar{b}, \T \upharpoonright \lambda) = \Q$ by Lemma
\ref{lem:Qstrunique}.

Since the branch $\bar{b}$ is uniquely definable from
$\T \upharpoonright \lambda$ and $\Q$, and we have that
$\T \upharpoonright \lambda, \Q \in M_{n-1}^\#(x)|\delta_x$, it
follows by homogeneity of the forcing $\Col(\omega, \gamma)$ that
actually $\bar{b} \in M_{n-1}^\#(x)|\delta_x$.

Thus we have that $\Sigma(\T) = \bar{b} \in M_{n-1}^\#(x)|\delta_x$
and our argument shows that in fact the operation
$\T \mapsto \Sigma(\T)$ for iteration trees
$\T \in H_{\delta_x}^{M_{n-1}^\#(x)}$ on $M$ of limit length is in the
model $M_{n-1}^\#(x)|\delta_x$ for the following reason.  Let $\T$ be
an iteration tree on $M$ of limit length such that
$\T \in H_{\delta_x}^{M_{n-1}^\#(x)}$. Then we showed in the first
part of this proof that $M_{n-1}^\#(x)|\delta_x$ can find a
$\Q$-structure $\Q$ for $\T$. Now $M_{n-1}^\#(x)|\delta_x$ can compute
$\Sigma(\T)$ from $\T$, because $\Sigma(\T)$ is the unique cofinal
branch $b$ through $\T$ such that $\Q(b, \T) = \Q$ and we showed that
this branch $b$ exists inside $M_{n-1}^\#(x)|\delta_x$. Therefore we
proved that $(2)$ holds.

To show $(3)$ assume now in addition that $N$ is an
$\omega_1$-iterable premouse such that every proper initial segment of
$N$ is $(n-1)$-small and $N$ has no definable Woodin
cardinals. Moreover let $x$ be a real coding $M \oplus N$. We again
assume that $\delta = \omega$ for notational simplicity.

By our hypothesis we have that $M_{n-1}^\#(x)$ exists, so we work
inside the model $M_{n-1}^\#(x)$. Moreover it follows from $(2)$ that
$M$ and $N$ are iterable inside $M_{n-1}^\#(x)$, particularly with
respect to trees in $H_{\omega_2}^{M_{n-1}^\#(x)}$, since
$M_{n-1}^\#(x)$ is countable in $V$ and thus
\[ \omega_2^{M_{n-1}^\#(x)} < \omega_1^V. \]
In particular $M$ and $N$ are $(\omega_1 + 1)$-iterable in
$M_{n-1}^\#(x)$ and the coiteration of $M$ and $N$ terminates
successfully inside the model $M_{n-1}^\#(x)$ by the usual Comparison
Lemma (see Theorem $3.11$ in \cite{St10}) applied inside
$M_{n-1}^\#(x)$. This shows that $(3)$ holds.

To prove $(4)$ we assume that we moreover have that $M$ and $N$ are
$\omega$-sound and that $\rho_\omega(M) = \omega$ and
$\rho_\omega(N) = \omega$. Then it follows as in Corollary $3.12$ in
\cite{St10} that we have $M \unlhd N$ or $N \unlhd M$.
\end{proof}

\begin{remark} As in the previous lemma let $n \geq 1$ and assume that
  $M_{n-1}^\#(x)$ exists and is $\omega_1$-iterable for all reals $x$.
  Then we can also successfully coiterate countable
  $\omega_1$-iterable premice $M$ and $N$ which agree up to a common
  cutpoint $\delta$ and are $(n-1)$-small above $\delta$ in the sense
  of Lemma \ref{lem:iterability} $(3)$, if we only assume that $M$ and
  $N$ both do not have Woodin cardinals, for the following reason.

  Let $x$ be a real coding $M \oplus N$. If $M \cap \Ord$ and
  $N \cap \Ord$ are both not definably Woodin over $M$ and $N$
  respectively, then this implies together with the assumption that
  $M$ and $N$ both do not have Woodin cardinals that $M$ and $N$ both
  do not have definable Woodin cardinals and we can apply Lemma
  \ref{lem:iterability}. So assume now for example that $M \cap \Ord$
  is definably Woodin over $M$. By the proof of Lemma
  \ref{lem:iterability} the coiteration of $M$ and $N$ can only fail
  on the $M$-side because of the lack of a $\Q$-structure for an
  iteration tree $\T$ of limit length on $M$ inside
  $M_{n-1}^\#(x)|\delta_x$. But in this case we must have that
  \[ \M(\T) = \M_\lambda^\T, \]
  where $\M_\lambda^\T$ is the limit model for $\T$ which exists in
  $V$. This implies that the coiteration on the $M$-side already
  finished because the $M$-side can no longer be iterated and by the
  same argument for $N$ we have that the coiteration is successful,
  even if $M \cap \Ord$ and $N \cap \Ord$ are both definably Woodin
  over $M$ and $N$ respectively.
\end{remark}

In what follows we also want to consider premice which are not fully
$\omega_1$-iterable but only $\Pi^1_n$-iterable for some
$n \in \omega$. This notion was defined by John Steel in \cite{St95}
and he proved that for a premouse $M$ the statement ``$M$ is
$\Pi^1_n$-iterable'' is $\Pi^1_n$-definable uniformly in any code for
$M$. See Definitions $1.4$ and $1.6$ in \cite{St95} for a precise
definition of $\Pi^1_n$-iterability.  He proves in Lemma $2.2$ in
\cite{St95} that for an $(n-1)$-small premouse $N$ which is
$\omega$-sound and such that $\rho_\omega(N) = \omega$,
$\Pi^1_n$-iterability is enough to perform the standard comparison
arguments with an $(\omega_1 + 1)$-iterable premouse which has the
same properties.

This implies that using Lemma \ref{lem:corr} and Lemma
\ref{lem:iterability} $(2)$ the following version of Lemma $2.2$
proven in \cite{St95} holds true for $\omega_1$-iterability (instead
of $(\omega_1 + 1)$-iterability as in \cite{St95}).

\begin{lemma}\label{lem:steel2.2}
  Let $n \geq 2$ and assume that $M_{n-1}^\#(x)$ exists and is
  $\omega_1$-iterable for all reals $x$. Let $M$ and $N$ be countable
  premice which have a common cutpoint $\delta$ such that $M$ and $N$
  are $(n-1)$-small above $\delta$. Assume that $M$ and $N$ are
  $\delta$-sound and that $\rho_\omega(M) \leq \delta$ and
  $\rho_\omega(N) \leq \delta$.
  \begin{enumerate}[(1)]
  \item Assume that $M$ is $\omega_1$-iterable above $\delta$. Let
    $\T$ be a normal iteration tree on $M$ above $\delta$ of length
    $\lambda$ for some limit ordinal $\lambda < \omega_1$ and let $b$
    be the unique cofinal well-founded branch through $\T$ such that
    $\Q(b, \T)$ is $\omega_1$-iterable above $\delta(\T)$. Then $b$ is
    the unique cofinal branch $c$ through $\T$ such that $\M_c^\T$ is
    well-founded and, if $n \geq 3$, $\Q(c, \T)$ is
    $\Pi^1_{n-1}$-iterable above $\delta(\T)$.
  \item Assume $M$ is $\omega_1$-iterable above $\delta$. Then $M$ is
    $\Pi^1_n$-iterable above $\delta$. 
  \item Assume $M$ is $\omega_1$-iterable above $\delta$ and $N$ is
    $\Pi^1_n$-iterable above $\delta$. Moreover assume that
    \[ M | \delta = N | \delta. \]
    Then we can successfully coiterate $M$ and $N$ above $\delta$,
    that means we have that
    \[ M \unlhd N \text{ or } N \unlhd M. \]
  \end{enumerate}
\end{lemma}

\begin{proof}
  Apply Lemma $2.2$ from \cite{St95} inside the model $M_{n-1}^\#(x)$,
  where $x$ is a real coding $M \oplus N$. This immediately gives
  Lemma \ref{lem:steel2.2} using Lemma \ref{lem:iterability} $(2)$,
  because we have that \[ \omega_2^{M_{n-1}^\#(x)} < \omega_1^V \]
  and therefore $M$ is $(\omega_1 + 1)$-iterable inside
  $M_{n-1}^\#(x)$ and by Lemma \ref{lem:corr} we have that $N$ is
  $\Pi^1_n$-iterable inside the model $M_{n-1}^\#(x)$ since
  $M_{n-1}^\#(x)$ is $\SIGMA^1_n$-correct in $V$.
\end{proof}

\begin{remark}
  Similarly as for Lemma \ref{lem:iterability} this lemma also holds
  true in the special case that $\delta = \omega$. More precisely in
  this case Lemma \ref{lem:steel2.2} $(3)$ holds true in the following
  sense. If $M$ and $N$ are $\omega$-sound, $\rho_\omega(M) = \omega$
  and $\rho_\omega(N) = \omega$, and if we assume
  $\omega_1$-iterability and $\Pi^1_n$-iterability for $M$ and $N$
  respectively, then we have that
  \[ M \unlhd N \text{ or } N \unlhd M. \]
\end{remark}

In Section \ref{ch:it} we in fact need the following strengthening of
Lemma \ref{lem:steel2.2} for odd $n$, which is proved in Lemma $3.3$
in \cite{St95}. That it holds for $\omega_1$-iterability instead of
$(\omega_1 + 1)$-iterability follows by the same argument as the one
we gave for Lemma \ref{lem:steel2.2} above. This lemma only holds for
odd $n$ because of the periodicity in the projective hierarchy. For
more details see \cite{St95}.

\begin{lemma} \label{lem:nofakeeven} Let $n \geq 1$ be \emph{odd} and
  assume that $M_{n-1}^\#(x)$ exists and is $\omega_1$-iterable for
  all reals $x$. Let $M$ and $N$ be countable premice which agree up
  to a common cutpoint $\delta$ such that $M$ and $N$ are
  $\delta$-sound and such that $\rho_\omega(M) \leq \delta$ and
  $\rho_\omega(N) \leq \delta$.  Assume that $M$ is $(n-1)$-small
  above $\delta$ and $N$ is not $(n-1)$-small above $\delta$. Moreover
  assume that $M$ is $\Pi^1_n$-iterable above $\delta$ and that $N$ is
  $\omega_1$-iterable above $\delta$. Then we have that
  \[ M \unlhd N. \]
\end{lemma}

For the proof of Lemma \ref{lem:KS} we need the following variant of
Lemma \ref{lem:steel2.2} which is a straightforward consequence of the
proof of Lemma $2.2$ in \cite{St95}, because the assumption that the
premice $M$ and $N$ both do not have definable Woodin cardinals yields
that $\Q$-structures exist in a coiteration of $M$ and $N$.

We say that an iteration tree $\U$ is a \emph{putative iteration
  tree}\index{putative iteration tree} if $\U$ satisfies all
properties of an iteration tree, but we allow the last model of $\U$
to be ill-founded, in case $\U$ has a last model.

\begin{lemma} \label{lem:itKS} Let $n \geq 2$ and assume that
  $M_{n-1}^\#(x)$ exists and is $\omega_1$-iterable for all reals $x$.
  Let $M$ and $N$ be countable premice which have a common cutpoint
  $\delta$ such that $M$ and $N$ are $(n-1)$-small above $\delta$ and
  solid above $\delta$. Assume that $M$ and $N$ both do not have
  definable Woodin cardinals and assume in addition that $M$ is
  $\omega_1$-iterable above $\delta$ and that $N$ is
  $\Pi^1_n$-iterable above $\delta$. Moreover assume that
  \[ M | \delta = N | \delta. \] Then we can successfully coiterate
  $M$ and $N$ above $\delta$ as in Lemma \ref{lem:iterability} $(3)$,
  that means here that there is an iteration tree $\T$ on $M$ and a
  putative iteration tree $\U$ on $N$ of length $\lambda + 1$ for some
  ordinal $\lambda$ such that we have
  \[ \M_\lambda^\T \unlhd \M_\lambda^\U \text{ or } \M_\lambda^\U
  \unlhd \M_\lambda^\T. \]
  So in the first case $\M_\lambda^\U$ need not be fully well-founded,
  but it is well-founded up to $\M_\lambda^\T \cap \Ord$. In the
  second case we have that $\M_\lambda^\U$ is fully well-founded and
  $\U$ is in fact an iteration tree.
\end{lemma}

As in the remark after the proof of Lemma \ref{lem:iterability} we get
that the following strengthening of Lemma \ref{lem:itKS} holds true,
where we replace ``no definable Woodin cardinals'' by ``no Woodin
cardinals''.

\begin{cor}\label{cor:itKS}
  Let $n \geq 2$ and assume that
  $M_{n-1}^\#(x)$ exists and is $\omega_1$-iterable for all reals $x$.
  Let $M$ and $N$ be countable premice which have a common cutpoint
  $\delta$ such that $M$ and $N$ are $(n-1)$-small above $\delta$ and
  solid above $\delta$. Assume that $M$ and $N$ both do not have
  Woodin cardinals and assume in addition that $M$ is
  $\omega_1$-iterable above $\delta$ and that $N$ is
  $\Pi^1_n$-iterable above $\delta$. Moreover assume that
  \[ M | \delta = N | \delta. \] Then we can successfully coiterate
  $M$ and $N$ above $\delta$ as in Lemma \ref{lem:iterability} $(3)$,
  that means here that there is an iteration tree $\T$ on $M$ and a
  putative iteration tree $\U$ on $N$ of length $\lambda + 1$ for some
  ordinal $\lambda$ such that we have
  \[ \M_\lambda^\T \unlhd \M_\lambda^\U \text{ or } \M_\lambda^\U
  \unlhd \M_\lambda^\T. \]
  So we again have that in the first case $\M_\lambda^\U$ need not be
  fully well-founded, but it is well-founded up to
  $\M_\lambda^\T \cap \Ord$, and in the second case we have that
  $\M_\lambda^\U$ is fully well-founded and $\U$ is in fact an
  iteration tree.\end{cor}

\begin{proof}
  To simplify the notation we again assume that $\delta = \omega$.  As
  in the remark after the proof of Lemma \ref{lem:iterability} we can
  just use Lemma \ref{lem:itKS} in the case that $M \cap \Ord$ and
  $N \cap \Ord$ are both not definably Woodin over $M$ and $N$
  respectively, because in this case $M$ and $N$ in fact do not have
  definable Woodin cardinals.

  So assume for example that $N \cap \Ord$ is definably Woodin over
  $N$. Let $x$ be a real coding $M \oplus N$ and consider the
  coiteration of $M$ and $N$ inside $M_{n-1}^\#(x)$. Let $\T$ and $\U$
  be the resulting trees on $M$ and $N$ respectively. Assume that the
  coiteration breaks down on the $N$-side, that means $\U$ is an
  iteration tree of limit length $\lambda$ such that there is no
  $\Q$-structure $\Q(\U)$ for $\U$ in $M_{n-1}^\#(x)$. Since $N$ has
  no Woodin cardinals, this can only be the case if
  \[ \M(\U) = \M_\lambda^\U, \]
  where $\M_\lambda^\U$ is the limit model for $\U$ which exists in
  $V$ since $N$ is $\Pi^1_n$-iterable in $V$. Therefore we have as in
  the remark after the proof of Lemma \ref{lem:iterability} that the
  coiteration on the $N$-side already finished.

  If we assume that $M \cap \Ord$ is definably Woodin over $M$ it
  follows by the same argument that the coiteration on the $M$-side
  already finished if it breaks down, because $M$ is
  $\omega_1$-iterable in $V$. 

  Therefore the coiteration of $M$ and $N$ is successful in the sense
  of Corollary \ref{cor:itKS} even if $M \cap \Ord$ or $N \cap \Ord$
  or both of them are definably Woodin over $M$ or $N$ respectively.
\end{proof}

We now aim to fix an order on $\OD$-sets. Therefore we first introduce
the following notion.

\begin{definition}\label{def:minimaltripleOD}
  Let $x \in \OD$. Then we say
  $(n, (\alpha_0, \dots, \alpha_n), \varphi)$ is the \emph{minimal
    triple defining $x$} iff
  $(n, (\alpha_0, \dots, \alpha_n), \varphi)$ is the minimal triple
  according to the lexicographical order on triples (using the
  lexicographical order on tuples of ordinals of length $n+1$ and the
  order on Gödel numbers for formulae) such that
  \[ x = \{ z \st V_{\alpha_0} \vDash \varphi(z, \alpha_1, \dots,
  \alpha_n) \}. \]
\end{definition}

\begin{definition} \label{def:orderODsets} Let $x,y \in \OD$. Moreover
  let $(n, (\alpha_0, \dots, \alpha_n), \varphi)$ and
  $(m, (\beta_0, \dots, \beta_m), \psi)$ be the minimal triples
  defining $x$ and $y$ respectively as in Definition
  \ref{def:minimaltripleOD}.  Then we say \emph{$x$ is less than $y$
    in the order on the $\OD$-sets} and write
  \[x <_{\OD} y,\]
  \index{ODorder@$x <_{\OD} y$}iff
  $(n, (\alpha_0, \dots, \alpha_n), \varphi)$ is smaller than
  $(m, (\beta_0, \dots, \beta_m), \psi)$ in the lexicographical order
  on triples, that means iff either
  \begin{enumerate}[$(a)$]
  \item $n < m$, or
  \item $n = m$ and
    $(\alpha_0, \dots, \alpha_n) <_{lex} (\beta_0, \dots, \beta_m)$,
    where $<_{lex}$ denotes the lexicographical order on tuples of
    ordinals of length $n+1$, or
  \item $(\alpha_0, \dots, \alpha_n) = (\beta_0, \dots, \beta_m)$ and
    $\ulcorner \varphi \urcorner < \ulcorner \psi \urcorner$, where
    $\ulcorner \varphi \urcorner$ and $\ulcorner \psi \urcorner$
    denote the Gödel numbers of the formulae $\varphi$ and $\psi$
    respectively.
  \end{enumerate}
\end{definition}

Fix the following notation for the rest of this paper.

\begin{definition}\label{def:L[E]constr}
  Let $x,y$ be reals, $M$ an $x$-premouse which is a model of $\ZFC$
  and $y \in M$.
  \begin{enumerate}
  \item $L[E](y)^{M}$ denotes the resulting model of a fully
    backgrounded extender construction above the real $y$ as in
    \cite{MS94} performed inside the model $M$, but with the smallness
    hypothesis weakened to allow $\omega$-small premice in the
    construction. \footnote{For more details about such a construction
      in a more general setting see also \cite{St93}.}
  \item We recursively say that a premouse $N$ is an \emph{iterated
      $L[E]$ in $M$} iff $N = M$ or $N$ is an iterated $L[E]$ in
    $L[E](z)^M$ for some $z \in M$ with $z \geq_T x$.
  \end{enumerate}
\end{definition}

Using the order on $\OD$-sets defined above we can prove the following
lemma, which will also be used in the proof of Lemma \ref{lem:KS}.

\begin{lemma}\label{lemma1und2}
  Let $n \geq 2$ and assume that $M_{n-1}^\#(z)$ exists and is
  $\omega_1$-iterable for all reals $z$. Let $x \leq_T y$ be
  reals and let $M$ be an $y$-premouse which is an iterated
  $L[E]$ in $M_{n-1}(x)$. Then
  \begin{enumerate}[$(a)$]
  \item $M$ is $(n-1)$-small, and
  \item the premice $M$ and $M_{n-1}(x)$ have the same
  sets of reals and the same $\OD$-sets of reals in the same order.
  \end{enumerate}
\end{lemma}

\begin{proof}
  We start with some general remarks about the premice we consider.
  The premouse $M$ still has $n-1$ Woodin cardinals and is
  $\omega_1$-iterable via an iteration strategy which is induced by
  the $\omega_1$-iteration strategy for $M_{n-1}(x)$ by §$11$ and
  §$12$ in \cite{MS94}. Therefore the premouse $L[E](x)^{M},$
  constructed inside the model $M$ as in Definition
  \ref{def:L[E]constr}, also has $n-1$ Woodin cardinals and is
  $\omega_1$-iterable. The following claim proves part $(a)$ of the
  lemma.

  \begin{claim}\label{cl:LEnsmall}
    The $y$-premouse $M$ is $(n-1)$-small.
  \end{claim}
  \begin{proof}
    Assume toward a contradiction that $M$ is not $(n-1)$-small and
    let $N_y^\#$ be the shortest initial segment of $M$ which is not
    $(n-1)$-small. That means we choose $N_y^\# \unlhd M$ such that it
    is not $(n-1)$-small, but every proper initial segment of $N_y^\#$
    is $(n-1)$-small. In particular $N_y^\#$ is an active
    $y$-premouse, so let $F$ be the top extender of $N_y^\#$.
    Moreover let $N_y$ be the model obtained from $N_y^\#$ by
    iterating the top extender $F$ out of the universe working inside
    the model $M_{n-1}(x)$.

    Now consider $L[E](x)^{N_y}$ and let $N$ be the active
    $x$-premouse obtained by adding $F \cap L[E](x)^{N_y}$ as a top
    extender to an initial segment of $L[E](x)^{N_y}$, as in Section
    $2$ in \cite{FNS10} to ensure that $N$ is a premouse. The main
    result in \cite{FNS10} yields that $N$ is $\omega_1$-iterable in
    $V$, because $M$, $N_y^\#$ and $L[E](x)^{N_y}$ inherit the
    iterability from $M_{n-1}(x)$ as in §11 and §12 in
    \cite{MS94}. Moreover $N$ is not $(n-1)$-small.

    Let $N_x^\#$ be the shortest initial segment of $N$ which is not
    $(n-1)$-small. By Lemma \ref{lem:iterability} we can successfully
    compare the $x$-premice $N_x^\#$ and $M_{n-1}^\#(x)$, because both
    premice are $\omega_1$-iterable in $V$ and every proper initial
    segment of one of them is $(n-1)$-small. Therefore we have that
    in fact $N_x^\# = M_{n-1}^\#(x)$ and thus
    \[ \mathbb{R} \cap N_x^\# = \mathbb{R} \cap M_{n-1}^\#(x) =
      \mathbb{R} \cap M_{n-1}(x). \] But $N_x^\#$ is by construction a
    countable premouse in $M_{n-1}(x)$, this is a contradiction.
  \end{proof}

  A simplified version of this argument shows that the $x$-premouse
  $L[E](x)^{M}$ is $(n-1)$-small.

  For the proof of $(b)$, we now fix a real $z$ which codes the
  $x$-premouse $M_{n-1}^\#(x)$ and work inside the model $M_{n-1}(z)$
  for a while.

  We have that every proper initial segment of $M_{n-1}^\#(x)$ is
  $(n-1)$-small, $\rho_\omega(M_{n-1}^\#(x)) = \omega$, and
  $M_{n-1}^\#(x)$ is $\omega$-sound. Since
  \[ \omega_2^{M_{n-1}^\#(z)} < \omega_1^V, \]
  this yields that $M_{n-1}^\#(x)$ is $(\omega_1 + 1)$-iterable inside
  $M_{n-1}^\#(z)$ by Lemma \ref{lem:iterability} $(2)$. So in
  particular, working inside $M_{n-1}(z)$, the $x$-premouse
  $M_{n-1}(x)$, obtained from $M_{n-1}^\#(x)$ by iterating the top
  measure out of the universe, is $(\omega_1 + 1)$-iterable.

  Let $\kappa$ denote the least measurable cardinal in
  $M_{n-1}(x)$. So in particular $\kappa$ is below the least
  measurable cardinal in $M$.

  \begin{claim}\label{cl:agreebelowleastmeas}
    The $x$-premice $M_{n-1}(x)$ and $L[E](x)^{M}$ agree below
    $\kappa$.
  \end{claim}

  \begin{proof}
    The $y$-premouse $M$ is $(\omega_1 + 1)$-iterable inside
    $M_{n-1}(z)$ via an iteration strategy which is induced by the
    iteration strategy for $M_{n-1}(x)$. Thus the $x$-premouse
    $L[E](x)^M$ is $(\omega_1 +1)$-iterable inside $M_{n-1}(z)$ via
    the iteration strategy inherited from $M$. In particular we can
    successfully compare the $x$-premice $M_{n-1}(x)$ and
    $L[E](x)^{M}$ inside the model $M_{n-1}(z)$ using this iteration
    strategy and they coiterate to the same premouse. This yields the
    claim.
  \end{proof}

  Now we can finally show the following claim, which will finish the
  proof of Lemma \ref{lemma1und2}.

  \begin{claim}
    $M_{n-1}(x)$ and $M$ have the same sets of reals and the same
    $\OD$-sets of reals in the same order.
  \end{claim}

  \begin{proof}
    Claim \ref{cl:agreebelowleastmeas} implies that
    \[ V_\kappa^{M_{n-1}(x)} \supseteq V_\kappa^{M}
    \supseteq V_\kappa^{L[E](x)^{M}} =
    V_\kappa^{M_{n-1}(x)}, \] and therefore
  \[ V_\kappa^{M_{n-1}(x)} = V_\kappa^{M}. \] Thus we can consider
  $M_{n-1}(x)$ and $M$ as $V_\kappa^{M_{n-1}(x)}$-premice and, still
  working in $M_{n-1}(z)$, we can successfully compare them by the
  same argument as in the proof of Claim \ref{cl:agreebelowleastmeas},
  using the iteration strategy for $M$ which is induced by the
  iteration strategy for $M_{n-1}(x)$. The
  $V_\kappa^{M_{n-1}(x)}$-premice $M_{n-1}(x)$ and $M$ coiterate to
  the same model $W$ and hence they have the same sets of reals and
  the same $\OD$-sets of reals in the same order by the following
  argument.

  Let $A$ be an $\OD$-set of reals in $M_{n-1}(x)$ and let
  $(n, (\alpha_0, \dots, \alpha_n), \varphi)$ be the minimal triple
  defining $A$. As $M_{n-1}(x)$ and $M$ have the same sets of reals,
  in particular $A \in M$. Let $i: M_{n-1}(x) \rightarrow W$ and
  $j:M \rightarrow W$ be the iteration embeddings. Then
  $(n, (i(\alpha_0), \dots, i(\alpha_n)), \varphi)$ is by elementarity
  the minimal triple defining $A$ in $W$ as $M_{n-1}(x)$ and $W$ agree
  about their $V_\kappa$. In particular $A$ is an $\OD$-set of reals
  in $W$. By elementarity of $j$, $A$ is also an $\OD$-set of reals in
  $M$. Let $(m, (\beta_0, \dots, \beta_m), \psi)$ be the minimal
  triple witnessing this. Then, by elementarity of $j$ again,
  $(m, (j(\beta_0), \dots, j(\beta_m)), \psi)$ is the minimal triple
  defining $A$ in $W$. Hence, by minimality
  $(n, (i(\alpha_0), \dots, i(\alpha_n)), \varphi) = (m, (j(\beta_0),
  \dots, j(\beta_m)), \psi)$. This shows that $M$ and $M_{n-1}(x)$
  have the same $\OD$-sets of reals and agree on the order on
  $\OD$-sets defined in Definition \ref{def:orderODsets}.
  \end{proof}
  This proves Lemma \ref{lemma1und2}.
\end{proof}

Motivated by this lemma we introduce the following notation.

\begin{definition}\label{def:sim}
  For premice $M$ and $N$ we write $M \sim N$ \index{sim@$M \sim N$}
  iff $M$ and $N$ have the same sets of reals and the same $\OD$-sets
  of reals in the same order.
\end{definition}

In the proof of Lemma \ref{lem:KS} we will also need the following
lemma.

\begin{lemma}\label{lemma3}
  Let $n \geq 2$ and assume that $M_{n-1}^\#(z)$ exists and is
  $\omega_1$-iterable for all reals $z$. Let $x \leq_T y$ be reals and
  let $M$ be an $y$-premouse which is an iterated $L[E]$ in
  $M_{n-1}(x)$. Let $\delta = \delta_M$ denote the least Woodin
  cardinal in $M$, let $z \in M$ be a real with $z \geq_T y$, and let
  \[ (\M_\xi, \mathcal{N}_\xi \st \xi \in \Ord) \] be the sequence of
  models obtained from a fully backgrounded extender construction
  inside $M$ above $z$ as in Definition \ref{def:L[E]constr}, such
  that
  \[ \M_{\xi + 1} = \mathcal{C}_\omega(\mathcal{N}_{\xi + 1}) \] and
  let $L[E](z)$ be the resulting model. Then we have for all
  $\xi \geq \delta$ that
  \[ \rho_\omega(\M_\xi) \geq \delta, \] and therefore
  \[ \M_\delta = L[E](z) | \delta. \]
\end{lemma}

\begin{proof}\footnote{The authors would like to thank John Steel for
    communicating this argument to them.}  Work in $M$. Assume not and
  let $\N_\xi$ be the least model with $\xi \geq \delta$ such that
  $\rho_\omega(\N_\xi) < \delta$. By construction (see Theorem $11.3$
  in \cite{MS94}) we have that $\delta$ is a Woodin cardinal in
  $\N_\xi$. An argument as in the proof of Claim \ref{cl:LEnsmall} in
  Lemma \ref{lemma1und2} shows that $\N_\delta$ is $(n-1)$-small, so
  in particular $\delta$ cannot be a limit of Woodin cardinals in
  $\N_\delta$. Hence there is an $\eta < \delta$ such that
  $\rho_\omega(\N_\xi) < \eta$ and all Woodin cardinals below $\delta$
  in $\N_\delta$ are in fact below $\eta$. Using the regularity of
  $\delta$, let $X$ be a fully elementary substructure of $\N_\xi$
  such that $\eta \subset X$, $\card{X} < \delta$ and
  $X \cap \delta \in \delta$.

  Let $\bar{X}$ be the Mostowski collapse of $X$ and let
  $\sigma: \bar{X} \rightarrow \N_\xi$ be the uncollapse map, i.e.
  \[ \bar{X} \overset{\sigma}{\cong} X \prec \N_\xi. \] Let
  $\bar{\delta} = X \cap \delta$, i.e.
  $\sigma(\bar{\delta}) = \delta$, and let $\mathcal{P}$ be the
  $\bar{\delta}$-core of $\bar{X}$. Then $\mathcal{P}$ is
  $\bar{\delta}$-sound and $\rho_\omega(\mathcal{P}) < \bar{\delta}$
  as $\rho_\omega(\bar{X}) < \bar{\delta}$. Moreover $\bar{\delta}$ is
  Woodin in $\bar{X}$ by elementarity and thus Woodin in
  $\mathcal{P}$.

  As $\eta < \bar{\delta} < \delta$, we have that $\bar{\delta}$
  cannot be Woodin in $\N_\delta$. Let
  $\mathcal{P}^\prime \unlhd \N_\delta$ be the shortest initial
  segment of $\N_\delta$ such that
  $\bar{\delta} \leq \mathcal{P}^\prime \cap \Ord$ and $\bar{\delta}$
  is not definably Woodin over $\mathcal{P}^\prime$. Then
  $\mathcal{P}^\prime$ is $\bar{\delta}$-sound and
  $\rho_\omega(\mathcal{P}^\prime) \leq \bar{\delta}$.

  $\mathcal{P}$ and $\mathcal{P}^\prime$ are both countably iterable
  in $M$, i.e. countable substructures of $\mathcal{P}$ and
  $\mathcal{P}^\prime$ are $(\omega_1 +1)$-iterable in $M$, by an
  argument similar to Lemma \ref{lem:iterability}. Therefore they can
  be successfully compared inside $M$ and in fact
  $\mathcal{P} = \mathcal{P}^\prime$. Thus the witness for the
  statement $\rho_\omega(\mathcal{P}) < \bar{\delta}$ is an element of
  $\N_\delta$. But $\bar{\delta}$ is a cardinal in
  $\mathcal{P} = \mathcal{P}^\prime \unlhd \N_\delta$, a
  contradiction.
\end{proof}


\subsection{$\OD$-Determinacy for an Initial Segment of $M_{n-1}$}

Now we can turn to the proof of Lemma \ref{lem:KS}, which is a
generalization of Theorem $3.1$ in \cite{KS85}. Recall

\begin{namedthm}{Lemma \ref{lem:KS}}
  Let $n \geq 1$. Assume that $M_{n-1}^\#(x)$ exists and is
  $\omega_1$-iterable for all $x \in {}^\omega\omega$ and that all
  $\Sigma^1_{n+1}$-definable sets of reals are determined. Then there
  exists a real $y_0$ such that for all reals $x \geq_T y_0$,
  \[ M_{n-1}(x)|\delta_x \vDash \ODdet, \]
  where $\delta_x$ denotes the least Woodin cardinal in $M_{n-1}(x)$
  if $n>1$ and $\delta_x$ denotes the least $x$-indiscernible in
  $M_0(x) = L[x]$ if $n = 1$.
\end{namedthm}

\begin{proof}[Proof of Lemma \ref{lem:KS}]
\setcounter{cl}{0}
  For $n = 1$ we have that Lemma \ref{lem:KS} immediately follows from
  Theorem $3.1$ in \cite{KS85}. So assume that $n > 1$ and assume
  further toward a contradiction that there is no such real $y_0$ as
  in the statement of Lemma \ref{lem:KS}.

  Then there are cofinally (in the Turing degrees) many $y \in \BS$
  such that there exists an $M \unlhd M_{n-1}(y)|\delta_y$ with
  $$ M \vDash \neg \ODdet. $$

  We want to consider $x$-premice $M$ for $x \in \BS$ which have the
  following properties:
\begin{enumerate}[(1)]
\item $M \vDash \ZFC$, $M$ is countable, and we have that the
  following formula holds true:
\begin{eqnarray*}
  & \text{for all } m < \omega \text{ and for all } (x_k \st k < m) \text{ and } (M^k
    \st k \leq m) \text{ such that } \\
  & x_k \in M^k \cap \BS, \; x_{k+1} \geq_T x_k \geq_T x, \\
  & M^0 = M, \, M^{k+1} = L[E](x_k)^{M^k}. \text{ We have}\\
  & \text{ for all } k+1 \leq m \text{ that } M^{k+1} \sim M^k,\\
  & M^k \text{ is } (n-1)\text{-small, and}\\
  & M^k \text{ does not have a
    Woodin cardinal,}
\end{eqnarray*}
where the relation ``$\sim$'' is as defined in Definition
\ref{def:sim},
\item \[ M \vDash \neg \ODdet, \]
\item for all $\xi < M \cap \OR$ such that $M|\xi$ satisfies property
  $(1)$, \[ M | \xi \vDash \ODdet, \] and
\item $M$ is $\Pi^1_n$-iterable.
\end{enumerate}

If $M$ is an $x$-premouse, then we have for every real $y$ such that
$y \geq_T x$ and $y \in M$, that
\[M \vDash (1) \; \; \Rightarrow \; \; L[E](y)^M \vDash (1), \] where
$L[E](y)^M$ is as in Definition \ref{def:L[E]constr}.

We first show that there exists a Turing cone of reals $x$ such that
there exists an $x$-premouse $M$ satisfying properties $(1)$ - $(4)$
as above.

\begin{claim}\label{lemma4}
For cofinally many $x \in \BS$, $M_{n-1}(x)|\delta_x$ satisfies
properties $(1)$ and $(2)$.
\end{claim}
\begin{proof}
  By assumption there are cofinally many $x \in \BS$ such that
  \[ M_{n-1}(x)|\delta_x \vDash \neg \ODdet. \] Pick such an
  $x \in \BS$. This already implies that property $(2)$ holds true for
  $M_{n-1}(x)|\delta_x$. In addition $M_{n-1}(x)|\delta_x$ is
  obviously a countable $\ZFC$ model without a Woodin cardinal.  To
  show the formula in property $(1)$, let $(x_k \st k < m)$ and
  $(M^k \st k \leq m)$ be as in property $(1)$ with
  $M^0 = M_{n-1}(x)|\delta_x$. Moreover let $\hat{M}^0 = M_{n-1}(x)$
  and $\hat{M}^{k+1} = L[E](x_k)^{\hat{M}^k}$. An inductive
  application of Lemma \ref{lemma3} implies that for all $k < m$,
  \[ \hat{M}^{k+1} | \delta_x = L[E](x_k)^{\hat{M}^k} | \delta_x =
    L[E](x_k)^{\hat{M}^k | \delta_x} = L[E](x_k)^{M^k} = M^{k+1}. \]
  Moreover Lemma \ref{lemma1und2} yields that $\hat{M}^k$ and thus
  $M^k$ is $(n-1)$-small and
  \[ \hat{M}^k \sim M_{n-1}(x) \sim \hat{M}^{k+1}, \] for all
  $k+1 \leq m$. Hence
  \[ M_{n-1}(x)|\delta_x \sim \hat{M}^{k+1}|\delta_x = M^{k+1}. \]
  Therefore the formula in property $(1)$ holds true for
  $M^0 = M_{n-1}(x)|\delta_x$ as we also have that no $M^k$ has a
  Woodin cardinal.
\end{proof}

By Claim \ref{lemma4} there are cofinally many $x \in \BS$ such that
there exists an $x$-premouse $M$ which satisfies properties
$(1) - (4)$ defined above. Such $x$-premice $M$ can be obtained by
taking the smallest initial segment of $M_{n-1}(x)|\delta_x$ which
satisfies properties $(1)$ and $(2)$ defined above. Moreover the set
\[ A = \{ x \in \BS \st \text{there is an } x \text{-premouse } M
\text{ with } (1) - (4)\} \]
is $\Sigma^1_{n+1}$-definable and Turing invariant. So by
$\Sigma^1_{n+1}$ (Turing-)determinacy there exists a cone of such
reals $x \in A$, because the set $A$ cannot be completely disjoint
from a cone of reals since there are cofinally many reals $x \in A$ as
argued above. Let $v$ be a base of this cone and consider a real
$x \geq_T v$ in the cone.

\begin{claim}
Let $M$ be an $\omega_1$-iterable $x$-premouse with properties $(1)$
and $(2)$ and let $N$ be an $x$-premouse satisfying properties
$(1)-(4)$. Then $N$ cannot win the comparison against $M$.
\end{claim}

\begin{proof}
  By Corollary \ref{cor:itKS} we can successfully coiterate the
  $x$-premice $M$ and $N$ because they are both $(n-1)$-small, solid
  and without Woodin cardinals and moreover $M$ is $\omega_1$-iterable
  and $N$ is $\Pi^1_n$-iterable. Let $\T$ and $\U$ be the resulting
  trees of length $\lambda + 1$ for some ordinal $\lambda$ on $M$ and
  $N$ respectively, in particular $\U$ might be a putative iteration
  tree.

  Assume toward a contradiction that $N$ wins the comparison against
  $M$ and consider the following two cases. 

\begin{minipage}{\textwidth}
  \bigskip \textbf{Case 1.} We have
  $\M_\lambda^\T \lhd \M_\lambda^\U$. \bigskip
\end{minipage}

In this case $\M_\lambda^\U$ need not be fully well-founded, but this
will not affect our argument to follow, because we have that
$\M_\lambda^\U$ is well-founded up to $\M_\lambda^\T \cap \Ord$.  So
there exists an ordinal $\alpha < \M_\lambda^\U \cap \Ord$ such that
$\M_\lambda^\U | \alpha = \M_\lambda^\T$ and we have by elementarity
that
\[ \M_\lambda^\T \vDash \neg \ODdet, \]
since property $(2)$ holds for $M$. Since
$\alpha = \M_\lambda^\T \cap \Ord$ it follows that
\[ \M_\lambda^\U | \alpha \vDash \neg \ODdet. \]
Moreover we have that $\M_\lambda^\U | \alpha \vDash (1),$ because we
have by elementarity that $\M_\lambda^\T \vDash (1)$.  If there is no
drop in the iteration from $N$ to $\M_\lambda^\U$, this contradicts
the minimality property $(3)$ for $N$ by elementarity. But even if
there is a drop on the main branch in $\U$ this statement is
transferred along the branch to $N$ by the following argument. 

Assume there is a drop at stage $\beta+1$ on the main branch through
$\U$, that means $\M_{\beta+1}^*$ is a proper initial segment of
$\M_\gamma$ for $\gamma = \pred_U(\beta + 1)$, where $\M_{\beta+1}^*$
is the model to which the next extender $F_\beta$ from the $\M_\beta$
sequence is applied in the iteration as introduced in Section $3.1$ in
\cite{St10}.  Since there can only be finitely many drops along the
main branch through the iteration tree $\U$, we can assume further
without loss of generality that this is the only drop along the main
branch through $\U$. (If there is more than one drop on the main
branch through $\U$, we repeat the argument to follow for the
remaining drops.) Then by elementarity there is an ordinal
$\alpha^\prime < \M_{\beta+1}^* \cap \Ord < \M_\gamma \cap \Ord$ such
that
\[ \M_{\beta+1}^* | \alpha^\prime \vDash (1) + \neg \ODdet. \]
But then also
\[ \M_\gamma | \alpha^\prime \vDash (1) + \neg \ODdet, \]
and therefore by elementarity there is an ordinal
$\alpha^{\prime \prime} < N \cap \Ord$ such that
\[ N | \alpha^{\prime \prime} \vDash (1) + \neg \ODdet. \]
This now again contradicts the minimality property $(3)$ for $N$.

\begin{minipage}{\textwidth}
  \bigskip \textbf{Case 2.} We have $\M_\lambda^\T = \M_\lambda^\U$,
  there is no drop on the main branch in the iteration from $M$ to
  $\M_\lambda^\T$, but there is a drop on the main branch in the
  iteration from $N$ to $\M_\lambda^\U$.  \bigskip
\end{minipage}

This immediately is a contradiction because it implies that we have
$\M_\lambda^\T \vDash \ZFC$, but at the same time
$\rho_\omega(\M_\lambda^\U) < \M_\lambda^\U \cap \Ord$.
\end{proof}

So in the situation of the above claim, the premouse $N$ is elementary
embedded into an $\omega_1$-iterable model and thus
$\omega_1$-iterable itself. As a corollary we obtain the following
claim.

\begin{claim}
\label{cl:it}
If there is an $\omega_1$-iterable $x$-premouse $M$ with properties $(1)$
and $(2)$, then every $x$-premouse $N$ satisfying properties $(1)-(4)$
is in fact already $\omega_1$-iterable.
\end{claim}

Since by Claim \ref{lemma4} there are cofinally many reals $x$ such
that the $\omega_1$-iterable $x$-premouse $M_{n-1}(x)|\delta_x$
satisfies properties $(1)$ and $(2)$, Claim \ref{cl:it} yields that
the following claim holds true.

\begin{claim}\label{cl:1-4w1it}
  There are cofinally many reals $x$ such that every $x$-premouse
  satisfying properties $(1)-(4)$ is in fact $\omega_1$-iterable.
\end{claim}

Consider the game $G$ which is a generalization of the
Kechris-Solovay game in \cite{KS85} and defined as follows.

\begin{figure}[h]
\centering
\begin{tabular}{c|cc}
 $\mathrm{I}$ & $x \oplus a$ & \\ \hline
 $\mathrm{II}$ & & $y \oplus b$ 
\end{tabular}
\; \; for $x,a,y,b \in {}^\omega\omega$.
\end{figure}

The players $\mathrm{I}$ and $\mathrm{II}$ alternate playing natural
numbers and the game lasts $\omega$ steps. Say player $\mathrm{I}$
produces a real $x \oplus a$ and player $\mathrm{II}$ produces a real
$y \oplus b$. Then player $\mathrm{I}$ wins $G$ iff there exists an
$(x \oplus y)$-premouse $M$ which satisfies properties $(1)-(4)$ and
if $A_M$ denotes the least $\OD$-set of reals in $M$ which is not
determined, then $a \oplus b \in A_M$.

This game is $\Sigma^1_{n+1}$-definable and therefore determined. So
say first player $\mathrm{I}$ has a winning strategy $\tau$ for
$G$. Recall that $v$ denotes a base for a cone of reals $x$ such that
there exists an $x$-premouse which satisfies properties $(1) - (4)$
and pick a real $z^* \geq_T x_\tau \oplus v$, where $x_\tau$ is a real
coding the winning strategy $\tau$ for player $\mathrm{I}$, such that
every $z^*$-premouse which satisfies properties $(1)-(4)$ is in fact
$\omega_1$-iterable (using Claim \ref{cl:1-4w1it}).

We now aim to construct a real $z \geq_T z^*$ such that there is a
$z$-premouse $N$ which satisfies properties $(1)-(4)$, is
$\omega_1$-iterable, and satisfies the following additional property
$(3^*)$.

\begin{enumerate}[$(3^*)$]
\item For all reals $y \in N$ such that $y \geq_T z$ there exists no
  ordinal $\xi < N \cap \Ord$ such that
  \[ L[E](y)^N | \xi \vDash (1) + \neg \ODdet. \]
\end{enumerate}

Let $M^0$ be an arbitrary $z^*$-premouse satisfying properties
$(1)-(4)$, which is therefore $\omega_1$-iterable. Assume that if we
let $N = M^0$ and $z = z^*$, then property $(3^*)$ is not
satisfied. So there is a real $y \in M^0$ with $y \geq_T z^*$
witnessing the failure of property $(3^*)$ in $M^0$. Let $y_0 \in M^0$
with $y_0 \geq_T z^*$ be a witness for that fact such that the ordinal
$\xi^0 < L[E](y_0)^{M^0} \cap \Ord$ with
\[ L[E](y_0)^{M^0} | \xi^0 \vDash (1) + \neg \ODdet \]
is minimal. Let $M^1 = L[E](y_0)^{M^0} | \xi^0$ and assume that
\[ M^1 \nvDash (3^*), \]
because otherwise we could pick $N = M^1$ and $z = y_0$ and the
construction would be finished. Then as before there is a real
$y_1 \in M^1$ with $y_1 \geq_T y_0$ and a minimal ordinal
$\xi^1 < L[E](y_1)^{M^1} \cap \Ord = M^1 \cap \Ord = \xi^0$ such
that \[ L[E](y_1)^{M^1} | \xi^1 \vDash (1) + \neg \ODdet. \]
This construction has to stop at a finite stage, because otherwise we
have that $\xi^0 > \xi^1 > \dots$ is an infinite descending chain of
ordinals. Therefore there is a natural number $n < \omega$ such that
\[M^n = L[E](y_{n-1})^{M^{n-1}} | \xi^{n-1}, \] and
\[ M^n \vDash (3^*). \] Let $z = y_{n-1}$ and $N = M^n$. Then we have
that $z \geq_T z^*$ and $N$ is a $z$-premouse which satisfies
properties $(1)$ and $(3^*)$. Moreover by minimality of the ordinal
$\xi^{n-1}$ we have that $N$ satisfies property $(3)$. From the
construction we also get that \[N \vDash \neg \ODdet.\] Furthermore
$N$ inherits property $(4)$ and the $\omega_1$-iterability from $M^0$
since it is obtained by performing multiple fully backgrounded
extender constructions inside the $\omega_1$-iterable premouse
$M^0$. Thus $N$ and $z$ are as desired.

Let $A_N$ denote the least non-determined $\OD$-set of reals in $N$.
We define a strategy $\tau^*$ for player $\mathrm{I}$ in the usual
Gale-Stewart game $G(A_N)$ with payoff set $A_N$ played inside the
model $N$ as follows. Assume player $\mathrm{II}$ produces the real
$b \in N$. Then we consider the following run of the original game $G$
defined above:

\begin{figure}[h]
\centering
\begin{tabular}{c|cc}
 $\mathrm{I}$ & $x \oplus a = \tau((z \oplus b) \oplus b)$ & \\ \hline
 $\mathrm{II}$ & & $(z \oplus b) \oplus b$ 
\end{tabular}
\end{figure}

Player $\mathrm{II}$ plays the real $(z \oplus b) \oplus b$ and player
$\mathrm{I}$ responds with the real $x \oplus a$ according to his
winning strategy $\tau$ in $G$. Note that this run of the game $G$ is
in the model $N$. We define the strategy $\tau^*$ such that in a run
of the game $G(A_N)$ inside $N$ according to $\tau^*$ player
$\mathrm{I}$ has to respond to the real $b$ with producing the real
$a$.

\begin{figure}[h]
\centering
\begin{tabular}{c|cc}
 $\mathrm{I}$ & $a = \tau^*(b)$ & \\ \hline
 $\mathrm{II}$ & & $b$ 
\end{tabular}
\end{figure}

\begin{claim} \label{cl:tau*} $\tau^*$ is a winning strategy for
  player $\mathrm{I}$ in the Gale-Stewart game with payoff set $A_N$
  played in $N$.
\end{claim}

This claim implies that $A_N$ is determined in $N$, contradicting that
$A_N$ was assumed to be the least non-determined $\OD$-set of reals in
$N$.

\begin{proof}[Proof of Claim \ref{cl:tau*}]
  Since $\tau$ is a winning strategy for player $\mathrm{I}$ in the
  original game $G$, there exists an
  $(x \oplus (z \oplus b))$-premouse $N^\prime$ which satisfies
  properties $(1)-(4)$ such that
  \[ a \oplus b \in A_{N^\prime}, \]
  where $A_{N^\prime}$ denotes the least non-determined $\OD$-set of
  reals in $N^\prime$.

  We want to show that \[ A_{N^\prime} = A_N \]
  in order to conclude that $\tau^*$ is a winning strategy for player
  $\mathrm{I}$ in the Gale-Stewart game with payoff set $A_N$ played
  in $N$.  

  Property $(1)$ yields that
  \[ L[E](x \oplus (z \oplus b))^{N} \sim N, \]
  because $x,z,b \in N$ and $x \oplus (z \oplus b) \geq_T
  z$.
  Therefore $L[E](x \oplus (z \oplus b))^{N}$ is an
  $(x \oplus (z \oplus b))$-premouse which satisfies property $(2)$,
  because it has the same sets of reals and the same $\OD$-sets of
  reals as $N$ and hence
  \[ L[E](x \oplus (z \oplus b))^{N} \vDash \neg \ODdet. \] Moreover
  $L[E](x \oplus (z \oplus b))^{N}$ inherits property $(1)$ from $N$.

  Since $L[E](x \oplus (z \oplus b))^{N}$ is the result of a fully
  backgrounded extender construction inside the $\omega_1$-iterable
  premouse $N$, it is $\omega_1$-iterable itself. Therefore Claim
  \ref{cl:it} yields that in particular the
  $(x \oplus (z \oplus b))$-premouse $N^\prime$ is also
  $\omega_1$-iterable, because it was choosen such that it satisfies
  properties $(1)-(4)$.

  So we can coiterate $L[E](x \oplus (z \oplus b))^{N}$ and $N^\prime$
  by the remark after Lemma \ref{lem:iterability} since they are both
  $(n-1)$-small $\omega_1$-iterable $(x \oplus (z \oplus b))$-premice
  which do not have Woodin cardinals. Thus by the minimality of
  $N^\prime$ from property $(3)$ and an argument analogous to the one
  we already gave in the proof of Claim \ref{cl:it}, we have that
  \[ L[E](x \oplus (z \oplus b))^{N} \geq^* N^\prime, \] where
  $\leq^*$ denotes the usual mouse order\footnote{We say an
    $\omega_1$-iterable premouse $M$ is smaller or equal in the
    \emph{mouse order} \index{mouseorder@mouse order $\leq^*$} than an
    $\omega_1$-iterable premouse $N$ and write ``$M \leq^* N$'' iff
    $M$ and $N$ successfully coiterate to premice $M^*$ and $N^*$ such
    that $M^* \unlhd N^*$ and there are no drops on the main branch in
    the iteration from $M$ to $M^*$. Moreover we say that $N$ and $M$
    are equal in the mouse order and write ``$M =^* N$'' iff
    $M \leq^* N \text{ and } N \leq^* M.$}. Moreover we have
  minimality for the premouse $L[E](x \oplus (z \oplus b))^{N}$ in the
  sense of property $(3^*)$ for $N$. This yields again by an argument
  analogous to the one we already gave in the proof of Claim
  \ref{cl:it} that
  \[ L[E](x \oplus (z \oplus b))^{N} \leq^* N^\prime. \]

  Therefore we have that in fact
  \[ L[E](x \oplus (z \oplus b))^{N} =^* N^\prime, \] and hence
  \[ L[E](x \oplus (z \oplus b))^{N} \sim N^\prime. \]
  Using $L[E](x \oplus (z \oplus b))^{N} \sim N$ it follows that
  \[ N \sim N^\prime, \] and thus $A_N = A_{N^\prime}$.
\end{proof}

Now suppose player $\mathrm{II}$ has a winning strategy $\sigma$ in
the game $G$ introduced above and recall that $v$ is a base of a cone
of reals $x$ such that there exists an $x$-premouse which satisfies
properties $(1) - (4)$.  As in the situation when player $\mathrm{I}$
has a winning strategy, we pick a real $z^* \geq_T x_\sigma \oplus v$,
where $x_\sigma$ is a real coding the winning strategy $\sigma$ for
player $\mathrm{II}$, such that every $z^*$-premouse which satisfies
properties $(1)-(4)$ is already $\omega_1$-iterable (see Claims
\ref{lemma4} and \ref{cl:it}).

As in the argument for player $\mathrm{I}$ we can construct a real
$z \geq_T z^*$ such that there exists a $z$-premouse $N$ which
satisfies properties $(1)-(4)$, is $\omega_1$-iterable, and satisfies
the additional property $(3^*)$. As before we let $A_N$ denote the
least non-determined $\OD$-set in $N$, which exists because the
$z$-premouse $N$ satisfies property $(2)$.

Now we define a strategy $\sigma^*$ for player $\mathrm{II}$ in the
usual Gale-Stewart game $G(A_N)$ with payoff set $A_N$ played inside
$N$ as follows. Assume that player $\mathrm{I}$ produces a real
$a \in N$ in a run of the game $G(A_N)$ inside the model $N$. Then we
consider the following run of the game $G$:

\begin{figure}[h]
\centering
\begin{tabular}{c|cc}
  $\mathrm{I}$ & $(z \oplus a) \oplus a$ & \\ \hline
  $\mathrm{II}$ & & $y \oplus b = \sigma((z \oplus a) \oplus a)$ 
\end{tabular}
\end{figure}

Player $\mathrm{I}$ plays a real $(z \oplus a) \oplus a$ and player
$\mathrm{II}$ responds with a real $y \oplus b$ according to his
winning strategy $\sigma$ in $G$. We define the strategy $\sigma^*$
such that in a run of the game $G(A_N)$ inside the model $N$ according
to $\sigma^*$ player $\mathrm{II}$ has to respond to the real $a$ with
producing the real $b$.

\begin{figure}[h]
\centering
\begin{tabular}{c|cc}
 $\mathrm{I}$ & $a$ & \\ \hline
 $\mathrm{II}$ & & $b = \sigma^*(a)$ 
\end{tabular}
\end{figure}

\begin{claim}\label{cl:sigma*}
  $\sigma^*$ is a winning strategy for player $\mathrm{II}$ in the
  Gale-Stewart game with payoff set $A_N$ played in $N$.
\end{claim}

This claim implies that $A_N$ is determined in $N$, again
contradicting that $A_N$ was assumed to be the least non-determined
$\OD$-set of reals in $N$.

\begin{proof}[Proof of Claim \ref{cl:sigma*}]
  We first want to show that the $((z \oplus a) \oplus y)$-premouse
  $L[E]((z \oplus a) \oplus y)^N$ satifies properties $(1) - (4)$.

  First property $(1)$ for $N$ yields that
  \[ L[E]((z \oplus a) \oplus y)^{N} \sim N, \]
  because we have $z,a,y \in N$.  Therefore
  $L[E]((z \oplus a) \oplus y)^{N}$ is a
  $((z \oplus a) \oplus y)$-premouse which satisfies property $(2)$,
  because it has the same sets of reals and the same $\OD$-sets of
  reals as $N$ and hence as before
  \[ L[E]((z \oplus a) \oplus y)^{N} \vDash \neg \ODdet. \] Moreover
  $L[E]((z \oplus a) \oplus y)^{N}$ inherits condition $(1)$ from $N$.

  Since $L[E]((z \oplus a) \oplus y)^{N}$ is a fully backgrounded
  extender construction inside the $\omega_1$-iterable mouse $N$ it is
  $\omega_1$-iterable itself by §$12$ in \cite{MS94}. Therefore
  $L[E]((z \oplus a) \oplus y)^{N}$ satisfies properties $(1), (2)$
  and $(4)$. By property $(3^*)$ for the $z$-premouse $N$ we
  additionally have that $L[E]((z \oplus a) \oplus y)^{N}$ also
  satisfies property $(3)$.

  Since $\sigma$ is a winning strategy for player $\mathrm{II}$ in the
  original game $G$, we have that if $a, b, z$ and $y$ are as above,
  then
  \[ a \oplus b \notin A_{N^\prime}, \]
  for all $((z \oplus a) \oplus y)$-premice $N^\prime$ satisfying
  properties $(1)-(4)$, where $A_{N^\prime}$ denotes the least
  non-determined $\OD$-set of reals in $N^\prime$. Now let
  \[ N^\prime = L[E]((z \oplus a) \oplus y)^{N}, \]
  so we have that in particular $N^\prime$ is a
  $((z \oplus a) \oplus y)$-premouse and satisfies properties
  $(1)-(4)$ as above.

  As in the previous case where we assumed that player $\mathrm{I}$
  has a winning strategy in $G$, we want to show that
  \[ A_{N^\prime} = A_N \]
  in order to conclude that $\sigma^*$ is a winning strategy for
  player $\mathrm{II}$ in the Gale-Stewart game with payoff set $A_N$
  played inside $N$, using that $N^\prime$ satisfies properties
  $(1)-(4)$ as argued above.

  Using $L[E]((z \oplus a) \oplus y)^{N} \sim N$ it follows that
  \[ N \sim N^\prime, \] and thus $A_N = A_{N^\prime}$, as desired.
\end{proof}
This finishes the proof of Lemma \ref{lem:KS}.
\end{proof}


\subsection{Applications}
\label{sec:CorollarysFromKS}

This section is devoted to two important corollaries of Lemma
\ref{lem:KS} which are going to be used in Sections
\ref{sec:SteelsArgumentProj} and \ref{sec:even}.

\begin{cor}
\label{cor:w1}
Let $n \geq 1$. Assume that $M_{n-1}^\#(x)$ exists and is
$\omega_1$-iterable for all $x \in {}^\omega\omega$ and that all
$\Sigma^1_{n+1}$-definable sets of reals are determined. Then
\[ \omega_1^{M_{n-1}(x)} \text{ is measurable in }
\HOD^{M_{n-1}(x)|\delta_x}, \] for a cone of reals $x$.
\end{cor}

\begin{proof}
  This follows from Lemma \ref{lem:KS} with a generalized version of
  Solovay's theorem that, under the Axiom of Determinacy $\AD$,
  $\omega_1$ is measurable. In fact Solovay's proof shows that
  $\OD$-determinacy implies that $\omega_1$ is measurable in
  $\HOD$. For the readers convenience, we will present a proof of this
  result, following the proof of the classical result as in Theorem
  $12.18$ $(b)$ in \cite{Sch14} or Lemma $6.2.2$ in \cite{SchSt}.

  Lemma \ref{lem:KS} yields that
  \[ M_{n-1}(x)|\delta_x \vDash \ODdet, \]
  for a cone of reals $x$. Let $x \in \BS$ be an arbitrary element of
  this cone and let us work inside $M_{n-1}(x)|\delta_x$ for the rest
  of the proof. We aim to define a $< \omega_1$-complete ultrafilter
  $\mathcal{U}$ inside $\HOD^{M_{n-1}(x)|\delta_x}$ on
  $\omega_1 \defeq \omega_1^{M_{n-1}(x)|\delta_x} =
  \omega_1^{M_{n-1}(x)}$,
  witnessing that $\omega_1$ is measurable in
  $\HOD^{M_{n-1}(x)|\delta_x}$.

  Let $n,m \mapsto \langle n,m \rangle$ be the Gödel pairing function
  for $n,m < \omega$ and recall that
  \[\mathsf{WO} \defeq \{ x \in \BS \st R_x \text{ is a well-ordering} \},\]
  where we let $(n,m) \in R_x$ iff $x( \langle n,m \rangle ) = 1$ for
  $x \in \BS$. For $y \in \mathsf{WO}$ we write $||y||$ for the order
  type of $R_y$ and for $x \in \BS$ we let
  \[ |x| = \sup \{ ||y|| \st y \in \mathsf{WO} \wedge y \equiv_T x
  \}. \]
  Consider the set $S = \{ |x| \st x \in \BS \}$ and let
  $\pi : \omega_1 \rightarrow S$ be an order isomorphism. Now we
  define the filter $\mathcal{U}$ on $\omega_1$ as follows. For
  $A \subset \omega_1$ such that $A \in \OD$ we let
  \[ A \in \mathcal{U} \; \text{ iff } \; \{ x \in \BS \st |x| \in \pi
  \pwimg A \} \text{ contains a cone of reals. } \]

  \begin{claim}
    $\mathcal{U} \cap \HOD$ is a $<\omega_1^{M_{n-1}(x)}$-complete
    ultrafilter in $\HOD$.
  \end{claim}

  \begin{proof}
    The set $\{ x \in \BS \st |x| \in \pi \pwimg A \}$ is Turing
    invariant. Therefore we have that $\OD$-Turing-determinacy implies
    that the set $\{ x \in \BS \st |x| \in \pi \pwimg A \}$ for
    $A \in \OD$ either contains a cone of reals or is completely
    disjoint from a cone of reals. Hence we have that
    $\mathcal{U} \cap \HOD$ is an ultrafilter on
    $\omega_1 = \omega_1^{M_{n-1}(x)}$ in $\HOD$. Moreover the
    following argument shows that this ultrafilter
    $\mathcal{U} \cap \HOD$ is $<\omega_1^{M_{n-1}(x)}$-complete in
    $\HOD$.

    Let
    $\{ A_\alpha \st \alpha < \gamma \} \subset \mathcal{U} \cap \HOD$
    be such that $\{ A_\alpha \st \alpha < \gamma \} \in \HOD$ for an
    ordinal $\gamma < \omega_1^{M_{n-1}(x)}$. Then there is a sequence
    $(a_\alpha \st \alpha < \gamma)$ of reals such that for each
    $\alpha < \gamma$, the real $a_\alpha$ is a base for a cone of
    reals contained in
    $\{ x \in \BS \st |x| \in \pi \pwimg A_\alpha \}$.  Since
    $\gamma < \omega_1^{M_{n-1}(x)}$, we can fix a bijection
    $f: \omega \rightarrow \gamma$ in $M_{n-1}(x)$. But then
    $\bigoplus_{n < \omega} a_{f(n)}$ is a base for a cone of reals
    contained in
    \[ \bigcap_{\alpha < \gamma} \{ x \in \BS \st |x| \in \pi \pwimg
    A_\alpha \} = \{ x \in \BS \st |x| \in \pi \pwimg \bigcap_{\alpha
      < \gamma} A_\alpha \}. \]

    So we have that
    $\bigcap_{\alpha < \gamma} A_\alpha \in \mathcal{U} \cap \HOD$ and
    thus the filter $\mathcal{U} \cap \HOD$ is
    $<\omega_1^{M_{n-1}(x)}$-complete.
  \end{proof}
  Therefore $\mathcal{U} \cap \HOD$ witnesses that
  $\omega_1^{M_{n-1}(x)}$ is measurable in
  $\HOD^{M_{n-1}(x)|\delta_x}$.
\end{proof}

In what follows we will prove that in the same situation as above
$\omega_2^{M_{n-1}(x)}$ is strongly inaccessible in
$\HOD^{M_{n-1}(x)|\delta_x}$, which is another consequence of Lemma
\ref{lem:KS}. This is going to be used later in Section
\ref{sec:even}.

In fact the following theorem holds true. It is due to the third
author and a consequence of the ``Generation Theorems'' in \cite{KW10}
(see Theorem $5.4$ in \cite{KW10}).

\begin{thm}\label{thm:genthm}
  Let $n \geq 1$. Assume that $M_{n-1}^\#(x)$ exists and is
  $\omega_1$-iterable for all $x \in {}^\omega\omega$ and that all
  $\Sigma^1_{n+1}$-definable sets of reals are determined. Then for a
  cone of reals $x$,
  \[ \omega_2^{M_{n-1}(x)} \text{ is a Woodin cardinal in }
  \HOD^{M_{n-1}(x)|\delta_x}. \]
\end{thm}

In order to make this paper more self-contained, we shall not use
Theorem \ref{thm:genthm} here, though, and we will give a proof of the
following theorem, which is essentially due to Moschovakis, and will
be used in Section \ref{sec:even}. A version of it can also be found
in \cite{KW10} (see Theorem $3.9$ in \cite{KW10}). The proof shows in
fact that $\OD$-determinacy implies that $\omega_2$ is inaccessible in
$\HOD$.

\begin{thm}\label{thm:w2inacc}
  Let $n \geq 1$. Assume that $M_{n-1}^\#(x)$ exists and is
  $\omega_1$-iterable for all $x \in {}^\omega\omega$ and that all
  $\Sigma^1_{n+1}$-definable sets of reals are determined. Then for a
  cone of reals $x$,
  \[ \omega_2^{M_{n-1}(x)} \text{ is strongly inaccessible in
  } \HOD^{M_{n-1}(x)|\delta_x}. \]
\end{thm}

\begin{proof}
  Using Lemma \ref{lem:KS} we have as above that there is a cone of
  reals $x$ such that
  \[ M_{n-1}(x)|\delta_x \vDash \ODdet. \]
  Let $x$ be an element of that cone. 
  \begin{claim}
    We have that $\omega_2^{M_{n-1}(x)} = (\Theta_0)^{M_{n-1}(x)}$,
    where
    \[ \Theta_0 = \sup\{ \alpha \st \text{ there exists an
    }\OD\text{-surjection } f : \BS \rightarrow \alpha\}. \]
  \end{claim}
  \begin{proof}
    Work inside the model $M_{n-1}(x)$. Since $\CH$ holds in
    $M_{n-1}(x)$, it follows that $\Theta_0 \leq \omega_2$.  For the
    other inequality let $\alpha < \omega_2$ be arbitrary. Then there
    exists an $\OD_x$-surjection $g : \BS \rightarrow \alpha$ because
    by definability (using the definability results from \cite{St95})
    we have that
    \[ M_{n-1}(x)|\omega_2^{M_{n-1}(x)} \subseteq
    \HOD_x^{M_{n-1}(x)}. \]
    This implies that there is an $\OD$-surjection
    $f : \BS \times \BS \rightarrow \alpha$ by varying $x$ and thus it
    follows that $\alpha \leq \Theta_0$.
  \end{proof}

  Work inside the model $M_{n-1}(x)|\delta_x$ from now on and note
  that it is trivial that $\Theta_0 = \omega_2$ is a regular cardinal
  in $\HOD$. So we focus on proving that $\omega_2$ is a strong
  limit. For this purpose we fix an arbitrary ordinal
  $\alpha < \omega_2$ and prove that
  $\card{\Pot(\alpha)^{\HOD}} < \omega_2$.

  Since $\alpha < \omega_2 = \Theta_0$, we can fix a surjection
  $f : \BS \rightarrow \alpha$ such that $f \in \OD$. This surjection
  $f$ induces a prewellordering $\leq_f \in \OD$ on $\BS$ if we let
  \[x \leq_f y \text{ iff } f(x) \leq f(y)\]
  for $x, y \in \BS$. Now consider the pointclass $\SIGMA^1_1(\leq_f)$
  which is defined as follows. For a set of reals $A$ (or analogously
  for a set $A \subset (\BS)^{k}$ for some $k<\omega$) we say
  $A \in \SIGMA^1_1(\leq_f)$ iff there is a $\Sigma_0$-formula
  $\varphi$ and a real $z \in \BS$ such that
  \[ A = \{ y \in \BS \st \exists x \in \BS \; \varphi(y,x,\leq_f,\BS
  \setminus \leq_f, z) \}. \]
  The pointclass $\Sigma^1_1(\leq_f)$ is defined analogous without the
  parameter $z$. We have that there exists a universal
  $\Sigma^1_1(\leq_f)$-definable set $U \subseteq \BS \times \BS$ for
  the pointclass $\SIGMA^1_1(\leq_f)$. So we have that for every
  $\SIGMA^1_1(\leq_f)$-definable set $A$ there exists a $z \in \BS$
  such that $A = U_z = \{ x \in \BS \st (z,x) \in U \}$. Now it
  suffices to prove the following claim.

  \begin{claim}\label{cl:w2inacc}
    Let $X \subset \alpha$ with $X \in \OD$ be arbitrary. Then there
    exists a $\SIGMA^1_1(\leq_f)$-definable set $A \subset \BS$ such
    that $X = f \pwimg A$.
  \end{claim}

  Using Claim \ref{cl:w2inacc} we can define a surjection
  \[ g : \BS \rightarrow \Pot(\alpha) \cap \OD \]
  such that $g \in \OD$ by letting $g(z) = f \pwimg U_z$ for
  $z \in \BS$. This yields that we have
  $\card{\Pot(\alpha)^{\HOD}} < \omega_2$ as desired.
  
  Therefore we are left with proving Claim \ref{cl:w2inacc} to finish
  the proof of Theorem \ref{thm:w2inacc}.
  The proof of Claim \ref{cl:w2inacc} is mainly a special case of
  Moschovakis' Coding Lemma as in Theorem $3.2$ in \cite{KW10}, so we
  will outline the proof in this special case. 

  \begin{proof}[Proof of Claim \ref{cl:w2inacc}]
    Let $X \in \Pot(\alpha) \cap \OD$ be arbitrary. We aim to show
    that there is a real $z \in \BS$ such that $X = f \pwimg U_z$. Let
    \[ B = \{ z \in \BS \st f \pwimg U_z \subseteq X \}. \]
    Moreover let $\alpha_z$ for $z \in B$ be the minimal ordinal
    $\beta$ such that $\beta \in X \setminus f \pwimg U_z$, if it
    exists. We aim to show that there exists a real $z \in B$ such
    that $\alpha_z$ does not exist. So assume toward a contradiction
    that the ordinal $\alpha_z$ exists for all $z \in B$.

    Now consider the following game $G$ of length $\omega$, where
    player $\mathrm{I}$ and player $\mathrm{II}$ alternate playing
    natural numbers such that in the end player $\mathrm{I}$ plays a
    real $x$ and player $\mathrm{II}$ plays a real $y$.
    \begin{figure}[h]
      \centering
      \begin{tabular}{c|cc}
        $\mathrm{I}$ & $x$ & \\ \hline
        $\mathrm{II}$ & & $y$ 
      \end{tabular}
      for $x, y \in \BS$. 
    \end{figure}

    We define that player $\mathrm{I}$ wins the game $G$ iff
    \[ x \in B \wedge (y \in B \rightarrow \alpha_x \geq \alpha_y). \]
    
    Note that we have $B \in \OD$ since $f, U, X \in \OD$. Therefore
    the game $G$ is $\OD$ and thus determined by our hypothesis.

    Assume first that player $\mathrm{I}$ has a winning strategy
    $\sigma$ in $G$. For a real $y$ let $(\sigma * y)_{\mathrm{I}}$
    denote player $\mathrm{I}$'s moves in a run of the game $G$, where
    player $\mathrm{II}$ plays $y$ and player $\mathrm{I}$ responds
    according to his winning strategy $\sigma$. Then there exists a
    real $z_0$ such that
    \[ U_{z_0} = \bigcup \{ U_{(\sigma * y)_{\mathrm{I}}} \st y \in
    \BS \} \]
    because the right hand side of this equation is
    $\SIGMA^1_1(\leq_f)$-definable by choice of $U$. Since $\sigma$ is
    a winning strategy for player $\mathrm{I}$, we have that
    $(\sigma * y)_{\mathrm{I}} \in B$ for all $y \in \BS$ and thus it
    follows that $z_0 \in B$. Moreover we have that
    $\alpha_{(\sigma * y)_{\mathrm{I}}} \leq \alpha_{z_0}$ for all
    $y \in \BS$ by definition of $z_0$.

    Now we aim to construct a play $z^*$ for player $\mathrm{II}$
    defeating the strategy $\sigma$. Since $f: \BS \rightarrow \alpha$
    is a surjection we can choose $a \in \BS$ such that
    $f(a) = \alpha_{z_0}$. Moreover we let $z^* \in \BS$ be such that
    $U_{z^*} = U_{z_0} \cup \{ a \}$. Then we have that
    \[ f \pwimg U_{z^*} = f \pwimg U_{z_0} \cup \{ f(a) \} = f \pwimg
    U_{z_0} \cup \{ \alpha_{z_0} \} \subset X, \]
    since $z_0 \in B$. Hence $z^* \in B$. Moreover we have that
    \[ \alpha_{z^*} > \alpha_{z_0} \geq \alpha_{(\sigma *
      y)_{\mathrm{I}}} \]
    for all $y \in \BS$. Therefore player $\mathrm{II}$ can defeat
    $\sigma$ by playing the real $z^*$, contradicting the fact that
    $\sigma$ is a winning strategy for player $\mathrm{I}$.

    Assume now that player $\mathrm{II}$ has a winning strategy $\tau$
    in the game $G$. Let \[ h_0 : \BS \times \BS \rightarrow \BS \]
    be a $\Sigma^1_1(\leq_f)$-definable function such that for all
    $y,z \in \BS$,
    \[U_{h_0(z, y)} = U_z \cap \{ x \in \BS \st f(x) < f(y)
    \}. \]
    Choose $h_1 : \BS \rightarrow \BS$ such that $h_1$ is
    $\SIGMA^1_1(\leq_f)$-definable and
    \[ U_{h_1(z)} = \bigcup \{ U_{(h_0(z, y) *
      \tau)_{\mathrm{II}}} \cap \{ x \in \BS \st f(x) = f(y) \} \st
    y \in \BS \}, \]
    where the notion $(h_0(z, y) * \tau)_{\mathrm{II}}$ is
    defined analogous to the corresponding notion for player
    $\mathrm{I}$ introduced above. By Kleene's Recursion Theorem (see
    for example Theorem $3.1$ in \cite{KW10}) there exists a fixed
    point for $h_1$ with respect to the set $U$, that means there
    exists a real $z^* \in \BS$ such that we have
    \[ U_{z^*} = U_{h_1(z^*)}. \]

    Now our first step is to prove that $z^* \in B$. Assume toward a
    contradiction that $(f \pwimg U_{z^*}) \setminus X \neq \emptyset$
    and let $\gamma_0 \in (f \pwimg U_{z^*}) \setminus X$ be
    minimal. Moreover let $y_0 \in U_{z^*}$ be such that $f(y_0) =
    \gamma_0$. Then
    \[ \gamma_0 \in f \pwimg U_{z^*} = f \pwimg U_{h_1(z^*)} \]
    and by definition of the function $h_1$ it follows that
    $\gamma_0 \in f \pwimg U_{(h_0(z^*, y_0) *
      \tau)_{\mathrm{II}}}$.
    Since $\gamma_0$ was picked to be minimal in
    $(f \pwimg U_{z^*}) \setminus X$, we have 
    $h_0(z^*, y_0) \in B$ because we have by definition that
    \[ U_{h_0(z^*, y_0)} = U_{z^*} \cap \{ x \in \BS \st f(x) < f(y_0)
    \} = U_{z^*} \cap \{ x \in \BS \st f(x) < \gamma_0 \} \]
    and thus $f \pwimg U_{h_0(z^*, y_0)} \subseteq X$. Since $\tau$ is
    a winning strategy for player $\mathrm{II}$, we have that
    $(h_0(z^*, y_0) * \tau)_{\mathrm{II}} \in B$. Taken all together
    it follows that
    \[ \gamma_0 \in f \pwimg U_{(h_0(z^*, y_0) *
      \tau)_{\mathrm{II}}} \subseteq X. \]
    This contradicts the fact that
    $\gamma_0 \in (f \pwimg U_{z^*}) \setminus X$.

    Recall that we assumed toward a contradiction that the ordinal
    $\alpha_{z^*}$ exists. Let $a^* \in \BS$ be such that
    \[f(a^*) = \alpha_{z^*}\]
    and note that such an $a^*$ exists since
    $f: \BS \rightarrow \alpha$ is a surjection and
    $\alpha_{z^*} < \alpha$. Then we have by definition of the
    function $h_0$ that $h_0(z^*, a^*) \in B$ because $z^* \in B$.
    Moreover we have that $\alpha_{z^*} = \alpha_{h_0(z^*, a^*)}$
    holds by definition of $\alpha_{z^*}$ since
    $f(a^*) = \alpha_{z^*}$. As $\tau$ is a winning strategy for
    player $\mathrm{II}$ in the game $G$, we finally have that
    \[ \alpha_{(h_0(z^*, a^*) * \tau)_{\mathrm{II}}} >
    \alpha_{h_0(z^*, a^*)} = \alpha_{z^*}, \]
    because $h_0(z^*, a^*) \in B$. This contradicts the fact that
    \[ U_{(h_0(z^*, a^*) * \tau)_{\mathrm{II}}} \subset
    U_{h_1(z^*)} = U_{z^*}, \]
    by definition of $\alpha_{z^*}$ and
    $\alpha_{(h_0(z^*, a^*) * \tau)_{\mathrm{II}}}$.
    Therefore the ordinal $\alpha_{z^*}$ does not exist and thus we
    finally have that $f \pwimg U_{z^*} = X$, as desired.
  \end{proof}

  This finishes the proof of Theorem \ref{thm:w2inacc}. 
\end{proof}


\subsection{A Proper Class Inner Model with $n$ Woodin Cardinals}
\label{sec:SteelsArgumentProj}

In this section we are now able to apply the results from the previous
sections to show the existence of a proper class inner model with $n$
Woodin cardinals from determinacy for $\PI^1_n$- and
$\Pi^1_{n+1}$-definable sets (if we assume inductively that $\PI^1_n$
determinacy implies that $M_{n-1}^\#(x)$ exists and is
$\omega_1$-iterable for all $x \in \BS$). This is done in the
following theorem, which is a generalization of Theorem $7.7$ in
\cite{St96} using Lemma \ref{lem:KS} and Corollary \ref{cor:w1}.

\begin{thm}\label{Steel}
  Let $n \geq 1$. If $M_{n-1}^\#(x)$ exists and is $\omega_1$-iterable
  for all $x \in {}^\omega\omega$ and all $\Sigma_{n+1}^1$-definable
  sets of reals are determined, then there exists a proper class inner
  model with $n$ Woodin cardinals.
\end{thm}

We are not claiming here that the model obtained in Theorem
\ref{Steel} is iterable in any sense. We will show how to construct an
$\omega_1$-iterable premouse with $n$ Woodin cardinals using this
model in the next section, but for that we need to assume slightly
more determinacy (namely a consequence of determinacy for all
$\SIGMA_{n+1}^1$-definable sets of reals).

\begin{proof}
  As before let $\delta_x$ denote the least Woodin cardinal in
  $M_{n-1}(x)$ if $n >1$ and let $\delta_x$ denote the least
  $x$-indiscernible in $L[x] = M_0(x)$ if $n = 1$. Then we have that
  according to Lemma \ref{lem:KS}, there is a real $x$ such that for
  all reals $y \geq_T x$,
  \[ M_{n-1}(y)|\delta_y \vDash \ODdet. \]
  Fix such a real $x$.  

  In the case $n=1$ we have that Theorem \ref{Steel} immediately
  follows from Theorem $7.7$ in \cite{St96}, so assume $n > 1$. 

  Let $(K^c)^{M_{n-1}(x)|\delta_x}$ denote the result of a
  robust-background-extender $K^c$-construction in the sense of
  Chapter $1$ in \cite{Je03} performed inside the model
  $M_{n-1}(x)|\delta_x$. Then we distinguish three cases as follows.

\begin{minipage}{\textwidth}
  \bigskip \textbf{Case 1.} Assume that $(K^c)^{M_{n-1}(x)|\delta_x}$
  has no Woodin cardinals and is fully iterable inside
  $M_{n-1}(x) | \delta_x$ via the iteration strategy $\Sigma$ which is
  guided by $\Q$-structures as in Definition \ref{def:Qstritstr}.
  \bigskip
\end{minipage}
  
In this case we can isolate the core model $K^{M_{n-1}(x)|\delta_x}$
below $\delta_x$ as in Theorem $1.1$ in \cite{JS13}.  Then the core
model $K^{M_{n-1}(x)|\delta_x}$ is absolute for all forcings of size
less than $\delta_x$ over $M_{n-1}(x)|\delta_x$ and moreover
$K^{M_{n-1}(x)|\delta_x}$ satisfies weak covering (building on work in 
\cite{MSchSt97} and \cite{MSch95}). That means we have that
$M_{n-1}(x)|\delta_x \vDash \text{``}(\alpha^+)^K = \alpha^+\text{''}$
for all singular cardinals $\alpha$.

Let $\alpha = \aleph_\omega^{M_{n-1}(x)}$. Then $\alpha$ is singular
in $M_{n-1}(x)$, so we have in particular that
\[ M_{n-1}(x)|\delta_x \vDash \text{``}(\alpha^+)^K =
\alpha^+\text{''}. \]
Moreover we have that $\alpha$ is a cutpoint of $M_{n-1}(x)$. So let
$z \in \BS$ be generic over $M_{n-1}(x)$ for $\Col(\omega,
\alpha)$. Then for $y = x \oplus z$ we have that
\[ M_{n-1}(x)[z] = M_{n-1}(y), \]
where we construe $M^\#_{n-1}(x)[z]$ as a $y$-mouse and as a $y$-mouse
$M^\#_{n-1}(x)[z]$ is sound and $\rho_\omega(M^\#_{n-1}(x)[z]) = y$
(see \cite{SchSt09} for the fine structural details).  Moreover we
have that
\[ M_{n-1}(y)|\delta_y \vDash \ODdet, \]
since $y \geq_T x$.  This implies that
\[ M_{n-1}(x)[z]|\delta_y \vDash \ODdet. \]
Now work in the model $M_{n-1}(x)[z]|\delta_y$. Then we have that
$\ODdet$ implies that $\omega_1$ is measurable in $\HOD$ as in
Corollary \ref{cor:w1}.

Since $K \subseteq \HOD$ and we have that $\omega_1 = (\alpha^+)^K$,
it follows that $\omega_1 = (\alpha^+)^{\HOD}$. But $\HOD \vDash \AC$,
so in particular in $\HOD$ all measurable cardinals are
inaccessible. This is a contradiction.

\begin{minipage}{\textwidth}
  \bigskip \textbf{Case 2.} Assume that there is a Woodin cardinal in
  $(K^c)^{M_{n-1}(x)|\delta_x}$.  \bigskip
\end{minipage}

In this case we aim to show that there exists a proper class inner
model with $n$ Woodin cardinals, which is obtained by performing a
fully backgrounded extender construction inside $M_{n-1}(x)$ on top of
the model
\[ (K^c)^{M_{n-1}(x)|\delta_x} \, | \, \delta, \]
where $\delta$ denotes the largest Woodin cardinal in
$(K^c)^{M_{n-1}(x)|\delta_x}$.

We can assume without loss of generality that there is a largest
Woodin cardinal in the model $(K^c)^{M_{n-1}(x)|\delta_x}$ if it has a
Woodin cardinal, because if there is no largest one, then
$(K^c)^{M_{n-1}(x)|\delta_x}$ already yields a proper class inner
model with $n$ Woodin cardinals by iterating some large enough
extender out of the universe. By the same argument we can in fact
assume that $(K^c)^{M_{n-1}(x)|\delta_x}$ is $(n-1)$-small above
$\delta$.

Let \[ (\M_\xi, \mathcal{N}_\xi \st \xi \in \Ord) \] be the sequence
of models obtained from a fully backgrounded extender construction
above $(K^c)^{M_{n-1}(x)|\delta_x} \, | \, \delta$ inside $M_{n-1}(x)$
as in Definition \ref{def:L[E]constr}, where
\[ \M_{\xi + 1} = \mathcal{C}_\omega(\mathcal{N}_{\xi + 1}) \]
and let
\[ L[E]((K^c)^{M_{n-1}(x)|\delta_x} \, | \, \delta)^{M_{n-1}(x)} \]
denote the resulting model.

\begin{minipage}{\textwidth}
  \bigskip \textbf{Case 2.1.} There is no $\xi \in \Ord$ such that
  $\delta$ is not definably Woodin over the model $\M_{\xi+1}$.
  \bigskip
\end{minipage}

In this case $\delta$ is a Woodin cardinal inside
$L[E]((K^c)^{M_{n-1}(x)|\delta_x} \, | \, \delta)^{M_{n-1}(x)}$ and it
follows by a generalization of Theorem $11.3$ in \cite{MS94} that we
have that
$L[E]((K^c)^{M_{n-1}(x)|\delta_x} \, | \, \delta)^{M_{n-1}(x)}$ is a
proper class inner model with $n$ Woodin cardinals, as desired.

\begin{minipage}{\textwidth}
  \bigskip \textbf{Case 2.2.} There exists a $\xi \in \Ord$ such that
  $\delta$ is not definably Woodin over the model
  $\M_{\xi+1}$. \bigskip
\end{minipage}

Let $\xi$ be the minimal such ordinal. In this case the premouse
$\M_{\xi+1}$ is $(n-1)$-small above $\delta$ (see the proof of Claim
\ref{cl:LEnsmall} in the proof of Lemma \ref{lemma1und2}) and we have
that
\[ \M_{\xi+1} \in M_{n-1}(x)|\delta_x. \]
Consider the coiteration of $\M_{\xi+1}$ and
$(K^c)^{M_{n-1}(x)|\delta_x}$ inside $M_{n-1}(x) | \delta_x$.
  
  \begin{claim}\label{cl:coitsucc}
    The coiteration of $\M_{\xi+1}$ and $(K^c)^{M_{n-1}(x)|\delta_x}$
    inside $M_{n-1}(x) | \delta_x$ is successful.
  \end{claim}

  \begin{proof}
    First of all we have that the coiteration takes place above
    $\delta$ and the premouse $\M_{\xi+1}$ is $\omega_1$-iterable
    above $\delta$ in $V$ by construction (see \cite{MS94}). Therefore
    the proof of Lemma \ref{lem:iterability} $(2)$ yields that in the
    model $M_{n-1}(x) | \delta_x$ we have that $\M_{\xi+1}$ is
    iterable for iteration trees in $H^{M_{n-1}(x)}_{\delta_x}$ which
    are above $\delta$, since $\M_{\xi+1} \in M_{n-1}(x) | \delta_x$
    is $(n-1)$-small above $\delta$ and
    $\rho_\omega(\M_{\xi+1}) \leq \delta$.

    Moreover we have that $(K^c)^{M_{n-1}(x)|\delta_x}$ is countably
    iterable above $\delta$ inside $M_{n-1}(x) | \delta_x$ by
    \cite{Je03} (building on the iterability proof in Chapter $9$ in
    \cite{St96}).

    Assume now toward a contradiction that the coiteration of
    $(K^c)^{M_{n-1}(x)|\delta_x}$ with $\M_{\xi+1}$ inside
    $M_{n-1}(x)|\delta_x$ is not successful. Since as argued above
    $\M_{\xi+1}$ is iterable above $\delta$ inside
    $M_{n-1}(x)|\delta_x$ and the coiteration takes place above
    $\delta$ this means that the coiteration has to fail on the
    $(K^c)^{M_{n-1}(x)|\delta_x}$-side. 

    The premouse $(K^c)^{M_{n-1}(x)|\delta_x}$ is assumed to be
    $(n-1)$-small above $\delta$ and therefore the fact that the
    coiteration of $(K^c)^{M_{n-1}(x)|\delta_x}$ and $\M_{\xi+1}$
    fails on the $(K^c)^{M_{n-1}(x)|\delta_x}$-side, implies that
    there exists an iteration tree $\T$ on
    $(K^c)^{M_{n-1}(x)|\delta_x}$ of limit length such that there is
    no $\Q$-structure $\Q(\T)$ for $\T$ such that
    $\Q(\T) \unlhd M_{n-2}^\#(\M(\T))$ and hence
    \[ M_{n-2}^\#(\M(\T)) \vDash \text{``} \delta(\T) \text{ is
      Woodin''.} \]
    In particular we have that the premouse $M_{n-2}^\#(\M(\T))$
    constructed in the sense of Definition \ref{def:M(N)} is not
    $(n-2)$-small above $\delta(\T)$ since otherwise it would already
    provide a $\Q$-structure $\Q(\T)$ for $\T$ which is $(n-2)$-small
    above $\delta(\T)$.

    Let $\bar{M}$ be the Mostowski collapse of a countable
    substructure of $M_{n-1}(x)|\delta_x$ containing the iteration
    tree $\T$. That means for a large enough natural number $m$ we let
    $\bar{M}, X$ and $\sigma$ be such that
    \[ \bar{M} \overset{\sigma}{\cong} X \prec_{\Sigma_m}
    M_{n-1}(x)|\delta_x, \]
    where \[ \sigma: \bar{M} \rightarrow M_{n-1}(x)|\delta_x \]
    denotes the uncollapse map such that we have a model $\bar{K}$ in
    $\bar{M}$ with
    $\sigma(\bar{K}| \gamma) = (K^c)^{M_{n-1}(x)|\delta_x}|
    \sigma(\gamma)$
    for every ordinal $\gamma < \bar{M} \cap \Ord$, and we have an
    iteration tree $\bar{\T}$ on $\bar{K}$ in $\bar{M}$ with
    $\sigma(\bar{\T}) = \T$. Moreover we let $\bar{\delta} \in
    \bar{M}$ be such that $\sigma(\bar{\delta}) = \delta$. 

    By the iterability proof of Chapter $9$ in \cite{St96} (in the
    version of \cite{Je03}) applied inside the model
    $M_{n-1}(x)|\delta_x$, there exists a cofinal well-founded branch
    $b$ through the iteration tree $\bar{\T}$ on $\bar{K}$ above
    $\bar{\delta}$. Moreover we have that
    \[ M_{n-2}^\#(\M(\bar{\T})) \vDash \text{``} \delta(\bar{\T})
    \text{ is Woodin''} \]
    and $M_{n-2}^\#(\M(\bar{\T}))$ is not $(n-2)$-small above
    $\delta(\bar{\T})$.

    Consider the coiteration of $\M_b^{\bar{\T}}$ with
    $M_{n-2}^\#(\M(\bar{\T}))$ and note that it takes place above
    $\delta(\bar{\T})$. Since $M_{n-2}^\#(\M(\bar{\T}))$ is
    $\omega_1$-iterable above $\delta(\bar{\T})$ and $\M_b^{\bar{\T}}$
    is iterable above $\bar{\delta} < \delta(\bar{\T})$ by the
    iterability proof of Chapter $9$ in \cite{St96} (in the version of
    \cite{Je03}) applied inside $M_{n-1}(x)|\delta_x$, the coiteration
    is successful using Lemma \ref{lem:iterability} $(2)$. We have
    that $\M_b^{\bar{\T}}$ cannot lose the coiteration by the
    following argument. If there is no drop along the branch $b$, then
    $\M_b^{\bar{\T}}$ cannot lose the coiteration, because then there
    is no definable Woodin cardinal in $\M_b^{\bar{\T}}$ above
    $\bar{\delta}$ by elementarity, but at the same time we have that
    \[ M_{n-2}^\#(\M(\bar{\T})) \vDash \text{``} \delta(\bar{\T}) >
    \bar{\delta} \text{ is Woodin''.} \]
    If there is a drop along $b$, then $\M_b^{\bar{\T}}$ also has to
    win the coiteration, because we have that
    $\rho_\omega(\M_b^{\bar{\T}}) < \delta(\bar{\T})$ and
    $\rho_\omega(M_{n-2}^\#(\M(\bar{\T}))) = \delta(\bar{\T})$.

    That means there is an iterate $\mathcal{R}^*$ of
    $\M_b^{\bar{\T}}$ and a non-dropping iterate $\M^*$ of
    $M_{n-2}^\#(\M(\bar{\T}))$ such that $\M^* \unlhd \mathcal{R}^*$.
    We have that $\M^*$ is not $(n-1)$-small above $\bar{\delta}$,
    because $M_{n-2}^\#(\M(\bar{\T}))$ is not $(n-1)$-small above
    $\bar{\delta}$ as argued above and the iteration from
    $M_{n-2}^\#(\M(\bar{\T}))$ to $\M^*$ is non-dropping. Therefore it
    follows that $\mathcal{R}^*$ is not $(n-1)$-small above
    $\bar{\delta}$ and thus $\M_b^{\bar{\T}}$ is not $(n-1)$-small
    above $\bar{\delta}$. By the iterability proof of Chapter $9$ in
    \cite{St96} (in the version of \cite{Je03}) we can re-embed the
    model $\M_b^{\bar{\T}}$ into a model of the
    $(K^c)^{M_{n-1}(x)|\delta_x}$-construction above
    $(K^c)^{M_{n-1}(x)|\delta_x} \, | \, \delta$. This yields that
    $(K^c)^{M_{n-1}(x)|\delta_x}$ is not $n$-small, contradicting our
    assumption that it is $(n-1)$-small above $\delta$.
  \end{proof}

  From Claim \ref{cl:coitsucc} it now follows by universality of
  $(K^c)^{M_{n-1}(x)|\delta_x}$ above $\delta$ (see Theorem $4$ in
  \cite{Je03}) that the $(K^c)^{M_{n-1}(x)|\delta_x}$-side has to win
  the comparison. That means there is an iterate $K^*$ of
  $(K^c)^{M_{n-1}(x)|\delta_x}$ and an iterate $N^*$ of $\M_{\xi+1}$
  which is non-dropping on the main branch such that
  \[ N^* \unlhd K^*. \]
  But this is a contradiction, because we assumed that $\delta$ is not
  definably Woodin over $\M_{\xi+1}$ and at the same time we have that
  \[ (K^c)^{M_{n-1}(x)|\delta_x} \vDash \text{``} \delta \text{ is a
    Woodin cardinal''}. \]
  This finishes the case that there is a Woodin cardinal in
  $(K^c)^{M_{n-1}(x)|\delta_x}$. 

\begin{minipage}{\textwidth}
  \bigskip \textbf{Case 3.} Assume that there is no Woodin cardinal in
  $(K^c)^{M_{n-1}(x)|\delta_x}$ and that the premouse
  $(K^c)^{M_{n-1}(x)|\delta_x}$ is not fully iterable inside
  $M_{n-1}(x)|\delta_x$ via the $\Q$-structure iteration strategy
  $\Sigma$ (see Definition \ref{def:Qstritstr}).\bigskip
\end{minipage}

The failure of the attempt to iterate $(K^c)^{M_{n-1}(x)|\delta_x}$
via the $\Q$-structure iteration strategy $\Sigma$ implies that there
exists an iteration tree $\mathcal{T}$ of limit length on
$(K^c)^{M_{n-1}(x)|\delta_x}$ in $M_{n-1}(x)|\delta_x$ such that there
exists no $\Q$-structure for $\T$ inside the model
$M_{n-1}(x)|\delta_x$.

Let \[ (\M_\xi, \mathcal{N}_\xi \st \xi \in \Ord) \] be the sequence
of models obtained from a fully backgrounded extender construction
above $\M(\T)$ inside $M_{n-1}(x)$ as in Definition
\ref{def:L[E]constr}, where
\[ \M_{\xi + 1} = \mathcal{C}_\omega(\mathcal{N}_{\xi + 1}) \]
and let
\[ L[E](\M(\T))^{M_{n-1}(x)} \]
denote the resulting model.

\begin{minipage}{\textwidth}
  \bigskip \textbf{Case 3.1.} There is no $\xi \in \Ord$ such that
  $\delta(\T)$ is not definably Woodin over the model $\M_{\xi+1}$.
  \bigskip
\end{minipage}

In this case $\delta(\T)$ is a Woodin cardinal inside
$L[E](\M(\T))^{M_{n-1}(x)}$ and it follows as in Case $2.1$ by a
generalization of Theorem $11.3$ in \cite{MS94} that
$L[E](\M(\T))^{M_{n-1}(x)}$ is a proper class inner model with $n$
Woodin cardinals, as desired.

\begin{minipage}{\textwidth}
  \bigskip \textbf{Case 3.2.} There exists a $\xi \in \Ord$ such that
  $\delta(\T)$ is not definably Woodin over the model
  $\M_{\xi+1}$. \bigskip
\end{minipage}

Let $\xi$ be the minimal such ordinal. In this case the premouse
$\M_{\xi+1}$ is $(n-1)$-small above $\delta(\T)$ and we have that
\[ \M_{\xi+1} \in M_{n-1}(x)|\delta_x. \]
But then $\M_{\xi+1} \rhd \M(\T)$ already provides a $\Q$-structure
for $\T$ inside the model $M_{n-1}(x)|\delta_x$ because $\delta(\T)$
is not definably Woodin over $\M_{\xi+1}$. This is a contradiction.
\end{proof}

Note that all results we proved in this section under a lightface
determinacy hypothesis relativize to all $x \in \BS$ if we assume the
analogous boldface determinacy hypothesis. We just decided to present
the results without additional parameters to simplify the notation.


\section{Proving Iterability}
\label{ch:it}

With Theorem \ref{Steel} we found a candidate for $M_n$ in the
previous section, but we still have to show its iterability. We will
in fact not prove that this candidate is iterable, but we will use it
to construct an $\omega_1$-iterable premouse $M_n^\#$ in the case that
$n$ is odd. (Note here already that we will give a different argument
if $n$ is even.)

We do parts of this in a slightly more general context and therefore
introduce the concept of an $n$-suitable premouse in Section
\ref{sec:prep}, which will be a natural candidate for the premouse
$M_n | (\delta_0^+)^{M_n}$, where $\delta_0$ denotes the least Woodin
cardinal in $M_n$. Using $n$-suitable premice we will show inductively
in Sections \ref{sec:odd} and \ref{sec:even} that under a determinacy
hypothesis $M_n^\#$ exists and is $\omega_1$-iterable.

In this section again all results we are going to prove under a
lightface determinacy hypothesis relativize to all $x \in \BS$ under
the analogous boldface determinacy hypothesis.


\subsection{Existence of $n$-suitable Premice}
\label{sec:prep}

After introducing pre-$n$-suitable premice and proving their existence
from the results in the previous section, we aim to show in this
section that pre-$(2n-1)$-suitable premice, which are premice with one
Woodin cardinal which satisfy certain fullness conditions, also
satisfy a weak form of iterability, namely short tree iterability. In
fact we are going to show a slightly stronger form of iterability
which includes that fullness properties are preserved during
non-dropping iterations. This will in particular enable us to perform
certain comparison arguments for $(2n-1)$-small premice and will
therefore help us to conclude $\omega_1$-iterability for some
candidate for $M_{2n-1}^\#$.

Recall that in what follows by ``$M_n^\#$ exists'' we always mean that
``$M_n^\#$ exists and is $\omega_1$-iterable''.

A lot of the results in this section only hold true for premice at the
odd levels of our argument, namely $(2n-1)$-suitable premice. This
results from the periodicity in the projective hierarchy in terms of
the uniformization property (see Theorem $6C.5$ in \cite{Mo09}) and
the periodicity in the correctness of $M_n^\#$ (see Lemmas
\ref{lem:corr} and \ref{lem:oddnotcorr}). This behaviour forces us to
give a different proof for the even levels of our argument in Section
\ref{sec:even}.

We start by introducing pre-$n$-suitable and $n$-suitable premice. Our
definition will generalize the notion of suitability from Definition
$3.4$ in \cite{StW16} to $n > 1$. For technical reasons our notion
slightly differs from $n$-suitability as defined in Definition $5.2$
in \cite{Sa13}.

\begin{definition}\label{def:prensuitable}
  Let $n \geq 1$ and assume that $M_{n-1}^\#(x)$ exists for all
  $x \in \BS$. Then we say a countable premouse $N$ is
  \emph{pre-$n$-suitable} \index{prensuitable@pre-$n$-suitable} iff
  there is an ordinal $\delta < \omega_1^V$ such that
\begin{enumerate}[(1)]
\item
  $N \vDash \text{``} \ZFC^- + \; \delta \text{ is the largest
    cardinal''}$,
  \[ N = M_{n-1}(N|\delta) \, | \, (\delta^+)^{M_{n-1}(N|\delta)}, \]
\item $M_{n-1}(N|\delta)$ is a proper class model and
  \[ M_{n-1}(N|\delta) \vDash \text{``} \delta \text{ is Woodin''}, \]
\item for every $\gamma < \delta$, $M_{n-1}(N | \gamma)$ is a set, or
  \[ M_{n-1}(N|\gamma) \nvDash \text{``} \gamma \text{ is
    Woodin''}, \] and
\item for every $\eta < \delta$,
  $M_{n-1}(N|\delta) \vDash \text{``} N|\delta \text{ is } \eta
  \text{-iterable''}$.
\end{enumerate}
\end{definition}

Recall the definition of the premouse $M_{n-1}(N|\delta)$ from
Definition \ref{def:M(N)}.  If $N$ is a pre-$n$-suitable premouse, we
denote the unique ordinal $\delta$ from Definition
\ref{def:prensuitable} by $\delta_N$, analogous to the notation fixed
in Section \ref{sec:introKS}.

Whenever we assume that some premouse $N$ is pre-$n$-suitable for some
$n \geq 1$, we in fact tacitly assume in addition that the premouse
$M_{n-1}^\#(x)$ exists for all $x \in \BS$ (or at least that the
premouse $M_{n-1}^\#(N|\delta)$ exists).

\begin{remark}
  Clearly, if it exists, $M_{n}|(\delta^{+})^{M_{n}}$ is a
  pre-$n$-suitable premouse for $n \geq 1$, whenever $\delta$ denotes
  the least Woodin cardinal in $M_{n}$.
\end{remark}

\begin{remark}
  We have that for $n \geq 1$, if $N$ is a pre-$n$-suitable premouse,
  then $N$ is $n$-small.
\end{remark}

We first show that the proper class inner model with $n$ Woodin
cardinals we constructed in the proof of Theorem \ref{Steel} yields a
pre-$n$-suitable premouse, if we cut it off at the successor of its
least Woodin cardinal and minimize it.

\begin{lemma}\label{lem:prensuitable}
  Let $n \geq 1$. Assume that $M_{n-1}^\#(x)$ exists for all
  $x \in {}^\omega\omega$ and that all $\Sigma^1_{n+1}$-definable sets
  of reals are determined. Then there exists a pre-$n$-suitable
  premouse.
\end{lemma}

\begin{proof}
  Let $W$ be the model constructed in Cases $2$ and $3$ in the proof
  of Theorem \ref{Steel}. Cutting off at the successor of the bottom
  Woodin cardinal $\delta$ yields a premouse $N = W | (\delta^+)^W$
  which satisfies conditions $(1)$ and $(2)$ in the definition of
  pre-$n$-suitability, in the case that the premouse
  $(K^c)^{M_{n-1}(x)|\delta_x}$ from the proof of Theorem \ref{Steel}
  is $(n-1)$-small above $\delta$. Otherwise we can easily consider an
  initial segment $N$ of $(K^c)^{M_{n-1}(x)|\delta_x}$ which satisfies
  conditions $(1)$ and $(2)$. Let $N^\prime$ be the minimal initial
  segment of $N$ which satisfies conditions $(1)$ and $(2)$ and note
  that we then have that $N^\prime$ satisfies condition $(3)$.

  This premouse $N^\prime|\delta^\prime$ is countably iterable inside
  $M_{n-1}(x)$ by Corollary $2$ in \cite{Je03} (using the iterability
  proof from Chapter $9$ in \cite{St96}), where $x$ is a real as in
  the proof of Theorem \ref{Steel} such that the model $W$ as above is
  constructed inside the model $M_{n-1}(x)$ and $\delta^\prime$
  denotes the largest cardinal in $N^\prime$. The $\Q$-structures for
  iteration trees $\T$ on $N^\prime|\delta^\prime$ are $(n-1)$-small
  above the common part model and are therefore contained in the model
  $M_{n-1}(N^\prime|\delta^\prime)$ by arguments we already gave
  several times before. Thus it follows that $N^\prime|\delta^\prime$
  is $\eta$-iterable inside $M_{n-1}(N^\prime|\delta^\prime)$ for all
  $\eta < \delta^\prime$. Therefore we have that condition $(4)$ holds
  as well for $N^\prime$.
\end{proof}

We can show that the following weak form of condensation holds for
pre-$n$-suitable premice.

\begin{lemma}[Weak Condensation Lemma]\label{lem:weakcond}
  Let $N$ be a pre-$n$-suitable premouse for some $n \geq 1$ and let
  $\delta_N$ denote the largest cardinal in $N$. Let $\gamma$ be a large
  enough countable ordinal in $V$ and let $H$ be the Mostowski
  collapse of $Hull_m^{M_{n-1}(N|\delta_N)|\gamma}(\{ \delta_N \})$ for
  some large enough natural number $m$. Then
  \[ H \lhd M_{n-1}(N|\delta_N). \]
\end{lemma}

\begin{proof}
  First consider
  $X = Hull_m^{M_{n-1}(N|\delta_N)|\gamma}(\alpha \cup \{ \delta_N
  \})$ with $X \cap \delta_N = \alpha$ and let $H^\prime$ be the
  Mostowski collapse of $X$. Then $H^\prime$ is $\omega_1$-iterable in
  $V$ above $\alpha$ as $M_{n-1}(N|\delta_N)$ is $\omega_1$-iterable
  in $V$ above $\delta_N$. Moreover condition $(4)$ in the definition
  of pre-$n$-suitability implies that $N|\delta_N$ is $\eta$-iterable
  inside $M_{n-1}(N|\delta_N)$ for every $\eta < \delta_N$.

  We want to show that $H^\prime \lhd N|\delta_N$, so assume this is
  not the case. As $H^\prime \in M_{n-1}(N|\delta_N)$ this yields that
  there is a $\xi < \delta_N$ such that over $N||\xi$ there is an
  $r\Sigma_k$-definable subset of $\alpha$ for some $k<\omega$
  witnessing $\rho_\omega(N||\xi) \leq \alpha$ which is not definable
  over $H^\prime$. Consider the coiteration of the premice $H^\prime$
  and $N||\xi$ inside $M_{n-1}(N|\delta_N)$ which takes place above
  $\alpha$ and let $\T$ and $\U$ be the corresponding iteration trees
  on $H^\prime$ and $N||\xi$ according to the $\Q$-structure iteration
  strategy. Assume towards a contradiction that the comparison does
  not terminate successfully. As $N||\xi$ is iterable inside
  $M_{n-1}(N|\delta_N)$ this means that $\T$ is of limit length and
  there is no cofinal well-founded branch through $\T$ according to
  the $\Q$-structure iteration strategy. As $H^\prime$ is
  $\omega_1$-iterable in $V$, there is such a branch $b$ in
  $V$. Moreover the $N||\xi$-side provides $\Q$-structures for the
  $H^\prime$-side of the coiteration inside
  $M_{n-1}(N|\delta_N)$. Hence the branch $b$ is in fact in
  $M_{n-1}(N|\delta_N)$ and the coiteration terminates
  successfully. So there are iterates $H^*$ of $H^\prime$ and $N^*$ of
  $N||\xi$ such that $H^* \unlhd N^*$. In fact,
  $H^\prime \unlhd N||\xi \lhd N|\delta_N$ since the coiteration takes
  place above $\alpha$, $\rho_\omega(H^\prime) \leq \alpha$,
  $\rho_\omega(N||\xi) \leq \alpha$, and both premice are sound above
  $\alpha$.

  Now let $H$ be the Mostowski collapse of
  $Hull_m^{M_{n-1}(N|\delta_N)|\gamma}(\{ \delta_N \})$. Then $H$ is
  equal to the Mostowski collapse of
  $Hull_m^{H^\prime}(\{ \delta_N \})$. Condition $(4)$ in the
  definition of pre-$n$-suitability suffices to prove the usual
  condensation lemma (see for example Theorem $5.1$ in \cite{St10})
  for proper initial segments of $N|\delta_N$, so in particular for
  $H^\prime$. Therefore it follows that
  $H \unlhd H^\prime \lhd N|\delta_N$, as desired.
\end{proof}

Analogous to Definitions $3.6$ and $3.9$ in \cite{StW16} we define a
notion of short tree iterability for pre-$n$-suitable
premice. Informally a pre-$n$-suitable premouse is short tree iterable
if it is iterable with respect to iteration trees for which there are
$\Q$-structures (see Definition \ref{def:Qstructure}) which are not
too complicated. For this definition we again tacitly assume that
$M_{n-1}^\#(x)$ exists for all reals $x$.

\begin{definition}\label{def:short}
  Let $\T$ be a normal iteration tree of length $< \omega_1^V$ on a
  pre-$n$-suitable premouse $N$ for some $n \geq 1$ which lives below
  $\delta^N$. We say $\T$ is \emph{short} \index{short} iff for all
  limit ordinals $\lambda < \lh(\T)$ the $\Q$-structure
  $\Q(\T \upharpoonright \lambda)$ exists, is $(n-1)$-small above
  $\delta(\T \upharpoonright \lambda)$ and we have that
  \[\Q(\T \upharpoonright \lambda) \unlhd \M_\lambda^\T,\]
  and if $\T$ has limit length we in addition have that $\Q(\T)$
  exists and
  \[\Q(\T) \unlhd M_{n-1}(\M(\T)).\]
  Moreover we say $\T$ is \emph{maximal} \index{maximal} iff $\T$ is
  not short.
\end{definition}

The premouse $M_{n-1}(\M(\T))$ in Definition \ref{def:short} is
defined as in Definition \ref{def:M(N)}.

\begin{definition}\label{def:sti}
  Let $N$ be a pre-$n$-suitable premouse for some $n \geq 1$. We say
  $N$ is \emph{short tree iterable} \index{short tree iterable} iff
  whenever $\T$ is a short tree on $N$,
  \begin{enumerate}[$(i)$]
  \item if $\T$ has a last model, then every putative\footnote{Recall
      that we say an iteration tree $\U$ is a \emph{putative iteration
        tree} if $\U$ satisfies all properties of an iteration tree,
      but we allow the last model of $\U$ to be ill-founded, in case
      $\U$ has a last model.}  iteration tree $\U$ extending $\T$ such
    that $\lh(\U) = \lh(\T) +1$ has a well-founded last model, and
  \item if $\T$ has limit length, then there exists a unique cofinal
    well-founded branch $b$ through $\T$ such that
    \[ \Q(b,\T) = \Q(\T). \]
  \end{enumerate}
\end{definition}

\begin{remark}
  At this point in contrast to the notion of short tree iterability
  for $1$-suitable premice in \cite{StW16} we do not require the
  iterate of a pre-$n$-suitable premouse via a short tree to be
  pre-$n$-suitable again. The reason for this is that in the general
  case for $n>1$ it is not obvious that this property holds assuming
  only our notion of short tree iterability as defined above. We will
  be able to prove later in Lemma \ref{lem:nsuitnew} that this
  property in fact does hold true.
\end{remark}

Because of the periodicity in the projective hierarchy (see also
\cite{St95} for the periodicity in the definition of
$\Pi^1_n$-iterability) the proof of the following lemma only works for
odd levels of suitability.

\begin{lemma}\label{lem:defstiodd}
  Let $n \geq 0$ and assume that $M_{2n}^\#(x)$ exists for all
  $x \in \BS$. Let $N$ be a pre-$(2n+1)$-suitable premouse. Then the
  statement ``$N$ is short tree iterable'' as in Definition
  \ref{def:sti} is $\Pi^1_{2n+2}$-definable uniformly in any code for
  the countable premouse $N$.
\end{lemma}

\begin{proof}
  The statement ``$N$ is short tree iterable'' can be phrased as
  follows. We first consider trees of limit length.
  \begin{align*}
    \forall \T & \text{ tree on } N \text{ of limit length } \forall \, (\Q_\lambda \st
                 \lambda \leq \lh(\T) \text{ limit ordinal}),\\
               & \text{if for all limit ordinals } \lambda \leq 
                 \lh(\T), \\
               & \; \; \; \; \;\; \; \Q_\lambda \text{ is }
                 \Pi^1_{2n+1}\text{-iterable above } \delta(\T
                 \upharpoonright \lambda), 2n\text{-small 
                 above } \delta(\T \upharpoonright \lambda), \\
               & \; \; \; \; \;\; \; \text{solid above } \delta(\T
                 \upharpoonright \lambda) \text{ and a } 
                 \Q\text{-structure for } \T \upharpoonright \lambda,
                 \text{ and}\\
               & \; \; \; \text{if for all limit ordinals } \lambda <
                 \lh(\T) \text{ we have } \Q_\lambda \unlhd \M_\lambda^\T, \text{
                 then}\\  
               & \; \; \; \; \; \exists b \text{ cofinal branch through } \T \text{ such
                 that } \Q_{\lh(\T)} \unlhd \M_b^\T. 
  \end{align*}
  This statement is $\Pi^1_{2n+2}$-definable uniformly in any code for
  $N$ since $\Pi^1_{2n+1}$-iterability above
  $\delta(\T \upharpoonright \lambda)$ for $\Q_\lambda$ is
  $\Pi^1_{2n+1}$-definable uniformly in any code for $\Q_\lambda$. For
  trees of successor length we get a similar statement as follows.
  \begin{align*}
    \forall \T & \text{ putative tree on } N \text{ of successor length }
                 \forall \, (\Q_\lambda \st 
                 \lambda < \lh(\T) \text{ limit ordinal}),\\
               & \text{if for all limit ordinals } \lambda <
                 \lh(\T), \\
               & \; \; \; \; \;\; \; \Q_\lambda \text{ is }
                 \Pi^1_{2n+1}\text{-iterable above } \delta(\T
                 \upharpoonright \lambda), 2n\text{-small 
                 above } \delta(\T \upharpoonright \lambda), \\
               & \; \; \; \; \;\; \; \text{solid above } \delta(\T
                 \upharpoonright \lambda) \text{ and a } 
                 \Q\text{-structure for } \T \upharpoonright \lambda,
                 \text{ and}\\
               & \; \; \; \text{if for all limit ordinals } \lambda <
                 \lh(\T) \text{ we have } \Q_\lambda \unlhd \M_\lambda^\T, \text{
                 then}\\  
               & \; \; \; \; \; \text{the last model of } \T \text{ is well-founded.}
  \end{align*}
  As above this statement is also $\Pi^1_{2n+2}$-definable uniformly
  in any code for $N$. Moreover the conjunction of these two
  statements is equivalent to the statement ``$N$ is short tree
  iterable'', because the relevant $\Q$-structures $\Q_\lambda$ for
  limit ordinals $\lambda \leq \lh(\T)$ are $2n$-small above
  $\delta(\T \upharpoonright \lambda)$ and thus Lemma
  \ref{lem:nofakeeven} implies that for them it is enough to demand
  $\Pi^1_{2n+1}$-iterability above
  $\delta(\T \upharpoonright \lambda)$ to identify them as a
  $\Q$-structure for $\T \upharpoonright \lambda$ since we assumed
  that $M_{2n}^\#(x)$ exists for all $x \in \BS$.
\end{proof}

From this we can obtain the following corollary using Lemma
\ref{lem:corr}.

\begin{cor}\label{cor:absoluteness}
  Let $n \geq 0$ and assume that $M_{2n}^\#(x)$ exists for all
  $x \in \BS$. If $N$ is a pre-$(2n+1)$-suitable premouse, then $N$ is
  short tree iterable iff $N$ is short tree iterable inside the model
  $M_{2n}(N|\delta_N)^{\Col(\omega, \delta_N)}$, where $\delta_N$
  again denotes the largest cardinal in $N$.
\end{cor}

\begin{proof}
  Let $N$ be an arbitrary pre-$(2n+1)$-suitable premouse. By Lemma
  \ref{lem:defstiodd} we have that short tree iterability for $N$ is a
  $\Pi^1_{2n+2}$-definable statement uniformly in any code for
  $N$. Therefore we have by Lemma \ref{lem:corr} that $N$ is short
  tree iterable inside the model
  $M_{2n}(N|\delta_N)^{\Col(\omega, \delta_N)}$ iff $N$ is short tree
  iterable in $V$, because the model
  $M_{2n}(N|\delta_N)^{\Col(\omega, \delta_N)}$ is
  $\SIGMA^1_{2n+2}$-correct in $V$.
\end{proof}

In what follows we aim to show that every pre-$(2n+1)$-suitable
premouse $N$ is short tree iterable. In fact we are going to show a
stronger form of iterability for pre-$(2n+1)$-suitable premice,
including for example fullness-preservation for short trees. This
means that for non-dropping short trees $\T$ on $N$ of length
$\lambda +1$ for some ordinal $\lambda < \omega_1^V$ the resulting
model of the iteration $\M_\lambda^\T$ is again
pre-$(2n+1)$-suitable. Here we mean by ``non-dropping'' that the tree
$\T$ does not drop on the main branch $[0,\lambda]_T$. If this
property holds for a pre-$(2n+1)$-suitable premouse $N$ we say that
\emph{$N$ has a fullness preserving iteration strategy for short
  trees}. \index{fullness preserving iteration strategy for short
  trees} Moreover we also want to show some form of iterability
including fullness-preservation for maximal trees on $N$. Premice
which satify all these kinds of iterability we will call
$(2n+1)$-suitable.

The exact form of iterability we are aiming for is introduced in the
following definition.

\begin{definition}\label{def:nsuitnew}
  Assume that $M_{n-1}^\#(x)$ exists for all $x \in \BS$ and let $N$
  be a pre-$n$-suitable premouse for some $n \geq 1$. Then we say that
  the premouse $N$ is \emph{$n$-suitable}
  \index{suitablen@$n$-suitable} iff
  \begin{enumerate}[$(i)$]
  \item $N$ is short tree iterable and whenever $\T$ is a short tree on $N$
  of length $\lambda + 1$ for some ordinal $\lambda < \omega_1^V$
  which is non-dropping on the main branch $[0, \lambda]_T$, then the
  final model $\M_\lambda^\T$ is pre-$n$-suitable, and
\item whenever $\T$ is a maximal iteration tree on $N$ of length
  $\lambda$ for some limit ordinal $\lambda < \omega_1^V$ according to
  the $\Q$-structure iteration strategy, then there exists a cofinal
  well-founded branch $b$ through $\T$ such that $b$ is non-dropping
  and the model $\M_b^\T$ is pre-$n$-suitable. In fact we have in this
  case that
  \[ \M_b^\T = M_{n-1}(\M(\T))|(\delta(\T)^+)^{M_{n-1}(\M(\T))}. \]
  \end{enumerate}
\end{definition}

Now we are ready to prove that every pre-$(2n+1)$-suitable premouse is
in fact already $(2n+1)$-suitable, using the iterability we build into
condition $(4)$ of Definition \ref{def:prensuitable} in form of the
Weak Condensation Lemma (see Lemma \ref{lem:weakcond}).

\begin{lemma}\label{lem:nsuitnew}
  Let $n \geq 0$ and assume that $M_{2n}^\#(x)$ exists for all
  $x \in \BS$. Let $N$ be an arbitrary pre-$(2n+1)$-suitable premouse.
  Then $N$ is $(2n+1)$-suitable.
\end{lemma}

\begin{proof}
  Let $N$ be an arbitrary pre-$(2n+1)$-suitable premouse and let
  $W = M_{2n}(N|\delta_N)$ be a premouse in the sense of Definition
  \ref{def:M(N)}, where $\delta_N$ as usual denotes the largest
  cardinal in $N$. That means in particular that
  $N = W | (\delta_N^+)^W$.
  
  We want to show that $N$ is $(2n+1)$-suitable. So we assume toward a
  contradiction that this is not the case. Say this is witnessed by an
  iteration tree $\T$ on $N$. 

  We want to reflect this statement down to a countable
  hull. Therefore let $m$ be a large enough natural number, let
  $\theta$ be a large enough ordinal such that in particular
  \[ W | \theta \prec_{\Sigma_m} W \] and
  \[ W | \theta \vDash \ZFC^-, \] where $\ZFC^-$ denotes $\ZFC$
  without the Power Set Axiom, and let
  \[\bar{W} \overset{\sigma}{\cong} Hull^{W|\theta}_m(\{ \delta_N \})
  \prec W | \theta, \]
  where $\bar{W}$ is the Mostowski collapse of
  $Hull^{W|\theta}_m(\{ \delta_N \})$ and
  \[ \sigma: \bar{W} \rightarrow Hull^{W|\theta}_m(\{ \delta_N \}) \]
  denotes the uncollapse map such that $\delta_N \in \ran(\sigma)$ and
  $\sigma(\bar{\delta}) = \delta_N$ for some ordinal
  $\bar{\delta}$ in $\bar{W}$. Then we have that $\bar{W}$ is sound,
  $\rho_{m+1}(\bar{W}) = \omega$, and the Weak Condensation Lemma
  \ref{lem:weakcond} yields that
  \[ \bar{W} \lhd W. \]

\begin{minipage}{\textwidth}
      \bigskip \textbf{Case 1.}  $\T$ is short and witnesses that $N$
      is not short tree iterable. \bigskip
    \end{minipage}  

    For simplicity assume in this case that $\T$ has limit length
    since the other case is easier. Then $\T$ witnesses that the
    following statement $\phi_1(N)$ holds in $V$.

    \begin{align*}
      \phi_1(N) \equiv \; & \exists \, \T \text{ tree on } N \text{ of
                            length } \lambda \text{ for
                            some limit ordinal } \lambda <
                            \omega_1^V \\
                          & \exists \, (\Q_\gamma \st 
                            \gamma \leq \lambda \text{ limit ordinal}), \text{
                            such that for all limit ordinals } \gamma \leq 
                            \lambda, \\
                          & \; \; \; \; \;\; \; \Q_\gamma \text{ is }
                            \Pi^1_{2n+1}\text{-iterable above } \delta(\T
                            \upharpoonright \gamma), 2n\text{-small 
                            above } \delta(\T \upharpoonright \gamma), \\
                          & \; \; \; \; \;\; \; \text{solid above } \delta(\T
                            \upharpoonright \gamma) \text{ and a } 
                            \Q\text{-structure for } \T \upharpoonright \gamma,
                            \text{ and}\\
                          & \; \; \; \text{for all limit ordinals } \gamma <
                            \lambda \text{ we have } \Q_\gamma \unlhd
                            \M_\gamma^\T, \text{ but}\\  
                          & \; \; \; \text{there exists no cofinal branch
                            } b \text{ through } \T \text{ such
                            that } \Q_{\lambda} \unlhd \M_b^\T. 
  \end{align*}

  We have that $\phi_1(N)$ is $\Sigma^1_{2n+2}$-definable uniformly in
  any code for $N$ as in the proof of Lemma \ref{lem:defstiodd}. See
  also the proof of Lemma \ref{lem:defstiodd} for the case that $\T$
  has successor length.

\begin{minipage}{\textwidth}
  \bigskip \textbf{Case 2.} $\T$ is a short tree on $N$ of length
  $\lambda +1$ for some ordinal $\lambda < \omega_1^V$ which is
  non-dropping on the main branch such that the final model
  $\M_\lambda^\T$ is not pre-$(2n+1)$-suitable. \bigskip
    \end{minipage}  

Assume that
\[ M_{2n}(\M_\lambda^\T|\delta_{\M_\lambda^\T}) \nvDash \text{``}
\delta_{\M_\lambda^\T} \text{ is Woodin'',} \]
where $\delta_{\M_\lambda^\T}$ denotes the largest cardinal in
$\M_\lambda^\T$. This means that we have
$\delta_{\M_\lambda^\T} = i_{0\lambda}^\T(\delta_N)$, where
$i_{0\lambda}^\T : N \rightarrow \M_\lambda^\T$ denotes the iteration
embedding, which exists since the iteration tree $\T$ is assumed to be
non-dropping on the main branch.

Then $\T$ witnesses that the following statement $\phi_2(N)$ holds
true in $V$.
\begin{align*}
  \phi_2(N) \equiv \; & \exists \, \T \text{ tree on } N \text{ of
                        length } \lambda + 1 \text{ for some } \lambda <
                        \omega_1^V \text{ such
                        that }\\ 
                      &\;  \T \text{ is non-dropping along } [0,\lambda]_T \text{ and } \\
                      & \; \forall \gamma < \lh(\T) \text{ limit } \exists \Q
                        \unlhd \M_\gamma^\T 
                        \text{ such that } \\  
                      & \;  \; \; \Q \text{ is } \Pi^1_{2n+1}\text{-iterable
                        above } \delta(\T \upharpoonright \gamma),  2n
                        \text{-small above } \delta(\T \upharpoonright \gamma),\\
                      & \;\;\; \text{solid above } \delta(\T \upharpoonright \gamma),
                        \text{ and a }\Q \text{-structure for } \T
                        \upharpoonright \gamma \text{, and }\\ 
                      & \; \exists \mathcal{P} \rhd \M_\lambda^\T |
                        \delta_{\M_\lambda^\T} \text{ such that } 
                        \mathcal{P} \text{ is } \Pi^1_{2n+1}\text{-iterable
                        above } i_{0\lambda}^\T(\delta_N),  \\
                      & \;\;\;  2n\text{-small }\text{above } i_{0\lambda}^\T(\delta_N), 
                        i_{0\lambda}^\T(\delta_N)\text{-sound, and }\\
                      & \;\;\; \; \; 
                        i_{0\lambda}^\T(\delta_N) \text{ is not definably 
                        Woodin over } \mathcal{P},
\end{align*}
where $\delta_N$ as above denotes the largest cardinal in $N$. Recall
Definition \ref{def:notdefWdn} for the notion of a definable Woodin
cardinal.

We have that $\phi_2(N)$ is $\Sigma^1_{2n+2}$-definable uniformly in
any code for $N$.

\begin{minipage}{\textwidth}
  \bigskip \textbf{Case 3.} $\T$ is a maximal tree on $N$ of length
  $\lambda$ for some limit ordinal $\lambda < \omega_1^V$ such that
  there is no cofinal well-founded branch $b$ through $\T$ or for
  every such branch $b$ the premouse $\M_b^\T$ is not
  pre-$(2n+1)$-suitable. \bigskip
    \end{minipage}  

    As $\T$ is maximal, we have that every such branch $b$ is
    non-dropping and in the case that for every such branch $b$ the
    premouse $\M_b^\T$ is not pre-$(2n+1)$-suitable, assume that we
    have
    \[ M_{2n}(\M_b^\T|\delta_{\M_b^\T}) \nvDash \text{``}
    \delta_{\M_b^\T} \text{ is Woodin'',} \]
    where $\delta_{\M_b^\T}$ denotes the largest cardinal in
    $\M_b^\T$. Then the iteration tree $\T$ witnesses that the
    following statement $\phi_3(N)$ holds true in $V$.
\begin{align*}
  \phi_3(N) \equiv \; & \exists \, \T \text{ tree on } N \text{ of
               length } \lambda \text{ for some limit ordinal } \lambda <
               \omega_1^V \text{ such
               that }\\ 
             & \; \forall \gamma < \lambda \text{ limit }\; \exists \Q \unlhd
               \M_\gamma^\T \text{ such that } \\  
             & \;  \; \; \Q \text{ is } \Pi^1_{2n+1}\text{-iterable
               above } \delta(\T \upharpoonright \gamma),  2n
               \text{-small above } \delta(\T \upharpoonright \gamma),\\
             & \;\;\; \text{solid above } \delta(\T \upharpoonright \gamma),
               \text{ and a }\Q \text{-structure for } \T
               \upharpoonright \gamma \text{, and }\\
             & \; \exists \mathcal{P} \rhd \M(\T)
               \text{ such that } \mathcal{P} 
               \text{ is } \Pi^1_{2n+1}\text{-iterable above }
               \delta(\T),  \\
             & \; \; \; \rho_\omega(\mathcal{P}) \leq \delta(\T), 
               \mathcal{P} \text{ is } 2n\text{-small above }
               \delta(\T), \delta(\T)\text{-sound, and } \\ 
             & \;\;\; \mathcal{P} \vDash \text{``} 
               \delta(\T) \text{ is 
               the largest cardinal + } \delta(\T) \text{ is Woodin'',} \\ 
             & \;\;\; \text{ and there is no branch } b \text{ through } \T \text{
               such that } \mathcal{P} \unlhd \M_b^\T.
\end{align*}

We have that $\phi_3(N)$ is $\Sigma^1_{2n+2}$-definable uniformly in
any code for $N$.

Now we consider all three cases together again and let
\[ \phi(N) = \phi_1(N) \vee \phi_2(N) \vee \phi_3(N). \]
As argued in the individual cases the iteration tree $\T$ witnesses
that $\phi(N)$ holds in $V$ (still assuming for simplicity that $\T$
has limit length if $\T$ is as in Case $1$).

Since $\phi$ is a $\SIGMA^1_{2n+2}$-definable statement and
$W = M_{2n}(N|\delta_N)$, we have by Lemma \ref{lem:corr} that
\[ \forces{W}{\Col(\omega, \delta_N)} \phi(N). \]
So since we picked $m$ and $\theta$ large enough we have that
\[ \bar{W}[g] \vDash \phi(\bar{N}), \]
if we let $\bar{N} \in \bar{W}$ be such that $\sigma(\bar{N}) = N$ and
if $g$ is $\Col(\omega, \bar{\delta})$-generic over $\bar{W}$. Let
$\bar{\T}$ be a tree on $\bar{N}$ in $\bar{W}[g]$ witnessing that
$\phi(\bar{N})$ holds. Since $\bar{N}$ is countable in $W$, we can
pick $g \in W$ and then have that $\bar{\T} \in W$.

We have that $\bar{W}[g]$ is $\SIGMA^1_{2n+1}$-correct in $V$ using
Lemma \ref{lem:ctblcorr}, because it is a countable model with $2n$ Woodin
cardinals. Since $\bar{\T}$ witnesses the statement $\phi(\bar{N})$ in
$\bar{W}[g]$, it follows that $\bar{\T}$ also witnesses
$\phi(\bar{N})$ in $V$, because $\phi(\bar{N})$ is
$\Sigma^1_{2n+2}$-definable in any code for $\bar{N}$. The
$\Q$-structures for $\bar{\T}$ in $\bar{W}[g]$ in the statement
$\phi(\bar{N})$ are $\Pi^1_{2n+1}$-iterable above
$\delta(\bar{\T} \upharpoonright \gamma)$ and $2n$-small above
$\delta(\bar{\T} \upharpoonright \gamma)$ for limit ordinals
$\gamma < \lh(\bar{\T})$ (and also for $\gamma = \lh(\bar{\T})$ if
$\bar{\T}$ witnesses that $\phi_1(\bar{N})$ holds in $\bar{W}[g]$).

Since this amount of iterability suffices to witness $\Q$-structures
using Lemma \ref{lem:nofakeeven} and since as mentioned above
$\bar{W}[g]$ is $\SIGMA^1_{2n+1}$-correct in $V$, the $\Q$-structures
for $\bar{\T}$ in $\bar{W}[g]$ are also $\Q$-structures for $\bar{\T}$
inside $V$. Since $W = M_{2n}(N|\delta_N)$ is also
$\SIGMA^1_{2n+1}$-correct in $V$ using Lemma \ref{lem:corr} and
$\bar{N}$ and $\bar{\T}$ are countable in $W$, it follows that the
$\Q$-structures for $\bar{\T}$ in $\bar{W}[g]$ (which are
$\Pi^1_{2n+1}$-iterable above
$\delta(\bar{\T} \upharpoonright \gamma)$ for $\gamma$ as above) are
also $\Q$-structures for $\bar{\T}$ inside $W$. Therefore the branches
choosen in the tree $\bar{\T}$ on $\bar{N}$ inside $\bar{W}[g]$ are
the same branches as the $\Q$-structure iteration strategy $\Sigma$ as
in Definition \ref{def:Qstritstr} would choose inside the model $W$
when iterating the premouse $\bar{W}$. That means if $\T^*$ is the
tree on $\bar{W}$ obtained by considering $\bar{\T}$ as a tree on
$\bar{W} \rhd \bar{N}$, then the iteration strategy $\Sigma$ picks the
same branches for the tree $\T^*$ as it does for the tree $\bar{\T}$.

Now we again distinguish three cases as before. 

\begin{minipage}{\textwidth}
  \bigskip \textbf{Case 1.} $\bar{\T}$ witnesses that
  $\phi_1(\bar{N})$ holds in $\bar{W}[g]$. \bigskip
    \end{minipage}  

    By the argument we gave above, $\bar{\T}$ is a short tree on
    $\bar{N}$ in $W$.

    Let $\bar{b}$ denote the cofinal branch through $\T^*$ which
    exists inside $W$ and is defined as follows. We have that a branch
    through the iteration tree $\T^*$ can be considered as a branch
    through $\bar{\T}$, where the latter is a tree on $\bar{N}$, and
    vice versa. Since by the Weak Condensation Lemma
    \ref{lem:weakcond} we have that
    \[ \bar{N} = \bar{W}|(\bar{\delta}^+)^{\bar{W}} \lhd W |
    \omega_1^W = M_{2n}(N|\delta_N) | \omega_1^{M_{2n}(N|\delta_N)} =
    N | \omega_1^N, \]
    there exists a cofinal well-founded branch $\bar{b} \in W$ through
    $\bar{\T}$ by property $(4)$ in the definition of
    pre-$n$-suitability (see Definition \ref{def:prensuitable}). We
    also consider this branch $\bar{b}$ as a branch through $\T^*$.

    Assume first that there is a drop along the branch $\bar{b}$. Then
    there exists a $\Q$-structure $\Q(\bar{\T}) \unlhd
    \M_{\bar{b}}^{\bar{\T}}$. Consider the statement
    \begin{eqnarray*}
      \psi(\bar{\T}, \Q(\bar{\T})) \equiv &\text{``there is a cofinal
                                            branch } b \text{ through } \bar{\T} \text{ such
                                            that } \\ 
                                          & \Q(\bar{\T}) \unlhd
                                            \M_{b}^{\bar{\T}}\text{''.}   
    \end{eqnarray*}
    This statement $\psi(\bar{\T}, \Q(\bar{\T}))$ is
    $\Sigma^1_1$-definable uniformly in any code for the parameters
    $\bar{\T}$ and $\Q(\bar{\T})$ and holds in the model $W$ as
    witnessed by the branch $\bar{b}$. Now a $\SIGMA^1_1$-absoluteness
    argument as the one given in the proof of Lemma
    \ref{lem:iterability} yields that this statement
    $\psi(\bar{\T}, \Q(\bar{\T}))$ also holds in $\bar{W}[g]$, which
    contradicts the fact that $\bar{\T}$ witnesses in $\bar{W}[g]$
    that $\bar{N}$ is not short tree iterable.

    Therefore we can assume that $\bar{b}$ does not drop.

    Since $\bar{\T}$ witnesses that $\phi_1(\bar{N})$ holds in $\bar{W}[g]$,
    we have that there exists a $\Q$-structure $\Q_\lambda$ for
    $\bar{\T}$ as in $\phi_1(\bar{N})$. In particular $\Q_\lambda$ is
    $2n$-small above $\delta(\T)$ and $\Pi^1_{2n+1}$-iterable above
    $\delta(\T)$ in $\bar{W}[g]$. 

    \begin{minipage}{\textwidth}
      \bigskip \textbf{Case 1.1.}
      $\delta(\bar{\T}) = i_{\bar{b}}^{\T^*}(\bar{\delta})$.  \bigskip
    \end{minipage}

    Consider the comparison of $\Q_\lambda$ with $\M_{\bar{b}}^{\T^*}$
    inside $W$.

    This comparison takes place above
    $i_{\bar{b}}^{\T^*}(\bar{\delta}) = \delta(\bar{\T})$ and the
    premouse $\M_{\bar{b}}^{\T^*}$ is $(\omega_1+1)$-iterable above
    $i_{\bar{b}}^{\T^*}(\bar{\delta})$ in $W$ using property $(4)$ in
    Definition \ref{def:prensuitable} because
    $(\omega_1+1)^W < \delta^N$ and $\T^*$ is an iteration tree on
    $\bar{W} \, \lhd \, W | \omega_1^W = N | \omega_1^N$ (using the
    Weak Condensation Lemma \ref{lem:weakcond} again).  Furthermore we
    have that $\bar{W}$ is $2n$-small above $\bar{\delta}$ and
    therefore $\M_{\bar{b}}^{\T^*}$ is $2n$-small above
    $i_{\bar{b}}^{\T^*}(\bar{\delta})$.

    Moreover $\Q_\lambda$ is $\Pi^1_{2n+1}$-iterable above
    $\delta(\T)$ in $\bar{W}[g]$ thus by $\SIGMA^1_{2n+1}$-correctness
    also inside $W$. The statement $\phi_1(\bar{N})$ yields that
    $\Q_\lambda$ is $2n$-small above
    $i_{\bar{b}}^{\T^*}(\bar{\delta})$.  We have that
    $\bar{W} \, \lhd \, W$ is sound by construction and thus the
    non-dropping iterate $\M_{\bar{b}}^{\T^*}$ is sound above
    $i_{\bar{b}}^{\T^*}(\bar{\delta})$. Moreover we have that
    $\rho_\omega(\M_{\bar{b}}^{\T^*}) \leq
    i_{\bar{b}}^{\T^*}(\bar{\delta})$ as
    $\rho_\omega(\M_{\bar{b}}^{\T^*}) = \rho_\omega(\bar{W}) =
    \omega$. In addition $\Q_\lambda$ is also sound above
    $i_{\bar{b}}^{\T^*}(\bar{\delta})$ and we have that
    $\rho_\omega(\Q_\lambda) \leq \delta(\bar{\T}) =
    i_{\bar{b}}^{\T^*}(\bar{\delta})$.  Hence Lemma $2.2$ in
    \cite{St95} (see the discussion before Lemma \ref{lem:steel2.2})
    implies that
    \[ \Q_\lambda \lhd \M_{\bar{b}}^{\T^*} \text{ or }
    \M_{\bar{b}}^{\T^*} \unlhd \Q_\lambda. \]

    So we again distinguish two different cases. 

    \begin{minipage}{\textwidth}
      \bigskip \textbf{Case 1.1.1.}
     $\Q_\lambda \lhd \M_{\bar{b}}^{\T^*}.$  \bigskip
    \end{minipage}

    By assumption $\bar{\delta}$ is a Woodin cardinal in $\bar{W}$,
    because $N$ is pre-$(2n+1)$-suitable and thus $\delta_N$ is a
    Woodin cardinal in $W$. Therefore we have by elementarity that
    \[ \M_{\bar{b}}^{\T^*} \vDash \text{``}
    i_{\bar{b}}^{\T^*}(\bar{\delta}) \text{ is Woodin''.} \]
    But since $\Q_\lambda$ is a $\Q$-structure for $\T$, we have that
    $\delta(\T) = i_{\bar{b}}^{\T^*}(\bar{\delta})$ is not definably
    Woodin over $\Q_\lambda$. This contradicts
    $\Q_\lambda \lhd \M_{\bar{b}}^{\T^*}.$

    \begin{minipage}{\textwidth}
      \bigskip \textbf{Case 1.1.2.}
      $\M_{\bar{b}}^{\T^*} \unlhd \Q_\lambda.$  \bigskip
    \end{minipage}

    In this case we have that
    \[ \Q_\lambda \cap \Ord < \bar{W} \cap \Ord \leq
    \M_{\bar{b}}^{\T^*} \cap \Ord \leq \Q_\lambda \cap \Ord, \]
    where the first inequality holds true since
    $\Q_\lambda \in \bar{W}[g]$. This contradiction finishes Case
    $1.1$.

    \begin{minipage}{\textwidth}
      \bigskip \textbf{Case 1.2.}
      $\delta(\bar{\T}) < i_{\bar{b}}^{\T^*}(\bar{\delta})$.
      \bigskip
    \end{minipage}

    In this case we have that
    \[ \M_{\bar{b}}^{\T^*} \vDash \text{``} \delta(\bar{\T}) \text{ is
      not Woodin''}, \]
    because otherwise $\M_{\bar{b}}^{\T^*}$ would not be
    $(2n+1)$-small. This implies that
    $\Q_\lambda = \Q(\bar{\T}) \lhd \M_{\bar{b}}^{\T^*}$ and therefore
    we have that
    \[ \Q(\bar{\T}) \unlhd \M_{\bar{b}}^{\bar{\T}}. \]

    Now we can again consider the statement
    \begin{eqnarray*}
      \psi(\bar{\T}, \Q(\bar{\T})) \equiv &\text{``there is a cofinal
                                            branch } b \text{ through } \bar{\T} \text{ such
                                            that } \\ 
                                          & \Q(\bar{\T}) \unlhd
                                            \M_{b}^{\bar{\T}}\text{''.}   
    \end{eqnarray*}
    Again $\psi(\bar{\T}, \Q(\bar{\T}))$ holds in the model $W$ as
    witnessed by the branch $\bar{b}$. By an absoluteness argument as
    above we have that it also holds in $\bar{W}[g]$, which
    contradicts the fact that $\bar{\T}$ witnesses in $\bar{W}[g]$
    that $\bar{N}$ is not short tree iterable.

\begin{minipage}{\textwidth}
  \bigskip \textbf{Case 2.} $\bar{\T}$ witnesses that
  $\phi_2(\bar{N})$ holds in $\bar{W}[g]$. \bigskip
    \end{minipage}  

    In this case $\bar{\T}$ is a tree of length $\bar{\lambda} +1$ for
    some ordinal $\bar{\lambda}$. 

    Since $\phi_2(\bar{N})$ holds true in $\bar{W}[g]$, there exists a
    model
    $\bar{\mathcal{P}} \unrhd
    \M_{\bar{\lambda}}^{\T^*}|i_{0\bar{\lambda}}^{\T^*}(\bar{\delta})$,
    which is $2n$-small above
    $i_{0\bar{\lambda}}^{\T^*}(\bar{\delta})$, sound above
    $i_{0\bar{\lambda}}^{\T^*}(\bar{\delta})$ and
    $\Pi^1_{2n+1}$-iterable above
    $i_{0\bar{\lambda}}^{\T^*}(\bar{\delta})$ in
    $\bar{W}[g]$. Moreover $i_{0\bar{\lambda}}^{\T^*}(\bar{\delta})$
    is not definably Woodin over $\bar{\mathcal{P}}$ and we have that
    $\rho_\omega(\bar{\mathcal{P}}) \leq
    i_{0\bar{\lambda}}^{\T^*}(\bar{\delta})$.

    Consider the comparison of $\bar{\mathcal{P}}$ with
    $\M_{\bar{\lambda}}^{\T^*}$ inside the model $W$. The comparison
    takes place above $i_{0\bar{\lambda}}^{\T^*}(\bar{\delta})$ and we
    have that $\M_{\bar{\lambda}}^{\T^*}$ is $2n$-small above
    $i_{0\bar{\lambda}}^{\T^*}(\bar{\delta})$ because $W$ is
    $2n$-small above $\delta_N$. The premouse
    $\M_{\bar{\lambda}}^{\T^*}$ is $(\omega_1+1)$-iterable above
    $i_{0\bar{\lambda}}^{\T^*}(\bar{\delta})$ in $W$ by the same
    argument we gave above in Case $1$ using property $(4)$ in
    Definition \ref{def:prensuitable}. Therefore the coiteration is
    successful using Lemma $2.2$ in \cite{St95} by the following
    argument.

    We have that $\bar{\mathcal{P}}$ is $\Pi^1_{2n+1}$-iterable inside
    the model $\bar{W}[g]$ and thus by $\SIGMA^1_{2n+1}$-correctness
    also inside $W$. The statement $\phi_2(\bar{N})$ yields that
    $\bar{\mathcal{P}}$ is also $2n$-small above
    $i_{0\bar{\lambda}}^{\T^*}(\bar{\delta})$.  We have that $\bar{W}$
    is sound by construction and thus the non-dropping iterate
    $\M_{\bar{\lambda}}^{\T^*}$ is sound above
    $i_{0\bar{\lambda}}^{\T^*}(\bar{\delta})$. Moreover we have that
    $\rho_\omega(\M_{\bar{\lambda}}^{\T^*}) \leq
    i_{0\bar{\lambda}}^{\T^*}(\bar{\delta})$.  In addition
    $\bar{\mathcal{P}}$ is also sound above
    $i_{0\bar{\lambda}}^{\T^*}(\bar{\delta})$ and we have that
    $\rho_\omega(\bar{\mathcal{P}}) \leq
    i_{0\bar{\lambda}}^{\T^*}(\bar{\delta})$ because of
    $\phi_2(\bar{N})$. Hence Lemma $2.2$ in \cite{St95} implies that
    \[ \bar{\mathcal{P}} \lhd \M_{\bar{\lambda}}^{\T^*} \text{ or }
    \M_{\bar{\lambda}}^{\T^*} \unlhd \bar{\mathcal{P}}. \]

    So we consider two different cases. 

    \begin{minipage}{\textwidth}
      \bigskip \textbf{Case 2.1.}
     $\bar{\mathcal{P}} \lhd \M_{\bar{\lambda}}^{\T^*}.$  \bigskip
    \end{minipage}

    By assumption $\bar{\delta}$ is a Woodin cardinal in $\bar{W}$,
    because $N$ is pre-$(2n+1)$-suitable and thus $\delta_N$ is a
    Woodin cardinal in $W$. Therefore we have by elementarity that
    \[ \M_{\bar{\lambda}}^{\T^*} \vDash \text{``}
    i_{0\bar{\lambda}}^{\T^*}(\bar{\delta}) \text{ is Woodin''.} \]
    Moreover we have by the statement $\phi_2(\bar{N})$ that
    $i_{0\bar{\lambda}}^{\T^*}(\bar{\delta})$ is not definably Woodin
    over $\bar{\mathcal{P}}$. This is a contradiction to
    $\bar{\mathcal{P}} \lhd \M_{\bar{\lambda}}^{\T^*}.$

    \begin{minipage}{\textwidth}
      \bigskip \textbf{Case 2.2.}
     $ \M_{\bar{\lambda}}^{\T^*} \unlhd \bar{\mathcal{P}}.$  \bigskip
    \end{minipage}

    In this case we have that 
    \[ \bar{\mathcal{P}} \cap \Ord < \bar{W} \cap \Ord \leq
    \M_{\bar{\lambda}}^{\T^*} \cap \Ord \leq \bar{\mathcal{P}} \cap
    \Ord, \]
    where the first inequality holds since
    $\bar{\mathcal{P}} \in \bar{W}[g]$. This is a contradiction.

    Therefore we proved that
    \[ M_{2n}(\M_\lambda^\T|\delta_{\M_\lambda^\T}) \vDash \text{``}
    \delta_{\M_\lambda^\T} \text{ is Woodin'',} \]
    if $\T$ is as in Case $2$ above, that means if $\T$ is a short
    iteration tree on $N$ of length $\lambda +1$ which is non-dropping
    on the main branch.

    This shows that there is an ordinal $\delta < \omega_1^V$ such
    that properties $(1)$ and $(2)$ in Definition
    \ref{def:prensuitable} hold for the premouse $\M_\lambda^\T$. That
    property $(3)$ holds for $\M_\lambda^\T$ follows from property
    $(3)$ for the pre-$(2n+1)$-suitable premouse $N$ by a similar
    argument and property $(4)$ follows from the corresponding
    property for $N$ as well. Thus $\M_\lambda^\T$ is
    pre-$(2n+1)$-suitable, as desired.

    \begin{minipage}{\textwidth}
      \bigskip \textbf{Case 3.} $\bar{\T}$ witnesses that
      $\phi_3(\bar{N})$ holds in $\bar{W}[g]$. \bigskip
    \end{minipage}

    Let $\bar{\mathcal{P}} \rhd \M(\bar{\T})$ witness that
    $\phi_3(\bar{N})$ holds inside $\bar{W}[g]$. We have that inside
    $W$ there exists a cofinal well-founded branch $\bar{b}$ through
    the iteration tree $\T^*$ by property $(4)$ in the definition of
    pre-$(2n+1)$-suitability for $N$ as above in Case $1$.

    \begin{minipage}{\textwidth}
      \bigskip \textbf{Case 3.1.}
      There is a drop along $\bar{b}$. \bigskip
    \end{minipage}

    Then we have as in Case $1$ that there exists a $\Q$-structure
    $\Q(\bar{\T}) \unlhd \M_{\bar{b}}^{\bar{\T}}$ for
    $\bar{\T}$. Consider the statement
    \begin{eqnarray*}
      \psi(\bar{\T}, \Q(\bar{\T})) \equiv &\text{``there is a cofinal
                                            branch } b \text{ through } \bar{\T} \text{ such
                                            that } \\ 
                                          & \Q(\bar{\T}) \unlhd
                                            \M_{b}^{\bar{\T}}\text{''.}   
    \end{eqnarray*}
    This statement $\psi(\bar{\T}, \Q(\bar{\T}))$ is
    $\Sigma^1_1$-definable from the parameters $\bar{\T}$ and
    $\Q(\bar{\T})$ and holds in the model $W$ as witnessed by the
    branch $\bar{b}$. By an absoluteness argument as above it follows
    that it also holds in $\bar{W}[g]$, which contradicts the fact
    that $\bar{\T}$ witnesses in $\bar{W}[g]$ that $\phi_3(\bar{N})$
    holds.

    \begin{minipage}{\textwidth}
      \bigskip \textbf{Case 3.2.}  There is no drop along
      $\bar{b}$. \bigskip
    \end{minipage}

    Then we can consider the coiteration of $\bar{\mathcal{P}}$ and
    $\M_{\bar{b}}^{\T^*}$ inside the model $W$. We have that both
    premice are $2n$-small above $\delta(\bar{\T})$. Moreover this
    coiteration takes place above $\delta(\bar{\T})$ since we have
    that $\bar{\mathcal{P}} \rhd \M(\bar{\T})$. Therefore the
    coiteration is successful inside $W$ using Lemma $2.2$ in
    \cite{St95} by the same argument as the one we gave in Cases $1$
    and $2$, because in $W$ we have that $\bar{\mathcal{P}}$ is
    $\Pi^1_{2n+1}$-iterable above $\delta(\bar{\T})$ and
    $\M_{\bar{b}}^{\T^*}$ is $(\omega_1+1)$-iterable above
    $\delta(\bar{\T})$ in $W$ using property $(4)$ in Definition
    \ref{def:prensuitable}. That means we have that
    \[ \M_{\bar{b}}^{\T^*} \unlhd \bar{\mathcal{P}} \text{ or }
    \bar{\mathcal{P}} \unlhd \M_{\bar{b}}^{\T^*}. \]

    \begin{minipage}{\textwidth}
      \bigskip \textbf{Case 3.2.1.}
      $\M_{\bar{b}}^{\T^*} \unlhd \bar{\mathcal{P}}.$ \bigskip
    \end{minipage}

    In this case we have that
    \[ \bar{\mathcal{P}} \cap \Ord < \bar{W} \cap \Ord \leq
    \M_{\bar{b}}^{\T^*} \cap \Ord \leq \bar{\mathcal{P}} \cap \Ord, \]
    where the first inequality holds true since
    $\bar{\mathcal{P}} \in \bar{W}[g]$. This is a contradiction.

    \begin{minipage}{\textwidth}
      \bigskip \textbf{Case 3.2.2.}
      $\bar{\mathcal{P}} \unlhd \M_{\bar{b}}^{\T^*}.$ \bigskip
    \end{minipage}

    Then we have that in fact
    \[ \bar{\mathcal{P}} \unlhd \M_{\bar{b}}^{\bar{\T}}, \]
    because $\delta(\bar{\T})$ is the largest cardinal in
    $\bar{\mathcal{P}}$. This contradicts $\phi_3(\bar{N})$.

    Therefore it follows for an iteration tree $\T$ as in Case $3$
    that there exists a cofinal well-founded branch through $\T$ and
    if there exists a non-dropping such branch $b$, then the premouse
    $\M_b^\T$ is pre-$(2n+1)$-suitable as in the argument at the end
    of Case $2$ above.

    Now the argument we just gave for Case $3$ shows that in this case
    we have that in fact
    \[ \M_b^\T = M_{2n}(\M(\T))|(\delta(\T)^+)^{M_{2n}(\M(\T))}. \]
\end{proof}

\subsection{Correctness for $n$-suitable Premice}
\label{sec:corrnsuit}

In the following lemmas we prove some correctness results for suitable
premice in the sense of Definition \ref{def:nsuitnew}. We are stating
these lemmas only for the levels $(2n+1)$ at which we proved that
there exists a $(2n+1)$-suitable premouse (see Lemma
\ref{lem:nsuitnew}).

\begin{lemma}
\label{lem:colcorr}
Let $n \geq 0$ and $z \in \BS$. Assume that $M_{2n}^\#(x)$ exists for
all $x \in \BS$ and let $N$ be a $(2n+1)$-suitable $z$-premouse. Let
$\varphi$ be an arbitrary $\Sigma^1_{2n+3}$-formula and let
$a \in N \cap \BS$ be arbitrary. Then we have
\[ \varphi(a) \; \leftrightarrow \; \; \forces{N}{\Col(\omega,
  \delta_N)} \varphi(a), \]
where $\delta_N$ as usual denotes the largest cardinal in $N$.
\end{lemma}

Here we again write $a$ for the standard name $\check{a}$ for a real
$a \in N$.

\begin{proof}
  Let $n \geq 0$ and $z \in \BS$ be arbitrary and let $N$ be a
  $(2n+1)$-suitable $z$-premouse. Let $\varphi(a)$ be a
  $\Sigma^1_{2n+3}$-formula for a parameter $a \in N \cap \BS$. That
  means
  \[ \varphi(a) \equiv \exists x \forall y \, \psi(x,y,a) \]
  for a $\Sigma^1_{2n+1}$-formula $\psi(x,y,a)$. We first want to
  prove the downward implication, that means we want to prove that if
  $\varphi(a)$ holds in $V$, then
  \[ \forces{N}{\Col(\omega, \delta_N)} \varphi(a). \]

  Let $x^* \in V$ be a witness for the fact that $\varphi(a)$ holds in
  $V$. That means $x^*$ is a real such that
  \[ V \vDash \forall y \, \psi(x^*,y,a). \] Use Corollary $1.8$ from
  \cite{Ne95} to make the real $x^*$ generic over an iterate of $N$
  for the collapse of the image of $\delta_N$.  Since $N$ is
  $(2n+1)$-suitable, we have enough iterability to apply Corollary
  $1.8$ from \cite{Ne95}, so there exists a non-dropping iterate $N^*$
  of $N$ such that $N^*$ is $2n$-iterable\footnote{see Definition
    $1.1$ in \cite{Ne95}} and whenever $g$ is
  $\Col(\omega, \delta_{N^*})$-generic over $N^*$, then
  $x^* \in N^*[g]$. Moreover let $i: N \rightarrow N^*$ denote the
  corresponding iteration embedding. Since $N$ is $(2n+1)$-suitable
  and the iteration from $N$ to $N^*$ is non-dropping, we have that
  $N^*$ is pre-$(2n+1)$-suitable.

  Let $g \in V$ be $\Col(\omega, \delta_{N^*})$-generic over $N^*$.
  Since we have that
  $N^* = M_{2n}(N^*|\delta_{N^*}) |
  (\delta_{N^*}^+)^{M_{2n}(N^*|\delta_{N^*})}$,
  it follows that $g$ is also $\Col(\omega, \delta_{N^*})$-generic
  over the proper class model $M_{2n}(N^*|\delta_{N^*})$. Moreover we
  can construe the model $M_{2n}(N^*|\delta_{N^*})[g]$ as a
  $y$-premouse for some real $y$ (in fact $y = z \oplus x^*$, see for
  example \cite{SchSt09} for the fine structural details) which has
  $2n$ Woodin cardinals. This yields by Lemma \ref{lem:corr} that the
  premouse $M_{2n}(N^*|\delta_{N^*})[g]$ is $\SIGMA^1_{2n+2}$-correct
  in $V$, because $M_{2n}(N^*|\delta_{N^*})[g]$ construed as a
  $y$-premouse is $\omega_1$-iterable above $\delta_{N^*}$. Thus we
  have that
  \[ M_{2n}(N^*|\delta_{N^*})[g] \vDash \forall y \, \psi(x^*,y,a) \]
  as $x^* \in M_{2n}(N^*|\delta_{N^*})[g]$. This in fact can be
  obtained using only $\SIGMA^1_{2n+1}$-correctness of
  $M_{2n}(N^*|\delta_{N^*})[g]$ and downward absoluteness.

  Since the premice $M_{2n}(N^*|\delta_{N^*})[g]$ and $N^*[g]$
  agree on their reals it follows that
  \[ N^*[g] \vDash \forall y \, \psi(x^*,y,a). \]
  By homogeneity of the forcing $\Col(\omega, \delta_{N^*})$ we now
  obtain that
  \[ \forces{N^*}{\Col(\omega, \delta_{N^*})} \exists x \forall y \,
  \psi(x,y,a). \]
  Since $a$ is a real in $N$ it follows by elementarity that
  \[ \forces{N}{\Col(\omega, \delta_{N})} \exists x \forall y \,
  \psi(x,y,a), \] as desired.

  For the upward implication let
  $\varphi(a) \equiv \exists x \forall y \, \psi(x,y,a)$ again be a
  $\Sigma^1_{2n+3}$-formula for a real $a$ in $N$ and a
  $\Sigma^1_{2n+1}$-formula $\psi(x,y,a)$ and assume that we have
  \[ \forces{N}{\Col(\omega, \delta_{N})} \exists x \forall y \,
  \psi(x,y,a). \]
  Let $g$ be $\Col(\omega, \delta_N)$-generic over the premouse
  $M_{2n}(N|\delta_N)$ and pick a real $x^* \in M_{2n}(N|\delta_N)[g]$
  such that
  \[ M_{2n}(N|\delta_N)[g] \vDash \forall y \, \psi(x^*,y,a). \]
  Since $M_{2n}(N|\delta_N)|(\delta_N^+)^{M_{2n}(N|\delta_N)} = N$ is
  countable in $V$, we can pick $g \in V$ and then get that
  $x^* \in V$. As above we can consider $M_{2n}(N|\delta_N)[g]$ as an
  $\omega_1$-iterable $y$-premouse for some real $y$ and therefore we
  have by Lemma \ref{lem:corr} again that
  \[ M_{2n}(N|\delta_N)[g] \prec_{\SIGMA^1_{2n+2}} V. \]
  Hence we have that
  \[ V \vDash \exists x \forall y \, \psi(x,y,a), \]
  witnessed by the real $x^*$, because
  ``$\forall y \, \psi(x^*,y,a)$'' is a $\Pi^1_{2n+2}$-formula.
\end{proof}

\begin{lemma}\label{lem:nsuitclosed}
  Let $n \geq 0$ and assume that $M_{2n}^\#(x)$ exists for all
  $x \in \BS$.  Let $N$ be a $(2n+1)$-suitable $z$-premouse for some
  $z \in \BS$.  Then $N|\delta_N$ is closed under the operation
  \[ A \mapsto M_{2n}^\#(A). \]
\end{lemma}

\begin{proof} It is enough to consider sets $A$ of the form $N | \xi$
  for some ordinal $\xi < \delta_N$ for the following reason. Let
  $A \in N|\delta_N$ be arbitrary. Then there exists an ordinal
  $\xi < \delta_N$ such that $A \in N | \xi$. Assume first that the
  ordinal $\xi$ is not overlapped by an extender on the $N$-sequence,
  that means there is no extender $E$ on the $N$-sequence such that
  $\crit(E) \leq \xi < \lh(E)$. We will consider the case that $\xi$
  is overlapped by an extender on the $N$-sequence later. Moreover
  assume that we already proved that
  \[ M_{2n}^\#(N|\xi) \lhd N | \delta_N. \] Then we also have that
  $M_{2n}^\#(A) \in N | \delta_N$ by the following argument. Consider
  the model $M_{2n}(N | \xi)$. Let $L[E](A)^{M_{2n}(N|\xi)}$ be as in
  Definition \ref{def:L[E]constr}. Then add the top measure of the
  active premouse $M_{2n}^\#(N | \xi)$ (intersected with
  $L[E](A)^{M_{2n}(N|\xi)}$) to an initial segment of
  $L[E](A)^{M_{2n}(N|\xi)}$ as described in Section $2$ of
  \cite{FNS10}. The main result in \cite{FNS10} yields that the model
  we obtain from this construction is again $\omega_1$-iterable and
  not $2n$-small. Thus it follows that the $\omega_1$-iterable
  premouse $M_{2n}^\#(A)$ exists inside $N | \delta_N$ as
  $M_{2n}^\#(N|\xi) \lhd N | \delta_N$.

  So let $\xi < \delta_N$ be an ordinal and assume as above first that
  $\xi$ is not overlapped by an extender on the $N$-sequence. Then we
  consider the premouse $M_{2n}^\#(N | \xi)$, which exists in $V$ by
  assumption and is not $2n$-small above $\xi$ because $\xi$ is
  countable in $V$. In this case we are left with showing that
  \[ M_{2n}^\#(N|\xi) \lhd N|\delta_N. \]
  Let $x$ be a real in $V$ which codes the countable premice
  $M_{2n}^\#(N|\xi)$ and $N|\delta_N$. Work inside the model
  $M_{2n}^\#(x)$ and coiterate $M_{2n}^\#(N|\xi)$ with
  $N|\delta_N$. We have that $N | \delta_N$ is short tree iterable
  inside $M_{2n}^\#(x)$, because for a pre-$(2n+1)$-suitable premouse
  short tree iterability is a $\Pi^1_{2n+2}$-definable statement by
  Lemma \ref{lem:defstiodd} and the model $M_{2n}^\#(x)$ is
  $\SIGMA^1_{2n+2}$-correct in $V$ by Lemma \ref{lem:corr}. In fact
  Lemma \ref{lem:corr} implies that $N$ has an iteration strategy
  which is fullness preserving in the sense of Definition
  \ref{def:nsuitnew} inside the model $M_{2n}^\#(x)$ by the proof of
  Lemma \ref{lem:nsuitnew}.

  Note that the coiteration takes place above $N|\xi$ and that
  $M_{2n}^\#(N|\xi)$ is $\omega_1$-iterable above $N|\xi$ in $V$ by
  definition. Therefore Lemma \ref{lem:iterability} $(2)$ implies that
  the comparison inside $M_{2n}^\#(x)$ cannot fail on this side of the
  coiteration. Say that the coiteration yields an iteration tree $\T$
  on $M_{2n}^\#(N|\xi)$ and an iteration tree $\U$ on $N|\delta_N$.

  We have that $\U$ is a short tree on $N|\delta_N$, because the
  $M_{2n}^\#(N|\xi)$-side of the coiteration provides
  $\Q$-structures. So the coiteration terminates successfully and
  there is an iterate $M^*$ of $M_{2n}^\#(N|\xi)$ via $\T$ and an
  iterate $R$ of $N|\delta_N$ via $\U$.

  \begin{claim}\label{cl:M2ncannotwin}
    The $M_{2n}^\#(N|\xi)$-side cannot win the comparison against
    $N|\delta_N$.
  \end{claim}

\begin{proof}
  Assume toward a contradiction that the $M_{2n}^\#(N|\xi)$-side wins
  the comparison. That means there is no drop on the main branch on
  the $N|\delta_N$-side of the coiteration and
  \[ M^* \unrhd R. \] We can consider $\U$ as a tree on the premouse
  $N \rhd N | \delta_N$ with final model $N^*$ such that there is an
  ordinal $\delta^*$ with
  \[ N^*|\delta^* = R. \]

  Then $\U$ is a short tree of length $\lambda + 1$ for some ordinal
  $\lambda$ and as there is no drop on the main branch on the
  $N|\delta_N$-side of the coiteration, we have that
  \[ \delta^* = i_{0\lambda}^\U(\delta_N), \]
  where $i_{0\lambda}^\U: N \rightarrow N^*$ denotes the corresponding
  iteration embedding. Moreover recall that $N$ is $(2n+1)$-suitable
  and so $\U$ is obtained using the iteration strategy for $N$ which
  is fullness preserving for short trees which do not drop on their
  main branch in the sense of Definition \ref{def:nsuitnew}. Therefore
  we have that $M_{2n}^\#(N^*|\delta^*)$ is not $2n$-small above
  $\delta^*$ and
  \[ M_{2n}^\#(N^*|\delta^*) \vDash \text{``} \delta^* \text{ is
    Woodin''}. \]

  We have for the other side of the coiteration that
  $\rho_\omega(M^*) < \delta^*$, because
  $\rho_\omega(M_{2n}^\#(N|\xi)) \leq \xi < \delta^*$. Let $\Q$ be the
  least initial segment of $M^*$ such that $\delta^*$ is not definably
  Woodin over $\Q$. Recall that this means that $\Q \unlhd M^*$ is
  such that
  \[ \Q \vDash \text{``} \delta^* \text{ is Woodin'',} \] but if
  $\Q = J_\alpha^{M^*}$ for some $\alpha < M^* \cap \Ord$ then
  \[ J_{\alpha + 1}^{M^*} \vDash \text{``} \delta^* \text{ is not
    Woodin''}, \]
  and if $\Q = M^*$ then $\rho_\omega(\Q) < \delta^*$ or there exists
  an $m < \omega$ and an $r\Sigma_m$-definable set
  $A \subset \delta^*$ such that there is no $\kappa < \delta^*$ such
  that $\kappa$ is strong up to $\delta^*$ with respect to $A$ as
  witnessed by extenders on the sequence of $\Q$. In the latter case
  we have that in particular $\rho_\omega(\Q) \leq \delta^*$. Moreover
  $\Q$ is $2n$-small above $\delta^*$, because we have $\Q \unlhd M^*$
  and
  \[ \Q \vDash \text{``} \delta^* \text{ is Woodin''}. \]
  Furthermore $\Q$ is $\omega_1$-iterable above $\delta^*$.

  By construction we have that the premouse $M_{2n}^\#(N^*|\delta^*)$
  is $\omega_1$-iterable above $\delta^*$. Consider the
  coiteration of $\Q$ and $M_{2n}^\#(N^*|\delta^*)$ inside the model
  $M_{2n}^\#(y)$, where $y$ is a real coding the countable premice
  $\Q$ and $M_{2n}^\#(N^* | \delta^*)$. Since
  \[ N^*|\delta^* = R \unlhd M^* \] this coiteration takes place above
  $\delta^* \geq \delta(\U)$. We have that
  $\rho_\omega(\Q) \leq \delta^*$ and
  \[ \rho_\omega(M_{2n}^\#(N^*|\delta^*)) \leq \delta^*. \]
  Moreover $\Q$ and $M_{2n}^\#(N^*|\delta^*)$ are both sound above
  $\delta^*$. Thus the comparison is successful by Lemma
  \ref{lem:iterability} and we have that
  \[ \Q \unlhd M_{2n}^\#(N^*|\delta^*) \text{ or }
  M_{2n}^\#(N^*|\delta^*) \unlhd \Q. \]
  The premouse $M_{2n}^\#(N^*|\delta^*)$ is not $2n$-small above
  $\delta^*$ because $N^*$ is pre-$(2n+1)$-suitable. Since $\Q$ is
  $2n$-small above $\delta^*$ we have in fact that
  \[ \Q \lhd M_{2n}^\#(N^*|\delta^*). \]
  But this implies by our choice of the premouse $\Q$ that $\delta^*$
  is not Woodin in $M_{2n}^\#(N^*|\delta^*)$, which is a contradiction
  because by fullness preservation we have that $\delta^*$ is a Woodin
  cardinal in $M_{2n}^\#(N^*|\delta^*)$.

  Therefore we have that the $M_{2n}^\#(N|\xi)$-side cannot win the
  comparison against $N|\delta_N$.
  \end{proof}

  \begin{claim}\label{cl:firstsidedoesnotmove}
    The $M_{2n}^\#(N|\xi)$-side does not move in the coiteration with
    $N | \delta_N$. 
  \end{claim}
  
  \begin{proof}
    Assume toward a contradiction that the $M_{2n}^\#(N|\xi)$-side
    moves in the coiteration. As the coiteration takes place above
    $\xi$ this means that it has to drop on the
    $M_{2n}^\#(N|\xi)$-side. Then $M^* \unrhd R$ and there is no drop
    on the main branch on the $N|\delta_N$-side of the coiteration,
    i.e. the $M_{2n}^\#(N|\xi)$-side wins the comparison. This
    contradicts Claim \ref{cl:M2ncannotwin}.
  \end{proof}

  \begin{claim}\label{cl:secondsidedoesnotmove}
    The $N | \delta_N$-side does not move in the coiteration with
    $M_{2n}^\#(N|\xi)$.
  \end{claim}
  \begin{proof}
    Assume toward a contradiction that the $N|\delta_N$-side moves in
    this coiteration. Since the coiteration takes place above $\xi$
    this means that there is an ordinal $\gamma > \xi$ with
    $\gamma \leq M_{2n}^\#(N | \xi) \cap \Ord$ such that there is an
    extender $E_\gamma^N$ indexed at $\gamma$ on the $N$-sequence
    which is used in the coiteration. In particular $\gamma$ is a
    cardinal in the iterate $R$ of $N|\delta_N$. By Claim
    \ref{cl:firstsidedoesnotmove} we have
    $M_{2n}^\#(N | \xi) \unlhd R$, in particular $\gamma$ is a
    cardinal in $M_{2n}^\#(N | \xi)$. In fact,
    $M_{2n}^\#(N | \xi) \lhd R$ by the following argument. If
    $M_{2n}^\#(N | \xi) = R$ there must be a drop in the iteration to
    $R$ since otherwise we obtain a contradiction as in the proof of
    Claim \ref{cl:M2ncannotwin}. But then $R$ cannot be sound,
    contradicting the fact that $M_{2n}^\#(N | \xi)$ is sound. Hence,
    $\gamma$ is a cardinal in both $M_{2n}^\#(N|\xi)$ and $R$ and
    $M_{2n}^\#(N|\xi) \lhd R$. This is a contradiction because
  \[ \rho_\omega(M_{2n}^\#(N|\xi)) \leq \xi < \gamma. \]
  Therefore the $N|\delta_N$-side also does not move in the
  comparison. 
  \end{proof}
  By Claims \ref{cl:M2ncannotwin}, \ref{cl:firstsidedoesnotmove} and
  \ref{cl:secondsidedoesnotmove} we finally have that
  \[ M_{2n}^\#(N|\xi) \lhd N|\delta_N. \]

  We now have to consider the case that $A \in N | \xi$ for an ordinal
  $\xi < \delta_N$ such that $\xi$ is overlapped by an extender $E$ on
  the $N$-sequence. That means there is an extender $E$ on the
  $N$-sequence such that $\crit(E) \leq \xi < \lh(E)$. Let $E$ be the
  least such extender, that means the index of $E$ is minimal among
  all critical points of extenders on the $N$-sequence overlapping
  $\xi$. By the definition of a ``fine extender sequence'' (see
  Definition $2.4$ in \cite{St10}) we have that $A \in \Ult(N; E)$ and
  that the ordinal $\xi$ is no longer overlapped by an extender on the
  $\Ult(N;E)$-sequence. Let $M = \Ult(N;E)$ and consider the premouse
  $M | \xi$, so we have in particular that $A \in M | \xi$. The same
  argument as above for $N | \xi$ replaced by $M | \xi$ proves that
  \[ M_{2n}^\#(M | \xi) \lhd N | \delta_N \]
  and therefore we finally have that $M_{2n}^\#(A) \in N | \delta_N$
  by repeating the argument we already gave at the beginning of this
  proof.
\end{proof}

From Lemma \ref{lem:nsuitclosed} we can obtain the following lemma as
a corollary. 

\begin{lemma}
\label{lem:2n+1suitcorr}
Let $n \geq 0$ and $z \in \BS$. Assume that $M_{2n}^\#(x)$ exists for
all $x \in \BS$ and let $N$ be a $(2n+1)$-suitable $z$-premouse. Then
$N$ is $\SIGMA^1_{2n+2}$-correct in $V$ for real parameters in $N$ and
we write \[ N \prec_{\SIGMA^1_{2n+2}} V. \]
\end{lemma}

\begin{proof}
  Let $a$ be a real in $N$.  By Lemma \ref{lem:nsuitclosed} we have
  that \[ M_{2n}^\#(a) \in N. \] Moreover we have by Lemma
  \ref{lem:corr} that $M_{2n}^\#(a)$ is $\SIGMA^1_{2n+2}$-correct in
  $V$. Therefore it follows that $N$ is $\SIGMA^1_{2n+2}$-correct in
  $V$ by the following inductive argument. 

  Shoenfield's Absoluteness Theorem implies that $M_{2n}(N|\delta_N)$
  is $\SIGMA^1_2$-correct in $V$. Let $k \leq n$ and assume
  inductively that $M_{2n}(N|\delta_N)$ is $\SIGMA^1_{2k}$-correct in
  $V$. Let $\varphi(a)$ be a $\Sigma^1_{2k+2}$-formula for a parameter
  $a \in N \cap \BS$, i.e.
  $\varphi(a) \equiv \exists x \forall y \, \psi(x,y,a)$ for a
  $\Sigma^1_{2k}$-formula $\psi(x,y,a)$. First assume that
  $\varphi(a)$ holds in $V$. As $M_{2n}^\#(a)$ is
  $\SIGMA^1_{2k+2}$-correct in $V$, we have that $\varphi(a)$ holds in
  $M_{2n}^\#(a)$. Let $x \in M_{2n}^\#(a)$ be a real such that
  $M_{2n}^\#(a) \vDash \forall y \, \psi(x,y,a)$. Then in particular
  $V \vDash \forall y \, \psi(x,y,a)$. As $M_{2n}^\#(a) \in N$ we have
  that $x \in M_{2n}(N|\delta_N)$ and by downwards absoluteness
  combined with the inductive hypothesis it follows that
  $M_{2n}(N|\delta_N) \vDash \forall y \, \psi(x,y,a)$ as this
  statement is true in $V$.

  For the other direction assume that $\varphi(a)$ is true in
  $M_{2n}(N|\delta_N)$ and let $x \in M_{2n}(N|\delta_N)$ be a real
  such that $M_{2n}(N|\delta_N) \vDash \forall y \, \psi(x,y,a)$. Let
  $y \in M_{2n}^\#(a \oplus x) \cap \BS$ be arbitrary. Since
  $M_{2n}^\#(a \oplus x) \cap \BS \subset M_{2n}(N|\delta_N) \cap
  \BS$, we have that $M_{2n}(N|\delta_N) \vDash \psi(x,y,a)$. As
  $\psi(x,y,a)$ is a $\Sigma^1_{2k}$-formula, the inductive hypothesis
  implies that $V \vDash \psi(x,y,a)$. Since $M_{2n}^\#(a \oplus x)$
  is $\SIGMA^1_{2n+2}$-correct in $V$, it follows that
  $M_{2n}^\#(a \oplus x) \vDash \psi(x,y,a)$. But
  $y \in M_{2n}^\#(a \oplus x) \cap \BS$ was arbitrary, so
  $M_{2n}^\#(a \oplus x) \vDash \exists x \forall y \,
  \psi(x,y,a)$. Therefore the fact that $M_{2n}^\#(a \oplus x)$ is
  $\SIGMA^1_{2k+2}$-correct in $V$ implies $V \vDash \varphi(a)$.
  
  This inductive argument shows that $M_{2n}(N|\delta_N)$ is
  $\SIGMA^1_{2n+2}$-correct in $V$. But $N$ and $M_{2n}(N|\delta_N)$
  agree on their reals and hence $N$ is $\SIGMA^1_{2n+2}$-correct in
  $V$, as desired.
\end{proof}

Note that we could have also proven Lemma \ref{lem:2n+1suitcorr} by an
argument similar to the one we gave in the proof of Lemma
\ref{lem:colcorr} but with additionally using the uniformization
property as in the proof of Lemma \ref{lem:corr} $(1)$.


\subsection{Outline of the Proof}
\label{sec:outline}

Our main goal for the rest of this paper is to give a proof of the
following theorem.

\begin{thm}\label{cor:newWoodin1}
  Let $n \geq 0$ and assume there is no $\SIGMA^1_{n+2}$-definable
  $\omega_1$-sequence of pairwise distinct reals. Then the
  following are equivalent.
\begin{enumerate}[(1)]
\item $\PI^1_n$ determinacy and $\Pi^1_{n+1}$ determinacy,
\item for all $x \in \BS$, $M_{n-1}^\#(x)$ exists and is
  $\omega_1$-iterable, and $M_n^\#$ exists and is $\omega_1$-iterable,
\item $M_n^\#$ exists and is $\omega_1$-iterable.
\end{enumerate}
\end{thm}

For $n=0$ this is due to L. Harrington (see \cite{Ha78}) and
D. A. Martin (see \cite{Ma70}). 

The results that $(3)$ implies $(2)$ and that $(2)$ implies $(1)$ for
$\omega_1$-iterable premice $M_n^\#$ are due to the third author for
odd $n$ (unpublished) and due to Neeman for even $n>0$ (see Theorem
$2.14$ in \cite{Ne02}), building on work of Martin and Steel (see
\cite{MaSt89}). Moreover the results that $(3)$ implies $(2)$ and that
$(2)$ implies $(1)$ hold without the background hypothesis that every
$\SIGMA^1_{n+2}$-definable sequence of pairwise distinct reals is
countable.

We will focus on the proof of the following theorem, which is the
implication ``$(1) \Rightarrow (3)$'' in Theorem \ref{cor:newWoodin1}
and due to the third author.

\begin{thm}\label{thm:newWoodin1}
  Let $n \geq 1$ and assume there is no $\SIGMA^1_{n+2}$-definable
  $\omega_1$-sequence of pairwise distinct reals. Moreover assume
  that $\PI^1_n$ determinacy and $\Pi^1_{n+1}$ determinacy hold. Then
  $M_n^\#$ exists and is $\omega_1$-iterable.
\end{thm}

\begin{remark}
  Let $n \geq 0$. Then we say that there exists a
  \emph{$\SIGMA^1_{n+2}$-definable $\omega_1$-sequence of pairwise
    distinct
    reals}\index{SIGMA1ndefinableomega1sequence@$\SIGMA^1_{n+2}$-definable
    $\omega_1$-sequence of pairwise distinct reals} iff there exists a
  well-order $\leq^*$ of ordertype $\omega_1$ for reals such that if
  we let $X_{\leq^*} = field(\leq^*)$, that means if we have for all
  $y \in \BS$ that
  \[y \in X_{\leq^*} \; \Leftrightarrow \; \exists x \; (x \leq^* y
  \vee y \leq^* x), \]
  then there exists a $\SIGMA^1_{n+2}$-definable relation $R$ such
  that we have for all $x,y \in \BS$,
  \[ R(x,y) \; \Leftrightarrow \; x,y \in X_{\leq^*} \; \wedge \; x
    \leq^* y. \]
\end{remark}

We will first show in Section \ref{sec:assumption} how boldface
determinacy for a level $\PI^1_{n+1}$ of the projective hierarchy can
be used to prove that every sequence of pairwise distinct reals
which is $\SIGMA^1_{n+2}$-definable is in fact countable.

This will enable us to conclude Theorem \ref{cor:bfDetSharps} from
Theorem \ref{thm:newWoodin1}. The odd levels in the inductive proof of
Theorem \ref{thm:newWoodin1} will finally be proven in Section
\ref{sec:odd} and the even levels in Section \ref{sec:even}. As
mentioned before the proof for the odd and the even levels of the
projective hierarchy will be different because of the periodicity in
the projective hierarchy in terms of uniformization and correctness.


\subsection{Making use of the Right Hypothesis}
\label{sec:assumption}

In this section we will provide one ingredient for the proof of
Theorem \ref{cor:bfDetSharps}. We will see that the results from
Sections \ref{sec:odd} and \ref{sec:even} can actually be obtained
from boldface determinacy at the right level of the projective
hierarchy. The following lemma makes this precise. The result
essentially goes back to Mansfield \cite{Man75}; the proof we give is
from \cite{Ke78} (see also \cite[Theorem 25.39]{Jec03}).

\begin{lemma} 
\label{lem:bfDetSeq}
Let $n \geq 0$. Then $\PI^1_{n+1}$ determinacy implies that every
$\SIGMA^1_{n+2}$-definable sequence of pairwise distinct reals is
countable.
\end{lemma}

\begin{remark}
  In fact, instead of $\PI^1_{n+1}$ determinacy we could also just
  assume the weaker hypothesis that $\SIGMA^1_{n+2}$ has the perfect
  set property. The proof only uses this consequence of $\PI^1_{n+1}$
  determinacy.
\end{remark}

\begin{proof}[Proof of Lemma \ref{lem:bfDetSeq}]
  Suppose there is a well-order $\leq^*$ of ordertype $\omega_1$ for
  reals and a $\SIGMA^1_{n+2}$-definable relation $R$ as in the remark
  after Theorem \ref{thm:newWoodin1}. In what follows we identify
  reals with elements of $2^\omega$ and let $X_{\leq^*} =
  field(X)$. By $\PI^1_{n+1}$ determinacy, $\SIGMA^1_{n+2}$ has the
  perfect set property (see \cite[27.14]{Ka03}). Therefore, since
  $X_{\leq^*}$ is uncountable, $X_{\leq^*}$ has a perfect subset
  $P \subseteq X_{\leq^*}$. Now it suffices to show the following
  claim.
  
  \begin{claim}\label{cl:bfDetSeq}
    Suppose $P_0$ is a perfect set and $f_0 \colon P_0 \rightarrow P$
    is a continuous injection. Then there is a perfect set
    $P_1 \subseteq P_0$ and a continuous injection
    $f_1 \colon P_1 \rightarrow P$ such that $f_1(x) <^* f_0(x)$ for
    all $x \in P_1$.
  \end{claim}
  
  Using this claim, we can define a sequence $(P_n \st n < \omega)$ of
  perfect sets $P_n$ together with continuous injections
  $f_n \colon P_n \rightarrow P$ such that for all $n < \omega$,
  $P_{n+1} \subseteq P_n$ and $f_{n+1}(x) <^* f_n(x)$ for all
  $x \in P_{n+1}$. Since $2^\omega$ is a compact space, all $P_n$'s
  are closed, and intersections of finitely many of the $P_n$'s are
  non-empty, $\bigcap_{n<\omega} P_n$ is non-empty. But if we let
  $x \in \bigcap_{n<\omega} P_n$, then $f_{n+1}(x) <^* f_n(x)$ for all
  $n<\omega$, a contradiction.

  \begin{proof}[Proof of Claim \ref{cl:bfDetSeq}]
    Let $h \colon P_0 \rightarrow P$ and
    $h^* \colon P_0 \rightarrow P$ be two continuous bijections such
    that $h(x) \neq h^*(x)$ for all $x \in P_0$. Such bijections can
    easily be obtained by considering homeomorphisms
    $g_0 \colon P_0 \rightarrow 2^{\omega}$ and
    $g \colon P \rightarrow 2^{\omega}$. Then we can define
    $h = g^{-1} \circ g_0$ and $h^*= g^{-1} \circ \pi \circ g_0$,
    where $\pi \colon 2^{\omega} \rightarrow 2^{\omega}$ is the
    homeomorphism flipping $0$ and $1$ in each coordinate.

    Let \[ A = \{ x \in P_0 \st f_0(x) >^* h(x)\} \] and \[ A^* = \{ x
      \in P_0 \st f_0(x) >^* h^*(x)\}. \]

    \begin{subclaim}
      $A$ or $A^*$ has a perfect subset.
    \end{subclaim}
    \begin{proof}
      Assume not. Then, since $A$ and $A^*$ are
      $\Sigma^1_{2n+2}$-definable, the perfect set property implies
      that $A \cup A^*$ is countable. Let $\mathcal{O}$ be the orbit
      of $A \cup A^*$ under $f_0$, $h$, and $h^*$, i.e. $\mathcal{O}$
      is the smallest set $Q$ such that $A \cup A^* \subseteq Q$ and
      for all $x \in Q$,
      $f_0(x), f_0^{-1}(x), h(x), h^{-1}(x), h^*(x), (h^*)^{-1}(x) \in
      Q$, whenever these maps are defined on $x$.

      Then $\mathcal{O}$ is countable, so we can let $z$ be the
      $<^*$-minimal element in $P \setminus \mathcal{O}$. Let
      $x, x^* \in P_0$ be such that $h(x) = z = h^*(x^*)$. Then
      $x \neq x^*$ and $x, x^* \notin \mathcal{O}$. This implies that
      $f_0(x) \neq f_0(x^*)$ and
      $f_0(x), f_0(x^*) \notin \mathcal{O}$. By minimality of $z$, we
      have $z \leq^* f_0(x)$ and $z \leq^* f_0(x^*)$. But since
      $f_0(x) \neq f_0(x^*)$, this yields that $h(x) = z <^* f_0(x)$
      or $h^*(x) = z <^* f_0(x^*)$. This means, $x \in A$ or
      $x \in A^*$, contradicting the fact that
      $x, x^* \notin \mathcal{O}$.
    \end{proof}

    Suppose $A$ has a perfect subset $P_A$. Then
    $P_1 = P_A \subseteq P_0$ and $f_1 = h \upharpoonright P_A$ are as
    desired. Analogously, suppose $B$ has a perfect subset $P_B$. Then
    $P_1 = P_B \subseteq P_0$ and $f_1 = h^* \upharpoonright P_B$ are
    as desired.
  \end{proof}
\end{proof}


\subsection{$M_{2n-1}^\#(x)$ from Boldface $\Pi^1_{2n}$ Determinacy} 
\label{sec:odd}

The goal of this section is to prove Theorem \ref{thm:newWoodin1} in
the case that $n$ is odd. The proof which is presented in the
following argument only works if $n$ is odd because of the periodicity
in the projective hierarchy in terms of uniformization (see Theorem
$6C.5$ in \cite{Mo09}). The even levels in the statement of Theorem
\ref{thm:newWoodin1} have to be treated differently (see Section
\ref{sec:even}). So we are going to prove the following theorem.

\begin{thm}
\label{thm:newWoodinodd}
Let $n \geq 1$ and assume that every $\PI^1_{2n-1}$-definable set of
reals and every $\Pi^1_{2n}$-definable set of reals is
determined. Moreover assume that there is no
$\SIGMA^1_{2n+1}$-definable $\omega_1$-sequence of pairwise distinct
reals. Then the premouse $M_{2n-1}^\#$ exists and is
$\omega_1$-iterable.
\end{thm}

The proof of Theorem \ref{thm:newWoodinodd} uses the following
Determinacy Transfer Theorem, which is due to A. S. Kechris and the
third author (see \cite{KW08}). In the version as stated below it
follows from \cite{KW08} using \cite{Ne95} and that Theorem
\ref{cor:newWoodin1} holds below $2n-1$ inductively.

\begin{thm}[Determinacy Transfer Theorem]
  \label{thm:transferthm}
  Let $n \geq 1$. Assume determinacy for every $\PI^1_{2n-1}$- and
  every $\Pi^1_{2n}$-definable set of reals. Then we have determinacy
  for all \emph{$\game^{(2n-1)} ( <\omega^2 - \Pi^1_1)$}-definable
  sets of reals.\footnote{See Section $6D$ in
  \cite{Mo09} for a definition and some basic facts about the game
  quantifier ``\,$\game$\,''. We let ``\,$\game^{(n)}$\,'' denote $n$
  successive applications of the game quantifier $\game$. For the
  definition of the difference hierarchy and in particular of the
  pointclass $\alpha-\Pi^1_1$ for an ordinal $\alpha < \omega_1$ see
  for example Section $31$ in \cite{Ka03}.}
\end{thm}

\begin{proof}
  The lightface version of Theorem $1.10$ in \cite{KW08} (see p. $369$
  in \cite{KW08}) gives that for all $n \geq 1$,
  \[ \Det(\game^{(2n-2)}(<\omega^2 - \PI^1_1)) \rightarrow
  [\Det(\Delta^1_{2n}) \leftrightarrow \Det(\game^{(2n-1)}(<\omega^2 -
  \Pi^1_1))]. \]
  As in the addendum §5 in \cite{KW08} we need to argue that
  $\PI^1_{2n-1}$ determinacy implies
  $\game^{(2n-2)}(<\omega^2 - \PI^1_1)$ determinacy to obtain that it
  implies the right-hand side of the implication, i.e.
  ``$\Det(\Delta^1_{2n}) \leftrightarrow \Det(\game^{(2n-1)}(<\omega^2
  - \Pi^1_1))$''.

  Recall that we assume inductively that Theorem \ref{cor:newWoodin1}
  holds for all $m < 2n-1$. Together with Lemma \ref{lem:bfDetSeq}
  this implies that the premouse $M_{2n-2}^\#(x)$ exists and is
  $\omega_1$-iterable for all $x \in \BS$ from $\PI^1_{2n-1}$
  determinacy. By Theorem $2.5$ in \cite{Ne95} this yields
  $\game^{(2n-2)}(<\omega^2 - \PI^1_1)$ determinacy. Therefore we have
  that
  \[ \Det(\PI^1_{2n-1}) \rightarrow [\Det(\Delta^1_{2n})
  \leftrightarrow \Det(\game^{(2n-1)}(<\omega^2 - \Pi^1_1))]. \]
  
  We have
  \[ \Det(\Delta^1_{2n}) \leftrightarrow \Det(\Pi^1_{2n}) \]
  by Theorem $5.1$ in \cite{KS85} which is due to Martin (see
  \cite{Ma73}). So in particular $\PI^1_{2n-1}$ determinacy and
  $\Pi^1_{2n}$ determinacy together imply that
  $\game^{(2n-1)}(<\omega^2 - \Pi^1_1)$ determinacy holds.
\end{proof}

Martin proves in \cite{Ma08} that under the assumption that $x^\#$
exists for every real $x$, \[ A \in \game( <\omega^2 - \Pi^1_1) \]
iff there is a formula $\phi$ such that for all $x \in \BS$
$$ x \in A \text{ iff } L[x] \vDash \phi[x, \gamma_1, \dots,
\gamma_k], $$
where $\gamma_1, \dots, \gamma_k$ are Silver indiscernibles for
$x$. In the light of this result (see also Definition $2.7$ in
\cite{Ne02} for the general case) we can obtain the following
corollary of the Determinacy Transfer Theorem \ref{thm:transferthm}.

\begin{cor}
\label{cor:transferthm}
Let $n \geq 1$. Assume that $\PI^1_{2n-1}$ determinacy and
$\Pi^1_{2n}$ determinacy hold. Suppose $Q$ is a set of reals such that
there is an $m < \omega$ and a formula $\phi$ such that for all
$x \in \BS$
$$ x \in Q \text{ iff } M_{2n-2}(x) \vDash \phi(x,E,\gamma_1, \dots,
\gamma_m), $$
where $E$ is the extender sequence of $M_{2n-2}(x)$ and
$\gamma_1, \dots, \gamma_m$ are the first $m$ indiscernibles of
$M_{2n-2}(x)$. Then $Q$ is determined.
\end{cor}

This follows from the Determinacy Transfer Theorem
\ref{thm:transferthm} as the set $Q$ defined in Corollary
\ref{cor:transferthm} is
$\game^{(2n-1)}(<\omega^2 - \Pi^1_1)$-definable. Moreover note in the
statement of Corollary \ref{cor:transferthm} that $\PI^1_{2n-1}$
determinacy inductively implies that the premouse $M_{2n-2}^\#(x)$
exists for every real $x$.

Now we are ready to prove Theorem \ref{thm:newWoodinodd}.

\begin{proof}[Proof of Theorem \ref{thm:newWoodinodd}]
  Let $n \geq 1$. We assume inductively that Theorem
  \ref{thm:newWoodin1} holds for $2n-2$, that means we assume that
  $\PI^1_{2n-1}$ determinacy implies that $M_{2n-2}^\#(x)$ exists and
  is $\omega_1$-iterable for all reals $x$. These even levels will be
  proven in Theorem \ref{thm:newWoodineven}. Recall that then our
  hypothesis implies by Lemmas \ref{lem:prensuitable} and
  \ref{lem:nsuitnew} that there exists a $(2n-1)$-suitable premouse.

  Let $x \in \BS$ and consider the simultaneous comparison of all
  $(2n-1)$-suitable premice $N$ such that $N$ is coded by a real
  $y_N \leq_T x$, using the iterability stated in Definition
  \ref{def:nsuitnew}. That means analogous to a usual comparison as in
  Theorem $3.14$ in \cite{St10} we iterate away the least disagreement
  that exists between any two of the models we are considering.

  These premice successfully coiterate to a common model since they
  are all $(2n-1)$-suitable and call this common iterate $N_x$. We
  could have performed this simultaneous comparison inside an inner
  model of $V$ of height $\omega_1^V$ which contains $x$ and is closed
  under the operation $a \mapsto M_{2n-2}^\#(a)$ and therefore the
  resulting premouse $N_x$ is countable in $V$.

  Since $N_x$ results from a successful comparison using iterability
  in the sense of Definition \ref{def:nsuitnew}, there either exists a
  $(2n-1)$-suitable premouse $N$ and a non-dropping iteration from $N$
  to $N_x$ via a short iteration tree or there exists a maximal
  iteration tree $\T$ on a $(2n-1)$-suitable premouse such that
  \[ N_x = M_{2n-2}(\M(\T)) \, | \,
  (\delta(\T)^+)^{M_{2n-2}(\M(\T))}. \]

  Let $\delta_{N_x}$ as usual denote the largest cardinal in $N_x$.
  Then we have that in the first case $N_x$ is a $(2n-1)$-suitable
  premouse again by fullness preservation and so in particular the
  model $M_{2n-2}(N_x|\delta_{N_x})$ constructed in the sense of
  Definition \ref{def:M(N)} is a well-defined proper class premouse
  with $2n-1$ Woodin cardinals. In the second case we have by
  maximality of $\T$ that
  \[ M_{2n-2}(N_x|\delta_{N_x}) \vDash \text{``} \delta(\T) \text{ is
    Woodin''}, \]
  with $\delta_{N_x} = \delta(\T)$, and $M_{2n-2}(N_x|\delta_{N_x})$
  is again a well-defined premouse with $2n-1$ Woodin cardinals.

  For each formula $\phi$ and each $m < \omega$ let $Q^\phi_m$ be the
  set of all $x \in \BS$ such that
  \[ M_{2n-2}(N_x|\delta_{N_x}) \vDash \phi(E, \gamma_1, \dots,
  \gamma_m), \]
  where $E$ is the extender sequence of $M_{2n-2}(N_x|\delta_{N_x})$
  and $\gamma_1, \dots, \gamma_m$ are indiscernibles of
  $M_{2n-2}(N_x|\delta_{N_x})$.

  \begin{claim}\label{cl:QmphiDet}
    For all formulas $\phi$ and for all $m < \omega$ the set
    $Q^\phi_m$ is determined.
  \end{claim}

  \begin{proof}
    We aim to reduce determinacy for the set $\Q_m^\phi$ to
    determinacy for a set $Q$ as in Corollary \ref{cor:transferthm}.

    Recall that by Lemma \ref{lem:nsuitnew} we have that a premouse
    $N$ is $(2n-1)$-suitable iff it is pre-$(2n-1)$-suitable, that
    means iff it satisfies the following properties for an ordinal
    $\delta_0$.
    \begin{enumerate}[(1)]
    \item
      $N \vDash \text{``} \ZFC^- + \; \delta_0 \text{ is the largest
        cardinal''}$,
      \[ N = M_{2n-2}(N|\delta_0) \, | \,
      (\delta_0^+)^{M_{2n-2}(N|\delta_0)}, \]
      and for every $\gamma < \delta_0$,
      \[ M_{2n-2}(N|\gamma) \, | \, (\gamma^+)^{M_{2n-2}(N|\gamma)}
      \lhd N, \]
    \item $M_{2n-2}(N|\delta_0)$ is a proper class model and
      \[ M_{2n-2}(N|\delta_0) \vDash \text{``} \delta_0 \text{ is Woodin''}, \]
    \item for every $\gamma < \delta_0$, $M_{2n-2}(N | \gamma)$ is a set,
      or
      \[ M_{2n-2}(N|\gamma) \nvDash \text{``} \gamma \text{ is
        Woodin''}, \] and
    \item for every $\eta < \delta_0$,
      $M_{2n-2}(N|\delta_0) \vDash \text{``} N|\delta_0 \text{ is }
      \eta \text{-iterable''}$.
    \end{enumerate}

    Formally the definition of pre-$(2n-1)$-suitability requires that
    in the background universe the premouse $M_{2n-2}^\#(z)$ exists
    for all $z \in \BS$. But for a premouse $N$ which is coded by a
    real $y_N \leq_T x$, the model $M_{2n-2}(x)$ can compute if $N$ is
    pre-$(2n-1)$-suitable in $V$, i.e. if $N$ satisfies properties
    $(1)$ - $(4)$ above in $V$, by considering a fully backgrounded
    extender construction as in Definition \ref{def:L[E]constr} above
    $N|\delta_0$, where $\delta_0$ denotes the largest cardinal in
    $N$. Therefore we can make sense of these $(2n-1)$-suitable
    premice inside the model $M_{2n-2}(x)$. For the same reason the
    simultaneous comparison of all such $(2n-1)$-suitable premice
    which are coded by a real $y_N \leq_T x$ as introduced above can
    be carried out inside $M_{2n-2}(x)$ and therefore we have that
    $N_x \in M_{2n-2}(x)$.
    
    Now consider the following formula $\psi_\phi$, where $\phi$ as
    above is an arbitrary formula.
   \begin{align*}
      \psi_\phi(x, E, \gamma_1, \dots, \gamma_m) \,
      \equiv \, & \text{``Write } N_x \text{ for the result of the simultaneous}\\
                & \text{comparison of all } (2n-1)\text{-suitable
                  premice } N \\
                & \text{which are coded by a real } y_N \leq_T x
                  \text{ and}\\
                & \delta_{N_x} \text{ for the largest cardinal in }
                  N_x, \text{ then}\\
                & \; \; \; \; \; L[\bar{E}](N_x|\delta_{N_x}) \vDash \phi(E^*, \gamma_1,
                  \dots, \gamma_m), \\
                & \text{where } E^* \text{ denotes the extender
                  sequence} \\ & \text{of }
                                 L[\bar{E}](N_x|\delta_{N_x}),
                                 \text{which in turn is computed from } E
      \\ & \text{via a fully backgrounded
           extender construction}\\ &  \;\;\; \text{over }
                                      N_x|\delta_{N_x} \text{ as
                                      in Definition
                                      \ref{def:L[E]constr}}\text{.''}    
    \end{align*}

    Now let $E$ denote the extender sequence of the $x$-premouse
    $M_{2n-2}(x)$ and let $\gamma_1, \dots, \gamma_m$ denote
    indiscernibles of the model
    $L[M_{2n-2}(x) | \delta] = M_{2n-2}(x)$, where $\delta$ is the
    largest Woodin cardinal in $M_{2n-2}(x)$. Then we have in
    particular that $\gamma_1, \dots, \gamma_m$ are indiscernibles of
    the model $L[\bar{E}](N_x|\delta_{N_x})$, constructed via a fully
    backgrounded extender construction inside $M_{2n-2}(x)$ as defined
    in the formula $\psi_\phi(x, E, \gamma_1, \dots, \gamma_m)$ above.

    Therefore we have that
    \[ x \in \Q_m^\phi \Leftrightarrow M_{2n-2}(x) \vDash \psi_\phi(x,
    E, \gamma_1, \dots, \gamma_m), \]
    because the premice $L[\bar{E}](N_x|\delta_{N_x})$ as defined
    above and $M_{2n-2}(N_x | \delta_{N_x})$ as in the definition of
    the set $\Q_m^\phi$ coiterate to the same model.

    Thus Corollary \ref{cor:transferthm} implies that $\Q_m^\phi$ is
    determined for any formula $\phi$ and any $m < \omega$. 
  \end{proof}

  The sets $Q^\phi_m$ are Turing invariant, since the premouse $N_x$
  by definition only depends on the Turing degree of $x$.

  Let $Th^{M_{2n-2}(N_x|\delta_{N_x})}$ denote the theory of
  $M_{2n-2}(N_x|\delta_{N_x})$ with indiscernibles (computed in
  $V$). That means
  \begin{align*} Th^{M_{2n-2}(N_x|\delta_{N_x})} = \{ \phi \st
    M_{2n-2}(N_x|&\delta_{N_x}) \vDash \phi(E, \gamma_1, \dots,
    \gamma_m), \\ & \;\;\; m < \omega, \phi \text{ formula}
    \}, \end{align*} where as above $E$ denotes the extender sequence
  of $M_{2n-2}(N_x|\delta_{N_x})$ and $\gamma_1, \dots, \gamma_m$ are
  indiscernibles of $M_{2n-2}(N_x|\delta_{N_x})$. Then we have that
  the theory of $M_{2n-2}(N_x|\delta_{N_x})$ stabilizes on a cone of
  reals $x$ as in the following claim.

  \begin{claim}\label{cl:Hconst}
    There exists a real $x_0 \geq_T x$ such that for all reals
    $y \geq_T x_0$,
    \[ Th^{M_{2n-2}(N_{x_0}|\delta_{N_{x_0}})} =
    Th^{M_{2n-2}(N_y|\delta_{N_y})}. \]
  \end{claim}

  \begin{proof}
    By Claim \ref{cl:QmphiDet} the set $Q^\phi_m$ is determined and as
    argued above it is also Turing invariant for all formulas $\phi$
    and all $m < \omega$. That means the set $Q^\phi_m$ either
    contains a cone of reals or is completely disjoint from a cone of
    reals.

    For each formula $\phi$ and each natural number $m$ let
    $x^\phi_m \in \BS$ be such that either $y \in Q^\phi_m$ for all
    $y \geq_T x^\phi_m$ or else $y \notin Q^\phi_m$ for all
    $y \geq_T x^\phi_m$. Let
    \[ x_0 = \bigoplus \{ x^\phi_m \st \phi \text{ formula}, m <
    \omega \}. \]
    Then we have by construction for all $y \geq_T x_0$ that
    \[ Th^{M_{2n-2}(N_{x_0}|\delta_{N_{x_0}})} =
    Th^{M_{2n-2}(N_y|\delta_{N_y})}, \] as desired.
  \end{proof}

  Let $x_0 \in \BS$ be as in Claim \ref{cl:Hconst}. We want to show
  that the unique theory $T = Th^{M_{2n-2}(N_x|\delta_{N_x})}$ of
  $M_{2n-2}(N_x|\delta_{N_x})$ with indiscernibles as defined above
  for $x \geq_T x_0$ in fact gives a candidate for the theory of the
  premouse $M_{2n-1}^\#$ in $V$ to conclude that $M_{2n-1}^\#$ exists
  and is $\omega_1$-iterable in $V$. By coding a formula $\phi$ by its
  unique Gödel number $\ulcorner \phi \urcorner$ we can code the
  theory $T$ by a real $x_T$.

  Fix a real $z$ such that $z \geq_T x_0 \oplus x_T$. Moreover we can
  pick the real $z$ such that it in addition codes a $(2n-1)$-suitable
  premouse by Lemmas \ref{lem:prensuitable} and \ref{lem:nsuitnew}.

  Using the existence of the premouse $M_{2n-2}^\#(x)$ in $V$ for
  every real $x$, we can as described below close under the operation
  \[ a \mapsto M_{2n-2}^\#(a) \]
  to construct a transitive model $M_z$ from $z$. Moreover we will
  define an order of construction for elements of the model $M_z$
  along the way. 

  The fact that the model $M_z$ will be closed under
  $a \mapsto M_{2n-2}^\#(a)$ directly yields that $M_z$ is
  $\SIGMA^1_{2n}$-correct (see Claim \ref{cl:MzOddProp} below).

  In the construction of the model $M_z$ we aim to construct models by
  expanding the usual definition of Gödels constructible universe $L$
  by adding additional elements at the successor steps of the
  construction. Therefore recall the following definition which is
  used in the definition of Jensens $J$-hierarchy for $L$ (see §1 and
  §2 in \cite{Je72}).

  \begin{definition}\label{def:rud}
    Let $X$ be an arbitrary set. Then $\TC(X)$ denotes the transitive
    closure of $X$ and $rud(X)$ denotes the closure of
    $\TC(X) \cup \{\TC(X)\}$ under rudimentary functions (see for
    example Definition $1.1$ in \cite{SchZe10} for the definition of a
    rudimentary function).
  \end{definition}

  We construct a sequence of models
  $(W_\alpha \st \alpha \leq \omega_1^V)$ and in particular the model
  $M_z = W_{\omega_1^V}$ level-by-level in a construction of length
  $\omega_1^V$, starting from $z$ and taking unions at limit steps of
  the construction. So we let $W_0 = \{ z \}.$

  \textbf{Successor steps:} At a successor level $\alpha+1$ of the
  construction we close the previous model $W_{\alpha}$ under the
  operation $a \mapsto M_{2n-2}^\#(a)$ before closing under
  rudimentary functions. More precisely assume that we already
  constructed $W_\alpha$ and let $a \in W_\alpha$ be arbitrary. Then
  the $a$-premouse $M_{2n-2}^\#(a)$ exists in $V$ because $a$ is
  countable in $V$ and we inductively assume that Theorem
  \ref{cor:newWoodin1} holds for $2n-2$. Let $\M$ be a countable
  $a$-premouse in $V$ with the following properties.
  \begin{enumerate}[$(i)$]
  \item $\M$ is $(2n-1)$-small, but not $(2n-2)$-small,
  \item all proper initial segments of $\M$ are $(2n-2)$-small,
  \item $\M$ is $a$-sound and $\rho_\omega(\M) = a$, and
  \item $\M$ is $\Pi^1_{2n}$-iterable.
  \end{enumerate}
  These properties uniquely determine the $a$-premouse
  $M_{2n-2}^\#(a)$ in $V$. We let $W_{\alpha+1}$ be the model obtained
  by taking the closure under rudimentary functions of $W_\alpha$
  together with all such $a$-premice $\M$ as above for all
  $a \in W_\alpha$, i.e. 
  \begin{eqnarray*}
     W_{\alpha+1} = rud(W_\alpha \cup \{ \M
    \st \exists a \in W_\alpha \text{ such that } \M \text{ is a
      countable } \\ a\text{-premouse satisfying }(i)-(iv)\}).
  \end{eqnarray*}

  \textbf{Order of construction:} For an $a$-premouse $\M_a$ and a
  $b$-premouse $\M_b$ satifying properties $(i)-(iv)$ for
  $a, b \in W_\alpha$, we say that $\M_a$ is defined before $\M_b$ if
  $a$ is defined before $b$ in the order of construction for elements
  of $W_\alpha$, which exists inductively. For elements added by the
  closure under rudimentary functions we define the order of
  construction as in the definition of the order of construction for
  $L$.

  \textbf{Limit steps:} At a limit step of the construction we let
  \[ W_\lambda = \bigcup_{\alpha < \lambda} W_\alpha \]
  for all limit ordinals $\lambda < \omega_1^V$ and we finally let
  \[ M_z = W_{\omega_1^V} = \bigcup_{\alpha < \omega_1^V} W_\alpha. \]

  \textbf{Order of construction:} The order of construction at the
  limit steps is defined as in the definition of the order of
  construction for $L$ (see Lemma $5.26$ in \cite{Sch14}).

  Now we get that $M_z$ is a model of $\ZFC$ from the background
  hypothesis that there is no $\SIGMA^1_{2n+1}$-definable
  $\omega_1$-sequence of pairwise distinct reals as in the following
  claim.

    \setcounter{cl}{2}
  \begin{claim}\label{cl:MzZFCnew}
    $M_z \vDash \ZFC$.
  \end{claim}
  \begin{proof}
    Assume not. Then the power set axiom has to fail. So let $\gamma$
    be a countable ordinal such that
    \[ \Pot(\gamma) \cap M_z \notin M_z. \]
    This yields that the set $\Pot(\gamma) \cap M_z$ has size
    $\aleph_1$.

    Let $W_\gamma = M_z | \gamma$ be the $\gamma$-th level in the
    construction of $M_z$. Then we can fix a real $a$ in $V$ which
    codes the countable set $W_\gamma$. That means $a \in \BS$ codes a
    set $E \subset \omega \times \omega$ such that there is an
    isomorphism $\pi: (\omega,E) \rightarrow (W_\gamma, \in)$.

    If it exists, we let $A_\xi$ for $\gamma < \xi < \omega_1^V$ be
    the smallest subset of $\gamma$ in
    \[ M_z | (\xi+1) \, \setminus \, M_z | \xi \]
    according to the order of construction. Moreover we let $X$ be the
    set of all $\xi$ with $\gamma < \xi < \omega_1^V$ such that
    $A_\xi$ exists. Then $X$ is cofinal in $\omega_1^V$.

    Finally we let $a_\xi$ be a real coding the set $A_\xi$ relative
    to the code $a$ for $W_\gamma$. That means $a_\xi \in \BS$ codes
    the real $a$ together with some set $a^\prime \subset \omega$ such
    that $b \in a^\prime$ iff $\pi(b) \in A_\xi$, where $\pi$ is the
    isomorphism given by $a$ as above. For $\xi \in X$ we have that
    $A_\xi \in \Pot(\gamma) \cap M_z$ and thus
    $A_\xi \subseteq W_\gamma$, so the canonical code $a_\xi$ for
    $A_\xi$ relative to $a$ exists.

    Now consider the following $\omega_1^V$-sequence of reals
    \[ A = ( a_\xi \in \BS \st \xi \in X). \]

    We have that a real $y$ codes an element of the model
    $M_z | (\xi+1) \, \setminus \, M_z | \xi$ for some
    $\xi < \omega_1^V$ iff there is a sequence of countable models
    $(W_\beta \st \beta \leq \xi+1)$ and an element $Y \in W_{\xi+1}$
    such that
  \begin{enumerate}[$(1)$]
  \item $W_0 = \{ z \}$,
  \item $W_{\beta+1}$ is constructed from $W_\beta$ as described in
    the construction above for all $\beta \leq \xi$,
  \item $W_\lambda = \bigcup_{\beta < \lambda} W_\beta$ for all limit
    ordinals $\lambda \leq \xi$, 
  \item $y$ does not code an element of $W_\xi$, and
  \item $y$ codes $Y$.
  \end{enumerate}

  This states that the real $y$ codes an element of the model
  $M_z | (\xi+1) \, \setminus \, M_z | \xi$ as the construction of the
  levels of the model $M_z$ is defined in a unique way. Now we have to
  argue that this statement is definable enough for our purposes.

  For the successor levels of the construction we have that the
  properties $(i)-(iv)$ as in the construction are
  $\Pi^1_{2n}$-definable uniformly in any code for the countable
  premouse $\M$.

  So this shows that the statement ``$a_\xi$ codes the smallest subset
  of $\gamma$ in $M_z | (\xi+1) \, \setminus \, M_z | \xi$'' is
  $\Sigma^1_{2n+1}$-definable uniformly in $a$, $z$ and any code for
  $\xi$ and $\gamma$.

  Therefore the $\omega_1$-sequence $A$ as defined above is
  $\Sigma^1_{2n+1}$-definable uniformly in $a$ and $z$ in the sense of
  the remark after the statement of Theorem \ref{thm:newWoodin1}. Thus
  $A$ contradicts the assumption that every
  $\SIGMA^1_{2n+1}$-definable sequence of pairwise distinct reals is
  countable.
  \end{proof}
  
  Moreover we have the following claim.

  \begin{claim}
    \label{cl:MzOddProp}
    The resulting model $M_z$ has the following properties.
    \begin{enumerate}[(1)]
    \item $M_z \cap \OR = \omega_1^V$, $z \in M_z$,
    \item $M_z \prec_{\SIGMA^1_{2n}} V$, 
    \item $M_z$ is closed under the operation
      \[ a \mapsto M_{2n-2}^\#(a), \]
      and moreover $M_{2n-2}^\#(a)$ is $\omega_1$-iterable in $M_z$
      for all $a \in M_z$.
    \end{enumerate}
  \end{claim}

  \begin{proof} Property $(1)$ immediately follows from the construction.

    \medskip \textbf{Proof of $(3)$, first part:} Let $a \in M_z$ be
    arbitrary. Then we have that $a \in W_\alpha$ for some ordinal
    $\alpha < \omega_1^V$. Recall that by our assumptions the
    $\omega_1$-iterable $a$-premouse $M_{2n-2}^\#(a)$ exists in $V$ as
    $a$ is countable in $V$. Moreover it satisfies properties
    $(i)-(iv)$ from the construction, i.e.
  \begin{enumerate}[$(i)$]
  \item $M_{2n-2}^\#(a)$ is $(2n-1)$-small, but not $(2n-2)$-small,
  \item all proper initial segments of $M_{2n-2}^\#(a)$ are $(2n-2)$-small,
  \item $M_{2n-2}^\#(a)$ is $a$-sound and
    $\rho_\omega(M_{2n-2}^\#(a)) = a$, and
  \item $M_{2n-2}^\#(a)$ is $\Pi^1_{2n}$-iterable.
  \end{enumerate}
  Therefore the $a$-premouse $\M = M_{2n-2}^\#(a)$, which is
  $\omega_1$-iterable in $V$, has been added to the model $M_z$ at
  some successor level of the construction and thus $M_z$ is closed
  under the operation $a \mapsto M_{2n-2}^\#(a)$.

  \medskip \textbf{Proof of $(2)$:} The fact that the model $M_z$ is
  closed under the operation $a \mapsto M_{2n-2}^\#(a)$ immediately
  implies property $(2)$, because the $a$-premouse $M_{2n-2}^\#(a)$ is
  $\SIGMA^1_{2n}$-correct in $V$ by Lemma \ref{lem:corr}.

  \medskip \textbf{Proof of $(3)$, second part:} We now want to show
  that the $a$-premouse $\M$ as above, which has been added to the
  model $M_z$ at some successor level of the construction, is
  $\omega_1$-iterable inside $M_z$ via the iteration strategy $\Sigma$
  which is guided by $\Q$-structures (see Definition
  \ref{def:Qstritstr}).

  Let $\T$ be an iteration tree of length $\lambda$ on the
  $a$-premouse $\M$ for some limit ordinal $\lambda < \omega_1$ in
  $M_z$ such that $\T$ is according to the iteration strategy
  $\Sigma$. Then $\T$ is guided by $\Q$-structures which are
  $(2n-2)$-small above $\delta(\T \upharpoonright \gamma)$,
  $\omega_1$-iterable above $\delta(\T \upharpoonright \gamma)$ and
  thus also $\Pi^1_{2n-1}$-iterable above
  $\delta(\T \upharpoonright \gamma)$ in $M_z$ for all limit ordinals
  $\gamma \leq \lambda$.

  By $(2)$ we have that $M_z$ is $\SIGMA^1_{2n}$-correct in $V$ for
  real parameters in $M_z$. Therefore it follows that the
  $\Q$-structures $\Q(\T \upharpoonright \gamma)$ which are
  $(2n-2)$-small above $\delta(\T \upharpoonright \gamma)$ and guiding
  the iteration tree $\T$ in $M_z$ are $\Pi^1_{2n-1}$-iterable above
  $\delta(\T \upharpoonright \gamma)$ in $V$ for all limit ordinals
  $\gamma \leq \lambda$. Since $M_{2n-2}^\#(x)$ exists in $V$ for all
  $x \in \BS$ by our assumptions and in particular $M_{2n-2}^\#(a)$
  exists in $V$, we have by Lemma \ref{lem:nofakeeven} that these
  $\Q$-structures also witness that $\T$ is according to the
  $\Q$-structure iteration strategy $\Sigma$ in $V$. Therefore there
  exists a unique cofinal well-founded branch $b$ through $\T$ in $V$
  such that we have $\Q(b, \T) = \Q(\T)$. By an absoluteness argument
  as given several times before (see for example the proof of Lemma
  \ref{lem:iterability}), it follows that the unique cofinal
  well-founded branch $b$ through $\T$ in $V$ for which
  $\Q(b, \T) = \Q(\T)$ holds, is also contained in $M_z$ since we have
  that $\T, \Q(\T) \in M_z$.

  Therefore $\M=M_{2n-2}^\#(a)$ exists and is $\omega_1$-iterable in
  $M_z$ via the iteration strategy $\Sigma$.
  \end{proof}

  Now we aim to show that a $K^c$-construction in $M_z$ reaches
  $M_{2n-1}^\#$, meaning that the premouse $(K^c)^{M_z}$ is not
  $(2n-1)$-small. Here and in what follows we consider a
  $K^c$-construction in the sense of \cite{MSch04} as this
  construction does not assume any large cardinals in the background
  model.

  \begin{claim}\label{cl:Kc}
    $(K^c)^{M_z}$ is not $(2n-1)$-small.
  \end{claim}
  \begin{proof}
    We work inside the model $M_z$ and distinguish numerous different
    cases. Moreover we assume toward a contradiction that
    $(K^c)^{M_z}$ is $(2n-1)$-small.

    First we show the following subclaim.

\begin{subclaim}\label{subcl:Kcit}
  Let $\delta$ be a cutpoint in $(K^c)^{M_z}$. If $(K^c)^{M_z}$ does
  not have a Woodin cardinal above $\delta$, then $(K^c)^{M_z}$ is
  fully iterable above $\delta$ in $M_z$.
\end{subclaim}

We allow $\delta = \omega$, i.e. the case that $(K^c)^{M_z}$ has no
Woodin cardinals, in the statement of Subclaim \ref{subcl:Kcit}.

\begin{proof}[Proof of Subclaim \ref{subcl:Kcit}]
  $(K^c)^{M_z}$ is iterated via the $\Q$-structure iteration strategy
  $\Sigma$ (see Definition \ref{def:Qstritstr}). That means for an
  iteration tree $\U$ of limit length on $(K^c)^{M_z}$ we let
  \[ \Sigma(\U) = b \text{ iff } \Q(b,\U) = \Q(\U) \text{ and } \Q(\U)
  \unlhd M_{2n-2}^\#(\M(\U)), \]
  since $(K^c)^{M_z}$ is assumed to be $(2n-1)$-small.

  It is enough to show that $(K^c)^{M_z}$ is $\omega_1$-iterable above
  $\delta$ inside $M_z$, because then an absoluteness argument as the
  one we gave in the proof of Lemma \ref{lem:iterability} yields that
  $(K^c)^{M_z}$ is fully iterable above $\delta$ inside $M_z$ as the
  iteration strategy for $(K^c)^{M_z}$ is guided by $\Q$-structures.

  Assume toward a contradiction that $(K^c)^{M_z}$ is not
  $\omega_1$-iterable above $\delta$ in $M_z$. Then there exists an
  iteration tree $\T$ on $(K^c)^{M_z}$ above $\delta$ of limit length
  $<\omega_1$ inside $M_z$ such that there exists no $\Q$-structure
  $\Q(\T)$ for $\T$ with $\Q(\T) \unlhd M_{2n-2}^\#(\M(\T))$ and hence
  \[ M_{2n-2}^\#(\M(\T)) \vDash \text{``} \delta(\T) \text{ is
    Woodin''.} \]
  The premouse $M_{2n-2}^\#(\M(\T))$, constructed in the sense of
  Definition \ref{def:M(N)}, exists in $M_z$ and is not $(2n-2)$-small
  above $\delta(\T)$, because otherwise it would already provide a
  $\Q$-structure for $\T$.

  Let $\bar{M}$ be the Mostowski collapse of a countable substructure
  of $M_z$ containing the iteration tree $\T$. That means we choose a
  large enough natural number $m$ and let $\bar{M}, X$ and $\sigma$ be
  such that
  \[ \bar{M} \overset{\sigma}{\cong} X \prec_{\Sigma_m} M_z, \]
  where \[ \sigma: \bar{M} \rightarrow M_z \]
  denotes the uncollapse map such that we have a model $\bar{K}$ in
  $\bar{M}$ with
  $\sigma(\bar{K}| \gamma) = (K^c)^{M_z}| \sigma(\gamma)$ for every
  ordinal $\gamma < \bar{M} \cap \Ord$, and we have an iteration tree
  $\bar{\T}$ on $\bar{K}$ above an ordinal $\bar{\delta}$ in $\bar{M}$
  with $\sigma(\bar{\T}) = \T$ and $\sigma(\bar{\delta}) =
  \delta$. Moreover we have that
  \[ M_{2n-2}^\#(\M(\bar{\T})) \vDash \text{``} \delta(\bar{\T})
  \text{ is Woodin''.} \]

  By the iterability proof of Chapter $9$ in \cite{St96} (adapted as
  in \cite{MSch04}) applied inside $M_z$, there exists a cofinal
  well-founded branch $b$ through the iteration tree $\bar{\T}$ on
  $\bar{K}$ in $M_z$ such that we have a final model
  $\M_b^{\bar{\T}}$.

  Consider the coiteration of $\M_b^{\bar{\T}}$ with
  $M_{2n-2}^\#(\M(\bar{\T}))$ and note that it takes place above
  $\delta(\bar{\T})$. Since $M_{2n-2}^\#(\M(\bar{\T}))$ is
  $\omega_1$-iterable above $\delta(\bar{\T})$ inside $M_z$ and
  $\M_b^{\bar{\T}}$ is iterable in $M_z$ by the iterability proof of
  Chapter $9$ in \cite{St96} (adapted as in \cite{MSch04}) applied
  inside the model $M_z$, the coiteration is successful using Lemma
  \ref{lem:iterability} inside $M_z$.

  If there is no drop along the branch $b$, then $\M_b^{\bar{\T}}$
  cannot lose the coiteration, because otherwise there exists a
  non-dropping iterate $R^*$ of $\M_b^{\bar{\T}}$ and an iterate
  $\M^*$ of $M_{2n-2}^\#(\M(\bar{\T}))$ such that $R^* \unlhd \M^*$.
  But we have that there is no Woodin cardinal in $\M_b^{\bar{\T}}$
  above $\bar{\delta}$ by elementarity and at the same time we have
  that $\bar{\delta} < \delta(\bar{\T})$ and
\[ M_{2n-2}^\#(\M(\bar{\T})) \vDash \text{``} \delta(\bar{\T}) \text{
  is Woodin''.} \]
If there is a drop along $b$, then $\M_b^{\bar{\T}}$ also has to win
the coiteration, because we have
$\rho_\omega(\M_b^{\bar{\T}}) < \delta(\bar{\T})$ and
$\rho_\omega(M_{2n-2}^\#(\M(\bar{\T}))) = \delta(\bar{\T})$. 

That means in both cases there is an iterate $R^*$ of
$\M_b^{\bar{\T}}$ and a non-dropping iterate $\M^*$ of
$M_{2n-2}^\#(\M(\bar{\T}))$ such that $\M^* \unlhd R^*$. We have that
$\M^*$ is not $(2n-1)$-small, because $M_{2n-2}^\#(\M(\bar{\T}))$ is
not $(2n-1)$-small as argued above and the iteration from
$M_{2n-2}^\#(\M(\bar{\T}))$ to $\M^*$ is non-dropping. Therefore it
follows that $R^*$ is not $(2n-1)$-small and thus $\M_b^{\bar{\T}}$ is
not $(2n-1)$-small. By the iterability proof of Chapter $9$ in
\cite{St96} (adapted as in \cite{MSch04}) applied inside $M_z$ we can
re-embed the model $\M_b^{\bar{\T}}$ into a model of the
$(K^c)^{M_z}$-construction. This yields that $(K^c)^{M_z}$ is not
$(2n-1)$-small, contradicting our assumption that it is
$(2n-1)$-small.
\end{proof}

Now we distinguish the following cases.

\begin{minipage}{\textwidth}
  \bigskip
  \textbf{Case 1.} Assume that
  \[ (K^c)^{M_z} \vDash \text{ ``there is a Woodin cardinal''.} \]
  \medskip
\end{minipage}
Then we can assume that there is a largest Woodin cardinal in
$(K^c)^{M_z}$, because otherwise $(K^c)^{M_z}$ has infinitely many
Woodin cardinals and is therefore not $(2n-1)$-small. So let $\delta$
denote the largest Woodin cardinal in $(K^c)^{M_z}$.
 
Consider the premouse $\M = M_{2n-2}^\#((K^c)^{M_z}| \delta)$ in the
sense of Definition \ref{def:M(N)} and note that $\M$ exists inside
$M_z$ by Claim \ref{cl:MzOddProp}. Now try to coiterate $(K^c)^{M_z}$
with $\M$ inside $M_z$.

Since the comparison takes place above $\delta$ and the premouse
\[ \M = M_{2n-2}^\#((K^c)^{M_z}| \delta) \]
is $\omega_1$-iterable above $\delta$, the coiteration is successful,
using Lemma \ref{lem:iterability} $(2)$ and Subclaim \ref{subcl:Kcit}
since the iteration strategies for the premice $(K^c)^{M_z}$ and
$M_{2n-2}^\#((K^c)^{M_z}|\delta)$ above $\delta$ are guided by
$\Q$-structures which are $(2n-2)$-small above $\delta(\T)$ for any
iteration tree $\T$ of limit length (see proof of Claim
\ref{cl:MzOddProp}).

So there is an iterate $R$ of $(K^c)^{M_z}$ and an iterate $\M^*$ of
$\M$ such that the coiteration terminates with $R \unlhd \M^*$ or
$\M^* \unlhd R$. By universality of the model $(K^c)^{M_z}$ inside
$M_z$ (see Section $3$ in \cite{MSch04}) we have that $\M^* \unlhd R$
and that there is no drop on the $\M$-side of the coiteration. This
implies that the premouse $\M = M_{2n-2}^\#((K^c)^{M_z}| \delta)$ in
the sense of Definition \ref{def:M(N)} is not $(2n-2)$-small above
$\delta$, as otherwise we have that $\M$ is not fully sound and since
$\M^* \unlhd R$ this yields a contradiction because of soundness.

Therefore $\M$ is not $(2n-1)$-small and thus $\M^*$ is also not
$(2n-1)$-small. That means $R$ and thereby $(K^c)^{M_z}$ is not
$(2n-1)$-small, which is the desired contradiction.

This finishes the proof of Claim \ref{cl:Kc} in the
case that there is a Woodin cardinal in $(K^c)^{M_z}$.

\begin{minipage}{\textwidth}
  \bigskip
  \textbf{Case 2.} Assume that
  \[ (K^c)^{M_z} \nvDash \text{ ``there is a Woodin cardinal''.} \]
  \medskip
\end{minipage}

Recall that by our choice of the real $z$, there exists a premouse $N$
(coded by the real $z$) in $M_z$ which is $(2n-1)$-suitable in $V$ and
try to coiterate $(K^c)^{M_z}$ with $N$ inside $M_z$. We again iterate
$(K^c)^{M_z}$ via the $\Q$-structure iteration strategy $\Sigma$ as in
Subclaim \ref{subcl:Kcit}. So using Subclaim \ref{subcl:Kcit} the
coiteration cannot fail on the $(K^c)^{M_z}$-side.

We have that the premouse $N$ is $(2n-1)$-suitable in $V$. Since the
statement ``$N$ is pre-$(2n-1)$-suitable'' is $\Pi^1_{2n}$-definable
uniformly in any code for $N$, it follows that $N$ is
pre-$(2n-1)$-suitable inside $M_z$, because $M_z$ is closed under the
operation $a \mapsto M_{2n-2}^\#(a)$ and hence $\SIGMA^1_{2n}$-correct
in $V$. Therefore Lemma \ref{lem:nsuitnew} implies that $N$ is
$(2n-1)$-suitable inside $M_z$. In particular $N$ has an iteration
strategy which is fullness preserving for non-dropping iteration
trees.

So we have that the coiteration of $(K^c)^{M_z}$ with $N$ is
successful. Let $\T$ and $\U$ be the resulting trees on $(K^c)^{M_z}$
and $N$ of length $\lambda+1$ for some ordinal $\lambda$. It follows
that $(K^c)^{M_z}$ wins the comparison by universality of the premouse
$(K^c)^{M_z}$ inside $M_z$ (see Section $3$ in
\cite{MSch04}). Therefore we have that there exists an iterate
$R = \M_\lambda^\T$ of $(K^c)^{M_z}$ and an iterate
$N^* = \M_\lambda^\U$ of $N$ such that $N^* \unlhd R$.  Moreover there
is no drop on the main branch on the $N$-side of the coiteration.

Since $N$ is $(2n-1)$-suitable we have that
$M_{2n-2}^\#(N^*|\delta_{N^*})$ is a premouse with $2n-1$ Woodin
cardinals, which is $\omega_1$-iterable above $\delta_{N^*}$ and not
$(2n-2)$-small above $\delta_{N^*}$. So we can consider the
coiteration of $R$ with $M_{2n-2}^\#(N^*|\delta_{N^*})$. This
coiteration is successful using Lemma \ref{lem:iterability} $(2)$
because we have
\[ N^*|\delta_{N^*} = R |\delta_{N^*} \]
and $R$ and $M_{2n-2}^\#(N^*|\delta_{N^*})$ are both iterable above
$\delta_{N^*}$, using Subclaim \ref{subcl:Kcit} for the iterate $R$ of
$(K^c)^{M_z}$. If there is no drop on the main branch in $\T$, then
$R = \M_\lambda^\T$ wins the comparison by universality of
$(K^c)^{M_z}$ inside $M_z$ (again see Section $3$ in
\cite{MSch04}). If there is a drop on the main branch in $\T$, then
$R$ also wins the comparison, because in this case we have that
\[ \rho_\omega(R) < \delta_{N^*} \text{ and }
\rho_\omega(M_{2n-2}^\#(N^*|\delta_{N^*})) = \delta_{N^*}. \]
Therefore we have that there exists an iterate $R^*$ of $R$ and an
iterate $M^*$ of $M_{2n-2}^\#(N^*|\delta_{N^*})$ such that we have
$M^* \unlhd R^*$ in both cases and the iteration from
$M_{2n-2}^\#(N^*|\delta_{N^*})$ to $M^*$ is non-dropping. Since
$M_{2n-2}^\#(N^*|\delta_{N^*})$ is not $(2n-1)$-small, we have that
$M^*$ and therefore $R^*$ is not $(2n-1)$-small. Thus $R$ is not
$(2n-1)$-small. But $R$ is an iterate of $(K^c)^{M_z}$ and therefore
this implies that $(K^c)^{M_z}$ is not $(2n-1)$-small, contradicting
our assumption.

This finishes the proof of Claim \ref{cl:Kc}.
\end{proof}

Claim \ref{cl:Kc} now implies that $M_{2n-1}^\#$ exists in $M_z$ as
the minimal $\omega_1$-iterable premouse which is not $(2n-1)$-small.

Work inside the model $M_z$ and let $x \in M_z$ be a real which codes
$x_0$, the theory $T$ and the premouse $(M_{2n-1}^\#)^{M_z}$. Let
\[ N^* = (M_{2n-1}^\# | (\delta_0^+)^{M_{2n-1}^\#})^{M_z} \]
denote the \emph{suitable initial segment} of $(M_{2n-1}^\#)^{M_z}$,
where $\delta_0$ denotes the least Woodin cardinal in
$(M_{2n-1}^\#)^{M_z}$. Then we have in particular that $N^*$ is a
$(2n-1)$-suitable premouse (in $M_z$).

  Recall that by Lemma \ref{lem:nsuitnew} we have that $N^*$ is a
  $(2n-1)$-suitable premouse iff it is pre-$(2n-1)$-suitable, that
  means iff it satisfies properties $(1)-(4)$ listed in the proof of
  Claim \ref{cl:QmphiDet}. Therefore the statement ``$N^*$ is a
  $(2n-1)$-suitable premouse'' is $\Pi^1_{2n}$-definable uniformly in
  any code for $N^*$. Hence by $\SIGMA^1_{2n}$-correctness of $M_z$ in
  $V$ it follows that $N^*$ is also a $(2n-1)$-suitable premouse in
  $V$.

  Recall that $N_x$ denotes the common iterate of all
  $(2n-1)$-suitable premice $N$ which are coded by a real $y_N$
  recursive in $x$. Since $M_z$ is $\SIGMA^1_{2n}$-correct in $V$ it
  follows that the premouse $N_x$ is the same computed in $M_z$ or in
  $V$. This yields by correctness of $M_z$ in $V$ again that the
  premouse $M_{2n-2}^\#(N_x|\delta_{N_x})$ is the same computed in
  $M_z$ or in $V$ by the following argument. The premouse
  $(M_{2n-2}^\#(N_x|\delta_{N_x}))^{M_z}$ is $\Pi^1_{2n}$-iterable
  above $\delta_{N_x}$ in $M_z$ and by $\SIGMA^1_{2n}$-correctness of
  $M_z$ also in $V$. Therefore we can successfully coiterate the
  premice $(M_{2n-2}^\#(N_x|\delta_{N_x}))^{M_z}$ and
  $(M_{2n-2}^\#(N_x|\delta_{N_x}))^V$ inside $V$ by Lemma
  \ref{lem:steel2.2} since the latter premouse is $\omega_1$-iterable
  in $V$ above $\delta_{N_x}$ and the comparison takes place above
  $\delta_{N_x}$. It follows that in fact
  \[ (M_{2n-2}^\#(N_x|\delta_{N_x}))^{M_z} =
  (M_{2n-2}^\#(N_x|\delta_{N_x}))^V. \]

  As mentioned before, we have that
  $N^* = (M_{2n-1}^\# | (\delta_0^+)^{M_{2n-1}^\#})^{M_z}$ is a
  $(2n-1)$-suitable premouse in $M_z$ and in $V$ and since $x$ codes
  the premouse $(M_{2n-1}^\#)^{M_z}$, the premouse $N^*$ is coded by a
  real recursive in $x$.

  Consider the comparison of $N^*$ with $N_x = (N_x)^{M_z}$ inside an
  inner model of $M_z$ of height $\omega_1^{M_z}$ which is closed
  under the operation $a \mapsto M_{2n-2}^\#(a)$. The premouse
  $(M_{2n-1}^\#)^{M_z}$ is $\omega_1$-iterable in $M_z$ and therefore
  it follows that $N^*$ is $\omega_1$-iterable in $M_z$. Thus
  arguments as in the proof of Lemma \ref{lem:iterability} yield that
  $N_x$ is in fact a non-dropping iterate of $N^*$, because $N^*$ is
  one of the models giving rise to $N_x$.

  The same argument shows that $N_x$ does not move in the comparison
  with $(M_{2n-1}^\#)^{M_z}$. So in fact there is a non-dropping
  iterate $M$ of $(M_{2n-1}^\#)^{M_z}$ below
  $(\delta_0^+)^{(M_{2n-1}^\#)^{M_z}}$ such that $N_x \unlhd M$. Since the
  iteration from $(M_{2n-1}^\#)^{M_z}$ to $M$ is fullness preserving
  in the sense of Definition \ref{def:nsuitnew} and it takes place
  below $(\delta_0^+)^{(M_{2n-1}^\#)^{M_z}}$ in $M_z$, we have that in fact
  \[ M = M_{2n-2}^\#(N_x|\delta_{N_x}), \]
  because $(M_{2n-1}^\#)^{M_z} = M_{2n-2}^\#(N^*|\delta_0)$. Therefore
  we have that $M_{2n-2}^\#(N_x | \delta_{N_x})$ is a non-dropping
  iterate of $M_{2n-2}^\#(N^*|\delta_0)$ and hence
  \[ Th^{M_{2n-2}(N_x|\delta_{N_x})} = Th^{M_{2n-2}(N^*|\delta_0)}. \]

  Recall that by Claim \ref{cl:Hconst} we picked the real $x_0$ such
  that $Th^{M_{2n-2}(N_x|\delta_{N_x})}$ and thus
  $Th^{M_{2n-2}(N^*|\delta_0)}$ is constant for all $x \geq_T x_0$.
  This now implies that the theory of $(M_{2n-1}^\#)^{M_z}$ is
  constant for all $z \geq_T x_0 \oplus x_T$, where $x_T$ is as above
  a real coding the theory $T$.

  Thus if we now work in $V$ and let $N = (M_{2n-1}^\#)^{M_z}$ for
  $z \geq_T x_0 \oplus x_T$, then we have
  \[ (M_{2n-1}^\#)^{M_z} = N = (M_{2n-1}^\#)^{M_y} \]
  for all $y \geq_T x_0 \oplus x_T$. We aim to show that
  \[ (M_{2n-1}^\#)^V = N, \] so in particular that $(M_{2n-1}^\#)^V$
  exists. 

  For this reason we inductively show that the premouse $N$ is
  $\omega_1$-iterable in $V$ via the $\Q$-structure iteration strategy
  $\Sigma$ (see Definition \ref{def:Qstritstr}). So assume that $\T$
  is an iteration tree via $\Sigma$ of limit length $< \omega_1$ on
  $N$ (in $V$). So we have that the branch $b$ through the iteration
  tree $\T \upharpoonright \lambda$ is given by $\Q$-structures, i.e.
  $\Q(b, \T \upharpoonright \lambda) = \Q(\T \upharpoonright
  \lambda)$, for every limit ordinal $\lambda < \lh(\T)$.

  Pick $z \in \BS$ with $z \geq_T x_0 \oplus x_T$ such that
  $\T \in M_z$ and $\lh(\T) < \omega_1^{M_z}$.  Since $\T$ is an
  iteration tree on $N = (M_{2n-1}^\#)^{M_z}$ according to the
  iteration strategy $\Sigma$ in $V$, we have that for all limit
  ordinals $\lambda < \lh(\T)$ the $\Q$-structure
  $\Q(\T \upharpoonright \lambda)$ exists in $V$ and is $(2n-1)$-small
  above $\delta(\T \upharpoonright \lambda)$. In fact
  $\Q(\T \upharpoonright \lambda)$ is not more complicated than the
  least active premouse which is not $(2n-2)$-small above
  $\delta(\T \upharpoonright \lambda)$. So in this case we have that
  $\Pi^1_{2n}$-iterability for these $\Q$-structures is enough to
  determine a unique cofinal well-founded branch $b$ through $\T$.
  Since $M_z$ is $\SIGMA^1_{2n}$-correct in $V$ it follows that $\T$
  is also according to the $\Q$-structure iteration strategy $\Sigma$
  inside $M_z$. Moreover recall that
  \[ (M_{2n-1}^\#)^{M_z} = N. \]
  Therefore there exists a cofinal well-founded branch $b$ through
  $\T$ in $M_z$. As above this branch is determined by $\Q$-structures
  $\Q(\T \upharpoonright \lambda)$ which are $\omega_1$-iterable above
  $\delta(\T \upharpoonright \lambda)$ and therefore
  $\Pi^1_{2n}$-iterable above $\delta(\T \upharpoonright \lambda)$ in
  $M_z$ for all limit ordinals $\lambda \leq \lh(\T)$. That means in
  particular that we have $\Q(b, \T) = \Q(\T)$. Moreover $\Q(\T)$ is
  also $\Pi^1_{2n}$-iterable in $V$ and therefore it follows that $b$
  is the unique cofinal well-founded branch determined by these
  $\Q$-structures in $V$ as well. So $N$ is $\omega_1$-iterable in $V$
  via the $\Q$-structure iteration strategy $\Sigma$.

  Thus we now finally have that
  \[ V \vDash \text{``} M_{2n-1}^\# \text{ exists and is } \omega_1
    \text{-iterable''}. \] This finishes the proof of Theorem
  \ref{thm:newWoodinodd}, i.e. Theorem \ref{thm:newWoodin1} for odd
  $n < \omega$.
\end{proof}


\subsection{$M_{2n}^\#(x)$ from Boldface $\Pi^1_{2n+1}$ Determinacy} 
\label{sec:even}

In this section we will finish the proof of Theorem
\ref{cor:bfDetSharps} by proving Theorem \ref{thm:newWoodineven} which
will yield Theorem \ref{thm:newWoodin1} and finally Theorem
\ref{cor:bfDetSharps} for even levels $n$ in the projective hierarchy
(using Lemma \ref{lem:bfDetSeq}). Therefore together with the previous
section we will have Theorem \ref{cor:bfDetSharps} for arbitrary
levels $n$.

\begin{thm}\label{thm:newWoodineven}
  Let $n \geq 1$ and assume that there is no
  $\SIGMA^1_{2n+2}$-definable $\omega_1$-sequence of pairwise distinct
  reals. Moreover assume that $\PI^1_{2n}$ determinacy and
  $\Pi^1_{2n+1}$ determinacy hold. Then $M_{2n}^\#$ exists and is
  $\omega_1$-iterable.
\end{thm}

Recall that Lemma \ref{lem:bfDetSeq} gives that $\PI^1_{2n+1}$
determinacy suffices to prove that every $\SIGMA^1_{2n+2}$-definable
sequence of pairwise distinct reals is countable.

In order to prove Theorem \ref{thm:newWoodineven} we are considering
slightly different premice than before.  The main advantage of these
models is that they only contain partial extenders on their extender
sequence and therefore behave nicer in some arguments to follow. The
premice we want to consider are defined as follows.

\begin{definition}
\label{def:Lp}
Let $A$ be an arbitrary countable transitive swo'd
\footnote{i.e. self-wellordered as for example defined in Definition
  $3.1$ in \cite{SchT}.} set. With $\Lpn(A)$
\index{lowerpartmode@$\Lpn(A)$} we denote the unique model of height
$\omega_1^V$ called \emph{lower part model above $A$} \index{lower
  part model} which is given by the following construction of length
$\omega_1^V$. We start with $N_0 = A$ and construct countable
$A$-premice along the way. Assume that we already constructed
$N_\alpha$. Then we let $N_{\alpha+1}$ be the model theoretic union of
all countable $A$-premice $M \unrhd N_\alpha$ such that
  \begin{enumerate}[$(1)$]
    \item $M$ is $n$-small above $N_\alpha \cap \Ord$,
    \item $\rho_\omega(M) \leq N_\alpha \cap \Ord$,
    \item $M$ is sound above $N_\alpha \cap \Ord$,
    \item $N_\alpha \cap \Ord$ is a cutpoint of $M$, and
    \item $M$ is $\omega_1$-iterable above $N_\alpha \cap \Ord$.
  \end{enumerate}
  For the limit step let $\lambda$ be a limit ordinal and assume that
  we already defined $N_\alpha$ for all $\alpha < \lambda$. Then we
  let $N_\lambda$ be the model theoretic union of all $A$-premice
  $N_\alpha$ for $\alpha < \lambda$.

  We finally let $\Lpn(A) = N_{\omega_1^V}$.
\end{definition}

A special case of this is $\Lpn(x)$ for $x \in \BS$.

\begin{remark}
  We have that every lower part model as defined above does not
  contain total extenders.
\end{remark}

We will state the following lemmas for lower part models constructed
above reals $x$ instead of swo'd sets $A$ as this will be our main
application. But they also hold (with the same proofs) if we replace
$x$ with a countable transitive swo'd set $A$ as in Definition
\ref{def:Lp}.

\begin{lemma}\label{lem:LpWD}
  Let $n \geq 1$ and assume that $\PI^1_{2n}$ determinacy
  holds. Moreover let $x \in \BS$ be arbitrary. Then the lower part
  model above $x$, $\Lpoddn(x)$, is well-defined, that means
  $\Lpoddn(x)$ is an $x$-premouse.
\end{lemma}

\begin{proof}
  Recall that we assume inductively that Theorem
  \ref{thm:newWoodinodd} holds. That means $\PI^1_{2n}$ determinacy
  implies that $M_{2n-1}^\#(x)$ exists for all $x \in \BS$. Therefore
  we have by Lemma \ref{lem:iterability} that whenever $M$ and
  $M^\prime$ are two $x$-premice extending some $x$-premouse
  $N_\alpha$ (as in Definition \ref{def:Lp}) and satisfying properties
  $(1)$-$(5)$ in Definition \ref{def:Lp} for some $x \in \BS$, then we
  have that in fact $M \unlhd M^\prime$ or $M^\prime \unlhd
  M$. Therefore $\Lpoddn(x)$ is a well-defined $x$-premouse.
\end{proof}

\begin{lemma}\label{lem:LpM}
  Let $n \geq 1$ and assume that $\PI^1_{2n}$ determinacy
  holds. Moreover let $x \in \BS$ be arbitrary. Let $M$ denote the
  $\omega_1^V$-th iterate of $M_{2n-1}^\#(x)$ by its least measure and
  its images. Then \[ M | \omega_1^V = \Lpoddn(x). \]
\end{lemma}

\begin{proof}
  Let $x \in \BS$ and let $(N_\alpha \st \alpha \leq \omega_1^V)$ be
  the sequence of models from the definition of $\Lpoddn(x)$ (see
  Definition \ref{def:Lp}). We aim to show inductively that
  \[ N_\alpha \unlhd M \]
  for all $\alpha < \omega_1^V$, where $M$ denotes the $\omega_1^V$-th
  iterate of $M_{2n-1}^\#(x)$ by its least measure and its images as
  in the statement of Lemma \ref{lem:LpM}. Fix an
  $\alpha < \omega_1^V$ and assume inductively that
  \[ N_\beta \unlhd M \]
  for all $\beta \leq \alpha$. Let $z$ be a real which codes the
  countable premice $N_{\alpha+1}$ and $M_{2n-1}^\#(x)$.

  Since $M_{2n-1}^\#(x)$ is $\omega_1$-iterable in $V$ and has no
  definable Woodin cardinals, we have by Lemma \ref{lem:iterability}
  $(2)$ that it is $(\omega_1 + 1)$-iterable inside the model
  $M_{2n-1}(z)$. Therefore we have in particular that $M$ is
  $(\omega_1+1)$-iterable inside $M_{2n-1}(z)$.

  Recall that $N_{\alpha+1}$ is by definition the model theoretic
  union of all countable $x$-premice $N \unrhd N_\alpha$ such that
  \begin{enumerate}[$(1)$]
    \item $N$ is $(2n-1)$-small above $N_\alpha \cap \Ord$,
    \item $\rho_\omega(N) \leq N_\alpha \cap \Ord$,
    \item $N$ is sound above $N_\alpha \cap \Ord$,
    \item $N_\alpha \cap \Ord$ is a cutpoint of $N$, and
    \item $N$ is $\omega_1$-iterable above $N_\alpha \cap \Ord$.
    \end{enumerate}
    In particular Lemma \ref{lem:iterability} $(2)$ implies that all
    these $x$-premice $N \unrhd N_\alpha$ satisfying properties
    $(1)-(5)$ are $(\omega_1+1)$-iterable above $N_\alpha \cap \Ord$
    inside the model $M_{2n-1}(z)$ since they have no definable Woodin
    cardinals above $N_\alpha \cap \Ord$. In particular $N_{\alpha+1}$
    is well-defined as in Lemma \ref{lem:LpWD} and it follows that
    $N_{\alpha+1}$ is $(\omega_1+1)$-iterable above
    $N_\alpha \cap \Ord$ inside $M_{2n-1}(z)$.

    Hence we can consider the comparison of the $x$-premice $M$ and
    $N_{\alpha+1}$ inside the model $M_{2n-1}(z)$. This comparison is
    successful because by our inductive hypothesis it takes place
    above $N_\alpha \cap \Ord$. Now we distinguish the following
    cases.

    \begin{minipage}{\textwidth}
      \bigskip \textbf{Case 1.} Both sides of the comparison move.
      \bigskip
    \end{minipage}

    In this case both sides of the comparison have to drop since we
    have that $\rho_\omega(N) \leq N_\alpha \cap \Ord$ for all
    $x$-premice $N$ occurring in the definition of $N_{\alpha+1}$,
    $N_{\alpha+1} \cap \Ord < \omega_1^V$ and $M$ only has partial
    extenders on its sequence below $\omega_1^V$. This is a
    contradiction by the proof of the Comparison Lemma (see Theorem
    $3.11$ in \cite{St10}), so only one side of the coiteration can
    move.

    \begin{minipage}{\textwidth}
      \bigskip \textbf{Case 2.} Only the $M$-side of the comparison
      moves.  \bigskip
    \end{minipage}

    As above we have in this case that the $M$-side drops. So there is
    an iterate $M^*$ of $M$ such that $N_{\alpha+1}$ is an initial
    segment of $M^*$. Let $E_\beta$ for some ordinal $\beta$ be the
    first extender used on the $M$-side in the coiteration of $M$ with
    $N_{\alpha+1}$.  In particular $E_\beta$ is an extender indexed on
    the $M$-sequence above $N_\alpha \cap \Ord$ and below
    $\omega_1^V$. Since $E_\beta$ has to be a partial extender, there
    exists a $(2n-1)$-small sound countable $x$-premouse $N \lhd M$
    such that $E_\beta$ is a total extender on the $N$-sequence and
    $\rho_m(N) \leq crit(E_\beta)$ for some natural number
    $m$. Moreover we have that $N$ is $\omega_1$-iterable by the
    iterability of $M$. Furthermore we have that $crit(E_\beta)$ is a
    cardinal in $N_{\alpha+1}$, because the extender $E_\beta$ is used
    in the coiteration and we have that
    \[ N_{\alpha+1} \unlhd M^*. \]
    Therefore the premouse $N$ is contained in $N_{\alpha+1}$ by the
    definition of a lower part model. But this contradicts the
    assumption that $E_\beta$ was used in the coiteration, because
    then we have that there is no disagreement between $M$ and
    $N_{\alpha+1}$ at $E_\beta$.

    \begin{minipage}{\textwidth}
      \bigskip \textbf{Case 3.} Only the $N_{\alpha+1}$-side of the
      comparison moves.  \bigskip
    \end{minipage}

    In this case there is an iterate $N^*$ of $N_{\alpha+1}$ above
    $N_\alpha \cap \Ord$ such that the iteration from $N_{\alpha+1}$
    to $N^*$ drops and we have that \[ M \unlhd N^*. \]
    Recall that $M$ denotes the $\omega_1^V$-th iterate of
    $M_{2n-1}^\#(x)$ by its least measure and its images and is
    therefore in particular not $(2n-1)$-small above $\omega_1^V$. But
    then the same holds for $N^*$ and thus it follows that
    $N_{\alpha+1}$ is not $(2n-1)$-small above $N_\alpha \cap \Ord$,
    which is a contradiction, because $N_{\alpha+1}$ is the model
    theoretic union of premice which are $(2n-1)$-small above
    $N_\alpha \cap \Ord$.

    This proves that \[ N_{\alpha+1} \unlhd M \]
    for all $\alpha < \omega_1^V$ and since
    $\Lpoddn(x) \cap \Ord = N_{\omega_1^V} \cap \Ord = \omega_1^V$ we
    finally have that
    \[ \Lpoddn(x) = M | \omega_1^V. \]
\end{proof}

\begin{remark}
  Lemma \ref{lem:LpM} also implies that the lower part model
  $\Lpoddn(x)$ is closed under the operation
  $a \mapsto M_{2n-2}^\#(a)$ by the proof of Lemma
  \ref{lem:iterability} $(1)$ as this holds for $M|\omega_1^V$, where
  $M$ again denotes the $\omega_1^V$-th iterate of $M_{2n-1}^\#(x)$ by
  its least measure and its images.
\end{remark}

Using this representation of lower part models we can also prove the
following lemma.

\begin{lemma}\label{lem:LpXY}
  Let $n \geq 1$ and assume that $\PI^1_{2n}$ determinacy holds. Let
  $x,y \in \BS$ be such that $x \in \Lpoddn(y)$. Then we have that
  \[ \Lpoddn(x) \subseteq \Lpoddn(y) \]
and in fact $\Lpoddn(x)$ is definable inside $\Lpoddn(y)$ from the parameter
$x$ plus possibly finitely many ordinal parameters.
  If in addition we have that $y \leq_T x$, then
  \[ \Lpoddn(x) = (\Lpoddn(x))^{\Lpoddn(y)}. \]
\end{lemma}

Here $(\Lpoddn(x))^{\Lpoddn(y)}$ denotes the model of height
$\omega_1^V$ which is constructed as in Definition \ref{def:Lp}, but
with models $M$ which are $\omega_1$-iterable above
$N_\alpha \cap \Ord$ inside $\Lpoddn(y)$ instead of inside $V$.

\begin{proof} Let $x,y \in \BS$ be such that $x \in \Lpoddn(y)$. We
  first prove that
  \[ \Lpoddn(x) \subseteq \Lpoddn(y). \]
  By Lemma \ref{lem:LpM} we have that
  $\Lpoddn(x) = M(x) | \omega_1^V$, where $M(x)$ denotes the
  $\omega_1^V$-th iterate of $M_{2n-1}^\#(x)$ by its least measure and
  its images. Moreover we have that $\Lpoddn(y) = M(y) | \omega_1^V$,
  where $M(y)$ denotes the $\omega_1^V$-th iterate of $M_{2n-1}^\#(y)$
  by its least measure and its images. Let $M^*(y)$ denote the result
  of iterating the top measure of $M(y)$ out of the universe.

  Let $L[E](x)^{M^*(y)}$ be as in Definition \ref{def:L[E]constr}.
  Moreover let $M^\#_x$ denote the model obtained from
  $L[E](x)^{M^*(y)}$ by adding the top measure (intersected with
  $L[E](x)^{M^*(y)}$) of the active premouse $M(y)$ to an initial
  segment of $L[E](x)^{M^*(y)}$ as in Section $2$ of \cite{FNS10}.

  We can successfully compare the active $x$-premice $M^\#_x$ and
  $M(x)$ inside the model $M_{2n-1}(z)$, where $z$ is a real coding
  $M_{2n-1}^\#(y)$ and $M_{2n-1}^\#(x)$, by an argument using Lemma
  \ref{lem:iterability} as in the proof of Lemma
  \ref{lemma1und2}. Therefore it follows that $M^\#_x = M(x)$ and thus
  we have that in fact
  \[ L[E](x)^{M^*(y)}|\omega_1^V = M^\#_x | \omega_1^V = M(x) |
  \omega_1^V. \] Since $L[E](x)^{M^*(y)} \subseteq M^*(y)$ it follows that
  \[ \Lpoddn(x) = M(x) | \omega_1^V = L[E](x)^{M^*(y)} | \omega_1^V
  \subseteq M^*(y) | \omega_1^V = \Lpoddn(y). \]

Now we prove that $\Lpoddn(x)$ is definable inside $\Lpoddn(y)$ from
the parameter $x$ plus possibly finitely many ordinal parameters.  Let
$(N_\alpha \st \alpha \leq \omega_1^V)$ be the sequence of models from
the definition of $\Lpoddn(x)$ in $V$ (see Definition \ref{def:Lp})
and let $(N_\alpha^{\Lpoddn(y)} \st \alpha \leq \omega_1^V)$ denote
the corresponding sequence of models from the definition of
$(\Lpoddn(x))^{\Lpoddn(y)}$. Let $\alpha < \omega_1^V$ be such that
$N_\alpha = N_\alpha^{\Lpoddn(y)}$ and let $N \unrhd N_\alpha$ be an
$x$-premouse which satisfies properties $(1)-(5)$ in Definition
\ref{def:Lp} in $V$. So in particular we have that $N \in \Lpoddn(x)$
and $N$ is $(2n-1)$-small above $N_\alpha \cap \Ord$. As argued above
we have that $N \in \Lpoddn(y)$ and we first want to show that as $N$
is $\omega_1$-iterable above $N_\alpha \cap \Ord$ in $V$ it follows
that $N$ is $\omega_1$-iterable above $N_\alpha \cap \Ord$ inside
$\Lpoddn(y)$.

So assume that $N$ is $\omega_1$-iterable above $N_\alpha \cap \Ord$
in $V$ and recall that the model $\Lpoddn(y)$ is closed under the
operation $a \mapsto M_{2n-2}^\#(a)$. Since $N$ is $(2n-1)$-small
above $N_\alpha \cap \Ord$ and has no definable Woodin cardinal above
$N_\alpha \cap \Ord$, we have that for an iteration tree $\T$ on $N$
of length $< \omega_1$ in $\Lpoddn(y)$ above $N_\alpha \cap \Ord$ the
iteration strategy $\Sigma$ is guided by $\Q$-structures which are
$(2n-2)$-small above $\delta(\T \upharpoonright \lambda)$ for every
limit ordinal $\lambda \leq \lh(\T)$. Therefore the $\Q$-structures
for $\T$ are contained in the model $\Lpoddn(y)$ and we have that $N$
is $\omega_1$-iterable inside $\Lpoddn(y)$ above $N_\alpha \cap \Ord$
if we argue as in the proof of Lemma \ref{lem:iterability} $(2)$.

  Assume now that $\Lpoddn(x) \neq (\Lpoddn(x))^{\Lpoddn(y)}$, as
  otherwise the definability of $\Lpoddn(x)$ inside $\Lpoddn(y)$ is
  trivial.  That means there is an ordinal $\alpha < \omega_1^V$ such
  that $N_\alpha = N_\alpha^{\Lpoddn(y)}$ and there exists a premouse
  $N \rhd N_\alpha$ which satifies properties $(1) - (5)$ in the
  definition of $(\Lpoddn(x))^{\Lpoddn(y)}$, so in particular $N$ is
  $\omega_1$-iterable above $N_\alpha \cap \Ord$ inside $\Lpoddn(y)$,
  but $N$ is not $\omega_1$-iterable above $N_\alpha \cap \Ord$ in
  $V$.

  Recall the model $M^*(y)$ from the first part of this proof. We have
  that $\Lpoddn(x) = L[E](x)^{M^*(y)} | \omega_1^V$ and $N$ are both
  sufficiently iterable in $M^*(y)$, so we can consider the
  coiteration of $\Lpoddn(x)$ and $N$ inside $M^*(y)$ and distinguish
  the following cases.

\begin{minipage}{\textwidth}
      \bigskip \textbf{Case 1.} Both sides of the comparison move.
      \bigskip
    \end{minipage}

    In this case both sides of the comparison have to drop since we
    have that $\rho_\omega(N) \leq N_\alpha \cap \Ord$ and
    $\Lpoddn(x)$ only has partial extenders on its sequence. As in
    Case $1$ in the proof of Lemma \ref{lem:LpM} this yields a
    contradiction.

    \begin{minipage}{\textwidth}
      \bigskip \textbf{Case 2.} Only the $\Lpoddn(x)$-side of the
      comparison moves.  \bigskip
    \end{minipage}

    Then the $\Lpoddn(x)$-side drops and we have that there is an
    iterate $M$ of $\Lpoddn(x)$ such that $N \unlhd M$. But this would
    imply that $N$ is $\omega_1$-iterable in $V$, contradicting our
    choice of $N$.

    \begin{minipage}{\textwidth}
      \bigskip \textbf{Case 3.} Only the $N$-side of the
      comparison moves.  \bigskip
    \end{minipage}
    
    In this case there exists an iterate $N^*$ of $N$ such that
    $\Lpoddn(x) \unlhd N^*$. In fact the iteration from $N$ to $N^*$
    only uses measures of Mitchell order $0$ as the iteration cannot
    leave any total measures behind. Since there are only finitely
    many drops along the main branch of an iteration tree, this
    implies that the whole iteration from $N$ to $N^*$ can be defined
    over $L[N]$ from $N$ and a finite sequence of ordinals as the
    iteration is linear. As $N \lhd (\Lpoddn(x))^{\Lpoddn(y)}$ we
    therefore get that $\Lpoddn(x)$ is definable inside $\Lpoddn(y)$
    from the parameter $x$ plus finitely many ordinal parameters.

    Now we prove the final part of Lemma \ref{lem:LpXY}, so assume
    that we in addition have that $y \leq_T x$. By the argument we
    just gave,

    \[ \Lpoddn(x) \subsetneq (\Lpoddn(x))^{\Lpoddn(y)} \subseteq
    \Lpoddn(y). \]

    This contradicts the following claim.
    
    \begin{claim}
      For reals $x,y$ such that $x \in \Lpoddn(y)$ and $y \leq_T x$ we
      have that
      \[ \Lpoddn(x) = \Lpoddn(y). \]
    \end{claim}
    \begin{proof}
      As $x \in \Lpoddn(y)$ we have that $x \in M(y)$ and thus
      $x \in M_{2n-1}^\#(y)$, because $M(y)$ is obtained from
      $M_{2n-1}^\#(y)$ by iterating its least measure and its images.
      
      As in the proof of Lemma \ref{lemma1und2} we can consider the
      premice $L[E](x)^{M_{2n-1}(y)}$ and
      $L[E](y)^{L[E](x)^{M_{2n-1}(y)}}$ and if we let $\kappa$ denote
      the least measurable cardinal in $M_{2n-1}(y)$, then we get as
      in the proof of Lemma \ref{lemma1und2} that
      \[ V_\kappa^{M_{2n-1}(y)} = V_\kappa^{L[E](x)^{M_{2n-1}(y)}}. \]
      Another comparison argument as in Lemma \ref{lemma1und2} yields
      that $V_\nu^{M_{2n-1}(x)} = V_\nu^{L[E](x)^{M_{2n-1}(y)}}$,
      where $\nu$ denotes the minimum of the least measurable cardinal
      in $M_{2n-1}(x)$ and the least measurable cardinal in
      $L[E](x)^{M_{2n-1}(y)}$. Therefore
      \[ V_\mu^{M_{2n-1}(y)} = V_\mu^{L[E](x)^{M_{2n-1}(y)}} =
        V_\mu^{M_{2n-1}(x)}, \] where $\mu = \min(\kappa, \nu)$,
      i.e. $\mu$ is the minimum of the least measurable cardinal in
      $M_{2n-1}(y)$ and the least measurable cardinal in
      $M_{2n-1}(x)$. In particular we can consider $M_{2n-1}(x)$ and
      $M_{2n-1}(y)$ as $V_\mu^{M_{2n-1}(x)}$-premice and as such they
      successfully coiterate to the same model (see again Lemma
      \ref{lemma1und2} for a similar argument). This implies that
      $\Lpoddn(x) = \Lpoddn(y)$ as they can be obtained by iterating
      $M_{2n-1}(x)$ and $M_{2n-1}(y)$ via their least measure and its
      images.
    \end{proof}
\end{proof}

In general it need not be the case that $\Lpoddn(x) =
(\Lpoddn(x))^{\Lpoddn(y)}$ as the following counterexample shows. 

\begin{lemma}
  \label{lem:LpInLpFake}
  Assume that $\PI^1_2$ determinacy holds. Then there exists a real
  $x$ and a premouse $M$ such that $M$ is $\omega_1$-iterable in
  $Lp^1(x)$, but $M$ is not $\omega_1$-iterable in $V$.
\end{lemma}

\begin{proof}[Proof sketch]
  Consider the premouse $M_1$ and let $\delta$ denote the Woodin
  cardinal in $M_1$. Force with the countable stationary tower (see
  \cite{La04}) over $M_1$ and let \[ j : M_1 \rightarrow N \] be the
  corresponding generic elementary embedding with
  $\crit(j) = \omega_1^{M_1}$ and $j(\omega_1^{M_1}) =
  \delta$. Moreover let $M$ denote the shortest cutpoint initial
  segment\footnote{i.e. $M \lhd N$ and $M \cap \Ord$ is a cutpoint in
    $N$} of $N$ such that $M_1 \cap \mathbb{R} \subset M$ and
  $\rho_\omega(M) = \omega$. Then we can identify the premouse $M$
  with a canonical real $m$ coding $M$.

$M_1|\delta$ thinks that all of its own initial segments are fully iterable,
so that $N|j(\delta)$ thinks that $M$ is fully iterable. Moreover, 
a short argument shows that $N$ is short tree iterable in $V$,
  i.e. $N$ is iterable in $V$ according to iteration trees $\T$ on $N$
  such that $\Q(\T) \unlhd L[\M(\T)]$. This is proven using the fact
  that $M_1$ and thus $N$ is a model of the statement ``I am short
  tree iterable'' together with an absoluteness argument similar to
  the one given for example in the proof of Lemma \ref{lem:nsuitnew}.

  Now consider the coiteration of $M_1(m)$ and $N$ (construed as an
  $m$-premouse) in $V$ and let $\T$ and $\U$ denote the resulting
  trees on $M_1(m)$ and $N$ respectively, where the cofinal branches
  at limit steps in $\U$ are given by the short tree iteration
  strategy for $N$. So there exists an iterate $M^*$ of $M_1(m)$ and a
  pseudo-normal-iterate (see Definition $3.13$ in \cite{StW16}) $N^*$
  of $N$ such that $M^* = N^*$. Here we either have that $N^*$ is the
  last model of $\U$, or we have that the iteration tree $\U$ is
  maximal and $N^* = L[\M(\U)]$. Moreover there is no drop on the main
  branch in $\T$, which is leading from $M_1(m)$ to $M^*$, and
  therefore the iteration embedding $i: M_1(m) \rightarrow M^*$
  exists. 
As $M$ is fully iterable inside $N|j(\delta)$, $M$ will also be fully iterable
inside $N^*$, cut off at its Woodin cardinal (to conclude this, we don't need
that there is a map from $N$ to $N^*$). Hence $M$ is fully iterable inside $M^*$,
cut off at its Woodin cardinal, 
and then by elementarity it follows
  that $M$ is $\omega_1$-iterable inside $M_1(m)$ and thus
  \[ Lp^1(m) \vDash \text{``}M \text{ is }
    \omega_1\text{-iterable''}. \]

  But the premouse $M$ is not $\omega_1$-iterable in $V$. 
\end{proof}

We also have a version of Lemma \ref{lem:KS} for lower part models
$\Lpoddn(x)$ as follows.

\begin{lemma}\label{lem:versionKS}
  Let $n \geq 1$ and assume that $\PI^1_{2n}$ determinacy and
  $\Pi^1_{2n+1}$ determinacy hold. Then there exists a real $x$ such
  that we have for all reals $y \geq_T x$ that
  \[ \Lpoddn(y) \vDash \ODdet. \]
\end{lemma}

\begin{proof}
  Recall that we assume inductively that $\PI^1_{2n}$ determinacy
  implies that $M_{2n-1}^\#(x)$ exists for all $x \in \BS$. Then Lemma
  \ref{lem:versionKS} follows from Lemma \ref{lem:KS} by the following
  argument. For $x \in \BS$ we have that
  \[ M_{2n-1}(x)|\delta_x \vDash \ODdet \]
  implies that \[ M_{2n-1}(x)|\kappa \vDash \ODdet, \]
  where $\kappa$ is the least measurable cardinal in $M_{2n-1}(x)$,
  because whenever a set of reals $A$ is ordinal definable in
  $M_{2n-1}(x)|\kappa$, then it is also ordinal definable in
  $M_{2n-1}(x)|\delta_x$, since $M_{2n-1}(x)|\kappa$ and
  $M_{2n-1}(x)|\delta_x$ have the same sets of reals. This yields by
  elementarity that \[ M | \omega_1^V \vDash \ODdet, \]
  where $M$ denotes the $\omega_1^V$-th iterate of $M_{2n-1}(x)$ by
  its least measure and its images. By Lemma \ref{lem:LpM} we have
  that $M | \omega_1^V = \Lpoddn(x)$, so it follows that
  \[ \Lpoddn(x) \vDash \ODdet. \]
\end{proof}

This immediately yields that we have the following variant of Theorem
\ref{thm:w2inacc}.

\begin{thm}\label{thm:varw2inacc}
  Let $n \geq 1$ and assume that $\PI^1_{2n}$ determinacy and
  $\Pi^1_{2n+1}$ determinacy hold. Then there exists a real $x$ such
  that we have for all reals $y \geq_T x$ that
  \[ \HOD^{\Lpoddn(y)} \vDash \text{``} \omega_2^{\Lpoddn(y)} \text{
    is inaccessible''.} \]
\end{thm}

Now we can turn to the proof of Theorem \ref{thm:newWoodineven}, which
is going to yield Theorem \ref{thm:newWoodin1} in the
case that $n$ is even.

\begin{proof}[Proof of Theorem \ref{thm:newWoodineven}]
  We start with constructing a $\ZFC$ model $M_x$ of height
  $\omega_1^V$ for some $x \in \BS$ such that $M_x$ is
  $\SIGMA^1_{2n+2}$-correct in $V$ for real parameters in $M_x$ and
  closed under the operation $a \mapsto M_{2n-1}^\#(a)$. To prove that
  this construction yields a model of $\ZFC$, we are as before going
  to use the hypothesis that there is no $\SIGMA^1_{2n+2}$-definable
  $\omega_1$-sequence of pairwise distinct reals. The construction
  will be similar to the construction we gave in Section
  \ref{sec:odd}, but here we have to manually ensure that $M_x$ is
  $\SIGMA^1_{2n+2}$-correct in $V$.

  Fix an arbitrary $x \in \BS$ and construct a sequence of models
  $(W_\alpha \st \alpha \leq \omega_1^V)$. So the model
  $M_x = W_{\omega_1^V}$ is build level-by-level in a construction of
  length $\omega_1^V$. We are starting from $W_0 = \{x\}$ and are
  taking unions at limit steps of the construction. At an odd
  successor level $\alpha + 1$ we will close the model $W_{\alpha}$
  under Skolem functions for $\SIGMA^1_{2n+2}$-formulas. At the same
  time we will use the even successor levels $\alpha +2$ to ensure
  that $M_x$ will be closed under the operation
  $a \mapsto M_{2n-1}^\#(a)$. As before the order of construction for
  elements of the model $M_x$ can be defined along the way. 

Before we are describing this construction in more detail, we fix a
$\Pi^1_{2n+1}$-definable set $U$ which is universal for the pointclass
$\PI^1_{2n+1}$. Pick the set $U$ such that we have
$U_{\ulcorner \varphi \urcorner ^\frown a} = A_{\varphi,a}$ for every
$\Pi^1_{2n+1}$-formula $\varphi$ and every $a \in \BS$, where
$\ulcorner \varphi \urcorner$ denotes the Gödel number of the formula
$\varphi$ and
  \[ A_{\varphi,a} = \{ x \st \varphi(x,a) \}. \] 

  Then the uniformization property (see Theorem $6C.5$ in \cite{Mo09})
  yields that there exists a $\Pi^1_{2n+1}$-definable function $F$
  uniformizing the universal set $U$. So we have for all
  $z \in \dom(F)$ that
  \[ (z,F(z)) \in U, \]
  where $\dom(F) = \{ z \st \exists y \, (z,y) \in U \}$. 

  \textbf{Odd successor steps:} For the odd successor steps of the
  construction assume now that we already constructed the model
  $W_\alpha$ such that $\alpha + 1$ is odd and that there exists a
  $\PI^1_{2n+1}$-formula $\varphi$ with a real parameter $a$ from
  $W_\alpha$ such that $\exists x \varphi(x,a)$ holds in $V$ but not
  in the model $W_\alpha$. In this case we aim to add a real
  $x_{\varphi,a}$ constructed as described below to $W_{\alpha+1}$
  such that $\varphi(x_{\varphi,a},a)$ holds. This real will witness
  that $\exists x \varphi(x,a)$ holds true inside $W_{\alpha+1}$.

  We aim to build these levels of $M_x$ in a
  $\Sigma^1_{2n+2}$-definable way, so we choose reals $x_{\varphi,a}$
  carefully. Therefore we add $F(z)$ for all
  $z \in \dom(F) \cap W_\alpha$ to the current model $W_\alpha$. We
  will see in Claims \ref{cl:MxZFCEven} and \ref{cl:MxEvenProp} that
  this procedure adds reals $x_{\varphi,a}$ as above in a
  $\Pi^1_{2n+1}$-definable way to the model $M_x$.

  So let $\varphi_F$ be a $\Pi^1_{2n+1}$-formula such that for all
  $x,y \in \BS$, \[ F(x) = y \text{ iff } \varphi_F(x,y). \]
  Then we let
  \[ W_{\alpha+1} = rud(W_\alpha \cup \{ y \in \BS \st \exists x \in
  W_\alpha \cap \BS \; \varphi_F(x,y) \}). \]

\textbf{Order of construction:} We inductively define the order of
construction for elements of $W_{\alpha+1}$ as follows. First we say
for $F(x) \neq F(x^\prime)$ with
$x,x^\prime \in \dom(F) \cap W_\alpha$ that $F(x)$ is constructed
before $F(x^\prime)$ iff $x$ is constructed before $x^\prime$ in the
order of construction for elements of $W_\alpha$ where $x$ and
$x^\prime$ are the minimal (according to the order of construction in
$W_\alpha$) reals $y$ and $y^\prime$ in $\dom(F) \cap W_\alpha$ such
that $F(y) = F(x)$ and $F(y^\prime) = F(x^\prime)$. Then we define the
order of construction for elements added by the closure under
rudimentary functions as in the definition of the order of
construction for $L$.

\textbf{Even successor steps:} At an even successor level $\alpha+2$
of the construction we close the previous model $W_{\alpha+1}$ under
the operation $a \mapsto M_{2n-1}^\#(a)$ before closing under
rudimentary functions. Assume that we already constructed
$W_{\alpha+1}$ and let $a \in W_{\alpha+1}$ be arbitrary. The
$a$-premouse $M_{2n-1}^\#(a)$ exists in $V$ because we as usual
inductively assume that Theorem \ref{cor:newWoodin1} holds for
$2n-1$. As in the proof of Theorem \ref{thm:newWoodinodd} let $\M$ be
a countable $a$-premouse in $V$ with the following properties.
  \begin{enumerate}[$(i)$]
  \item $\M$ is $2n$-small, but not $(2n-1)$-small,
  \item all proper initial segments of $\M$ are $(2n-1)$-small,
  \item $\M$ is $a$-sound and $\rho_\omega(\M) = a$, and
  \item $\M$ is $\Pi^1_{2n+1}$-iterable.
  \end{enumerate}
  We have that these properties $(i) - (iv)$ uniquely determine the
  $a$-premouse $M_{2n-1}^\#(a)$ in $V$ and we add such $a$-premice
  $\M$ for all $a \in W_{\alpha+1}$ to the model $W_{\alpha+2}$ before
  closing under rudimentary functions as in the usual construction of
  $L$, i.e. \begin{eqnarray*}
     W_{\alpha+2} = rud(W_{\alpha+1} \cup \{ \M
    \st \exists a \in W_{\alpha+1} \text{ such that } \M \text{ is a
      countable } \\ a\text{-premouse satisfying }(i)-(iv)\}).
  \end{eqnarray*}

  \textbf{Order of construction:} For an $a$-premouse $\M_a$ and a
  $b$-premouse $\M_b$ satifying properties $(i)-(iv)$ for
  $a, b \in W_{\alpha+1}$, we say that $\M_a$ is defined before $\M_b$
  if $a$ is defined before $b$ in the order of construction for
  elements of $W_{\alpha+1}$, which exists inductively. For elements
  added by the closure under rudimentary functions we define the order
  of construction as in the definition of the order of construction
  for $L$.

  \textbf{Limit steps:} Finally we let
  \[ W_\lambda = \bigcup_{\alpha < \lambda} W_\alpha \]
  for limit ordinals $\lambda < \omega_1^V$ and
  \[ M_x = W_{\omega_1^V} = \bigcup_{\alpha < \omega_1^V} W_\alpha. \]

  \textbf{Order of construction:} The order of construction at the
  limit steps is defined as in the definition of the order of
  construction for $L$.

  As in Section \ref{sec:odd} we are now able to show that this model
  $M_x$ satisfies $\ZFC$, using the background hypothesis that every
  $\SIGMA^1_{2n+2}$-definable sequence of pairwise distinct reals is
  countable. 

  \begin{claim}\label{cl:MxZFCEven}
    $M_x \vDash \ZFC.$
  \end{claim}

  \begin{proof}
    For the even successor levels of the construction we have that
    properties $(i)-(iv)$ as in the construction are
    $\Pi^1_{2n+1}$-definable uniformly in any code for the countable
    premouse $\M$.

    For the odd successor levels recall that $U$ is a
    $\Pi^1_{2n+1}$-definable set and that $F$ is a
    $\Pi^1_{2n+1}$-definable function uniformizing $U$.

    Therefore the argument given in the proof of Claim
    \ref{cl:MzZFCnew} in the proof of Theorem \ref{thm:newWoodinodd}
    yields a contradiction to the background hypothesis that every
    $\SIGMA^1_{2n+2}$-definable sequence of pairwise distinct reals is
    countable.
  \end{proof}

Moreover we can prove the following claim. 

  \begin{claim}
    \label{cl:MxEvenProp}
    The model $M_x$ as constructed above has the following properties.
    \begin{enumerate}[$(1)$]
    \item $M_x \cap \Ord = \omega_1^V$, 
    \item $x \in M_x$,
    \item $M_x$ is $\SIGMA^1_{2n+2}$-correct in $V$ for real
      parameters in $M_x$, that means we have that
      \[ M_x \prec_{\SIGMA^1_{2n+2}} V, \] 
    \item $M_x$ is closed under the operation
      \[ a \mapsto M_{2n-1}^\#(a), \]
      and moreover $M_{2n-1}^\#(a)$ is $\omega_1$-iterable in $M_x$
      for all $a \in M_x$.
    \end{enumerate}
  \end{claim}

  \begin{proof}
    Properties $(1)$ and $(2)$ immediately follow from the
    construction.

    \medskip \textbf{Proof of $(3)$:} The proof is organized as an
    induction on $m < 2n+1$. We have that $M_x$ is
    $\SIGMA^1_2$-correct in $V$ using Shoenfield’s Absoluteness
    Theorem (see for example Theorem 13.15 in \cite{Ka03}). We assume
    inductively that $M_x$ is $\SIGMA^1_{m}$-correct in $V$ and prove
    that $M_x$ is $\SIGMA^1_{m+1}$-correct in $V$. Since the upward
    implication follows easily as in the proof of Lemma
    \ref{lem:corr}, we focus on the proof of the downward implication.

    So let $\psi$ be a $\Sigma^1_{m+1}$-formula and let
    $a \in M_x \cap \BS$ be such that $\psi(a)$ holds in $V$. Say
    \[ \psi(a) \equiv \exists x \varphi(x,a) \]
    for a $\Pi^1_{m}$-formula $\varphi$.

    Let $y = \ulcorner \varphi \urcorner ^\frown a \in \BS$. Then we
    have that $y \in \dom(F)$ since $\psi(a)$ holds in $V$ and
    $m \leq 2n+1$. Therefore $\varphi_F(y,F(y))$ holds and $F(y)$ is
    added to the model $M_x$ at an odd successor level of the
    construction because we have that $y \in M_x$. Recall that $F(y)$
    is choosen such that $(y, F(y)) \in U$, that means we have that
    \[ F(y) \in U_y = U_{\ulcorner \varphi \urcorner ^\frown a} = \{ x
    \st \varphi(x,a) \}, \]
    by our choice of $U$. Now the inductive hypothesis implies that
    \[ M_x \vDash \varphi(F(y), a) \]
    and therefore it follows that $M_x \vDash \psi(a)$, as desired.

    Property $(4)$ now follows by the same argument as the one we gave
    in the proof of property $(3)$ in Claim \ref{cl:MzOddProp} in the
    proof of Theorem \ref{thm:newWoodinodd}.
  \end{proof}

  The following additional property of the model $M_x$ is a key point
  in proving that $M_{2n}^\#$ exists and is $\omega_1$-iterable in
  $V$.

  \begin{claim}\label{cl:M2nexinMx} For all $x \in \BS$ in the cone of
    reals given in Theorem \ref{thm:varw2inacc},
    \[ M_x \vDash \text{``}M_{2n}^\# \text{ exists and is }
    \omega_1\text{-iterable.''} \]
  \end{claim}

  The proof of this claim is now different from the proof of the
  analogous claim in the previous section. The reason for this is that
  at the even levels we cannot assume that we have a $2n$-suitable
  premouse to compare the model $K^c$ with (which at the odd levels
  was given by Lemmas \ref{lem:prensuitable} and
  \ref{lem:nsuitnew}). This is why we have to give a different
  argument here.

  \begin{proof}[Proof of Claim \ref{cl:M2nexinMx}]
    Assume this is not the case. Let $(K^c)^{M_x}$ denote the model
    resulting from a $2n$-small robust-background-extender
    $K^c$-construction as in \cite{Je03} inside $M_x$. Since $M_x$ is
    closed under the operation $a \mapsto M_{2n-1}^\#(a)$,
    $(K^c)^{M_x}$ is fully iterable inside $M_x$ by a generalization
    of Theorem $2.11$ in \cite{St96} (see Corollary $3$ in
    \cite{Je03}).  Therefore we can build the core model $K^{M_x}$
    inside $M_x$ by a generalization of Theorem $1.1$ in \cite{JS13}
    due to Jensen and Steel. The core model $K^{M_x}$ has to be
    $2n$-small, because otherwise we would have that
    \[M_x \vDash \text{``There exists a model which is fully iterable
        and not }2n\text{-small''.}\] This would already imply that
    $M_{2n}^\#$ exists and is fully iterable inside $M_x$, so in this
    case there is nothing left to show.

    \begin{subclaim}\label{subcl:clsharps}
      $K^{M_x}$ is closed under the operation \[ a \mapsto
      M_{2n-1}^\#(a). \]
    \end{subclaim}
    \begin{proof}
      We start with considering sets of the form $a = K^{M_x} | \xi$
      where $\xi < K^{M_x} \cap \Ord$ is not overlapped by an extender
      on the $K^{M_x}$-sequence. That means there is no extender $E$
      on the $K^{M_x}$-sequence such that
      $\crit(E) \leq \xi < \lh(E)$. We aim to prove that in fact
      \[ M_{2n-1}^\#(K^{M_x}|\xi) \lhd K^{M_x}.\]
      We have that the premouse $M_{2n-1}^\#(K^{M_x}|\xi)$ exists
      inside the model $M_x$ since we have that
      $\xi < K^{M_x} \cap \Ord = M_x \cap \Ord$ and $M_x$ is closed
      under the operation
      \[ a \mapsto M_{2n-1}^\#(a) \] by property $(4)$ in Claim
      \ref{cl:MxEvenProp}. Consider the coiteration of the premice
      $K^{M_x}$ and $M_{2n-1}^\#(K^{M_x}|\xi)$ inside $M_x$. This
      coiteration takes place above $\xi$ and thus both premice are
      iterable enough such that the comparison is successful by Lemma
      \ref{lem:iterability} since $M_{2n-1}^\#(K^{M_x}|\xi)$ is
      $\omega_1$-iterable above $\xi$ in $M_x$. By universality of
      $K^{M_x}$ inside $M_x$ (by Theorem $1.1$ in \cite{JS13} and
      Lemma $3.5$ in \cite{St96} applied inside $M_x$), we have that
      there is an iterate $M^*$ of $M_{2n-1}^\#(K^{M_x}|\xi)$ and an
      iterate $K^*$ of $K^{M_x}$ such that $M^* \unlhd K^*$ and the
      iteration from $M_{2n-1}^\#(K^{M_x}|\xi)$ to $M^*$ is
      non-dropping on the main branch. Since
      \[ \rho_\omega(M_{2n-1}^\#(K^{M_x}|\xi)) \leq \xi \]
      and since the coiteration takes place above $\xi$, we have that
      the iterate $M^*$ of $M_{2n-1}^\#(K^{M_x}|\xi)$ is not sound, if
      any extender is used on this side of the coiteration. Therefore
      it follows that in fact
      \[ M_{2n-1}^\#(K^{M_x}|\xi) \lhd K^*. \]
      Assume that the $K^{M_x}$-side moves in the coiteration, that
      means we have that $K^{M_x} \neq K^*$. Let $\alpha$ be an
      ordinal such that $E_\alpha$ is the first extender on the
      $K^{M_x}$-sequence which is used in the coiteration. Then we
      have that $\alpha > \xi$. We have in particular that $\alpha$ is
      a cardinal in $K^*$. But then since we have that
      $\rho_\omega(M_{2n-1}^\#(K^{M_x}|\xi)) \leq \xi < \alpha$ and
      $M_{2n-1}^\#(K^{M_x}|\xi) \lhd K^*$, this already implies that
      \[ \alpha > M_{2n-1}^\#(K^{M_x}|\xi) \cap \Ord. \]
      Therefore there was no need to iterate $K^{M_x}$ at all and we
      have that
      \[ M_{2n-1}^\#(K^{M_x}|\xi) \lhd K^{M_x}.\]
   
      Now let $a \in K^{M_x}$ be arbitrary. Then there exists an
      ordinal $\xi < K^{M_x} \cap \Ord$ such that $a \in K^{M_x}|\xi$
      and $\xi$ is not overlapped by an extender on the
      $K^{M_x}$-sequence. We just proved that
      \[ M_{2n-1}^\#(K^{M_x}|\xi) \lhd K^{M_x}.\]

      As we argued several times before, by considering
      $L[E](a)^{M_{2n-1}(K^{M_x}|\xi)}$ as in Definition
      \ref{def:L[E]constr} and adding the top extender of the active
      premouse $M_{2n-1}^\#(K^{M_x}|\xi)$ (intersected with
      $L[E](a)^{M_{2n-1}(K^{M_x}|\xi)}$) to an initial segment of the
      model $L[E](a)^{M_{2n-1}(K^{M_x}|\xi)}$ as described in Section
      $2$ in \cite{FNS10} we obtain that
      \[M_{2n-1}^\#(a) \in K^{M_x},\] as desired.
    \end{proof}

    Using the Weak Covering Lemma from \cite{JS13} (building on
    \cite{MSchSt97} and \cite{MSch95}) we can pick a cardinal
    $\gamma \in M_x$ such that $\gamma$ is singular in $M_x$ and
    $\gamma^+$ is computed correctly by $K^{M_x}$ inside $M_x$. That
    means we pick $\gamma$ such that we have
    \[ (\gamma^+)^{K^{M_x}} = (\gamma^+)^{M_x}. \]
    For later purposes we want to pick $\gamma$ such that it
    additionally satisfies
    \[ \cf(\gamma)^{M_x} \geq \omega_1^{M_x}. \]

    \begin{subclaim}\label{subcl:almostdisjcoding}
      There exists a real $z \geq_T x$ such that 
      \begin{enumerate}[$(1)$]
      \item $(\gamma^+)^{M_x} = \omega_2^{\Lpoddn(z)},$
        and
      \item $K^{M_x}|(\gamma^+)^{M_x} \in \Lpoddn(z).$
      \end{enumerate}
    \end{subclaim}
    \begin{proof}
      We are going to produce the real $z$ via a five step forcing
      using an almost disjoint coding. For an introduction into this
      kind of forcing see for example \cite{FSS14} for a survey or
      \cite{Sch00} where a similar argument is given. 

      We force over the ground model
      \[\Lpoddn(x, K^{M_x}|(\gamma^+)^{M_x}).\] We have that $\Lpoddn(x,
      K^{M_x}|(\gamma^+)^{M_x})$
      is a definable subset of $M_x$ because we have by property $(4)$
      in Claim \ref{cl:MxEvenProp} that
      \[ M_{2n-1}^\#(x, K^{M_x}|(\gamma^+)^{M_x}) \in M_x \]
      and by Lemma \ref{lem:LpM} the lower part model
      $\Lpoddn(x, K^{M_x}|(\gamma^+)^{M_x})$ is obtained by iterating
      the least measure of $M_{2n-1}^\#(x, K^{M_x}|(\gamma^+)^{M_x})$
      and its images $\omega_1^V$ times and cutting off at
      $\omega_1^V$.

      This implies that in particular
      \[ \cf(\gamma)^{\Lpoddn(x, K^{M_x}|(\gamma^+)^{M_x})} \geq
      \omega_1^{M_x}. \]

      \textbf{Step 1:} Write
      $V_0 = \Lpoddn(x, K^{M_x}|(\gamma^+)^{M_x})$ for the ground
      model. We start with a preparatory forcing that collapses
      everything below $\omega_1^{M_x}$ to $\omega$. Afterwards we
      collapse $\gamma$ to $\omega_1^{M_x}$. 

      So let $G_0 \in V$ be $\Col(\omega, <\omega_1^{M_x})$-generic
      over $V_0$ and let
      \[ V_0^\prime = V_0[G_0]. \]
      Moreover let $G_0^\prime \in V$ be
      $\Col(\omega_1^{M_x}, \gamma)$-generic over $V_0^\prime$ and let
      \[ V_1 = V_0^\prime[G_0^\prime]. \]
      So we have that $\omega_1^{M_x} = \omega_1^{V_1}$ and by our
      choice of $\gamma$, i.e.
      $\cf(\gamma)^{V_0} \geq \omega_1^{M_x}$, we moreover have that
      $(\gamma^+)^{M_x} = (\gamma^+)^{K^{M_x}} = \omega_2^{V_1}$.  We
      write $\omega_1 = \omega_1^{V_1}$ and
      $\omega_2 = \omega_2^{V_1}$.

      Furthermore let $A^\prime$ be a set of ordinals coding $G_0$ and
      $G_0^\prime$, such that if we let $A \subset (\gamma^+)^{M_x}$
      be a code for
      \[ x \oplus (K^{M_x}|(\gamma^+)^{M_x}) \oplus A^\prime, \]
      then we have that $G_0, G_0^\prime \in \Lpoddn(A)$ and
      $K^{M_x}|(\gamma^+)^{M_x} \in \Lpoddn(A)$.

      We can in fact pick the set $A$ such that we have
      $V_1 = \Lpoddn(A)$ by the following argument: Recall that
      \[ \Lpoddn(A) = M(A) | \omega_1^V, \]
      where $M(A)$ denotes the $\omega_1^V$-th iterate of
      $M_{2n-1}^\#(A)$ by its least measure and its images for a set
      $A$ as above. Then we can consider $G_0$ as being generic over
      the model $M(x, K^{M_x} | (\gamma^+)^{M_x})$ and $G_0^\prime$ as
      being generic over the model
      $M(x, K^{M_x} | (\gamma^+)^{M_x})[G_0]$, where
      $M(x, K^{M_x} | (\gamma^+)^{M_x})$ denotes the $\omega_1^V$-th
      iterate of $M_{2n-1}^\#(x, K^{M_x} | (\gamma^+)^{M_x})$ by its
      least measure and its images. Since both forcings in Step $1$
      take place below $(\gamma^+)^{M_x} < \omega_1^V$, it follows as
      in the proof of Theorem \ref{Steel} that
      $M(x,K^{M_x}|(\gamma^+)^{M_x})[G_0][G_0^\prime] = M(A)$ for a
      set $A \subset (\gamma^+)^{M_x}$ coding
      $x, K^{M_x}|(\gamma^+)^{M_x}, G_0$ and $G_0^\prime$ and thus we
      have that $V_1 = M(A) | \omega_1^V$ for this set $A$, as
      desired.

      \textbf{Step 2:} Before we can perform the first coding using
      almost disjoint subsets of $\omega_1 = \omega_1^{V_1}$ we have
      to ``reshape'' the interval between
      $(\gamma^+)^{M_x} = \omega_2^{V_1}$ and $\omega_1$ to ensure
      that the coding we will perform in Step $3$ exists. Moreover we
      have to make sure that the reshaping forcing itself does not
      collapse $\omega_1$ and $(\gamma^+)^{M_x}$. We are going to show
      this by proving that the reshaping forcing is
      $< (\gamma^+)^{M_x}$-distributive.

      We are going to use the following notion of reshaping.
        \begin{definition}\label{def:reshaping}
          Let $\eta$ be a cardinal and let $X \subset \eta^+$. We say
          a function $f$ is \emph{$(X, \eta^+)$-reshaping}
          \index{reshaping@$(X, \eta^+)$-reshaping} iff
          $f: \alpha \rightarrow 2$ for some $\alpha \leq \eta^+$ and
          moreover for all $\xi \leq \alpha$ with $\xi < \eta^+$ we
          have that
          \begin{enumerate}[$(i)$]
          \item
            $L[X \cap \xi, f \upharpoonright \xi] \vDash \card{\xi}
            \leq \eta$, or
          \item there is a model $N$ and a $\Sigma_k$-elementary
            embedding \[ j: N \rightarrow \Lpoddn(X) | \eta^{++} \]
            for some large enough $k < \omega$ such that 
            \begin{enumerate}[$(a)$]
            \item $\crit(j) = \xi$, $j(\xi) = \eta^+$,
            \item $\rho_{k+1}(N) \leq \xi$, $N$ is sound above $\xi$,
              and
            \item definably over $N$ there exists a surjection $g :
              \eta \twoheadrightarrow \xi$. 
            \end{enumerate}
          \end{enumerate}
        \end{definition}

For future purposes notice that if $N$ is as in Clause (ii) above, then
$N \triangleleft \Lpoddn(X \cap \xi)$.

        Now we denote with $P_1$ the forcing that adds an
        $(A, (\gamma^+)^{M_x})$-reshaping function for
        $(\gamma^+)^{M_x} = \omega_2^{V_1}$, defined inside our new
        ground model $V_1 = \Lpoddn(A)$.

        We let $p \in P_1$ iff $p$ is an
        $(A, (\gamma^+)^{M_x})$-reshaping function with
        $\dom(p) < (\gamma^+)^{M_x}$ and we order two conditions $p$
        and $q$ in $P_1$ by reverse inclusion, that means we let
        $p \leq_{P_1} q$ iff $q \subseteq p$.

        First notice that the forcing $P_1$ is extendible, that means
        for every ordinal $\alpha < (\gamma^+)^{M_x}$ the set
        $D^\alpha = \{ p \in P_1 \st \dom(p) \geq \alpha \}$ is open
        and dense in $P_1$. In fact, for each $p \in P_1$ and every
        $\alpha < (\gamma^+)^{M_x}$ there is some $q \leq_{P_1} p$
        such that $\dom(q) \geq \alpha$ and
        $L[A \cap \xi, q \upharpoonright \xi] \vDash \card{\xi} \leq
        \eta$ for all $\xi$ with $\dom(p) < \xi \leq \dom(q)$.

        We now want to show that $P_1$ is
        $<(\gamma^+)^{M_x}$-distributive. For that we fix a condition
        $p \in P_1$ and open dense sets
        $(D_\beta \st \beta < \omega_1)$. We aim to find a condition
        $q \leq_{P_1} p$ such that $q \in D_\beta$ for all
        $\beta < \omega_1$.

        Consider, for some large enough fixed natural number $k$,
        transitive $\Sigma_{k}$-elementary substructures of the model
        $\Lpoddn(A) = V_1$. More precisely we want to pick a
        continuous sequence
        \[ (N_{\alpha}, \pi_{\alpha}, \xi_\alpha \st \alpha \leq
        \omega_1) \]
        of transitive models $N_{\alpha}$ of size
        $\card{\omega_1^{V_1}}$ together with $\Sigma_{k}$-elementary
        embeddings
        \[ \pi_{\alpha} : N_{\alpha} \rightarrow \Lpoddn(A) \]
        and an increasing sequence of ordinals $\xi_\alpha$ such that
        we have $p \in N_{0}$, and for all $\alpha \leq \omega_1$
        \begin{enumerate}[$(1)$]
        \item $\crit(\pi_{\alpha}) = \xi_\alpha$ with
          $\pi_{\alpha} (\xi_\alpha) = (\gamma^+)^{M_x}$,
        \item for all ordinals $\alpha < \omega_1$ we have that
          $\rho_{k+1}(N_{\alpha}) \leq \xi_\alpha$ and $N_{\alpha}$ is
          sound above $\xi_\alpha$, and
        \item
          $\{ p \} \cup \{ D_\beta \st \beta < \omega_1 \} \subset
          \ran(\pi_{\alpha})$.
        \end{enumerate}
        We can obtain $N_{\alpha}$ and $\pi_{\alpha}$ for all
        $\alpha \leq \omega_1$ with these properties inductively as
        follows.  Let $M_0$ be the (uncollapsed) $\Sigma_{k}$-hull of
        \[\gamma \cup \{ p \} \cup \{ D_\beta \st \beta <
        \omega_1 \}\]
        taken inside $\Lpoddn(A)$. Then let $N_{0}$ be the
        Mostowski collapse of $M_0$ and let
        \[ \pi_{0} : N_{0} \rightarrow M_0 \prec_{\Sigma_{k}}
        \Lpoddn(A) \]
        be the inverse of the embedding obtained from the Mostowski
        collapse with critical point $\xi_0$.

        Now assume we already constructed
        $(N_{\alpha}, \pi_{\alpha}, \xi_\alpha)$ and $M_\alpha$ for
        some $\alpha < \omega_1$. Then we let $M_{\alpha +1}$ be the
        (uncollapsed) $\Sigma_{k}$-hull of
        \[\gamma \cup \{ p \} \cup \{ D_\beta \st \beta <
        \omega_1 \} \cup M_{\alpha} \cup \{ M_{\alpha} \} \]
        taken inside $\Lpoddn(A)$. Further let $N_{\alpha + 1}$ be the
        Mostowski collapse of $M_{\alpha + 1}$ and let
        \[ \pi_{\alpha + 1} : N_{\alpha + 1} \rightarrow M_{\alpha +
          1} \prec_{\Sigma_{k}} \Lpoddn(A) \]
        be the inverse of the embedding obtained from the Mostowski
        collapse with critical point $\xi_{\alpha +1}$. Note that we
        have $\xi_{\alpha+1} > \xi_\alpha$.

        Moreover if we assume that
        $(N_{\alpha}, \pi_{\alpha},\xi_\alpha)$ is already constructed
        for all $\alpha < \lambda$ for some limit ordinal
        $\lambda \leq \omega_1$, then we let
        \[ M_{\lambda} = \bigcup_{\alpha < \lambda} M_{\alpha}, \] and
        we let $N_\lambda$ be the Mostowski collapse of $M_\lambda$
        with inverse collapse embedding
        \[ \pi_{\lambda} : N_\lambda \rightarrow M_\lambda, \] with
        critical point $\crit(\pi_\lambda) = \xi_\lambda$.

        Recall that we fixed open dense sets
        $(D_\beta \st \beta < \omega_1)$. We are now going to
        construct a sequence $(p_\alpha \st \alpha \leq \omega_1)$ of
        conditions such that $p_{\alpha +1} \leq_{P_1} p_\alpha$ and
        $p_{\alpha +1} \in D_\alpha$ for all $\alpha < \omega_1$.
        Moreover we are going to construct these conditions such that
        we inductively maintain
        $p_\alpha \in \pi^{-1}_\alpha(P_1) \subset N_{\alpha}$.

        We start with $p_0 = p \in N_{0}$. For the successor step
        suppose that we already defined
        $p_\alpha \in \pi^{-1}_\alpha(P_1) \subset N_{\alpha}$ for
        some $\alpha < \omega_1$. Then we have that
        $\dom(p_\alpha) < \xi_\alpha$ and $p_\alpha \in N_{\alpha +1}$
        by the definition of the models and embeddings
        $(N_{\alpha}, \pi_{\alpha} \st \alpha \leq \omega_1)$.  By
        extendibility of the forcing $\pi_{\alpha +1}^{-1}(P_1)$ and
        the density of the set
        $\pi_{\alpha +1}^{-1}(D_\alpha) \subseteq D_\alpha$, there
        exists a condition $p_{\alpha +1} \leq_{P_1} p_\alpha$ such
        that we have
        $p_{\alpha +1} \in \pi^{-1}_{\alpha+1}(P_1) \subset N_{\alpha
          +1}$,
        $p_{\alpha +1} \in D_{\alpha}$ and
        $\dom(p_{\alpha +1}) \geq \xi_\alpha$ since
        $\pi_{\alpha +1}^{-1}(P_1) \subseteq P_1$.

        For a limit ordinal $\lambda \leq \omega_1$ we simply let
        $p_\lambda = \bigcup_{\alpha < \lambda} p_\alpha$. 
We have that
      the sequence
        $(\xi_\alpha \st \alpha < \lambda)$ of critical points of
        $(\pi_{\alpha} \st \alpha < \lambda)$ is definable over
        $N_{\lambda}$ since for $\alpha < \lambda$ the model
        $N_{\alpha}$ is equal to the transitive collapse of a
        $\Sigma_{k}$-elementary submodel of $N_{\lambda}$ which is
        constructed inside $N_{\lambda}$ exactly as it was constructed
        inside $\Lpoddn(A)$ above. Therefore we have that
        \[ \cf^{N_\lambda}(\xi_\lambda) \leq \lambda
        \leq \omega_1 = \omega_1^{V_1} \]
       which implies
        that
        \[ N_\lambda \vDash \card{\xi_\lambda} \leq
        \omega_1^{V_1}. \] 
As $\dom(p_\lambda)=\xi_\lambda$, this buys us that
$p_\lambda$ is in fact a condition in the forcing $P_1$.

Now consider the function $q = p_{\omega_1}$. Then 
        $q \in P_1$ and $q \in D_\beta$ for all $\beta<\omega_1$.

        We have shown that the reshaping forcing $P_1$ is
        $< (\gamma^+)^{M_x}$-distributive and therefore does not
        collapse $\omega_1$ and $(\gamma^+)^{M_x} = \omega_2$.

        So let $G_1$ be $P_1$-generic over $V_1$ and let
        $V_2 = V_1[G_1]$. The extendability of the forcing $P_1$
        yields that $\bigcup G_1$ is an
        $(A, (\gamma^+)^{M_x})$-reshaping function with domain
        $(\gamma^+)^{M_x}$. Let $B^\prime$ be a subset of
        $(\gamma^+)^{M_x}$ which codes the function $\bigcup G_1$, for
        example the subset of $(\gamma^+)^{M_x}$ which has
        $\bigcup G_1$ as its characteristic function. Finally let
        $B \subset (\gamma^+)^{M_x}$ be a code for
        $A \oplus B^\prime$.

        As at the end of Step $1$ we can pick this code
        $B \subset (\gamma^+)^{M_x}$ such that the model $V_2$ is of
        the form $\Lpoddn(B)$ by the following argument: Recall that
        \[ \Lpoddn(B) = M(B) | \omega_1^V, \] where as above $M(B)$
        denotes the $\omega_1^V$-th iterate of $M_{2n-1}^\#(B)$ by its
        least measure and its images. Therefore we can consider $G_1$
        as being generic over $M(A)$. This yields as in the argument
        at the end of Step $1$ that we can pick $B$ such that
        $V_2 = M(B) | \omega_1^V$ because the ``reshaping forcing''
        $P_1$ takes place below $(\gamma^+)^{M_x} <
        \omega_1^V$. Therefore we have that
        \[ V_2 = \Lpoddn(B). \]

        \textbf{Step 3:} Now we can perform the first coding using
        almost disjoint subsets of
        $\omega_1 = \omega_1^{V_2} = \omega_1^{V_1}$. Since $B$ is
        ``reshaped'' we can inductively construct a sequence of almost
        disjoint subsets of $\omega_1$,
        \[ (A_\xi \st \xi < (\gamma^+)^{M_x}), \]
        as follows. Let $\xi < (\gamma^+)^{M_x}$ be such that we
        already constructed a sequence $(A_\zeta \st \zeta < \xi)$ of
        almost disjoint subsets of $\omega_1$. 

        \begin{minipage}{\textwidth}
          \bigskip \textbf{Case 1.} $L[B \cap \xi] \vDash \card{\xi}
          \leq \omega_1^{V_2}$. 
          \bigskip
        \end{minipage}

        Then we let $A_\xi$ be the least subset of $\omega_1$ in
        $L[B \cap \xi]$ which is almost disjoint from any
        $A_\zeta$ for $\zeta < \xi$ and which satisfies that
        \[ \card{\omega_1 \setminus \bigcup_{\zeta \leq \xi} A_\zeta}
          = \aleph_1. \] 

        \begin{minipage}{\textwidth}
          \bigskip \textbf{Case 2.} Otherwise.
          \bigskip
        \end{minipage}

        Let $N$ be the least initial segment of
        $\Lpoddn(A \cap \xi)^{\Lpoddn(B)}$ such that
        $\rho_\omega(N) \leq \xi$, $N$ is sound above $\xi$, $\xi$ is
        the largest cardinal in $N$, and definably over $N$ there
        exists a surjection
        $g: \omega_1^{V_2} \twoheadrightarrow \xi$. Now let $A_\xi$ be
        the least subset of $\omega_1^{V_2}$ which is definable over
        $N$, almost disjoint from any $A_\zeta$ for $\zeta < \xi$ and
        which satisfies that
        $\card{\omega_1 \setminus \bigcup_{\zeta \leq \xi} A_\zeta} =
        \aleph_1.$

        That the set $A_\xi$ is well-defined in this case follows from
        the fact that the set $B \subset (\gamma^+)^{M_x}$ is
        ``reshaped'' by the following argument. As $B$ is ``reshaped''
        we have in Case $2$ above that there exists a model $N$ as in
        Definition \ref{def:reshaping}
        $(ii)$. 
        We have that $N \lhd \Lpoddn(A \cap \xi)$. In general it need
        not be the case that $\Lpoddn(A \cap \xi)^{\Lpoddn(B)}$ is
        equal to $\Lpoddn(A \cap \xi)$ (see Lemma
        \ref{lem:LpInLpFake}), but as $\xi$ is the largest cardinal in
        $N$, it follows that in fact
        $N \lhd \Lpoddn(A \cap \xi)^{\Lpoddn(B)}$.  Hence any $N$
        witnessing at $\xi$ that $B$ is ``reshaped'' is an initial
        segment of $\Lpoddn(A \cap \xi)^{\Lpoddn(B)}$, so that the set
        $A_\xi$ is indeed well-defined.

     The sequence $(A_\xi \st \xi < (\gamma^+)^{M_x})$ is now
        definable in $V_2 = \Lpoddn(B)$.

        Now let $P_2$ be the forcing for coding $B$ by a subset of
        $\omega_1$ using the almost disjoint sets
        $(A_\xi \st \xi < (\gamma^+)^{M_x})$. That means a condition
        $p \in P_2$ is a pair $(p_l, p_r)$ such that
        $p_l : \alpha \rightarrow 2$ for some
        $\alpha < \omega_1$ and $p_r$ is a countable subset of
        $(\gamma^+)^{M_x}$. We say
        $p = (p_l, p_r) \leq_{P_2} (q_l, q_r) = q$ iff
        $q_l \subseteq p_l$, $q_r \subseteq p_r$, and for all
        $\xi \in q_r$ we have that if $\xi \in B$, then 
        \[ \{ \beta \in \dom(p_l) \setminus \dom(q_l) \st p_l(\beta) =
        1 \} \cap A_\xi = \emptyset. \]

        An easy argument shows that the $(\gamma^+)^{M_x}$-c.c. holds
        true for the forcing $P_2$. Moreover it is $\omega$-closed and
        therefore no cardinals are collapsed.
        
        Let $G_2$ be $P_2$-generic over $V_2$ and let
        \[ C^\prime = \bigcup_{p \in G_2} \{ \beta \in \dom(p_l) \st
        p_l(\beta) = 1 \}. \]
        Then $C^\prime \subset \omega_1$ and we have that for all
        $\xi < (\gamma^+)^{M_x}$,
        \[ \xi \in B \text{ iff } \card{C^\prime \cap A_\xi} \leq
        \aleph_0. \]
        Finally let $V_3 = V_2[G_2]$. By the same argument as we gave
        at the end of Step $2$ we can obtain that
        \[ V_3 = \Lpoddn(C) \] for some set $C \subset \omega_1$
        coding $C^\prime$ and the real $x$, as the model $\Lpoddn(C)$
        can successfully decode the set $B \subset (\gamma^+)^{M_x}$
        by the following argument. We show inductively that for every
        $\xi < (\gamma^+)^{M_x}$,
        $(A_\zeta \st \zeta < \xi) \in \Lpoddn(C)$ and
        $B \cap \xi \in \Lpoddn(C)$. This yields that
        $B \in \Lpoddn(C)$.

        For the inductive step let $\xi < (\gamma^+)^{M_x}$ be an
        ordinal and assume inductively that we have
        \[ (A_\zeta \st \zeta < \xi) \in \Lpoddn(C). \]
        Since for all $\zeta < \xi$,
        \[ \zeta \in B \text{ iff } \card{C^\prime \cap A_\zeta} \leq
          \aleph_0, \] we have that $B \cap \xi \in \Lpoddn(C)$.

        In Case $1$, i.e. if
        $L[B \cap \xi] \vDash \card{\xi} \leq \omega_1^{V_2}$, the set
        $A_\xi$ can easily be identified inside $\Lpoddn(C)$. In Case
        $2$ let $N$ be the least initial segment of
        $\Lpoddn(A \cap \xi)^{\Lpoddn(C)}$ such that
        $\rho_\omega(N) \leq \xi$, $N$ is sound above $\xi$, $\xi$ is
        the largest cardinal in $N$, and definably over $N$ there
        exists a surjection $g: \omega_1 \twoheadrightarrow \xi$.
        Such an $N$ exists due to the fact that $B$ is ``reshaped'':
        there is even some $N \triangleleft \Lpoddn(A \cap \xi)$ such
        that $\rho_\omega(N) \leq \xi$, $N$ is sound above $\xi$,
        $\xi$ is the largest cardinal in $N$, and definably over $N$
        there exists a surjection
        $g: \omega_1 \twoheadrightarrow \xi$; and
        $\Lpoddn(A \cap \xi)^{\Lpoddn(C)} \subseteq \Lpoddn(A \cap
        \xi)$.  Thus the least initial segment $N$ of
        $\Lpoddn(A \cap \xi)^{\Lpoddn(C)}$ with these properties exists
        and it will also be an initial segment of
        $\Lpoddn(A \cap \xi)$, and $N$ will then also be an initial
        segment of $\Lpoddn(A \cap \xi)^{\Lpoddn(B)}$ (for the $B$ as
        above) which was used to identify $A_\xi$.  We have shown that
        in each case $A_\xi$ can be identified inside $\Lpoddn(C)$.

        As the identification of the next $A_\xi$ was according to a
        uniform procedure, we get that in fact
        \[ (A_\zeta \st \zeta \leq \xi ) \in \Lpoddn(C). \]

        \textbf{Step 4:} Before we can ``code down to a real'', that
        means before we can find a real $z$ such that
        $K^{M_x}|(\gamma^+)^{M_x} \in \Lpoddn(z)$, we have to perform
        another ``reshaping'' similar to the one in Step $2$. So let
        $P_3$ be the forcing for adding a $(C, \omega_1)$-reshaping
        function working in $V_3$ as the new ground model, where
        $\omega_1 = \omega_1^{V_3} = \omega_1^{V_2}$. That means we
        let $p \in P_3$ iff $p$ is a $(C, \omega_1)$-reshaping
        function with $\dom(p) < \omega_1$. The order of two
        conditions $p$ and $q$ in $P_3$ is again by reverse inclusion,
        that means $p \leq_{P_3} q$ iff $q \subseteq p$.

        The forcing $P_3$ is extendable and $<\omega_1$-distributive
        by the same arguments as we gave in Step $2$ since we have
        that $V_3 = \Lpoddn(C)$. Therefore $P_3$ does not collapse
        $\omega_1$.

        Let $G_3$ be $P_3$-generic over $V_3$ and let
        $V_4 = V_3[G_3]$. We again have that $\bigcup G_3$ is a
        $(C, \omega_1)$-reshaping function with domain $\omega_1$,
        because $P_3$ is extendable. Let $D^\prime$ be a subset of
        $\omega_1$ which codes $\bigcup G_3$, for example the subset
        of $\omega_1$ which has $\bigcup G_3$ as its characteristic
        function. Finally let $D \subset \omega_1$ code
        $C \oplus D^\prime$.

        By the same argument as the one we gave at the end of Step $2$
        we can obtain that in fact \[ V_4 = \Lpoddn(D). \]

        \textbf{Step 5:} Now we are ready to finally ``code down to a
        real''. Since $D$ is ``reshaped'' we can consider a uniformly
        defined sequence \[ (B_\xi \st \xi < \omega_1) \] of almost
        disjoint subsets of $\omega$ as in Step $3$, where
        $\omega_1 = \omega_1^{V_4} = \omega_1^{V_3}$.

        Now we let $P_4$ be the forcing for coding $D$ by a subset of
        $\omega$ using the almost disjoint sets
        $(B_\xi \st \xi < \omega_1)$. That means a condition
        $p \in P_4$ is a pair $(p_l, p_r)$ such that
        $p_l : \alpha \rightarrow 2$ for some $\alpha < \omega$ and
        $p_r$ is a finite subset of $\omega_1$. We say
        $p = (p_l, p_r) \leq_{P_4} (q_l, q_r) = q$ iff
        $q_l \subseteq p_l$, $q_r \subseteq p_r$, and for all
        $\xi \in q_r$ we have that if $\xi \in D$, then
        \[ \{ \beta \in \dom(p_l) \setminus \dom(q_l) \st p_l(\beta) =
        1 \} \cap B_\xi = \emptyset. \]

        As in Step $3$ above an easy argument shows that the forcing
        $P_4$ has the c.c.c. and therefore no cardinals are collapsed.
        
        Finally let $G_4$ be $P_4$-generic over $V_4$ and let
        \[ E^\prime = \bigcup_{p \in G_4} \{ \beta \in \dom(p_l) \st
        p_l(\beta) = 1 \}. \]
        Then $E^\prime \subset \omega$ and we have that for all
        $\xi < \omega_1$,
        \[ \xi \in D \text{ iff } \card{E^\prime \cap B_\xi} <
        \aleph_0. \]
        Let $V_5 = V_4[G_4]$ and finally let $z$ be a real coding
        $E^\prime$ and the real $x$. Analogous to the arguments given
        at the end of Step $3$ we can pick the real $z \geq_T x$ such
        that we have
        \[ V_5 = \Lpoddn(z) \]
        and the model $\Lpoddn(z)$ is able to successfully decode the
        set $D$ and thereby also the set $A$.

        This ultimately yields that we have a real $z \geq_T x$ such
        that
        \[ (\gamma^+)^{M_x} = \omega_2^{\Lpoddn(z)} \] and
        \[ K^{M_x}|(\gamma^+)^{M_x} \in \Lpoddn(z). \]
    \end{proof}

    \setcounter{subcl}{2}

    \begin{subclaim}\label{subcl:KMxit}
      $K^{M_x} | (\gamma^+)^{M_x}$ is fully iterable inside
      $\Lpoddn(z)$.
    \end{subclaim}
    \begin{proof}
      It is enough to show that the $2n$-small premouse
      $K^{M_x} | (\gamma^+)^{M_x}$ is $\omega_1$-iterable inside
      $\Lpoddn(z)$ because once we showed this, an absoluteness
      argument, as for example similar to the one we already gave in
      the proof of Lemma \ref{lem:nsuitnew}, yields that
      $K^{M_x} | (\gamma^+)^{M_x}$ is in fact fully iterable inside
      $\Lpoddn(z)$ since the iteration strategy for
      $K^{M_x} | (\gamma^+)^{M_x}$ is given by $\Q$-structures
      $\Q(\T)$ for iteration trees $\T$ on
      $K^{M_x} | (\gamma^+)^{M_x}$ which are $(2n-1)$-small above
      $\delta(\T)$ and such $\Q$-structures $\Q(\T)$ are contained in
      every lower part model at the level $2n-1$, as for example
      $\Lpoddn(z)$, by definition of the lower part model.

      In a first step we show that there is a tree $T$ such that
      $p[T]$ is a universal $\Pi^1_{2n}$-set in
      $\Lpoddn(x, K^{M_x} | (\gamma^+)^{M_x})$ and moreover for every
      forcing $\mathbb{P}$ of size at most $(\gamma^{+})^{M_x}$ in
      $\Lpoddn(x, K^{M_x} | (\gamma^+)^{M_x})$ and every
      $\mathbb{P}$-generic $G$ over
      $\Lpoddn(x, K^{M_x} | (\gamma^+)^{M_x})$ we have that
      \[ \Lpoddn(x, K^{M_x} | (\gamma^+)^{M_x})[G] \vDash
      \text{``}p[T] \text{ is a universal }
      \Pi^1_{2n}\text{-set''.} \]
      Let $\varphi$ be a $\Pi^1_{2n}$-formula defining a universal
      $\Pi^1_{2n}$-set, i.e. $\{ y \in \BS \st \varphi(y) \}$ is a
      universal $\Pi^1_{2n}$-set. Then we let
      $T \in \Lpoddn(x, K^{M_x} | (\gamma^+)^{M_x})$ be a tree of
      height $\omega$ searching for $y, H, \M, \sigma, \mathbb{Q}$ and
      $g$ such that
    \begin{enumerate}[$(1)$]
    \item $y \in \BS$,
    \item $\M$ is a countable $(x,H)$-premouse,
    \item
      $\sigma: \M \rightarrow \Lpoddn(x, K^{M_x} | (\gamma^+)^{M_x})$
      is a sufficiently elementary embedding,
    \item $\sigma(H) = K^{M_x} | (\gamma^+)^{M_x}$, and
    \item $\mathbb{Q} \in \M$ is a partial order of size at most
      $H \cap \Ord$ in $\M$ and $g$ is $\mathbb{Q}$-generic over $\M$
      such that
      \[ \M[g] \vDash \varphi(y). \]
    \end{enumerate}
    This tree $T$ has the properties we claimed above by the following
    argument. The model $\Lpoddn(x, K^{M_x} | (\gamma^+)^{M_x})$ is
    closed under the operation $a \mapsto M_{2n-2}^\#(a)$ and is
    therefore $\SIGMA^1_{2n}$-correct in $V$. The same holds for the
    model $\Lpoddn(x, K^{M_x} | (\gamma^+)^{M_x})[G]$, where $G$ is
    $\mathbb{P}$-generic over the model
    $\Lpoddn(x, K^{M_x} | (\gamma^+)^{M_x})$ for some forcing
    $\mathbb{P}$ of size at most $(\gamma^{+})^{M_x}$ in the model
    $\Lpoddn(x, K^{M_x} | (\gamma^+)^{M_x})$. Moreover we have that if
    $\M, H, \sigma, \mathbb{Q}$ and $g$ are as searched by the tree
    $T$, then the forcing $\mathbb{Q} \in \M$ has size at most
    $H \cap \Ord$ in $\M$ and the embedding
    $\sigma: \M \rightarrow \Lpoddn(x, K^{M_x} | (\gamma^+)^{M_x})$ is
    sufficiently elementary, so it follows that $\M[g]$ is
    $\SIGMA^1_{2n}$-correct in $V$. This easily yields that $p[T]$ is
    a universal $\Pi^1_{2n}$-set in
    $\Lpoddn(x, K^{M_x} | (\gamma^+)^{M_x})[G]$ for every generic set
    $G$ as above, in fact
    \begin{align*}
      p[T] \cap \Lpoddn(x, K^{M_x} &| (\gamma^+)^{M_x})[G] = \\ & \{ y \in
    \BS \st \varphi(y) \} \cap \Lpoddn(x, K^{M_x} |
    (\gamma^+)^{M_x})[G]. 
    \end{align*}

    In a second step we now define another tree $U$ whose
    well-foundedness is going to witness that
    $K^{M_x} | (\gamma^+)^{M_x}$ is $\omega_1$-iterable. So we define
    the tree $U \in \Lpoddn(x, K^{M_x} | (\gamma^+)^{M_x})$ such that
    $U$ is searching for $\bar{K}, \pi, \T, \N, \sigma$ and a sequence
    $(\Q_\lambda \st \lambda \leq \lh(\T), \lambda \text{ limit})$
    with the following properties.
    \begin{enumerate}[$(1)$]
    \item $\bar{K}$ is a countable premouse,
    \item $\pi: \bar{K} \rightarrow K^{M_x} | (\gamma^+)^{M_x}$ is an
      elementary embedding,
    \item $\T$ is a countable putative\footnote{Recall that we say
        that a tree $\T$ is a \emph{putative iteration tree} if $\T$
        satisfies all properties of an iteration tree, but we allow
        the last model of $\T$ to be ill-founded, in case $\T$ has a
        last model.} iteration tree on $\bar{K}$ such that for all
      limit ordinals $\lambda < \lh(\T)$,
      \[ \Q_\lambda \unlhd \M_\lambda^\T, \]
    \item for all limit ordinals $\lambda \leq \lh(\T)$, $\Q_\lambda$
      is a $\Pi^1_{2n}$-iterable (above
      $\delta(\T \upharpoonright \lambda)$) $\Q$-structure for
      $\T \upharpoonright \lambda$ and $\Q_\lambda$ is $(2n-1)$-small
      above $\delta(\T \upharpoonright \lambda)$ (where this
      $\Pi^1_{2n}$-statement is witnessed using the tree $T$ defined
      above),
    \item $\N$ is a countable model of $\ZFC^-$ such that either
      \[ \N \vDash \text{``} \T \text{ has a last ill-founded
        model''}, \]
      or else \begin{eqnarray*} \N \vDash & \text{``} \lh(\T) \text{
          is a limit ordinal and there is no cofinal branch }b \\ &
        \text{ through }\T \text{ such that } \Q_{\lh(\T)} \unlhd
        \M_b^\T \text{''}, \end{eqnarray*} and
    \item $\sigma: \N \cap \Ord \hookrightarrow \omega_1$.
    \end{enumerate}
    Here a code for the sequence
    $(\Q_\lambda \st \lambda \leq \lh(\T), \lambda \text{ limit})$ of
    $\Q$-structures satisfying property $(4)$ can be read off from
    $p[T]$.

    Recall that we have that the core model $K^{M_x}$ is $2n$-small
    and thereby fully iterable in the model $M_x$ via an iteration
    strategy which is guided by $\Q$-structures $\Q(\T)$ for iteration
    trees $\T$ on $K^{M_x}$ of limit length which are $(2n-1)$-small
    above $\delta(\T)$. This implies that the premouse
    $K^{M_x} | (\gamma^+)^{M_x}$ is iterable inside the model
    $\Lpoddn(x, K^{M_x} | (\gamma^+)^{M_x})$ (which as argued earlier
    is a definable subset of $M_x$) by an argument as in the proof of
    Lemma \ref{lem:iterability}, because $K^{M_x} | (\gamma^+)^{M_x}$
    is a $2n$-small premouse and the $\omega_1$-iterable
    $\Q$-structures $\Q(\T)$ for iteration trees $\T$ on
    $K^{M_x} | (\gamma^+)^{M_x}$ in
    $\Lpoddn(x,K^{M_x} | (\gamma^+)^{M_x})$ are contained in the model
    $\Lpoddn(x,K^{M_x} | (\gamma^+)^{M_x})$ as they are $(2n-1)$-small
    above $\delta(\T)$.

    We aim to show that the tree $U$ defined above is well-founded
    inside the model $\Lpoddn(x, K^{M_x} | (\gamma^+)^{M_x})$. So
    assume toward a contradiction that $U$ is ill-founded in
    $\Lpoddn(x, K^{M_x} | (\gamma^+)^{M_x})$ and let
    $\bar{K}, \pi, \T, \N, \sigma$ and
    $(\Q_\lambda \st \lambda \leq \lh(\T), \lambda \text{ limit})$ be
    as above satisfying properties $(1)-(6)$ in the definition of the
    tree $U$.

    Assume that the iteration tree $\T$ has limit length since the
    other case is easier. Since as argued above the premouse
    $K^{M_x} | (\gamma^+)^{M_x}$ is countably iterable inside
    $\Lpoddn(x, K^{M_x} | (\gamma^+)^{M_x})$, we have that in
    $\Lpoddn(x, K^{M_x} | (\gamma^+)^{M_x})$ there exists a cofinal
    well-founded branch $b$ through the iteration tree $\T$ on
    $\bar{K}$ and $\omega_1$-iterable $\Q$-structures
    $\Q(\T \upharpoonright \lambda)$ for all limit ordinals
    $\lambda \leq \lh(\T)$ such that we have $\Q(\T) = \Q(b, \T)$ by
    an argument as in Case $1$ in the proof of Lemma
    \ref{lem:nsuitnew}. Moreover every $\Q$-structure
    $\Q(\T \upharpoonright \lambda)$ is $(2n-1)$-small above
    $\delta(\T \upharpoonright \lambda)$.

    Since the model $\Lpoddn(x, K^{M_x} | (\gamma^+)^{M_x})$ is closed
    under the operation $a \mapsto M_{2n-2}^\#(a)$, the proofs of
    Lemmas \ref{lem:iterability} and \ref{lem:steel2.2} imply that we
    can successfully compare $\Q(\T \upharpoonright \lambda)$ and
    $\Q_\lambda$ for all limit ordinals $\lambda \leq \lh(\T)$. This
    implies that in fact
    \[ \Q(\T \upharpoonright \lambda) = \Q_\lambda. \]
    Therefore it follows by an absoluteness argument as the one we
    already gave in the proof of Lemma \ref{lem:iterability} that
    \begin{eqnarray*} \N \vDash & \text{``there exists a cofinal
        branch }b\text{ through }\T \text{ such that } \\ &
      \Q_{\lh(\T)} \unlhd \M_b^\T \text{''}, \end{eqnarray*} because
    we in particular have that
    $\Q_{\lh(\T)} = \Q(b,\T) \unlhd \M_b^\T$.

    Therefore the tree $U$ is well-founded in the model
    $\Lpoddn(x, K^{M_x} | (\gamma^+)^{M_x})$ and by absoluteness of
    well-foundedness this implies that the tree $U$ is also
    well-founded inside $\Lpoddn(z)$.

    We have by construction of the tree $T$ that a code for the
    sequence
    $(\Q_\lambda \st \lambda \leq \lh(\T), \lambda \text{ limit})$ can
    still be read off from $p[T]$ in $\Lpoddn(z)$ because the forcing
    we performed over the ground model
    $\Lpoddn(x, K^{M_x} | (\gamma^+)^{M_x})$ to obtain the real $z$
    has size at most $(\gamma^{+})^{M_x}$ in
    $\Lpoddn(x, K^{M_x} | (\gamma^+)^{M_x})$. Therefore the
    well-foundedness of the tree $U$ in $\Lpoddn(z)$ implies that
    $K^{M_x} | (\gamma^+)^{M_x}$ is $\omega_1$-iterable in
    $\Lpoddn(z)$.
    \end{proof}

    Following \cite{MSch04}, we now aim to make sense of
    \[ (K^c)^{\Lpoddn(z)}. \] The issue is that the definition of \emph{certified} in the construction of the
      model $K^c$ in the sense of \cite{MSch04} makes reference to
      some fixed class of ordinals $A$ such that $V = L[A]$. Therefore the
      model $K^c$ constructed in this sense is in general not
      contained in $\HOD$, because whether an extender is certified in
      the sense of Definition $1.6$ in \cite{MSch04} may depend on the
      choice of $A$. The following claim is supposed to express the
      fact that we may choose $A$ with $\Lpoddn(z) \models$ ``$V=L[A]$''
in such a way that the resulting $K^c$ is a subclass of $\HOD$ from the
point of view of $\Lpoddn(z)$. 


    \begin{subclaim}\label{subcl:KcHOD}
      $(K^c)^{\Lpoddn(z)} \subset \HOD^{\Lpoddn(z)}$.
    \end{subclaim}

    \begin{proof}
%
      Let $z^\prime \in \Lpoddn(z)$ be an arbitrary real with
      $z \leq_T z^\prime$. In particular,
      $\Lpoddn(z) = \Lpoddn(z^\prime)$ by Lemma \ref{lem:LpXY}.
Let us write
      $A_{z^\prime} = z^\prime {}^\frown E_{z^\prime}$,
      where $E_{z^\prime}$ codes the extender sequence of
      $\Lpoddn(z^\prime)$. We may construe $A_{z^\prime}$ as a set of
ordinals in a canonical fashion.

      \begin{claim}
Let $z^\prime \in \Lpoddn(z)$ be a real with
      $z \leq_T z^\prime$.
        An extender $E$ is certified with respect to $A_z$ iff $E$
        is certified with respect to $A_{z^\prime}$, in the sense
        of Definition $1.6$ in \cite{MSch04}.
      \end{claim}

      \begin{proof}
        It follows from Lemma \ref{lem:LpXY} that
        \[ \Lpoddn(z^\prime) = (\Lpoddn(z^\prime))^{\Lpoddn(z)} \] and
        we even have that the set
        $A_{z^\prime} \cap \omega_1^{\Lpoddn(z)}$ may be easily computed from
        $A_z \cap \omega_1^{\Lpoddn(z)}$. Symmetrically, we have that the
        set $A_{z} \cap \omega_1^{\Lpoddn(z)}$ may be easily computed from
        $A_{z^\prime} \cap \omega_1^{\Lpoddn(z)}$. Moreover the
        extender sequences of $\Lpoddn(z)$ and $\Lpoddn(z^\prime)$
literally agree above
        $\omega_1^{\Lpoddn(z)}$. Therefore it follows easily that
        an extender $E$ is certified with respect to $A_z$ iff $E$
        is certified with respect to $A_{z^\prime}$.
      \end{proof}

We have shown that inside
      ${\Lpoddn(z)}$ there is a cone of reals $z^\prime$ 
such that the output of the $K^c$ construction performed \`a la \cite{MSch04} and
using $A_{z^\prime}$
is not sensitive to the particular choice of $z^\prime$ from that cone.
Hence if we let \[ (K^c)^{\Lpoddn(z)} \] be this common output, then
     \[ (K^c)^{\Lpoddn(z)} \subset \HOD^{\Lpoddn(z)}, \] as desired.
    \end{proof}

    In what follows we will also need the following notion of
    iterability. 

    \begin{definition}\label{def:itstrLambda}
      Let $N$ be a countable premouse. Then we inductively define an
      iteration strategy $\Lambda$ for $N$ as follows. Assume that
      $\T$ is a normal iteration tree on $N$ of limit length according
      to $\Lambda$ such that there exists a $\Q$-structure
      $\Q \unrhd \M(\T)$ for $\T$ which is fully iterable above
      $\delta(\T)$. Then we define that $\Lambda(\T) = b$ iff $b$ is a
      cofinal branch through $\T$ such that either
      \begin{enumerate}[$(i)$]
      \item $\Q \unlhd \M_b^\T$, or
      \item $b$ does not drop (so in particular the iteration
        embedding $i_b^\T$ exists), there exists an ordinal
        $\delta < N \cap \Ord$ such that $i_b^\T(\delta) = \delta(\T)$
        and \[ N \vDash \text{``} \delta \text{ is Woodin''}, \]
        and there exists a $\tilde{\Q} \unrhd N | \delta$ such that
        \[ \tilde{\Q} \vDash \text{``} \delta \text{ is Woodin''}, \]
        but $\delta$ is not definably Woodin over $\tilde{\Q}$ and if
        we lift the iteration tree $\T$ to $\tilde{\Q}$, call this
        iteration tree $\T^*$, then
        \[ i_b^{\T^*} : \tilde{\Q} \rightarrow \Q. \]
      \end{enumerate}
    \end{definition}

    \begin{definition}\label{def:Qstrit}
      Let $N$ be a countable premouse. Then we say that $N$ is
      \emph{$\Q$-structure iterable} \index{Qstructure
        iterable@$\Q$-structure iterable} iff for every iteration tree
      $\T$ on $N$ which is according to the iteration strategy
      $\Lambda$ from Definition \ref{def:itstrLambda} the following
      holds.
      \begin{enumerate}[$(i)$]
      \item If $\T$ has limit length and there exists a $\Q$-structure
        $\Q \unrhd \M(\T)$ for $\T$ which is fully iterable above
        $\delta(\T)$, then there exists a cofinal well-founded branch
        $b$ through $\T$ such that $\Lambda(\T) = b$.
      \item If $\T$ has a last model, then every putative iteration
        tree $\U$ extending $\T$ such that $\lh(\U) = \lh(\T) + 1$ has
        a well-founded last model.
      \end{enumerate}
    \end{definition} 
   
    The premouse
    \[ (K^c)^{\Lpoddn(z)} \]
    is countably iterable in $\Lpoddn(z)$ by the iterability proof of
    Chapter $9$ in \cite{St96} adapted as in Section $2$ in
    \cite{MSch04}.

    \setcounter{subcl}{4}
    \begin{subclaim}\label{subcl:QstritA}
      In $\Lpoddn(z)$,
      \[ (K^c)^{\Lpoddn(z)} \text{ is } \Q\text{-structure
        iterable.} \]
    \end{subclaim}

    \begin{proof}
      Assume there exists an iteration tree $\T$ on
      $(K^c)^{\Lpoddn(z)}$ which witnesses that $(K^c)^{\Lpoddn(z)}$
      is not $\Q$-structure iterable inside $\Lpoddn(z)$. Since the
      other case is easier assume that $\T$ has limit length. That
      means in particular that there exists a $\Q$-structure
      $\Q(\T) \unrhd \M(\T)$ for $\T$ which is fully iterable above
      $\delta(\T)$, but there is no cofinal well-founded branch $b$
      through $\T$ in $\Lpoddn(z)$ such that $\Lambda(\T) = b$.

      For some large enough natural number $m$ let $H$ be the
      Mostowski collapse of $Hull_m^{\Lpoddn(z)}$ such that $H$ is
      sound and $\rho_\omega(H) = \rho_{m+1}(H)= \omega$. Furthermore
      let
      \[ \pi : H \rightarrow \Lpoddn(z) \]
      be the uncollapse map such that
      $\T, \Q(\T) \in \ran(\pi)$.
      Moreover let $\bar{\T}, \bar{\Q} \in H$ be such that
      $\pi(\bar{\T}) = \T$, $\pi(\bar{\Q}) = \Q(\T)$, and let
      $\bar{K} \subset H$ be such that
      $\pi(\bar{K} | \gamma) = (K^c)^{\Lpoddn(z)} | \pi(\gamma)$ for
      any $\gamma < H \cap \Ord$. That means in particular that
      $\bar{\T}$ is an iteration tree on $\bar{K}$.

      As argued above we have that $(K^c)^{\Lpoddn(z)}$ is
      countably iterable in $\Lpoddn(z)$. Therefore there exists a
      cofinal well-founded branch $\bar{b}$ through $\bar{\T}$ in
      $\Lpoddn(z)$.

      Now we consider two different cases. 

      \begin{minipage}{\textwidth}
        \bigskip \textbf{Case 1.} There is a drop along the branch
        $\bar{b}$. \bigskip
      \end{minipage}

      In this case there exists a $\Q$-structure
      $\Q^* \unlhd \M_{\bar{b}}^{\bar{\T}}$ for $\bar{\T}$. A standard
      comparison argument shows that $\Q^* = \bar{\Q}$ and thus
      $\bar{\Q} \unlhd \M_{\bar{b}}^{\bar{\T}}$ is a $\Q$-structure
      for $\bar{\T}$.

      Now consider the statement
      \begin{eqnarray*}
      \phi(\bar{\T}, \bar{\Q}) \equiv &\text{``there is a cofinal
                                            branch } b \text{ through } \bar{\T} \text{ such
                                            that } \\ 
                                          & \bar{\Q} \unlhd
                                            \M_{b}^{\bar{\T}}\text{''.}   
    \end{eqnarray*}
    This statement $\phi(\bar{\T}, \bar{\Q})$ is
    $\Sigma^1_1$-definable from the parameters $\bar{\T}$ and
    $\bar{\Q}$ and holds in the model $\Lpoddn(z)$ as witnessed by the
    branch $\bar{b}$.

    By $\SIGMA^1_1$-absoluteness the statement
    $\phi(\bar{\T}, \bar{\Q})$ also holds in the model
    $H^{\Col(\omega, \eta)}$ as witnessed by some branch $b$,
    where $\eta < H \cap \Ord$ is a large enough ordinal such that
    $\bar{\T}, \bar{\Q} \in H^{\Col(\omega, \eta)}$ are countable
    inside $H^{\Col(\omega, \eta)}$. 

    Since $b$ is uniquely definable from $\bar{\T}$ and the
    $\Q$-structure $\bar{\Q}$, and $\bar{\T}, \bar{\Q} \in H$, we have
    by homogeneity of the forcing $\Col(\omega, \eta)$ that actually
    $b \in H$.  This contradicts the fact that the iteration tree $\T$
    witnesses that the premouse $(K^c)^{\Lpoddn(z)}$ is not
    $\Q$-structure iterable inside $\Lpoddn(z)$.

    \begin{minipage}{\textwidth}
      \bigskip \textbf{Case 2.} There is no drop along the branch
      $\bar{b}$. \bigskip
    \end{minipage}    

    In this case we consider two subcases as follows.

 \begin{minipage}{\textwidth}
   \bigskip \textbf{Case 2.1.} There is no Woodin cardinal
   $\bar{\delta}$ in $\bar{K}$ such that
   $i_{\bar{b}}^{\bar{\T}}(\bar{\delta}) = \delta(\bar{\T})$. \bigskip
 \end{minipage}

 In this case we have that $\delta(\bar{\T})$ is not Woodin in
 $\M_{\bar{b}}^{\bar{\T}}$. Thus there exists an initial segment
 $\Q^* \unlhd \M_{\bar{b}}^{\bar{\T}}$ such that
 \[ \Q^* \vDash \text{``} \delta(\bar{\T}) \text{ is a Woodin
   cardinal''}, \]
 and definably over $\Q^*$ there exists a witness for the fact that
 $\delta(\bar{\T})$ is not Woodin (in the sense of Definition
 \ref{def:notdefWdn}). As $(K^c)^{\Lpoddn(z)}$ is countably iterable
 in $\Lpoddn(z)$ we can successfully coiterate the premice $\Q^*$ and
 $\bar{\Q}$. Therefore we have that in fact $\Q^* = \bar{\Q}$ and thus
 as in Case $1$ it follows that
 $\bar{\Q} \unlhd \M_{\bar{b}}^{\bar{\T}}$ is a $\Q$-structure for
 $\bar{\T}$ and we can derive a contradiction from that as above.

 \begin{minipage}{\textwidth}
   \bigskip \textbf{Case 2.2.} There is a Woodin cardinal
   $\bar{\delta}$ in $\bar{K}$ such that
   $i_{\bar{b}}^{\bar{\T}}(\bar{\delta}) = \delta(\bar{\T})$. \bigskip
 \end{minipage}

 Let $\delta = \pi(\bar{\delta})$ be the corresponding Woodin cardinal
 in $(K^c)^{\Lpoddn(z)}$. All extenders appearing in the
 $K^c$-construction in the sense of \cite{MSch04} satisfy the axioms
 of the extender algebra by the definition of being \emph{certified}
 (see Definition $1.6$ in \cite{MSch04}), because certified extenders
 are countably complete witnessed by an embedding which is
 $\Sigma_{1+}$-elementary (see Definition $1.3$ in \cite{MSch04} for
 the definition of $\Sigma_{1+}$-formulae). Therefore the real $z$ is
 generic over the model $(K^c)^{\Lpoddn(z)}$ for the extender algebra
 at the Woodin cardinal $\delta$. In fact by the same argument the
 model $\Lpoddn(z)|\delta$ is generic over
 $(K^c)^{\Lpoddn(z)} | (\delta + \omega)$ for the $\delta$-version of
 the extender algebra $\mathbb{Q}_\delta$ (see the proof of Lemma
 $1.3$ in \cite{SchSt09} for a definition of the $\delta$-version of
 the extender algebra).

 Therefore we can perform a maximal $\mathcal{P}$-construction, which
 is defined as in \cite{SchSt09}, inside $\Lpoddn(z)$ over
 $(K^c)^{\Lpoddn(z)} | (\delta + \omega)$ to obtain a model
 $\mathcal{P}$. We have that
 $\mathcal{P} \vDash \text{``} \delta \text{ is Woodin''}$ by the
 definition of a maximal $\mathcal{P}$-construction.

 Let $\bar{\mathcal{P}}$ be the corresponding result of a maximal
 $\mathcal{P}$-construction in $H$ over
 $\bar{K} | (\bar{\delta} + \omega)$. Moreover let $\T^*$ be the
 iteration tree obtained by considering the tree $\bar{\T}$ based on
 $\bar{K} | \bar{\delta}$ as an iteration tree on
 $\bar{\mathcal{P}} \rhd \bar{K} | \bar{\delta}$. Let
 \[ i_{\bar{b}}^{\T^*} : \bar{\mathcal{P}} \rightarrow
 \M_{\bar{b}}^{\T^*} \]
 denote the corresponding iteration embedding, where the branch
 through $\T^*$ we consider is induced by the branch $\bar{b}$ through
 $\bar{\T}$ we fixed above, so we call them both $\bar{b}$. Then we
 have that
 \[ i_{\bar{b}}^{\T^*}(\bar{\delta}) =
 i_{\bar{b}}^{\bar{\T}}(\bar{\delta}) = \delta(\bar{\T}). \]

      \begin{minipage}{\textwidth}
        \bigskip \textbf{Case 2.2.1.} We have that \[ \mathcal{P} \cap \Ord
        < \Lpoddn(z) \cap \Ord. \]  \medskip
      \end{minipage}
      
      Then we have in particular that $\bar{\mathcal{P}} \cap \Ord < H
      \cap \Ord$ and thus $\bar{\mathcal{P}} \in H$. 

      Consider the coiteration of $\M_{\bar{b}}^{\T^*}$ with
      $\bar{\Q}$ inside $\Lpoddn(z)$. By the definition of a maximal
      $\mathcal{P}$-construction (see \cite{SchSt09}) we have that
      $\bar{\delta}$ is not definably Woodin over $\bar{\mathcal{P}}$
      since $\bar{\mathcal{P}} \cap \Ord < H \cap \Ord$. Since
      $\bar{b}$ is non-dropping this implies that
      $i_{\bar{b}}^{\T^*}(\bar{\delta})$ is not definably Woodin over
      $\M_{\bar{b}}^{\T^*}$.  Furthermore we have that
      \[ \M_{\bar{b}}^{\T^*} \vDash \text{``}
      i_{\bar{b}}^{\T^*}(\bar{\delta}) \text{ is Woodin''}. \]
      Concerning the other side of the coiteration we also have that
      $i_{\bar{b}}^{\T^*}(\bar{\delta}) = \delta(\bar{\T})$ is a
      Woodin cardinal in $\bar{\Q}$ but it is not definably Woodin
      over $\bar{\Q}$.

      Since the coiteration of $\M_{\bar{b}}^{\T^*}$ with $\bar{\Q}$
      takes place above
      $i_{\bar{b}}^{\T^*}(\bar{\delta}) = \delta(\bar{\T})$ we have
      that it is successful inside $\Lpoddn(z)$ using that
      $\M_{\bar{b}}^{\T^*}$ inherits the realization strategy for $H$
      above $i_{\bar{b}}^{\T^*}(\bar{\delta})$ and it follows that in
      fact
      \[ \M_{\bar{b}}^{\T^*} = \bar{\Q}. \]

      Consider the statement
      \begin{eqnarray*}
        \psi(\T^*, \bar{\Q}) \equiv &\text{``there is a cofinal
                                      branch } b \text{ through } \T^* \text{ such
                                      that } \\ 
                                    & \bar{\Q} =
                                      \M_{b}^{\T^*}\text{''.}   
      \end{eqnarray*}
      This statement $\psi(\T^*, \bar{\Q})$ is $\Sigma^1_1$-definable
      from the parameters $\T^*$ and $\bar{\Q}$ and holds in the model
      $\Lpoddn(z)$ as witnessed by the branch $\bar{b}$. We have that
      $\bar{\T}, \bar{\mathcal{P}} \in H$ and thus $\T^* \in H$.

    Therefore an absoluteness argument exactly as in Case $1$ yields
    that $\psi(\T^*, \bar{\Q})$ holds in $H^{\Col(\omega, \eta)}$,
    where $\eta < H \cap \Ord$ is an ordinal such that
    $\T^*, \bar{\Q} \in H^{\Col(\omega, \eta)}$ are countable inside
    the model $H^{\Col(\omega, \eta)}$. Thus it follows as before that
    $\bar{b} \in H$, which contradicts the fact that $\bar{\T}$
    witnesses in $H$ that $\bar{K}$ is not $\Q$-structure iterable.

      \begin{minipage}{\textwidth}
        \bigskip \textbf{Case 2.2.2.} We have that
        \[ \mathcal{P} \cap \Ord = \Lpoddn(z) \cap \Ord. \] \medskip
      \end{minipage}

      Then it follows that $\bar{\mathcal{P}} \cap \Ord = H \cap \Ord$
      and Lemma $1.5$ in \cite{SchSt09} applied to the maximal
      $\mathcal{P}$-construction inside $H$ yields that $\bar{\delta}$
      is not definably Woodin over $\bar{\mathcal{P}}$ since we have
      that $\rho_\omega(H) = \omega$. As $\bar{b}$ is non-dropping,
      this implies that $i_{\bar{b}}^{\T^*}(\bar{\delta})$ is not
      definably Woodin over $\M_{\bar{b}}^{\T^*}$.

      So as in Case $2.2.1$ we can successfully coiterate
      $\M_{\bar{b}}^{\T^*}$ and $\bar{\Q}$ inside $\Lpoddn(z)$ and
      obtain again that \[ \M_{\bar{b}}^{\T^*} = \bar{\Q}.\]
      
      Then it follows that
      \[ \bar{\Q} \cap \Ord < H \cap \Ord = \bar{\mathcal{P}} \cap
      \Ord \leq \M_{\bar{b}}^{\T^*} \cap \Ord, \]
      where the first inequality holds true since $\bar{\Q} \in H$.
      This is a contradiction to the fact that
      $\M_{\bar{b}}^{\T^*} = \bar{\Q}$.

      This finishes the proof of Subclaim \ref{subcl:QstritA}.
    \end{proof}

    Working in $\Lpoddn(z)$ we consider the coiteration of
    $K^{M_x}|\omega_2^{\Lpoddn(z)}$ with
    $(K^c|\omega_2)^{\Lpoddn(z)}$. From what we proved so far it
    easily follows that this coiteration is successful as shown in the
    next subclaim.

    \begin{subclaim}\label{subcl:coitsucc}
      The coiteration of $K^{M_x}|\omega_2^{\Lpoddn(z)}$ with
      $(K^c|\omega_2)^{\Lpoddn(z)}$ in $\Lpoddn(z)$ is successful.
    \end{subclaim}

    \begin{proof}
      Write $W = \Lpoddn(z)$. Then we can successfully coiterate the
      premice $K^{M_x}|\omega_2^W$ and $K^c|\omega_2^W$ inside the
      model $W$ since
      $K^{M_x}|\omega_2^W = K^{M_x} | (\gamma^+)^{M_x}$ is fully
      iterable in $W$ by Subclaim \ref{subcl:KMxit} and
      $K^c|\omega_2^W$ is $\Q$-structure iterable in $W$ by Subclaim
      \ref{subcl:QstritA}. In particular the $K^{M_x}|\omega_2^W$-side
      of the coiteration provides $\Q$-structures for the
      $K^c|\omega_2^W$-side and therefore the coiteration is
      successful.
    \end{proof}

    In what follows we want to argue that the
    $(K^c|\omega_2)^{\Lpoddn(z)}$-side cannot lose this
    coiteration. For that we want to use the following lemma, which we
    can prove similar to Theorem $3.8$ in \cite{MSch04}. As the
    version we aim to use is slightly stronger than what is shown in
    \cite{MSch04}, we will sketch a proof of this lemma.

    \begin{lemma}\label{lem:Kcuniv}
      Let $\kappa \geq \omega_2$ be a regular cardinal such that
      $\kappa$ is inaccessible in $K^c$ (constructed in the sense of
      \cite{MSch04}). Then $K^c|\kappa$ is universal with respect to
      every premouse $M$ with $M \cap \Ord = \kappa$ to which it can
      be successfully coiterated.
    \end{lemma}

    In \cite{MSch04} the universality of the premouse $K^c|\kappa$ is
    only proved with respect to smaller premice, i.e. premice $M$ such
    that $M \cap \Ord < \kappa$, to which $K^c|\kappa$ can be
    successfully coiterated. As shown below, their argument can easily
    be modified to yield Lemma \ref{lem:Kcuniv}.

    \begin{proof}[Proof of Lemma \ref{lem:Kcuniv}]
      Let $\kappa \geq \omega_2$ be a regular cardinal such that
      $\kappa$ is inaccessible in $K^c$ and write $N = K^c|\kappa$.

      Analogous to the notation in \cite{MSch04} we say that \emph{$M$
        iterates past $N$} iff $M$ is a premouse with
      $M \cap \Ord = \kappa$ and there are iteration trees $\T$ on $M$
      and $\U$ on $N$ of length $\lambda+1$ arising from a successful
      comparsion such that there is no drop along $[0,\lambda]_U$ and
      $\M_\lambda^\U \lhd \M_\lambda^\T$ or
      $\M_\lambda^\U = \M_\lambda^\T$ and we in addition have that
      there is a drop along $[0, \lambda]_T$.

      Assume toward a contradiction that there is a premouse $M$ which
      iterates past $N$ and let $\T$ and $\U$ be iteration trees of
      length $\lambda+1$ on $M$ and $N$ respectively witnessing
      this. Let $i_{0\lambda}^\U : N \rightarrow \M_\lambda^\U$ denote
      the corresponding iteration embedding on the $N$-side, which
      exists as there is no drop on the main branch through $\U$. 

      We distinguish the following cases. 

    \begin{minipage}{\textwidth}
      \bigskip \textbf{Case 1.} We have that for some $\xi < \kappa$,
      \[i_{0\lambda}^\U \pwimg \kappa \subset \kappa \text{ and }
      i_{\beta\lambda}^\T(\xi) \geq \kappa \]
      for some ordinal $\beta < \lambda$ such that the iteration
      embedding $i_{\beta\lambda}^\T$ is defined.  \bigskip
    \end{minipage}      

    In this case we can derive a contradiction as in Section $3$ in
    \cite{MSch04} because we can prove that the consequences of Lemma
    $3.5$ in \cite{MSch04} also hold in this setting. 

    We assume that the reader is familiar with the argument for Lemma
    $3.5$ in \cite{MSch04} and we use the notation from there. So we
    let
    \[ X \prec H_{\kappa^+} \]
    be such that $\card{X} < \kappa$, $X \cap \kappa \in \kappa$,
    $\{M,N, \T, \U, \beta, \xi \} \subset X$, and
    $X \cap \kappa \in (0,\kappa)_T \cap (0, \kappa)_U$, as in this
    case $\lambda = \kappa$. We write $\alpha = X \cap \kappa$. Let
    \[ \pi : \bar{H} \cong X \prec H_{\kappa^+} \]
    be such that $\bar{H}$ is transitive. Then we have that
    $\alpha = \crit(\pi)$ and $\pi(\alpha) = \kappa$. Let
    $\bar{\T}, \bar{\U} \in \bar{H}$ be such that
    $\pi(\bar{\T}, \bar{\U}) = (\T, \U)$.

    As in the proof of Lemma $3.5$ in \cite{MSch04} we aim to show
    that $\Pot(\alpha) \cap N \subset \bar{H}$. 

    Exactly as in \cite{MSch04} we get that \[ \M_\alpha^\T ||
    (\alpha^+)^{\M_\alpha^\T} = \M_\alpha^\U ||
    (\alpha^+)^{\M_\alpha^\U}. \] 

    The case assumption that $i_{\beta\kappa}^\T(\xi) \geq \kappa$ for
    ordinals $\beta, \xi < \kappa$ implies that there are ordinals
    $\bar{\beta}, \bar{\xi} < \alpha$ such that
    $i_{\bar{\beta}\alpha}^{\bar{\T}}(\bar{\xi}) \geq \alpha$. As
    $i_{\bar{\beta}\alpha}^{\bar{\T}} \upharpoonright \alpha =
    i_{\bar{\beta}\alpha}^{\T} \upharpoonright \alpha$
    and $\M_\alpha^{\bar{\T}} | \alpha = \M_\alpha^{\T} | \alpha$,
    this yields that
    $\Pot(\alpha) \cap \M_\alpha^\T = \Pot(\alpha) \cap
    \M_\alpha^{\bar{\T}}$. Therefore we have that
    \[ \Pot(\alpha) \cap \M_\alpha^\U = \Pot(\alpha) \cap \M_\alpha^\T
    = \Pot(\alpha) \cap \M_\alpha^{\bar{\T}} \in \bar{H}. \]

    Moreover we are assuming that
    $i_{0\kappa}^\U \pwimg \kappa \subset \kappa$ and we have
    \[ i_{0\alpha}^\U \upharpoonright \alpha = i_{0\alpha}^{\bar{\U}}
    \upharpoonright \alpha \in \bar{H}. \]

    Therefore we can again argue exactly as in the proof of Lemma
    $3.5$ in \cite{MSch04} to get that
    $\Pot(\alpha) \cap N \subset \bar{H}$. Following \cite{MSch04}
    this now yields a contradiction to the assumption that $M$
    iterates past $N$.

    \begin{minipage}{\textwidth}
      \bigskip \textbf{Case 2.} We have that
      \[i_{0\lambda}^\U \pwimg \kappa \subset \kappa \text{ and there
        are no } \xi, \beta \text { such that }
      i_{\beta\lambda}^\T(\xi) \geq \kappa, \]
      where $\xi < \kappa$ and $\beta < \lambda$. \bigskip
    \end{minipage}  

    By assumption we have that $\kappa$ is inaccessible in $N$. Assume
    first that $M$ has a largest cardinal $\eta$. In this case we have
    that there are no cardinals between the image of $\eta$ under the
    iteration embedding and $\kappa$ in $\M_\lambda^\T$ but there are
    cardinals between the image of $\eta$ under the iteration
    embedding and $\kappa$ in $\M_\lambda^\U$. This contradicts the
    fact that $\T$ and $\U$ were obtained by a successful comparison
    of $M$ and $N$ with $\M_\lambda^\U \unlhd \M_\lambda^\T$.

    Now assume that $M \vDash \ZFC$. Then the case assumption implies
    that in particular $\M_\lambda^\T \cap \Ord \leq \kappa$. As we
    assumed that $M$ iterates past $N$, this implies that
    $\M_\lambda^\T = \M_\lambda^\U$ and that there is a drop along
    $[0,\lambda]_T$. But then $\M_\lambda^\T$ is not $\omega$-sound,
    contradicting the soundness of $\M_\lambda^\U$.

    \begin{minipage}{\textwidth}
      \bigskip \textbf{Case 3.} We have that for some $\zeta < \kappa$,
      \[i_{0\lambda}^\U(\zeta) \geq \kappa \text{ and there are no } \xi,
      \beta \text { such that } i_{\beta\lambda}^\T(\xi) \geq \kappa, \]
      where $\xi < \kappa$ and $\beta < \lambda$. \bigskip
    \end{minipage}     

    This again easily contradicts the assumption that $M$ iterates
    past $N$.

    \begin{minipage}{\textwidth}
      \bigskip \textbf{Case 4.} We have that for some $\zeta < \kappa$
      and for some $\xi < \kappa$,
      \[i_{0\lambda}^\U(\zeta) \geq \kappa \text{ and }
      i_{\beta\lambda}^\T(\xi) \geq \kappa \]
      for some ordinal $\beta < \lambda$ such that the iteration
      embedding $i_{\beta\lambda}^\T$ is defined.  \bigskip
    \end{minipage}    

    In this case we have that $\lambda = \kappa$. For a reflection
    argument as in the proof of the Comparison Lemma (see Theorem
    $3.11$ in \cite{St10}) let
    \[ X \prec H_{\eta} \] for some large enough ordinal $\eta$ be
    such that $\card{X} < \kappa$, $\{M,N, \T, \U \} \subset X$, and
    $\zeta \cup \xi \cup \beta \subset X$. Following the notation in
    the proof of Theorem $3.11$ in \cite{St10} let
    \[ \pi : H \cong X \prec H_{\eta} \] be such that $H$ is
    transitive.

    Then an argument as in the proof of Theorem $3.11$ in \cite{St10}
    yields that there is an ordinal $\gamma < \kappa$ such that the
    embeddings $i_{\gamma\kappa}^\T$ and $i_{\gamma\kappa}^\U$
    agree. This implies that there are extenders $E_{\alpha}^\T$ used
    in $\T$ at stage $\alpha$ and $E_{\alpha^\prime}^\U$ used in $\U$
    at stage $\alpha^\prime$ which are compatible. Again as in the
    proof of Theorem $3.11$ in \cite{St10} this yields a
    contradiction.

    This finishes the proof of Lemma \ref{lem:Kcuniv}.
    \end{proof}

    Now we can use Lemma \ref{lem:Kcuniv} to prove the following
    subclaim. 

\setcounter{subcl}{6}
    \begin{subclaim}\label{subcl:o2succ}
      $\omega_2^{\Lpoddn(z)}$ is a successor cardinal in
      $(K^c)^{\Lpoddn(z)}$. 
    \end{subclaim}
    \begin{proof}
      Work in $W = \Lpoddn(z)$ and assume that this does not
      hold. That means we are assuming that $\omega_2^W$ is
      inaccessible in $(K^c)^W$.

      As above consider the successful coiteration of
      $K^{M_x}|\omega_2^{W}$ with $(K^c)^W|\omega_2^{W}$ and let $\T$
      and $\U$ be the resulting trees on $K^{M_x}|\omega_2^W$ and
      $(K^c)^W|\omega_2^W$ respectively of length $\lambda + 1$ for
      some ordinal $\lambda$. Since we assume that $\omega_2^W$ is
      inaccessible in $(K^c)^W$, it follows by Lemma \ref{lem:Kcuniv}
      that $(K^c)^W | \omega_2^W$ is universal in $W$ for the
      coiteration with premice of height $\leq \omega_2^W$. Therefore
      the $(K^c)^W | \omega_2^W$-side has to win the comparison.  That
      means we have that $\M_\lambda^\T \unlhd \M_\lambda^\U$ and
      there is no drop on the main branch through $\M_\lambda^\T$. In
      particular the iteration embedding
      \[ i_{0\lambda}^\T : K^{M_x}|\omega_2^W \rightarrow
      \M_\lambda^\T \] exists. Now we distinguish the following cases.

    \begin{minipage}{\textwidth}
      \bigskip \textbf{Case 1.} We have that
      \[i_{0\lambda}^\T \pwimg \omega_2^W \subset \omega_2^W.\] \medskip
    \end{minipage}      
    This means in particular that
    $(\gamma^+)^{K^{M_x}} = (\gamma^+)^{M_x} = \omega_2^W$ stays a
    successor cardinal in the model $\M_\lambda^\T$. So say that we
    have $(\eta^+)^{\M_\lambda^\T} = \omega_2^W$ for some cardinal
    $\eta < \omega_2^W$ in $\M_\lambda^\T$. In particular this means
    that there are no cardinals between $\eta$ and $\omega_2^W$ in
    $\M_\lambda^\T$. But by assumption we have that $\omega_2^W$ is
    inaccessible in $(K^c)^W$ and thus also in $\M_\lambda^\U$. In
    particular there are cardinals between $\eta$ and $\omega_2^W$ in
    $\M_\lambda^\U$. This contradicts the fact that $\M_\lambda^\T$
    and $\M_\lambda^\U$ were obtained by a successful comparison with
    $\M_\lambda^\T \unlhd \M_\lambda^\U$.

    \begin{minipage}{\textwidth}
      \bigskip \textbf{Case 2.} We have that
      \[i_{0\lambda}^\T \pwimg \omega_2^W \not\subset \omega_2^W.\]
      \medskip
    \end{minipage}     

        In this case we distinguish two subcases as follows.

    \begin{minipage}{\textwidth}
      \bigskip \textbf{Case 2.1.} We have that
      \[ \sup i_{0\lambda}^\T \pwimg \gamma < \omega_2^W. \] \medskip
    \end{minipage}  
    In this case we have that
    \[ (i_{0\lambda}^\T(\gamma)^+)^{\M_\lambda^\T} = \omega_2^W, \]
    because we are also assuming that
    $i_{0\lambda}^\T \pwimg \omega_2^W \not\subset \omega_2^W$.

    So in particular we again have that $\omega_2^W$ is a successor
    cardinal in the model $\M_\lambda^\T$ and so there are no
    cardinals between $i_{0\lambda}^\T(\gamma) < \omega_2^W$ and
    $\omega_2^W$ in $\M_\lambda^\T$. From this we can derive the same
    contradiction as in Case $1$, because $\omega_2^W$ is inaccessible
    in $(K^c)^W$.

    \begin{minipage}{\textwidth}
      \bigskip \textbf{Case 2.2.}  We have that
      \[ \exists \eta < \omega_2^W \text{ such that }
      i_{0\lambda}^\T(\eta) \geq \omega_2^W. \] \medskip
    \end{minipage}  
    Let $\alpha < \lambda$ be the least ordinal such that the
    iteration embedding $i_{\alpha \lambda}^\U$ is defined. That means
    the last drop on the main branch in $\U$ is at stage
    $\alpha$. Since $K^{M_x}|\omega_2^W$ has height $\omega_2^W$ we
    have by universality of $(K^c)^W | \omega_2^W$ in $W$ (see Lemma
    \ref{lem:Kcuniv}) that the case assumption implies that there
    exists an ordinal $\nu < \omega_2^W$ such that
    \[ i_{\alpha \lambda}^\U(\nu) \geq \omega_2^W. \]
    Moreover in this case we have in fact $\lambda = \omega_2^W$.

    Let $X \prec W | \theta$ for some large ordinal $\theta$ be such
    that we have $\card{X} < \omega_2^W$,
    $\{K^{M_x}|\omega_2^W, (K^c)^W | \omega_2^W, \T, \U \} \subset X$,
    and $\eta \cup \nu \cup \alpha \subset X$. Moreover let $H$ be the
    Mostowski collapse of $X$ and let $\pi: H \rightarrow W | \theta$
    be the uncollapse map. Then a reflection argument as in the proof
    of the Comparison Lemma (see Theorem $3.11$ in \cite{St10}) yields
    that there is an ordinal $\xi < \lambda$ such that the embeddings
    $i_{\xi\lambda}^\T$ and $i_{\xi\lambda}^\U$ agree. This implies
    that there are extenders $E_{\beta}^\T$ used in $\T$ at stage
    $\beta$ and $E_{\beta^\prime}^\U$ used in $\U$ at stage
    $\beta^\prime$ which are compatible. Again as in the proof of
    Theorem $3.11$ in \cite{St10} this yields a contradiction.

      This finishes the proof of Subclaim \ref{subcl:o2succ}.
    \end{proof}

    Recall that we have
    \[ \HOD^{\Lpoddn(z)} \vDash \text{``}\omega_2^{\Lpoddn(z)} \text{
      is inaccessible''} \]
    by Theorem \ref{thm:varw2inacc} as $z \geq_T x$. Since we have
    that \[ (K^c)^{\Lpoddn(z)} \subset \HOD^{\Lpoddn(z)} \]
    by Subclaim \ref{subcl:KcHOD}, this contradicts Subclaim
    \ref{subcl:o2succ} and thereby finishes the proof of Claim
    \ref{cl:M2nexinMx}.
  \end{proof}

  Work in $V$ now and let $x \in \BS$ be arbitrary in the cone of
  reals from Theorem \ref{thm:varw2inacc}. Then by Claim
  \ref{cl:M2nexinMx} we have that
  \[ M_x \vDash \text{``}(M_{2n}^\#)^{M_x} \text{ is } \omega_1
  \text{-iterable''}. \] Hence
  \[ M_x \vDash \text{``}(M_{2n}^\#)^{M_x} \text{ is } \Pi^1_{2n+2}
  \text{-iterable''}. \]
  Since $M_x$ is $\SIGMA^1_{2n+2}$-correct in $V$ we have that
  \[ V \vDash \text{``}(M_{2n}^\#)^{M_x} \text{ is } \Pi^1_{2n+2}
  \text{-iterable''}. \]
  By $\SIGMA^1_{2n+2}$-correctness in $V$ again we have for every real
  $y \geq_T x$ such that in particular $(M_{2n}^\#)^{M_x} \in M_y$
  that
  \[ M_y \vDash \text{``}(M_{2n}^\#)^{M_x} \text{ is } \Pi^1_{2n+2}
  \text{-iterable''}. \]
  Consider the comparison of the premice $(M_{2n}^\#)^{M_x}$ and
  $(M_{2n}^\#)^{M_y}$ inside the model $M_y$. This comparison is
  successful by Lemma \ref{lem:steel2.2} for all reals $y \geq_T x$ as
  above, since $(M_{2n}^\#)^{M_y}$ is $\omega_1$-iterable in $M_y$ and
  $(M_{2n}^\#)^{M_x}$ is $\Pi^1_{2n+2}$-iterable in $M_y$. Moreover
  both premice are $\omega$-sound and we have that
  $\rho_\omega((M_{2n}^\#)^{M_x}) = \rho_\omega((M_{2n}^\#)^{M_y}) =
  \omega$.
  Thus the premice $(M_{2n}^\#)^{M_x}$ and $(M_{2n}^\#)^{M_y}$ are in
  fact equal.

  Therefore we have that all premice $(M_{2n}^\#)^{M_x}$ for
  $x \in \BS$ in the cone of reals from Theorem \ref{thm:varw2inacc}
  are equal in $V$. Call this unique premouse $M_{2n}^\#$.

  We now finally show that this premouse $M_{2n}^\#$ is
  $\omega_1$-iterable in $V$ via the $\Q$-structure iteration strategy
  (see Definition \ref{def:Qstritstr}). So let $\T$ be an iteration
  tree on $M_{2n}^\#$ in $V$ of limit length $< \omega_1^V$ according
  to the $\Q$-structure iteration strategy. Pick $z \in \BS$ such that
  $M_{2n}^\#$ and $\T$ are in $M_z$ and $\lh(\T) < \omega_1^{M_z}$.

  Since $M_z$ is $\SIGMA^1_{2n+2}$-correct in $V$ we have that $\T$ is
  according to the $\Q$-structure iteration strategy in $M_z$, because
  $\Q(\T \upharpoonright \lambda)$ is $2n$-small above
  $\delta(\T \upharpoonright \lambda)$ for all limit ordinals
  $\lambda < \lh(\T)$ and therefore $\Pi^1_{2n+1}$-iterability
  above $\delta(\T \upharpoonright \lambda)$ is enough for
  $\Q(\T \upharpoonright \lambda)$ to determine a unique cofinal
  well-founded branch $b$ through $\T$ by Lemma
  \ref{lem:nofakeeven}. Moreover we have that
  $(M_{2n}^\#)^{M_z} = M_{2n}^\#$ and therefore
  \[ M_z \vDash \text{``}M_{2n}^\# \text{ is } \omega_1
  \text{-iterable''.} \]
  So in $M_z$ there exists a cofinal well-founded branch $b$ through
  $\T$, which is determined by $\Q$-structures
  $\Q(\T \upharpoonright \lambda)$ which are $\omega_1$-iterable above
  $\delta(\T \upharpoonright \lambda)$ and therefore also
  $\Pi^1_{2n+1}$-iterable above $\delta(\T \upharpoonright \lambda)$
  in $M_z$ for all limit ordinals $\lambda \leq \lh(\T)$. That means
  we in particular have that $\Q(b,\T) = \Q(\T)$. Since $M_z$ is
  $\SIGMA^1_{2n+2}$-correct in $V$, it follows as above that $b$ is
  also the unique cofinal well-founded branch in $V$ which is
  determined by the same $\Q$-structures as in $M_z$. Therefore
  \[ V \vDash \text{``}M_{2n}^\# \text{ exists and is } \omega_1
  \text{-iterable''} \] and we finished the proof of Theorem
  \ref{thm:newWoodineven}. 
\end{proof}


\section{Conclusion}
\label{ch:conclusion}

By the results proved in Sections \ref{sec:odd} and \ref{sec:even} we
have that the following theorem which is due to Itay Neeman and the
third author and was announced in Section \ref{sec:outline} holds
true.

\begin{namedthm}{Theorem \ref{cor:newWoodin1}}
  Let $n \geq 1$ and assume there is no $\SIGMA^1_{n+2}$-definable
  $\omega_1$-sequence of pairwise distinct reals. Then the following
  are equivalent.
\begin{enumerate}[(1)]
\item $\PI^1_n$ determinacy and $\Pi^1_{n+1}$ determinacy,
\item for all $x \in \BS$, $M_{n-1}^\#(x)$ exists and is
  $\omega_1$-iterable, and $M_n^\#$ exists and is $\omega_1$-iterable,
\item $M_n^\#$ exists and is $\omega_1$-iterable.
\end{enumerate}
\end{namedthm}

\begin{proof}
  This follows from Theorems \ref{thm:newWoodinodd} and
  \ref{thm:newWoodineven} together with Theorem $2.14$ in \cite{Ne02}.
\end{proof}

Moreover Theorems \ref{thm:newWoodinodd} and \ref{thm:newWoodineven}
together with Lemma \ref{lem:bfDetSeq} immediately imply the following
main theorem due to the third author.

\begin{namedthm}{Theorem \ref{cor:bfDetSharps}}
  Let $n \geq 1$ and assume $\PI^1_{n+1}$ determinacy holds. Then $M_n^\#(x)$
  exists and is $\omega_1$-iterable for all $x \in \BS$.
\end{namedthm}

\subsection{Applications}
\label{sec:Applications}

From these results we can now obtain a boldface version of the
Determinacy Transfer Theorem as in Theorem \ref{thm:transferthm} for
all projective levels $\PI^1_{n+1}$ of determinacy. The lightface
version of the Determinacy Transfer Theorem for even levels of
projective determinacy $\Pi^1_{2n}$ (see Theorem
\ref{thm:transferthm}) is used in the proof of Theorem
\ref{thm:newWoodinodd}, and therefore in the proof of Theorem
\ref{cor:bfDetSharps} for odd levels $n$.

\begin{cor}[Determinacy Transfer Theorem]
\label{cor:transferThm}
  Let $n \geq 1$. Then $\PI^1_{n+1}$ determinacy is equivalent to
  \emph{$\game^{(n)}(<\omega^2 - \PI^1_1)$} determinacy.
\end{cor}

\begin{proof}
  By Theorem $1.10$ in \cite{KW08} we have for the even levels that
  \[ \Det(\PI^1_{2n}) \leftrightarrow \Det(\boldsymbol{\Delta}^1_{2n})
  \leftrightarrow \Det(\game^{(2n-1)}(<\omega^2 - \PI^1_1)). \]
  Here the first equivalence is due to Martin (see \cite{Ma73}) and
  proven in Theorem $5.1$ in \cite{KS85}.  The second equivalence due
  to Kechris and the third author can be proven using purely
  descriptive set theoretic methods (see Theorem $1.10$ in
  \cite{KW08}).

  The results in this paper (using this version of the Determinacy
  Transfer Theorem for even levels, see Theorem \ref{thm:transferthm})
  together with results due to Itay Neeman yield the Determinacy
  Transfer Theorem for all levels $n$ as follows.

  By basic facts about the game quantifier ``$\game$'' we have that
  \[\Det(\game^{(n)}(<\omega^2 - \PI^1_1))\]
  implies $\PI^1_{n+1}$ determinacy.

  For the other direction assume that $\PI^1_{n+1}$ determinacy
  holds. Then Theorem \ref{cor:bfDetSharps} yields that the premouse
  $M_n^\#(x)$ exists and is $\omega_1$-iterable for all $x \in
  \BS$.
  This implies that \[ \Det(\game^{(n)}(<\omega^2 - \PI^1_1)) \]
  holds true by Theorem $2.5$ in \cite{Ne95}.
\end{proof}

\subsection{Open problems}
\label{sec:openproblemsproj}

We close this paper with the following open problem, which is the
lightface version of Theorem \ref{cor:bfDetSharps}.

\begin{namedthm}{Conjecture}
  Let $n > 1$ and assume that $\PI^1_{n}$ determinacy and
  $\Pi^1_{n+1}$ determinacy hold. Then $M_n^\#$ exists and is
  $\omega_1$-iterable.
\end{namedthm}

This conjecture holds true for $n=0$ which is due to L. Harrington
(see \cite{Ha78}) and for $n=1$ which is due to the third author (see
Corollary $4.17$ in \cite{StW16}), but it is open for $n>1$.


\bibliographystyle{abstract}
\bibliography{DetSharps} 
\nocite{*}

\end{document}